\documentclass[reqno]{amsart}

\usepackage{amsmath,amsfonts,amssymb,amscd,amsthm,amsbsy,bm,latexsym}
\usepackage{gensymb}
\usepackage[pdftex,bookmarksnumbered=true,bookmarksopen=true,          
            colorlinks=true,pdfborder=001,citecolor=blue,
            linkcolor=red,anchorcolor=green,urlcolor=blue]{hyperref}

\allowdisplaybreaks

\usepackage{color}

\newtheorem{lemma}{Lemma}[section]
\newtheorem{theorem}{Theorem}[section]

\newtheorem{proposition}{Proposition}[section]
\newtheorem{remark}{Remark}[section]
\newtheorem{corollary}{Corollary}[section]
\numberwithin{equation}{section}

\begin{document}
\bibliography{bibfile}

\title[Global-in-time VPB limit of the VMB system]{On the Vlasov-Poisson-Boltzmann Limit of the Vlasov-Maxwell-Boltzmann system}

\author[N. Jiang]{Ning Jiang}
\address[JN]{School of Mathematics and Statistics and Hubei Key Laboratory of Computational Science, Wuhan University, Wuhan 430072, P.R.~China}
\email{njiang@whu.edu.cn}

\author[Y.-J. Lei]{Yuanjie Lei}
\address[YJL]{School of Mathematics and Statistics, Huazhong University of Science and Technology, Wuhan 430074, P.R.~China}
\email{leiyuanjie@hust.edu.cn}

\author[H.-J. Zhao]{Huijiang Zhao}
\address[HJZ]{School of Mathematics and Statistics and Hubei Key Laboratory of Computational Science, Wuhan University, Wuhan 430072, China}
\email{hhjjzhao@whu.edu.cn}


\begin{abstract}
For the whole range of cutoff intermolecular interactions, we give a rigorous mathematical justification of the limit from the Vlasov-Maxwell-Boltzmann system to the Vlasov-Poisson-Boltzmann system as the light speed tends to infinity. Such a limit is shown to hold global-in-time in the perturbative framework. The key point in our analysis is to deduce certain {\it a priori} estimates which are independent of the light speed and the main difficulty is due to the degeneracy of the dissipative effect of the electromagnetic field for large light speed.
\end{abstract}

\maketitle
\thispagestyle{empty}
\tableofcontents

%



\section{Introduction}

\subsection{The problem}

\setcounter{equation}{0}
The motion of dilute ionized plasmas consisting of two-species particles (e.g., electrons and ions) under the influence of binary collisions and the self-consistent electromagnetic field can be modelled by the Vlasov-Maxwell-Boltzmann system (cf. \cite[Chapter 19]{Chapman-Cowling-1970} as well as \cite[Chapter 6.6]{Krall-Trivelpiece-1973})
\begin{eqnarray}
 \partial_tF^c_++ v  \cdot\nabla_xF^c_++\frac{e_+}{m_+}\left(E^c+\frac vc\times B^c\right)\cdot\nabla_{ v  }F^c_+&=&Q\left(F^c_+,F^c_+\right)+Q\left(F^c_+,F^c_-\right),\nonumber\\
 \partial_tF^c_-+ v  \cdot\nabla_xF^c_--\frac{e_-}{m_-}\left(E^c+\frac vc\times B^c\right)\cdot\nabla_{ v  }F^c_-&=&Q\left(F^c_-,F^c_+\right)+Q\left(F^c_-,F^c_-\right).\label{VMB}
\end{eqnarray}
Here the electromagnetic field $\left[E^c, B^c\right]=\left[E^c(t,x), B^c(t,x)\right]$ satisfies the  Maxwell equations
\begin{eqnarray}
 \frac1c\partial_tE^c-\nabla_x\times B^c&=&-\frac{4\pi}{c}{\displaystyle\int_{\mathbb{R}^3}}v\left(e_+F^c_+-e_-F^c_-\right)dv,\nonumber\\
 \frac{1}{c}\partial_tB^c+\nabla_x\times E^c&=&0,\label{Maxwell}\\
 \nabla_x\cdot E^c&=&4\pi{\displaystyle\int_{\mathbb{R}^3}}\left(e_+F^c_+-e_-F^c_-\right)dv,\nonumber\\
 \nabla_x\cdot B^c&=&0,\nonumber
\end{eqnarray}
where $\nabla_x=\left[\partial_{x_1}, \partial_{x_2},\partial_{x_3}\right], \nabla_v=\left[\partial_{v_1}, \partial_{v_2},\partial_{v_3}\right]$.  The unknown functions $F^c_\pm= F^c_\pm(t,x, v) \geq  0$ are the number density functions for the ions ($+$) and electrons ($-$) with position $x = (x_1, x_2, x_3)\in {\mathbb{R}}^3$ and velocity $ v=( v_1,  v_2,  v_3) \in {\mathbb{R}}^3$ at time $t\geq 0$, respectively, $e_\pm$ and $m_\pm$ the magnitudes of their charges and masses, and $c$ the light speed.

Let $F(v)$, $G(v)$ be two number density functions for two types of particles with masses $m_\pm$ and diameters $\sigma_\pm$, then the collision operator $Q(F,G)(v)$ for cutoff intermolecular interactions is defined as (cf. \cite{Chapman-Cowling-1970})
\begin{eqnarray*}
  Q(F,G)(v)
  &=&\frac{(\sigma_++\sigma_-)^2}{4}\int_{\mathbb{R}^3\times \mathbb{S}^2}|u-v|^{\gamma}{\bf b}\left(\frac{\omega\cdot(v-u)}{|u-v|}\right)\left\{F(v')G(u')-F(v)G(u)\right\}
  d\omega du\\
&\equiv&Q_{gain}(F,G)-Q_{loss}(F,G).
\end{eqnarray*}
Here $\omega\in\mathbb{S}^2$ and ${\bf b}$, the angular part of the collision kernel, satisfies Grad's cutoff assumption (cf.~\cite{Grad-1963})
\begin{equation}\label{cutoff-assump}
0\leq{\bf b}(\cos \theta)\leq C|\cos\theta|
\end{equation}
for some positive constant $C>0$. The deviation angle $\pi-2\theta$ satisfies $\cos\theta =\omega\cdot(v-u)/{|v-u|}$.
Moreover, for $m_{1}, m_2\in \{m_+,m_-\}$,
\begin{equation*}
v'=v-\frac{2m_2}{m_1+m_2}[(v-u)\cdot\omega]\omega,\quad u'=u+\frac{2m_1}{m_1+m_2}[(v-u)\cdot\omega]\omega,
\end{equation*}
which denote velocities $(v',u')$ after a collision of particles having velocities $(v, u)$ before the collision and vice versa. Notice that the above identities follow from the conservation of momentum $m_1v+m_2u$ and energy $\frac 12
m_1|v|^2+\frac 12m_2|u|^2$.

The exponent $\gamma\in(-3,1]$ in the kinetic part of the collision kernel is determined by the potential of intermolecular force, which is classified into the soft potential case for $-3<\gamma<0$, the Maxwell molecular case for $\gamma=0$, and the hard potential case for $0<\gamma\leq 1$ which includes the hard sphere model with $\gamma=1$ and ${\bf b}(\cos\theta)=C|\cos\theta|$ for some positive constant $C>0$. For the soft potentials, the case $-2\leq \gamma<0$ is called the moderately soft potentials while $-3<\gamma<-2$ is called the very soft potentials, cf. \cite{Villani-2002}.

For the Vlasov-Maxwell-Boltzmann system \eqref{VMB}, \eqref{Maxwell}, if one let the light speed $c\to+\infty$, one can then formally deduce that the solution $\left[F^c_+(t,x,v), F^c_-(t,x,v), E^c(t,x), B^c(t,x)\right]$ of the Vlasov-Maxwell-Boltzmann system \eqref{VMB}, \eqref{Maxwell} converges to $\left[F^{\infty}_+(t,x,v), F^{\infty}_-(t,x,v), E^{\infty}(t,x), B^{\infty}(t,x)\right]$ with $E^{\infty}(t,x)$ $\equiv \nabla_x \phi^{\infty}(t,x)$ for some scalar function $\phi^{\infty}(t,x)$, and $\left[F^{\infty}_+(t,x,v), F^{\infty}_-(t,x,v), \phi^{\infty}(t,x)\right]$ solves the following Vlasov-Poisson-Boltzmann system
\begin{eqnarray}\label{VPB-Original}
\partial_tF^{\infty}_++ v  \cdot\nabla_xF^{\infty}_++\frac{e_+}{m_+}\nabla_x\phi^{\infty}\cdot\nabla_{ v  }F^{\infty}_+&=&Q\left(F^{\infty}_+,F^{\infty}_+\right)+Q\left(F^{\infty}_+,F^{\infty}_-\right),\nonumber\\
\partial_tF^{\infty}_-+ v  \cdot\nabla_xF^{\infty}_--\frac{e_-}{m_-}\nabla_x\phi^{\infty}\cdot\nabla_{ v  }F^{\infty}_-&=&Q\left(F^{\infty}_-,F^{\infty}_+\right)+Q\left(F^{\infty}_-,F^{\infty}_-\right),\\
\nabla_x\phi^{\infty}&=&4\pi {\displaystyle\int_{\mathbb{R}^3}}\left(e_+F^{\infty}_+-e_-F^{\infty}_-\right)dv.\nonumber
\end{eqnarray}
It is easy to see that \eqref{VPB-Original}$_3$ is equivalent to the fact that $E^{\infty}(t,x)$ solves
\begin{eqnarray}\label{Poisson-Rewritten}
\partial_tE^{\infty}&=&-\int_{\mathbb{R}^3_v}v\left(e_+F^{\infty}_+-e_-F^{\infty}_-\right)dv,\nonumber\\
  \nabla_x\times E^{\infty}&=&0,\\
  \nabla_x\cdot E^{\infty}&=&4\pi\int_{\mathbb{R}^3_v}\left(e_+F^{\infty}_+-e_-F^{\infty}_-\right)dv.\nonumber
\end{eqnarray}

If we define
\begin{eqnarray}
\mu_+(v)&=&\frac{n_0}{e_+}\left(\frac{m_+}{2\pi\kappa_B T_0}\right)^{\frac 32}\exp\left(-\frac{m_+|v|^2}{2\kappa_B T_0}\right),\nonumber\\
\mu_-(v)&=&\frac{n_0}{e_-}\left(\frac{m_-}{2\pi\kappa_B T_0}\right)^{\frac 32}\exp\left(-\frac{m_-|v|^2}{2\kappa_B T_0}\right),\label{strong-magnetic}\\
\mathfrak{E}(x)&=&\left[\mathfrak{E}_1(x), \mathfrak{E}_2(x), \mathfrak{E}_3(x)\right],\nonumber\\
\mathfrak{B}(x)&=&\left[\mathfrak{B}_1(x),\mathfrak{B}_2(x),\mathfrak{B}_3(x)\right],\nonumber
\end{eqnarray}
where $\kappa_B>0$ is the Boltzmann constant, $n_0>0$ and $T_0>0$ are constant
reference number density and temperature, respectively, and the reference bulk
velocities have been chosen to be zero, then it is easy to check that $[\mu_+(v),\mu_-(v),\mathfrak{E}(x), \mathfrak{B}(x)]$ is a solution to the Vlasov-Maxwell-Boltzmann system \eqref{VMB}, \eqref{Maxwell} iff
\begin{equation*}
\mathfrak{E}(x)=0,\quad \nabla_x\cdot \mathfrak{B}(x)=0,\quad \nabla_x\times \mathfrak{B}(x)=0
\end{equation*}
hold for $x\in\mathbb{R}^3$, from which, the fact that $\mathfrak{B}(x)$ is always assumed to be bounded, and Liouville's theorem, one can deduce that
$\mathfrak{B}(x)\equiv\mathfrak{B}=\left[\mathfrak{B}_1, \mathfrak{B}_2, \mathfrak{B}_3\right]$, where $\left[\mathfrak{B}_1, \mathfrak{B}_2, \mathfrak{B}_3\right]$ is any constant vector.

The main purpose of this paper is to give a rigorous mathematical justification of the above formal limit, which is global-in-time, in the perturbative framework, i.e. in the regime when the solution $\left[F^c_+(t,x,v),\right.$ $\left. F^c_-(t,x,v), E^c(t,x), B^c(t,x)\right]$ of the Vlasov-Maxwell-Boltzmann system \eqref{VMB}, \eqref{Maxwell} and the solution $\left[F^{\infty}_+(t,x,v),\right.$ $\left. F^{\infty}_-(t,x,v), \phi^\infty(t,x)\right]$ of the Vlasov-Poisson-Boltzmann system \eqref{VPB-Original} are a small perturbation of the equilibrium states $[\mu_+(v),\mu_-(v), 0, \mathfrak{B}]$ and $[\mu_+(v),\mu_-(v), 0]$, respectively. Here and in the rest of this paper, $\mathfrak{B}=\left[\mathfrak{B}_1, \mathfrak{B}_2, \mathfrak{B}_3\right]$ is any given constant vector.

To this end, we focus on the Cauchy problem of the Vlasov-Maxwell-Boltzmann system \eqref{VMB}, \eqref{Maxwell} with prescribed initial data
\begin{equation}\label{VMB-in}
  F^c_\pm (0,x,v)=F^c_{0,\pm}(v,x), \quad E^c(0,x)=E^c_0(x), \quad B^c(0,x)=B^c_0(x),
\end{equation}
which satisfy the compatibility conditions
\begin{equation}\label{IC-Compatibility}
  \nabla_x\cdot E^c_0=\int_{\mathbb{R}^3}\left(F^c_{0,+}-F^c_{0,-}\right)dv, \quad \nabla_x\cdot B^c_0=0,
\end{equation}
and the Cauchy problem of the Vlasov-Poisson-Boltzmann system \eqref{VPB-Original} with initial data
\begin{equation}\label{VPB-IC-Original}
F^{\infty}_\pm(0,x,v)=F^{\infty}_{0,\pm}(x,v).
\end{equation}

Without loss of generality, we assume in the rest of this paper that all the physical constants such as $m_\pm, e_\pm, \sigma_\pm$, and the generic constant $4\pi$ except the light speed $c$ are chosen to be one. Under such an assumption, we can normalize the above Maxwellians accordingly as
$$
\mu(v)=\mu_+(v)=\mu_-(v)=(2\pi)^{-\frac 32}e^{-\frac{|v|^2}{2}}.
$$
Moreover, if we set $\epsilon=\frac1c$, then we will use $F^\epsilon_\pm(t,x,v), E^\epsilon(t,x), B^\epsilon(t,x), F^\epsilon_{0,\pm}(x,v), E_0^\epsilon(x), B_0^\epsilon(x)$ to denote $F^c_\pm(t,x,v),$ $E^c(t,x), B^c(t,x), F^c_{0,\pm}(x,v), E^c_0(x), B^c_0(x)$, respectively, in the rest of this paper.

\subsection{Notations}
To continue, we first introduce some notations used throughout the paper.
\begin{itemize}
\item $A\lesssim B$ means that there is a generic constant $C> 0$, which is independent of $\epsilon$, such that $A \leq   CB$. $A \sim B$ means $A\lesssim B$ and $B\lesssim A$;
\item The multi-indices $ \alpha= [\alpha_1,\alpha_2, \alpha_3]$ and $\beta = [\beta_1, \beta_2, \beta_3]$ will be used to record spatial and velocity derivatives, respectively. And $\partial^{\alpha}_{\beta}=\partial^{\alpha_1}_{x_1}\partial^{\alpha_2}_{x_2}\partial^{\alpha_3}_{x_3} \partial^{\beta_1}_{ v_1}\partial^{\beta_2}_{ v_2}\partial^{\beta_3}_{ v_3}$. Similarly, the notation $\partial^{\alpha}$ will be used when $\beta=0$ and likewise for $\partial_{\beta}$. The length of $\alpha$ is denoted by $|\alpha|=\alpha_1 +\alpha_2 +\alpha_3$. $\alpha'\leq  \alpha$ means that no component of $\alpha'$ is greater than the corresponding component of $\alpha$, and $\alpha'<\alpha$ means that $\alpha'\leq  \alpha$ and $|\alpha'|<|\alpha|$. And it is convenient to write $\partial_i=\partial_{x_i}$ for $i=1,2,3$ and $\left|\nabla_x^ku\right|=\sum\limits_{|\alpha|=k}\left|\partial^{\alpha}u\right|$;
\item $\langle\cdot,\cdot\rangle$ is used to denote the ${L^2_{ v}}\times L^2_{ v}$ inner product in ${\mathbb{ R}}^3_{ v}$, with the ${L^2}$ norm $|\cdot|_{L^2}$. For notational simplicity, $(\cdot, \cdot)$ denotes the ${L^2}\times L^2$ inner product either in ${\mathbb{ R}}^3_{x}\times{\mathbb{ R}}^3_{ v }$ or in ${\mathbb{ R}}^3_{x}$ with the ${L^2}\times L^2$ norm $\|\cdot\|$. Moreover $|f|_\nu\equiv\left|\sqrt{\nu}f\right|_{L^2}, \|f\|_\nu\equiv\left\|\sqrt{\nu}f\right\|$, where $\nu(v)$ is the collision frequency defined by \eqref{collision-frequency} satisfying $\nu(v)\sim \langle v\rangle^\gamma$, $\langle v\rangle=\sqrt{1+|v|^2}$;
\item For a complex-valued function $g(t,x,v)$, $\mathfrak{R}g$ and $\overline{g}$ denote the real part of $g(t,x,v)$ and the complex conjugate of $g(t,x,v)$, respectively. Moreover, $\langle [v,-v],[f,g]\rangle=\langle v, f-g\rangle,$ $( [v,-v],[f,g])=( v, f-g),$ $(f\ |\ g)=(f, \overline{g}), \langle f\ |\ g\rangle=\langle f, \overline{g}\rangle$;
\item $\|f(t,\cdot,\cdot)\|_{L^p_xL^q_v}=\left({\displaystyle\int_{\mathbb{R}^3_x}}
    \left({\displaystyle\int_{\mathbb{R}^3_v}}|f(t,x,v)|^qdv\right)^{\frac pq}dx\right)^{\frac 1p}$;
\item For $s\in\mathbb{R}$,
\[
\left(\Lambda^sg\right)(t,x,v)=\int_{\mathbb{R}^3}|\xi|^{s}\hat{g}(t,\xi,v)e^{2\pi ix\cdot\xi}d\xi
=\int_{\mathbb{R}^3}|\xi|^{s}\mathcal{F}[g](t,\xi,v)e^{2\pi ix\cdot\xi}d\xi
\]
with $\hat{g}(t,\xi,v)\equiv\mathcal{F}[g](t,\xi,v)$ being the Fourier transform of $g(t,x,v)$ with respect to $x$.  The homogeneous Sobolev space $\dot{H}^s\times \dot{H}^s$ is the Banach space consisting of all $g$ satisfying  $\|g\|_{\dot{H}^s}<+\infty$, where
\[
\|g(t)\|_{\dot{H}^s}\equiv\left\|\left(\Lambda^s g\right)(t,x,v)\right\|_{L^2_{x,v}}=\left\||\xi|^s\hat{g}(t,\xi,v)\right\|_{L^2_{\xi,v}}.
\]
\end{itemize}

\section{Main results and ideas}
In this section, we will state our main results and sketch our main ideas to prove them. To this end, we first give our main results in section 2.1.
\subsection{Main results}
To make the presentation clear, we divided this section into two parts, the first part is concentrated on the hard sphere model.

\subsubsection{Main result for the hard sphere model}

In the perturbative framework, we can further define the perturbation $f^\epsilon_\pm(t,x,v)$ by
\begin{eqnarray}\label{f-perturbation}
F^\epsilon_\pm(t,x,v)&=&\mu(v)+\mu^{1/2}(v)f^\epsilon_\pm(t,x,v)\geq 0, \ \ B^\epsilon(t,x)=\mathfrak{B}+\widetilde{B}^\epsilon(t,x)
\end{eqnarray}
with initial data
\begin{eqnarray}
F^\epsilon_{0,\pm}(x,v)=\mu(v)+ \mu^{1/2}(v)f^\epsilon_{0,\pm}(x,v),\quad\quad B^\epsilon_0(x)=\mathfrak{B}+\widetilde{B}^{\epsilon}_0(x)
\end{eqnarray}
respectively, then we can deduce from \eqref{f-perturbation} that
\begin{eqnarray} \label{f}
&&\partial_tf^\epsilon_\pm+v\cdot\nabla_xf^\epsilon_\pm\mp E^\epsilon \cdot v \mu^{1/2}\mp \frac {1} 2 E^\epsilon\cdot v f^\epsilon_\pm\pm E^\epsilon\cdot\nabla_{ v  }f^\epsilon_\pm\nonumber\\
&&\pm\epsilon \left\{v\times \left(\mathfrak{B}+\widetilde{B}^\epsilon\right)\right\}\cdot\nabla_vf^\epsilon_\pm+{ L}_\pm f^\epsilon\nonumber\\
  &=&{\Gamma}_\pm\left(f^\epsilon,f^\epsilon\right),\\
\epsilon\partial_tE^\epsilon-\nabla_x\times \widetilde{B}^\epsilon&=&-\epsilon{\displaystyle\int_{\mathbb{R}^3}} v\mu^{1/2}(v)\left(f^\epsilon_+-f^\epsilon_-\right)dv,\nonumber\\
\epsilon\partial_t\widetilde{B}^\epsilon+\nabla_x\times E^\epsilon&=&0,\nonumber\\
\nabla_x\cdot E^\epsilon&=&{\displaystyle\int_{\mathbb{R}^3}}\mu^{1/2}(v)\left(f^\epsilon_+-f^\epsilon_-\right)dv,\quad \nabla_x\cdot \widetilde{B}^\epsilon=0\nonumber
\end{eqnarray}
with prescribed initial data
\begin{equation}\label{f-initial}
f^\epsilon_\pm(0,x,v)=f^\epsilon_{0,\pm}(x,v),  \quad E^\epsilon(0,x)=E^\epsilon_0(x), \quad \widetilde{B}^\epsilon(0,x)=\widetilde{B}^\epsilon_0(x),
\end{equation}
which satisfy the compatibility conditions
\begin{equation}\label{compatibility conditions}
  \nabla_x\cdot E^\epsilon_0=\int_{\mathbb{R}^3}\mu^{1/2}(v)\left(f^\epsilon_{0,+}-f^\epsilon_{0,-}\right)dv, \quad \nabla_x\cdot \widetilde{B}^\epsilon_0=0.
\end{equation}

For later use, setting $g=\left[g_+,g_-\right]$, the first equation of $(\ref{f})$ can be also rewritten as
\begin{eqnarray}
&&\partial_tf^\epsilon+ v  \cdot\nabla_xf^\epsilon- E^\epsilon \cdot v \mu^{1/2}q_1\nonumber\\
&&+q_0 E^\epsilon\cdot\nabla_{ v  }f^\epsilon+ q_0\epsilon \left\{v\times \left(\mathfrak{B}+\widetilde{B}^\epsilon\right)\right\}\cdot\nabla_vf^\epsilon+{ L} f^\epsilon\label{f gn}\\
&=&\frac 1 2 q_0 E^\epsilon\cdot v f^\epsilon
 +{\Gamma}\left(f^\epsilon,f^\epsilon\right),\nonumber
\end{eqnarray}
where
\begin{equation}\label{def-q}
q_0={\textrm{diag}}(1,-1),\quad q_1=[1,-1],
 \end{equation}
 and the linearized collision operator $L=[L_+,L_-]$ together with the nonlinear collision operator $\Gamma=[\Gamma_+,\Gamma_-]$ are respectively defined by
\begin{equation}\label{def-L-Gamma}
Lg=[L_{+}g,L_{-}g],\quad\quad\quad\Gamma(g,h)=[\Gamma_+(g,h),\Gamma_-(g,h)]
\end{equation}
with
\begin{eqnarray*}
{L}_\pm g&=& -{\bf \mu}^{-1/2}
\left\{{Q\left(\mu,{\bf \mu}^{1/2}(g_\pm+g_\mp)\right)+ 2Q\left(\mu^{1/2}g_\pm, \mu\right)}\right\},\\
{ \Gamma}_{\pm}(g,h)&=&{\bf \mu}^{-1/2}\left\{Q\left({\bf \mu}^{1/2}g_{\pm},{\bf \mu}^{1/2}h_\pm\right)+Q\left({\bf \mu}^{1/2}g_{\pm},{\bf \mu}^{1/2}h_\mp\right)\right\}.
\end{eqnarray*}

For the linearized  collision operator $ L $, it is well known (cf.~\cite{Guo-Invent Math-2003}) that it is non-negative and the null space $\mathcal{N}$ of ${ L}$ is spanned by
\begin{equation*}
{\mathcal{ N}}={\textrm{span}}\left\{[1,0]\mu^{1/2} , [0,1]\mu^{1/2}, [v_i,v_i]{\mu}^{1/2} (1\leq i\leq3),[|v|^2,|v|^2]{\bf \mu}^{1/2}\right\}.
\end{equation*}

Let ${\bf P}$ be the orthogonal projection from $L^2({\mathbb{R}}^3_ v)\times L^2({\mathbb{R}}^3_ v)$ to $\mathcal{N}$, then for any given function $g(t, x, v )\in L^2({\mathbb{R}}^3_ v)\times L^2({\mathbb{R}}^3_ v)$, one has
\begin{equation}\label{def-pf}
  {\bf P}g ={a^{g}_+(t, x)[1,0]\mu^{1/2}+a^{g}_-(t, x)[0,1]\mu^{1/2}+\sum_{i=1}^{3}b^{g}_i(t, x) [1,1]v_i{\mu}^{1/2}+c^{g}(t, x)[1,1]\left(| v|^2-3\right)}{\bf \mu}^{1/2}
\end{equation}
with
\begin{eqnarray*}
  a^{g}_\pm(t,x)&=&\int_{{\mathbb{R}}^3}{\bf \mu}^{1/2}(v)g_\pm(t,x,v) d v,\\
  b^{g}_i(t,x)&=&\frac12\int_{{\mathbb{R}}^3} v  _i {\bf \mu}^{1/2}(v)\left(g_+(t,x,v)+g_-(t,x,v)\right)d v,\\
  c^{g}(t,x)&=&\frac{1}{12}\int_{{\mathbb{R}}^3}\left(| v|^2-3\right){\bf \mu}^{1/2}(v)\left(g_+(t,x,v)+g_-(t,x,v)\right) d v.
\end{eqnarray*}
The expression \eqref{def-pf} can be rewritten as ${\bf P}g=[{\bf P}g_+,{\bf P}g_-]$ with
\begin{equation*}\label{def-pf-1}
  {\bf P}_\pm g ={a^{g}_\pm(t, x)\mu^{1/2}+\sum_{i=1}^{3}b^{g}_i(t, x) v_i{\mu}^{1/2}+c^{g}(t, x)\left(| v|^2-3\right)}{\bf \mu}^{1/2}.
\end{equation*}
Therefore, we have the following macro-micro decomposition with respect to the given global Maxwellian $\mu(v)$, cf.  \cite{Guo-IUMJ-2004, Liu-Yang-Yu-Physica D-2004}
\begin{equation}\label{macro-micro}
 g(t,x, v)={\bf P}g(t,x, v)+\{{\bf I}-{\bf P}\}g(t, x, v),
\end{equation}
where ${\bf I}$ denotes the identity operator, and ${\bf P}g$ and $\{{\bf I}-{\bf P}\}g$ are called the macroscopic and the microscopic component of $g(t,x,v)$, respectively.
Moreover, it is easy to see that $L$ can be decomposed as
\begin{equation*}\label{decomposition-L}
{L}g=\nu g-Kg
\end{equation*}
with the collision frequency $\nu(v)$ and the nonlocal integral operator $K=[K_+,K_-]$ being defined by
\begin{equation}\label{collision-frequency}
\nu(v)=2Q_{loss}(1,\mu)=2\int_{\mathbb{R}^3\times {\mathbb{S}}^2}|v-u|^\gamma {\bf b}\left(\frac{\omega\cdot(v-u)}{|v-u|}\right)\mu(u)d\omega du\thicksim(1+|v|)^\gamma
\end{equation}
and
\begin{eqnarray}
\left(K_\pm g\right)(v)
&=&{\mu}^{-\frac 12}
\left\{2Q_{gain}\left(\mu^{\frac 12}g_\pm, \mu\right)-Q\left(\mu,{\mu}^{\frac 12}(g_\pm+g_\mp)\right)\right\}\nonumber\\
&=&\int_{\mathbb{R}^3\times \mathbb{S}^2}|u-v|^\gamma {\bf b}\left(\frac{\omega\cdot(v-u)}{|v-u|}\right)\mu^{\frac 12}(u)\label{Operator-K}\\
&&\times\left\{2\mu^{\frac 12}(u')g_\pm(v')-\mu^{\frac 12}(v') (g_\pm+g_\mp)(u')
+\mu^{\frac 12}(v)(g_\pm+g_\mp)(u)\right\}d\omega du,\nonumber
\end{eqnarray}
respectively.

Under Grad's angular cutoff assumption \eqref{cutoff-assump},
by  \cite[Lemma 1]{Guo-Invent Math-2003},  $L$ is locally coercive in the sense that
\begin{equation}\label{coercive-estimates}
-\left\langle f, Lf\right\rangle\geq \sigma_0\left|\{{\bf I}-{\bf P}\}f\right|_\nu^2\equiv \sigma_0\left\|\sqrt{\nu}\{{\bf I}-{\bf P}\}f\right\|^2_{L^2(\mathbb{R}^3_v)},\quad \nu(v)\sim(1+|v|)^{\gamma}
\end{equation}
holds for some positive constant $\sigma_0>0$.

To state our main result for hard sphere model, let
$\left[f^\epsilon(t,x,v), E^\epsilon(t,x), \widetilde{B}^\epsilon(t,x)\right]$ be the unique solution of the Cauchy problem (\ref{f}), (\ref{f-initial}), \eqref{compatibility conditions} and for some $n\in \mathbb{N},$ we need to define the following energy functional ${\mathcal{E}}_{n,f^\epsilon}(t)$ and the corresponding energy dissipation rate functional ${\mathcal{D}}_{n,f^\epsilon}(t)$ by
\begin{equation}\label{Energy-functional-VMB}
{\mathcal{E}}_{n,f^\epsilon}(t)\sim\sum_{|\alpha|+|\beta|\leq n}\left\|\partial^{\alpha}_{\beta}f^\epsilon(t)\right\|^2 +\left\|\left[E^\epsilon(t),\widetilde{B}^\epsilon(t)\right]\right\|_{H^n_x}^2
\end{equation}
and
\begin{eqnarray}\label{Energy-dissipation-rate-functional-VMB}
{\mathcal{D}}_{n,f^\epsilon}(t)&\sim&\sum_{1\leq|\alpha|\leq n}\left\|\partial^{\alpha}\left[a^{f^\epsilon}_{\pm}(t),b^{f^\epsilon}(t), c^{f^\epsilon}(t)\right]\right\|^2+\sum_{|\alpha|+|\beta|\leq n}\left\|
\partial^{\alpha}_{\beta}{\bf\{I-P\}}f^\epsilon(t)\right\|^2_{\nu}
+\left\|a^{f^\epsilon}_+(t)-a^{f^\epsilon}_-(t)\right\|^2\nonumber\\
&&+\epsilon^2\left\|\left\{E^\epsilon(t)+\epsilon b^{f^\epsilon}(t)\times \mathfrak{B}\right\}\right\|^2+\epsilon^2\left\|\left[\nabla_xE^\epsilon(t),
\nabla_x\widetilde{B}^\epsilon(t)\right]\right\|_{H^{n-2}_x}^2,
\end{eqnarray}
respectively.

With the above preparations in hand, we are now ready to state our first result for the hard sphere model.

\begin{theorem}\label{Th1.1}
For the hard sphere model, let $F^\epsilon(0,x,v)=\mu(v)+\sqrt{\mu(v)}f^\epsilon_0(x,v)\geq0$, if we assume further that there exists some positive constant $c_0$ independent of $\epsilon$ such that
\begin{equation}\label{small-condition}
Y_0\equiv\sum_{|\alpha|+|\beta|\leq 3}\left\|\partial_\beta^\alpha f^\epsilon_0\right\| +\left\|E^\epsilon_0\right\|_{H^3_x}+\left\|B^\epsilon_0-\mathfrak{B}\right\|_{H^3_x}\leq c_0,
\end{equation}
then for every $\epsilon\in (0,1]$, the Cauchy problem (\ref{f}), (\ref{f-initial}), \eqref{compatibility conditions} admits a unique global solution $\left[f^\epsilon(t,x,v),\right.$ $\left. E^\epsilon(t,x), \widetilde{B}^\epsilon(t,x)\right]$ satisfying \eqref{f-perturbation}
and the following estimates
\begin{equation}\label{th1.1-1}
\mathcal{E}_{3,f^\epsilon}(t)+\int^t_0\mathcal{D}_{3,f^\epsilon}(\tau)d\tau\lesssim
Y_0^2\lesssim c_0^2
\end{equation}
holds for all $t\in\mathbb{R}^+$.
\end{theorem}

As a direct consequence of the estimate \eqref{th1.1-1} and Sobolev's imbedding theorem, one can easily deduce that the unique global solution $\left[F^\epsilon(t,x,v), E^\epsilon(t,x), B^\epsilon(t,x)\right]$ of the Cauchy problem \eqref{VMB}, \eqref{Maxwell}, \eqref{VMB-in}, \eqref{IC-Compatibility} obtained in Theorem \ref{Th1.1} converges strongly to $\left[F^\infty(t,x,v), E^\infty(t,x), \mathcal{B}\right]$ as $\epsilon\to 0_+$ which satisfies $F^\infty_\pm(t,x,v)=\mu(v)+\mu^\frac 12(v)f^\infty_\pm(t,x,v)\geq 0$, the following estimate
\begin{equation*}
\mathcal{E}^\infty_3(t)+\int^t_0\mathcal{D}^\infty_3(\tau)d\tau\lesssim c_0^2,
\end{equation*}
and $\left[F^\infty(t,x,v), E^\infty(t,x)\right]$ solves the Cauchy problem of the Vlasov-Poisson-Boltzmann equations \eqref{VPB-Original}, \eqref{Poisson-Rewritten}, \eqref{VPB-IC-Original} with
\begin{equation*}
F^\infty_{0,\pm}(x,v)=\lim\limits_{\epsilon\to0_+}F^\epsilon_{0,\pm}(x,v).
\end{equation*}
Here
\begin{eqnarray*}
\mathcal{E}^\infty_n(t)&\sim&\sum_{|\alpha|+|\beta|\leq n}\left\|\partial_\beta^\alpha f^\infty(t)\right\|^2+\left\|E^\infty(t)\right\|^2_{H^n_x},\\
\mathcal{D}^\infty_n(t)&\sim&
\sum_{1\leq|\alpha|\leq n}\left\|\partial^{\alpha}\left[a^{f^\infty}_{\pm}(t),b^{f^\infty}(t), c^{f^\infty}(t)\right]\right\|^2+\sum_{|\alpha|+|\beta|\leq n}\left\|
\partial^{\alpha}_{\beta}{\bf\{I-P\}}f^\infty(t)\right\|^2_{\nu}
+\left\|a^{f^\infty}_+(t)-a^{f^\infty}_-(t)\right\|^2.
\end{eqnarray*}

Our second result in this section is to deduce the error estimate between $\left[F^\epsilon(t,x,v), E^\epsilon(t,x), B^\epsilon(t,x)\right]$ and $\left[F^\infty(t,x,v), E^\infty(t,x), \mathcal{B}\right]$ provided that certain conditions are imposed on the difference between their initial data. In fact, suppose that there exists some sufficiently small $\epsilon-$independent positive constant $\tilde{c}>0$ such that
\begin{equation}\label{small-condition-vpb}
Y_{P,0}\equiv\sum_{|\alpha|+|\beta|\leq 4}\left\|\partial_\beta^\alpha f^\infty_0\right\| +\left\|E^\infty_0\right\|_{H^4_x}\leq \tilde{c}_0,
\end{equation}
then we can show further that $\left[f^\infty(t,x,v), E^\infty(t,x)\right]$ obtained above satisfies
\begin{equation}\label{th1.1-1-b}
\mathcal{E}^\infty_4(t)+\int^t_0\mathcal{D}^\infty_4(\tau)d\tau\lesssim \tilde{c}^2,
\end{equation}
from which and \eqref{th1.1-1}, we can get that

\begin{corollary}\label{c-limit}
In addition to the assumptions listed in Theorem \ref{Th1.1}, we assume further that \eqref{small-condition-vpb} holds and
\begin{eqnarray}\label{th1.1-1-c}
\sum_{|\alpha|+|\beta|\leq 3}\left\|\partial_\beta^\alpha \left(f^\epsilon_0-f^\infty_0\right)\right\| +\left\|E^\epsilon_0-E^\infty_0\right\|_{H^3_x}+\left\|B^\epsilon_0-\mathfrak{B}\right\|_{H^3_x}
\lesssim \epsilon^{\eta}
\end{eqnarray}
is true for any positive constant $\eta>0$, then we can deduce for all $t\in\mathbb{R}^+$ that
\begin{eqnarray}\label{th1.1-1-d}
&&\sum_{|\alpha|+|\beta|\leq 3}\left\|\partial_\beta^\alpha (f^\epsilon(t)-f^\infty(t))\right\|^2+\left\|E^\epsilon(t)-E^\infty(t)\right\|^2_{H^3_x}+\left\|B^\epsilon(t)-\mathfrak{B}\right\|^3_{H^3_x}
\lesssim \epsilon^{\eta}.
\end{eqnarray}
\end{corollary}

\subsubsection{Main result for soft potentials}
Now we turn to deal with the case of soft potentials.  Recall that for the hard sphere model, we can deduce that the solution $\left[f^\epsilon(t,x,v), E^\epsilon(t,x), \widetilde{B}^\epsilon(t,x)\right]$ of the Cauchy problem  (\ref{f}), (\ref{f-initial}), \eqref{compatibility conditions} satisfies the estimates \eqref{th1.1-1} for all $\epsilon\in(0,1]$, from which one can deduce the strong convergence of $\left[f^\epsilon(t,x,v), E^\epsilon(t,x), {B}^\epsilon(t,x)\right]$ to $\left[f^\infty(t,x,v), E^\infty(t,x), \mathfrak{B}\right]$ as $\epsilon\to 0_+$ provided that the initial data $\left[f_0^\epsilon(x,v), E_0^\epsilon(x),\right.$ $\left.   B^\epsilon_0-\mathfrak{B}\right]$ satisfies the assumption \eqref{small-condition}.

The above argument can not be used to treat the case of soft potentials since, unlike the hard sphere model, to guarantee the global solvability of the Cauchy problem (\ref{f}), (\ref{f-initial}), \eqref{compatibility conditions} for non-hard sphere model, a time-velocity-weighted energy method should be introduced and the temporal decay of the solution itself plays an essential role in closing the {\it a priori} estimates, cf. \cite{Duan-Lei-Yang-Zhao-CMP-2017, Duan_Liu-Yang_Zhao-KRM-2013, Fan-Lei-Liu-Zhao-SCM-2017} and the references cited therein, and for the case when $\epsilon$ is chosen sufficiently small, we had to deal with the difficulty caused by the
degeneracy of the dissipative effect of the electromagnetic field $\left[E^\epsilon(t,x), \widetilde{B}^\epsilon(t,x)\right]$.

To overcome such a difficulty, we need to assume that both the solution $\left[F^\epsilon(t,x,v), E^\epsilon(t,x), B^\epsilon(t,x)\right]$ of the Cauchy problem  (\ref{VMB}), (\ref{Maxwell}), \eqref{VMB-in} and the corresponding initial data $\left[F^\epsilon_{0,\pm}(x,v), E^\epsilon_0(x), B_0^\epsilon(x)\right]$ satisfies the following expansions
\begin{eqnarray}\label{F-S-Expansion}
F^\epsilon_\pm(t,x,v)&=&F^{P,\epsilon}_\pm(t,x,v)+\sum_{i=1}^{m-1}\epsilon^i F_{\pm}^{i,\epsilon}(t,x,v)+\epsilon^m F^{m,\epsilon}_{\pm}(t,x,v),\nonumber\\
E^\epsilon(t,x)&=&E^{P,\epsilon}(t,x)+\sum_{i=1}^{m-1}\epsilon^{i}E^{i,\epsilon}(t,x) +\epsilon^{m} E^{m,\epsilon}(t,x),\\
B^\epsilon(t,x)&=&B^{P,\epsilon}(t,x)+\sum_{i=1}^{m-1}\epsilon^iB^{i,\epsilon}(t,x) +\epsilon^m B^{m,\epsilon}(t,x),\ m\geq 1,\nonumber
\end{eqnarray}
and
\begin{eqnarray}\label{F-S-Expansion-IC}
F^\epsilon_{0,\pm}(x,v)&=&F^{P,\epsilon}_{0,\pm}(x,v)+\sum_{i=1}^{m-1}\epsilon^i F_{0,\pm}^{i,\epsilon}(x,v)+\epsilon^m F^{m,\epsilon}_{0,\pm}(x,v),\nonumber\\
E^\epsilon_0(x)&=&E^{P,\epsilon}_0(x)+\sum_{i=1}^{m-1}\epsilon^{i}E^{i,\epsilon}_{0}(x) +\epsilon^{m} E^{m,\epsilon}_{0}(x),\\
B^\epsilon_0(x)&=&B^{P,\epsilon}_0(x)+\sum_{i=1}^{m-1}\epsilon^iB^{i,\epsilon}_{0}(x) +\epsilon^m B^{m,\epsilon}_{0}(x),\ m\geq 1,\nonumber
\end{eqnarray}
then we will show that $B^{P,\epsilon}(t,x), B^{i,\epsilon}(t,x) (i=1,2,\cdots, m-1), B^{P,\epsilon}_0(x)$, and $B^{i,\epsilon}_{0}(x) (i=1,2,\cdot, m-1)$ are nothing but constant vectors (we will assume without loss of generality that all these constant vectors are independent of $\epsilon$), which will be denoted by $B^P, B^i (i=1,2,\cdots, m-1)$ in the rest of this paper.

In fact, by substituting the above ansatz into \eqref{VMB}, \eqref{Maxwell} and by comparing the order of $\epsilon $ in the resulting system, we can get from \eqref{VMB}, \eqref{Maxwell}, \eqref{VMB-in} that:
\begin{itemize}
\item [(i).]
 $\left[F^{P,\epsilon}_\pm(t,x,v), E^{P,\epsilon}(t,x)\right]$ solve the following Vlasov-Poisson-Boltzmann system with given magnetic field $B^{P,\epsilon}(t,x)+\sum_{i=1}^{m-1}\epsilon^{m-1}B^{i,\epsilon}(t,x)$
\begin{eqnarray}\label{VPB-s}
&&\partial_tF^{P,\epsilon}_\pm+ v  \cdot\nabla_xF^{P,\epsilon}_\pm\pm E^{P,\epsilon}\cdot\nabla_vF^{P,\epsilon}_\pm
\pm \epsilon \left\{v\times \left(B^{P,\epsilon}+\sum_{i=1}^{m-1}\epsilon^{m-1}B^{i,\epsilon}\right)\right\} \cdot\nabla_vF^{P,\epsilon}_\pm\nonumber\\
&=&Q\left(F^{P,\epsilon}_\pm,F^{P,\epsilon}_\pm\right) +Q\left(F^{P,\epsilon}_\pm,F^{P,\epsilon}_\mp\right),\\
\partial_tE^{P,\epsilon}&=&-\int_{\mathbb{R}^3_v}v\left(F^{P,\epsilon}_+ -F^{P,\epsilon}_-\right)dv,\nonumber\\
\partial_t B^{P,\epsilon}&=& \nabla_x\times B^{P,\epsilon}=\nabla_x\cdot B^{P,\epsilon}= \nabla_x\times E^{P,\epsilon}=0,\label{vpb-condition-s}\\
\nabla_x\cdot E^{P,\epsilon}&=&\int_{\mathbb{R}^3_v}\left(F^{P,\epsilon}_+-F^{P,\epsilon}_-\right)dv\nonumber
\end{eqnarray}
with prescribed initial condition
\begin{equation}\label{VPB-s-IC}
\left[F^{P,\epsilon}_\pm(0,x,v), E^{P,\epsilon}(0,x)\right]=\left[F^{P,\epsilon}_{0,\pm}(x,v), E^{P,\epsilon}_0(x)\right].
\end{equation}
\eqref{vpb-condition-s}$_3$ together with the Liouville theorem imply that $B^{P,\epsilon}(t,x)$ is nothing but a constant vector $B^{P,\epsilon}$, while from $\eqref{vpb-condition-s}_3$ one can conclude that there exists a potential function $\phi^{R,\epsilon}(t,x)$ for the electric field $E^{P,\epsilon}(t,x)$ such that $E^{P,\epsilon}(t,x)= \nabla_x\phi^{P,\epsilon}(t,x)$. Without loss of generality, we can assume that $\lim\limits_{|x|\to+\infty}\phi^{P,\epsilon}(t,x)=0$ and that the constant vector $B^{P,\epsilon}$ does not depend on $\varepsilon$, thus we can set $B^{P,\epsilon}(t,x)\equiv B^{P,\epsilon}_0(x)\equiv B^P$;
\item [(ii).]
For $i=1,2,\cdots, m-1$, one can deduce that $\left[F_\pm^{i,\epsilon}(t,x,v), E^{i,\epsilon}(t,x),B^{i,\epsilon}(t,x)\right]$ satisfies the following linear Vlasov-Poisson-Boltzmann equations
\begin{eqnarray}\label{F_i-s}
&&\partial_tF_\pm^{i,\epsilon}+ v\cdot\nabla_xF_\pm^{i,\epsilon} \pm E^{i,\epsilon}\cdot\nabla_vF^{P,\epsilon}_\pm\pm E^{P,\epsilon}\cdot\nabla_vF_\pm^{i,\epsilon}\nonumber\\
&&
\pm\sum_{j_1+j_2=i,\atop 0<j_1,j_2<i} E^{j_1,\epsilon}\cdot\nabla_vF_\pm^{j_2,\epsilon} \pm \epsilon \left\{v\times \left(B^{P}
+\sum_{j_1=1}^{m-1}\epsilon^{j_1}B^{j_1,\epsilon}\right)\right\}\cdot\nabla_vF_\pm^{i,\epsilon}\\
&=&Q\left(F^{P,\epsilon}_\pm,F_\pm^{i,\epsilon}+F_\mp^{i,\epsilon}\right) +Q\left(F_\pm^{i,\epsilon},F^{P,\epsilon}_\pm+F^{P,\epsilon}_{\mp}\right)
+\sum_{j_1+j_2=i,\atop 0<j_1,j_2<i} Q\left(F_\pm^{j_1,\epsilon},F_\pm^{j_2,\epsilon}+F_\mp^{j_2,\epsilon}\right),\nonumber
\end{eqnarray}
\begin{eqnarray}
 \partial_tE^{i,\epsilon}&=&-{\displaystyle\int_{\mathbb{R}^3_v}} v\left(F_+^{i,\epsilon}-F_-^{i,\epsilon}\right)dv,\nonumber\\
\nabla_x\cdot E^{i,\epsilon}&=&{\displaystyle\int_{\mathbb{R}^3_v}}\left(F_+^{i,\epsilon} -F_-^{i,\epsilon}\right)dv,\label{F_i-s-Maxwell-con}\\
 \partial_tB^{i,\epsilon}&=&\nabla_x\times B^{i,\epsilon}=\nabla_x\cdot B^{i,\epsilon}=\nabla_x\times E^{i,\epsilon}=0,\nonumber
\end{eqnarray}
and the initial data
\begin{equation}\label{F_i-s-IC}
\left[F_\pm^{i,\epsilon}(0,x,v),E^{i,\epsilon}(0,x)\right]=\left[ F_{0,\pm}^{i,\epsilon}(x,v), E_0^{i,\epsilon}(x)\right],
\end{equation}
which satisfy the compatibility conditions
\begin{equation}\label{F_i-s-IC-Compatibility}
  \nabla_x\cdot E_0^{i,\epsilon}=\int_{\mathbb{R}^3}
  \left(F_{0,+}^{i,\epsilon}-F_{0,-}^{i,\epsilon}\right)dv.
\end{equation}
Here again from \eqref{F_i-s-Maxwell-con}$_3$ together with the Liouville theorem, we can deduce for $i=1,2,\cdots, m-1$ that $B^{i,\epsilon}(t,x)$ are constant vectors $B^{i,\epsilon}$, while from $\eqref{F_i-s-Maxwell-con}_3$ one can further conclude that there exist potential functions $\phi^{i,\epsilon}(t,x)$ for the electric fields $E^{i,\epsilon}(t,x)$ such that $E^{i,\epsilon}(t,x)= \nabla_x\phi^{i,\epsilon}(t,x)$ hold. Without loss of generality, we can assume that $\lim\limits_{|x|\to+\infty}\phi^{i,\epsilon}(t,x)=0$ and that the constant vectors $B^{i,\epsilon}$ does not depend on $\varepsilon$, thus we have $B^{i,\epsilon}(t,x)\equiv B^{i,\epsilon}_0(x)\equiv B^i$ for $i=1,2,\cdots,m-1$;
\item [(iii).] The remainder $\left[F_\pm^{m,\epsilon}(t,x,v), E^{m,\epsilon}(t,x), B^{m,\epsilon}(t,x)\right]$ satisfies the following Vlasov-Maxwell-Boltzmann equations
\begin{eqnarray}\label{F-R-s-VMB-Remain-1}
&&\partial_tF_\pm^{m,\epsilon}+ v\cdot\nabla_xF_\pm^{m,\epsilon} \pm  E^{m,\epsilon}\cdot\nabla_vF^{P,\epsilon}_\pm\pm E^{P,\epsilon}\cdot\nabla_vF_\pm^{m,\epsilon}\nonumber\\
&&\pm\sum_{j_1+j_2\geq m,\atop 0<j_1,j_2<m} \epsilon^{j_1+j_2-m}E^{j_1,\epsilon}\cdot\nabla_vF_\pm^{j_2,\epsilon}\pm\sum_{0<j_1< m} \epsilon^{j_1}E^{j_1,\epsilon}\cdot\nabla_vF_\pm^{m,\epsilon}
\pm\sum_{0<j_1< m} \epsilon^{j_1}E^{m,\epsilon}\cdot\nabla_vF_\pm^{j_1,\epsilon}\nonumber\\
&&\pm\epsilon \left(v\times B^{m,\epsilon}\right)\cdot \nabla_v \left\{F^{P,\epsilon}_\pm+\sum_{i=1}^{m-1}\epsilon^i F_\pm^{i,\epsilon}\right\}\pm \epsilon \left\{v\times \left(B^{P}
+\sum_{i=1}^{m-1}\epsilon^{i}B^{i}\right)\right\}\cdot\nabla_vF_\pm^{m,\epsilon}\nonumber\\
&&\pm\epsilon^{m} E^{m,\epsilon}\cdot\nabla_vF_\pm^{m,\epsilon}
\pm\epsilon^{m+1}v\times B^{m,\epsilon}\cdot\nabla_vF_\pm^{m,\epsilon}\nonumber\\
&=&Q\left(F^{P,\epsilon}_\pm,F_\pm^{m,\epsilon}+F_\mp^{m,\epsilon}\right) +Q\left(F_\pm^{m,\epsilon},F^{P,\epsilon}_\pm+F_\mp^{P,\epsilon}\right)+\epsilon^m Q\left(F_\pm^{m,\epsilon},F_\pm^{m,\epsilon}+F_\mp^{m,\epsilon}\right)\\
&&+\sum_{j_1+j_2\geq m,\atop 0<j_1,j_2\leq m}\epsilon^{j_1+j_2-m}Q\left(F_\pm^{j_1,\epsilon},F_\pm^{j_2,\epsilon} +F_\mp^{j_2,\epsilon}\right),\nonumber
\end{eqnarray}
\begin{eqnarray}
 \epsilon\partial_tE^{m,\epsilon}-\nabla_x\times B^{m,\epsilon}&=&-\epsilon{\displaystyle\int_{\mathbb{R}^3_v}} v\left(F_+^{m,\epsilon}-F_-^{m,\epsilon}\right)dv,\nonumber\\
 \epsilon\partial_tB^{m,\epsilon}+\nabla_x\times E^{m,\epsilon}&=&0,\label{E-B-R-s-Maxwell-con}\\
\nabla_x\cdot E^{m,\epsilon}&=&{\displaystyle\int_{\mathbb{R}^3_v}} \left(F_+^{m,\epsilon}-F_-^{m,\epsilon}\right)dv,\nonumber\\
 \nabla_x\cdot B^{m,\epsilon}&=&0,\nonumber
\end{eqnarray}
and initial data
\begin{eqnarray}\label{VMB-s-IC-Error}
F_\pm^{m,\epsilon}(0,x,v)&=&F_{0,\pm}^{m,\epsilon}(x,v)=\frac 1{\epsilon^m}\left(F^\epsilon_{0,\pm}(x,v)-F^{P,\epsilon}_{0,\pm}(x,v) -\sum_{i=1}^{m-1}\epsilon^iF_{0,\pm}^{i,\epsilon}(x,v)\right),\nonumber\\
E^{m,\epsilon}(0,x)&=&E_0^{m,\epsilon}(x)=\frac 1{\epsilon^{m}}\left(E^\epsilon_0(x)-E^{P,\epsilon}_{0}(x) -\sum_{i=1}^{m-1}\epsilon^{i+1}E_0^{i,\epsilon}(x)\right),\\
B^{m,\epsilon}(0,x)&=&B_0^{m,\epsilon}(x)=\frac 1{\epsilon^m} \left(B^\epsilon_0(x)-B^{P}-\sum_{i=1}^{m-1}\epsilon^iB^{i}\right),\nonumber
\end{eqnarray}
which satisfy the compatibility conditions
\begin{equation}\label{IC-s-Compatibility}
\nabla_x\cdot E_0^{m,\epsilon}=\int_{\mathbb{R}^3}\left(F_{0,+}^{m,\epsilon} -F_{0,-}^{m,\epsilon}\right)dv, \quad \nabla_x\cdot B^{m,\epsilon}_0=0.
\end{equation}
\end{itemize}

In the perturbative framework, if we further define the perturbations $f_\pm^{P,\epsilon}(t,x, v)$, $f_\pm^{i,\epsilon}(t,x,v)$ for $1\leq i\leq m$, $f_{0,\pm}^{P,\epsilon}(x, v)$, and $f_{0,\pm}^{i,\epsilon}(x,v)$ for $1\leq i\leq m$ by
\begin{eqnarray}\label{f-s-perturbation}
F^{P,\epsilon}_\pm(t,x,v)&=&\mu(v)+ \mu^{1/2}(v)f^{P,\epsilon}_{\pm}(t,x,v),\nonumber\\
F_\pm^{i,\epsilon}(t,x,v)&=&\mu^{1/2}(v)f_\pm^{i,\epsilon}(t,x,v),\quad 1\leq i\leq m,\\
F^{P,\epsilon}_{0,\pm}(x,v)&=&\mu(v)+ \mu^{1/2}(v)f^{P,\epsilon}_{0,\pm}(x,v),\nonumber\\
F_{0,\pm}^{i,\epsilon}(x,v)&=&\mu^{1/2}(v)f_{0,\pm}^{i,\epsilon}(x,v),\quad 1\leq i\leq m,\nonumber
\end{eqnarray}
and set
\begin{eqnarray*}
f^{P,\epsilon}(t,x,v)&=&\left[f^{P,\epsilon}_+(t,x,v),f^{P,\epsilon}_-(t,x,v)\right],\\
f^{i,\epsilon}(t,x,v)&=&\left[f^{i,\epsilon}_+(t,x,v), f^{i,\epsilon}_-(t,x,v)\right],\quad 1\leq i\leq m,\\
f^{P,\epsilon}_0(x,v)&=&\left[f^{P,\epsilon}_{0,+}(x,v),f^{P,\epsilon}_{0,-}(x,v)\right],\\
f^{i,\epsilon}_0(x,v)&=&\left[f^{i,\epsilon}_{0,+}(x,v), f^{i,\epsilon}_{0,-}(x,v)\right],\quad 1\leq i\leq m,
\end{eqnarray*}
then, noticing that
$$
B^{P,\epsilon}(t,x)=B^P,\quad B^{i,\epsilon}(t,x)=B^i,\quad i=1,2,\cdots, m-1,
$$
we can get from \eqref{VPB-s}-\eqref{VPB-s-IC}, \eqref{F_i-s}-\eqref{F_i-s-IC-Compatibility}, and
 \eqref{F-R-s-VMB-Remain-1}-\eqref{IC-s-Compatibility}
 that
\begin{itemize}
\item [(i).] $\left[f^{P,\epsilon}(t,x,v), E^{P,\epsilon}(t,x)\right]$ satisfies the following Cauchy problem
\begin{eqnarray}
&&\partial_tf^{P,\epsilon}+ v  \cdot\nabla_xf^{P,\epsilon}- E^{P,\epsilon} \cdot v \mu^{1/2}q_1\nonumber\\
&&+q_0 E^{P,\epsilon}\cdot\nabla_{ v  }f^{P,\epsilon}+ q_0\epsilon \left\{v\times \left(B^P+\sum\limits_{j=1}^{m-1}\epsilon^jB^j\right)\right\}\cdot\nabla_vf^{P,\epsilon}+{ L} f^{P,\epsilon}\nonumber\\
&=&\frac 1 2 q_0 E^{P,\epsilon}\cdot v f^{P,\epsilon}
 +{\Gamma}\left(f^{P,\epsilon},f^{P,\epsilon}\right),\label{f-P-sign}\\
\partial_tE^{P,\epsilon}&=&-\int_{\mathbb{R}^3_v}v\mu^{\frac 12}(v)\left(f^{P,\epsilon}_+ -f^{P,\epsilon}_-\right)dv,\nonumber\\
\nabla_x\cdot E^{P,\epsilon}&=&\int_{\mathbb{R}^3_v}\mu^{\frac 12}(v)\left(f^{P,\epsilon}_+-f^{P,\epsilon}_-\right)dv,\quad \nabla_x\times E^{P,\epsilon} =0,\nonumber
\end{eqnarray}
\begin{equation}\label{VPB-sign-IC}
\left[f^{P,\epsilon}_\pm(0,x,v), E^{P,\epsilon}(0,x)\right]=\left[f^{P,\epsilon}_{0,\pm}(x,v), E^{P,\epsilon}_0(x)\right];
\end{equation}

\item [(ii).] For $i=1,2,\cdots, m-1$, $\left[f^{i,\epsilon}(t,x,v), E^{i,\epsilon}(t,x), B^{i,\epsilon}(t,x)\right]$ with $B^{i,\epsilon}(t,x)=B^i$ solves
\begin{eqnarray} \label{f-i-vector}
&&\partial_tf^{i,\epsilon}+ v\cdot\nabla_xf^{i,\epsilon}-E^{i,\epsilon} \cdot v \mu^{1/2}q_1-\frac12 q_0 E^{i,\epsilon}\cdot v f^{P,\epsilon}+ q_0E^{i,\epsilon}\cdot\nabla_{ v  }f^{P,\epsilon}\nonumber\\
&&- \frac12 q_0E^{P,\epsilon}\cdot v f^{i,\epsilon}+q_0 E^{P,\epsilon}\cdot\nabla_{ v  }f^{i,\epsilon}-\frac12\sum_{j_1+j_2=i,\atop 0<j_1,j_2<i}q_0E^{j_1,\epsilon}\cdot vf^{j_2,\epsilon}\\
&&+\sum_{j_1+j_2=i,\atop 0<j_1,j_2<i} q_0E^{j_1,\epsilon}\cdot\nabla_vf^{j_2,\epsilon} +\epsilon q_0\left\{v\times \left(B^{P}
+\sum_{j=1}^{m-1}\epsilon^{j}B^{j}\right)\right\}\cdot\nabla_vf^{i,\epsilon} +{ L} f^{i,\epsilon}\nonumber\\
  &=&{\Gamma}\left(f^{P,\epsilon},f^{i,\epsilon}\right)
  +{\Gamma}\left(f^{i,\epsilon},f^{P,\epsilon}\right)
  +\sum_{j_1+j_2=i,\atop 0<j_1,j_2<i} \Gamma\left(f^{j_1,\epsilon},f^{j_2,\epsilon}\right),\nonumber
\end{eqnarray}
\begin{eqnarray}
\partial_tE^{i,\epsilon}&=&-{\displaystyle\int_{\mathbb{R}^3}}v\mu^{1/2}(v) \left(f_+^{i,\epsilon}-f_-^{i,\epsilon}\right)dv,\nonumber\\
\nabla_x\times E^{i,\epsilon}&=&0,\label{f-i-sign-e-b}\\
\nabla_x\cdot E^{i,\epsilon}&=&{\displaystyle\int_{\mathbb{R}^3}}\mu^{1/2}(v) \left(f_+^{i,\epsilon}-f_-^{i,\epsilon}\right)dv,\nonumber
\end{eqnarray}
and the initial condition
\begin{equation}\label{f-i-initial}
\left[f^{i,\epsilon}(0,x,v), E^{i,\epsilon}(0,x)\right]=\left[f_0^{i,\epsilon}(x,v),  E_0^{i,\epsilon}(x)\right],
\end{equation}
which satisfy the compatibility conditions
\begin{equation}\label{f-i-e-b-compatibility conditions}
  \nabla_x\cdot E_0^{i,\epsilon}=\int_{\mathbb{R}^3}\mu^{1/2}(v)\left(f_{0,+}^{i,\epsilon}-f_{0,-}^{i,\epsilon}\right)dv;
\end{equation}

\item [(iii).]  $\left[f^{m,\epsilon}(t,x,v), E^{m,\epsilon}(t,x), B^{m,\epsilon}(t,x)\right]$ solves
\begin{eqnarray} \label{f-R-vector}
&&\partial_tf^{m,\epsilon}+ v\cdot\nabla_xf^{m,\epsilon}-
  E^{m,\epsilon} \cdot v \mu^{1/2}q_1- \frac12 q_0  E^{m,\epsilon}\cdot v f^{P,\epsilon}+ q_0E^{m,\epsilon}\cdot\nabla_{ v  }f^{P,\epsilon}\nonumber\\
&&- \frac12 q_0E^{P,\epsilon}\cdot v f^{m,\epsilon}+q_0 E^{P,\epsilon}\cdot\nabla_{ v  }f^{m,\epsilon}
-\frac12q_0\sum_{j_1+j_2\geq m,\atop 0<j_1,j_2<m} \epsilon^{j_1+j_2-m}E^{j_1,\epsilon}\cdot vf^{j_2,\epsilon}\nonumber\\
&&+q_0\sum_{j_1+j_2\geq m,\atop 0<j_1,j_2<m} \epsilon^{j_1+j_2-m}E^{j_1,\epsilon}\cdot\nabla_vf^{j_2,\epsilon} -\frac12q_0\sum_{0<j_1< m} \epsilon^{j_1}E^{j_1,\epsilon}
\cdot vf^{m,\epsilon}\nonumber\\
&&+q_0\sum_{0<j_1< m}\epsilon^{j_1}E^{j_1,\epsilon}
\cdot\nabla_vf^{m,\epsilon}-q_0\frac12\sum_{0<j_1< m} \epsilon^{j_1}E^{m,\epsilon}
\cdot vf^{j_1,\epsilon}+q_0\sum_{0<j_1< m} \epsilon^{j_1}E^{m,\epsilon}\cdot\nabla_vf^{j_1,\epsilon}\\
&&+\epsilon \left(v\times B^{m,\epsilon}\right)\cdot \nabla_v \left\{f^{P,\epsilon}+\sum_{i=1}^{m-1}\epsilon^i f^{i,\epsilon}\right\}+ q_0\epsilon \left\{v\times \left(B^{P}
+\sum_{i=1}^{m-1}\epsilon^{i}B^{i}\right)\right\}\cdot\nabla_vf^{m,\epsilon}\nonumber\\
&&- \frac {\epsilon^{m}} 2 q_0E^{m,\epsilon}\cdot v f^{m,\epsilon}+q_0 \epsilon^{m} E^{m,\epsilon}\cdot\nabla_{ v  }f^{m,\epsilon}+q_0\epsilon^{m+1} v\times B^{m,\epsilon}\cdot\nabla_vf^{m,\epsilon}
+{ L} f^{m,\epsilon} \nonumber\\
&=&{\Gamma}\left(f^{P,\epsilon},f^{m,\epsilon}\right)
  +{\Gamma}\left(f^{m,\epsilon},f^{P,\epsilon}\right)
  +\epsilon^m{\Gamma}\left(f^{m,\epsilon},f^{m,\epsilon}\right) +\sum_{j_1+j_2\geq m,\atop 0<j_1,j_2\leq m}\epsilon^{j_1+j_2-m}\Gamma\left(f^{j_1,\epsilon},f^{j_2,\epsilon}\right),\nonumber
  \end{eqnarray}

\begin{eqnarray}\label{f-r-s-e-b}
\epsilon\partial_tE^{m,\epsilon}-\nabla_x\times B^{m,\epsilon}&=&-{\displaystyle\int_{\mathbb{R}^3}}v\mu^{1/2}(v) \left(f_+^{m,\epsilon}-f_-^{m,\epsilon}\right)dv,\nonumber\\
\epsilon\partial_tB^{m,\epsilon}+\nabla_x\times E^{m,\epsilon}&=&0,\\
\nabla_x\cdot E^{m,\epsilon}&=&{\displaystyle\int_{\mathbb{R}^3}}\mu^{1/2}(v) \left(f_+^{m,\epsilon}-f_-^{m,\epsilon}\right)dv,\quad \nabla_x\cdot B^{m,\epsilon}=0\nonumber
\end{eqnarray}
with prescribed initial data
\begin{equation}\label{f-r-s-initial}
\left[f^{m,\epsilon}(0,x,v), E^{m,\epsilon}(0,x), B^{m,\epsilon}(0,x)\right]=\left[f_0^{m,\epsilon}(x,v), E_0^{m,\epsilon}(x), B_0^{m,\epsilon}(x)\right],
\end{equation}
which satisfy the compatibility conditions
\begin{equation}\label{f-r-s-compatibility conditions}
  \nabla_x\cdot E_0^{m,\epsilon}=\int_{\mathbb{R}^3}\mu^{1/2}(v)\left(f_{0,+}^{m,\epsilon}-f_{0,-}^{m,\epsilon}\right)dv, \quad \nabla_x\cdot B_0^{m,\epsilon}=0.
\end{equation}
\end{itemize}

Now let $\left[f^{m,\epsilon}(t,x,v), E^{m,\epsilon}(t,x), {B}^{m,\epsilon}(t,x)\right]$ be the unique solution of the Cauchy problem \eqref{f-R-vector}-\eqref{f-r-s-compatibility conditions} and for $k,n\in\mathbb{N}$ satisfying $k\leq n$, we can define the following energy functionals $\mathcal{E}_{f^{m,\epsilon},n}(t),$ $\mathcal{E}_{f^{m,\epsilon},n,\ell,-\gamma}(t),$
$\overline{\mathcal{E}}_{f^{m,\epsilon},n,\ell,-\gamma}(t),$ $\mathcal{E}^k_{f^{m,\epsilon},n}(t),$ $\mathcal{E}^k_{f^{m,\epsilon},n,\ell,-\gamma}(t),
$ $\mathcal{E}^{(n,j)}_{f^{m,\epsilon},n,\ell,1}(t)$, $\mathcal{E}^{(n,0)}_{f^{m,\epsilon},n,\ell,1}(t)$, $\mathcal{E}^{(0,0)}_{f^{m,\epsilon},n,\ell,1}(t)$, and the corresponding energy dissipation rate functionals $\mathcal{D}_{f^{m,\epsilon},n}(t)$, $\mathcal{D}_{f^{m,\epsilon},n,\ell,-\gamma}(t)$, $\overline{\mathcal{D}}_{f^{m,\epsilon},n,\ell,-\gamma}(t)$, $\mathcal{D}^k_{f^{m,\epsilon},n}(t)$, $\mathcal{D}^{k}_{f^{m,\epsilon},n,\ell,-\gamma}(t)$, $\mathcal{D}^{(n,j)}_{f^{m,\epsilon},n,\ell,1}(t)$, $\mathcal{D}^{(n,0)}_{f^{m,\epsilon},n,\ell,1}(t)$, $\mathcal{D}^{(0,0)}_{f^{m,\epsilon},n,\ell,1}(t)$ as follows:
\begin{eqnarray}\label{def-E-n-g}
 \mathcal{E}_{f^{m,\epsilon},n}(t)&\equiv& \sum_{|\alpha|\leq n}\left\|\partial^\alpha \left[f^{m,\epsilon}(t),E^{m,\epsilon}(t),B^{m,\epsilon}(t)\right]\right\|^2, \nonumber\\
\mathcal{E}_{f^{m,\epsilon},n,\ell,-\gamma}(t)&\equiv& \sum_{|\alpha|+|\beta|\leq n}\left\|w_{\ell-|\beta|,-\gamma}\partial^\alpha_\beta f^{m,\epsilon}(t)\right\|^2+\mathcal{E}_{f^{m,\epsilon},n}(t),\nonumber\\
\overline{\mathcal{E}}_{f^{m,\epsilon},n,\ell,-\gamma}(t)&\equiv& \mathcal{E}_{f^{m,\epsilon},n,\ell,-\gamma}(t)+\left\|\Lambda^{-\varrho}\left[f^{m,\epsilon}(t), E^{m,\epsilon}(t), B^{m,\epsilon}(t)\right]\right\|^2,\nonumber\\
\mathcal{E}^k_{f^{m,\epsilon},n}(t)&\equiv& \sum_{k\leq|\alpha|\leq n}\left\|\partial^\alpha \left[f^{m,\epsilon}(t), E^{m,\epsilon}(t),B^{m,\epsilon}(t)\right]\right\|^2,\\
\mathcal{E}^k_{f^{m,\epsilon},n,\ell,-\gamma}(t)&\equiv& \sum_{|\alpha|+|\beta|\leq n,\atop |\alpha|\geq k}\left\|w_{\ell-|\beta|,-\gamma}\partial^\alpha_\beta f^{m,\epsilon}(t)\right\|^2+\mathcal{E}^k_{f^{m,\epsilon},n}(t),\nonumber\\
\mathcal{E}^{(n,j)}_{f^{m,\epsilon},n,\ell,1}(t)&\equiv&\sum_{|\alpha|+|\beta|=n} \left\|w_{\ell-|\beta|,1}\partial^\alpha_\beta{\bf\{I-P\}}f^{m,\epsilon}(t)\right\|^2, \quad 1\leq |\beta|=j\leq n, n\geq1\nonumber,\\
\mathcal{E}^{(n,0)}_{f^{m,\epsilon},n,\ell,1}(t)&\equiv&\sum_{|\alpha|=n}\left\|w_{\ell,1} \partial^\alpha f^{m,\epsilon}(t)\right\|^2,\quad n\geq 1,\nonumber\\
\mathcal{E}^{(0,0)}_{f^{m,\epsilon},0,\ell,1}(t)&\equiv& \left\|w_{\ell,1}{\bf\{I-P\}}f^{m,\epsilon}(t)\right\|^2,\nonumber
\end{eqnarray}
and
\begin{eqnarray}\label{def-D-n-g}
\mathcal{D}_{f^{m,\epsilon},n}(t)&\equiv &\sum_{|\alpha|\leq n}\left\|\partial^\alpha\{{\bf I-P}\} f^{m,\epsilon}(t)\right\|_\nu^2 +\left\|\nabla_{x}{\bf P}f^{m,\epsilon}(t)\right\|^2_{H^{n-1}_x}+\epsilon^2\left\|\nabla_x\left[E^{m,\epsilon}(t),B^{m,\epsilon}(t)\right]\right\|_{H^{n-2}_x}^2\nonumber\\
&&+\epsilon^2\left\|E^{m,\epsilon}(t)+\epsilon b^{f^{m,\epsilon}}(t)\times \left\{B^{P}
+\sum_{j=1}^{m-1}\epsilon^{j}B^{j}\right\}\right\|^2
+\left\|a^{f^{m,\epsilon}}_+(t)-a^{f^{m,\epsilon}}_-(t)\right\|^2,\nonumber\\
\mathcal{D}_{f^{m,\epsilon},n,\ell,-\gamma}(t)&\equiv& \sum_{|\alpha|+|\beta|\leq n} \left\|w_{\ell-|\beta|,-\gamma}\partial^\alpha_\beta {\bf \{I-P\}}f^{m,\epsilon}(t)\right\|^2_\nu+\mathcal{D}_{f^{m,\epsilon},n}(t)\nonumber\\
&&+\frac1{(1+t)^{1+\vartheta}}\sum_{|\alpha|+|\beta|\leq n}
  \left\|w_{\ell-|\beta|,-\gamma}\partial^\alpha_\beta{\bf\{I-P\}}f^{m,\epsilon}(t)\langle v\rangle\right\|^2,\nonumber\\
\overline{\mathcal{D}}_{f^{m,\epsilon},n,\ell,-\gamma}(t)&\equiv& \mathcal{D}_{f^{m,\epsilon},n,\ell,-\gamma}(t) +\left\|\Lambda^{-\varrho}{\bf\{I-P\}}f^{m,\epsilon}(t)\right\|_\nu^2\nonumber\\
&&\left\|\Lambda^{-\varrho}\left[a^{f^{m,\epsilon}}_+(t) -a^{f^{m,\epsilon}}_-(t)\right]\right\|^2+\left\|\Lambda^{1-\varrho}\left[{\bf P}f^{m,\epsilon}(t),\epsilon E^{m,\epsilon}(t), \epsilon B^{m,\epsilon}(t)\right]\right\|^2,\\
\mathcal{D}^k_{f^{m,\epsilon},n}(t)&\equiv &\sum_{k\leq|\alpha|\leq n}\left\|\partial^\alpha{\{\bf I-P\} } f^{m,\epsilon}(t)\right\|_\nu^2 +\left\|\nabla^{k+1}_{x}{\bf P}f^{m,\epsilon}(t)\right\|^2_{H^{n-1-k}_x}\nonumber\\ &&+\epsilon^2\left\|\nabla^{k+1}_x \left[E^{m,\epsilon}(t),B^{m,\epsilon}(t)\right]\right\|_{H^{n-2-k}_x}^2,\nonumber\\
\mathcal{D}^{k}_{f^{m,\epsilon},n,\ell,-\gamma}(t)&\equiv& \sum_{|\alpha|+|\beta|\leq n, |\alpha|\geq k}\left\|w_{\ell-|\beta|,-\gamma}\partial^\alpha_\beta {\bf \{I-P\}}f^{m,\epsilon}(t)\right\|^2_\nu+\mathcal{D}^k_{f^{m,\epsilon},n}(t)\nonumber\\
&&+\frac1{(1+t)^{1+\vartheta}}\sum_{|\alpha|+|\beta|\leq n,|\alpha|\geq k}
  \left\|w_{\ell-|\beta|,-\gamma}\partial^\alpha_\beta{\bf\{I-P\}}f^{m,\epsilon}(t)\langle v\rangle\right\|^2,\nonumber\\
 \mathcal{D}^{(n,j)}_{f^{m,\epsilon},n,\ell,1}(t)&\equiv& \sum_{|\alpha|+|\beta|=n} \left\|w_{\ell-|\beta|,1}\partial^\alpha_\beta{\bf\{I-P\}}f^{m,\epsilon}(t)\right\|_{\nu}^2  \nonumber\\
 &&+\frac1{(1+t)^{1+\vartheta}}
  \sum_{|\alpha|+|\beta|=n} \left\|w_{\ell-|\beta|,1}\partial^\alpha_\beta{\bf\{I-P\}}f^{m,\epsilon}(t)\langle v\rangle\right\|^2, 1\leq |\beta|=j\leq n, n\geq 1,\nonumber\\
\mathcal{D}^{(n,0)}_{f^{m,\epsilon},n,\ell,1}(t)&\equiv& \sum_{|\alpha|=n} \left\|w_{\ell-|\beta|,1}\partial^\alpha f^{m,\epsilon}(t)\right\|_{\nu}^2+\frac1{(1+t)^{1+\vartheta}}\sum_{|\alpha|=n}
  \left\|w_{\ell-|\beta|,1}\partial^\alpha f^{m,\epsilon}(t)\langle v\rangle\right\|^2,\  n\geq 1,\nonumber\\
\mathcal{D}^{(0,0)}_{f^{m,\epsilon},0,\ell,1}(t)&\equiv& \left\|w_{\ell-|\beta|,1}{\bf\{I-P\}}f^{m,\epsilon}(t)\right\|_{\nu}^2 +\frac1{(1+t)^{1+\vartheta}}
  \left\|w_{\ell-|\beta|,1}{\bf\{I-P\}}f^{m,\epsilon}(t)\langle v\rangle\right\|^2.\nonumber
\end{eqnarray}

Similarly, let $\left[f^{P,\epsilon}(t,x,v), E^{P,\epsilon}(t,x)\right]$ be the unique solution of the Cauchy problems \eqref{f-P-sign}, \eqref{VPB-sign-IC} and for $k, n\in\mathbb{N}$ satisfying $k\leq n$, we can also define the following energy functionals $\mathcal{E}_{f^{P,\epsilon},n}(t), \mathcal{E}^k_{f^{P,\epsilon},n}(t), \mathcal{E}_{f^{P,\epsilon},n,\ell,-\gamma}(t),$ $ \mathcal{E}^k_{f^{P,\epsilon},n,\ell,-\gamma}(t), \overline{\mathcal{E}}_{f^{P,\epsilon},n,\ell,-\gamma}(t)$, $\mathcal{E}^{(n,j)}_{f^{P,\epsilon},n,\ell,1}(t)$ and  the corresponding energy dissipation rate functionals $\mathcal{D}_{f^{P,\epsilon},n}(t),$ $\mathcal{D}^k_{f^{P,\epsilon},n}(t),$ $\mathcal{D}_{f^{P,\epsilon},n,\ell,-\gamma}(t), \mathcal{D}^k_{f^{P,\epsilon},n,\ell,-\gamma}(t), \overline{\mathcal{D}}_{f^{P,\epsilon},n,\ell,-\gamma}(t)$, $\mathcal{D}^{(n,j)}_{f^{P,\epsilon},n,\ell,1}(t)$:
\begin{eqnarray*}
\mathcal{E}^k_{f^{P,\epsilon},n}(t)&=&\sum\limits_{k\leq|\alpha|\leq n}\left\|\partial^\alpha\left[f^{P,\epsilon}(t), E^{P,\epsilon}(t)\right]\right\|^2,\\
\mathcal{E}_{f^{P,\epsilon},n}(t)&=&\mathcal{E}^0_{f^{P,\epsilon},n}(t),\\
\mathcal{E}^k_{f^{P,\epsilon},n,\ell,-\gamma}(t)&=&\sum\limits_{|\alpha|+|\beta|\leq n,|\alpha|\geq k}\left\|w_{\ell-|\beta|,-\gamma}\partial^\alpha_\beta f^{P,\epsilon}(t)\right\|^2 +\mathcal{E}^k_{f^{P,\epsilon},n}(t),\\
\mathcal{E}_{f^{P,\epsilon},n,\ell,-\gamma}(t)&=& \mathcal{E}^0_{f^{P,\epsilon},n,\ell,-\gamma}(t),\\
\overline{\mathcal{E}}_{f^{P,\epsilon},n,\ell,-\gamma}(t)&=& \mathcal{E}_{f^{P,\epsilon},n,\ell,-\gamma}(t)
+\left\|\Lambda^{-\varrho}\left[f^{P,\epsilon}(t), E^{P,\epsilon}(t)\right]\right\|^2,\nonumber\\
\mathcal{E}^{(n,j)}_{f^{P,\epsilon},n,\ell,1}(t)&\equiv&\sum_{|\alpha|+|\beta|=n} \left\|w_{\ell-|\beta|,1}\partial^\alpha_\beta{\bf\{I-P\}}f^{P,\epsilon}(t)\right\|^2, \quad 1\leq |\beta|=j\leq n, n\geq1,\nonumber
\end{eqnarray*}
and
\begin{eqnarray*}
\mathcal{D}^k_{f^{P,\epsilon},n}(t)&=&\sum\limits_{k\leq|\alpha|\leq n}\left\|\partial^\alpha\{{\bf I}-{\bf P}\}f^{P,\epsilon}(t)\right\|^2_\nu +\left\|\nabla_x{\bf P}f^{P,\epsilon}(t)\right\|^2_{H^{n-1}_x} +\left\|a_+^{f^{P,\epsilon}}(t)-a_-^{f^{P,\epsilon}}(t)\right\|^2,\\
\mathcal{D}_{f^{P,\epsilon},n}(t)&=&\mathcal{D}^0_{f^{P,\epsilon},n}(t),\\
\mathcal{D}^k_{f^{P,\epsilon},n,\ell,-\gamma}(t)&=&\sum\limits_{|\alpha|+|\beta|\leq n,|\alpha|\geq k}\left\|w_{\ell-|\beta|,-\gamma}\partial^\alpha_\beta \{{\bf I}-{\bf P}\}f^{P,\epsilon}(t)\right\|^2_\nu +\mathcal{D}^k_{f^{P,\epsilon},n}(t)\\
&&+\frac1{(1+t)^{1+\vartheta}}\sum_{|\alpha|+|\beta|\leq n,|\alpha|\geq k}
  \left\|w_{\ell-|\beta|,-\gamma}\partial^\alpha_\beta\{{\bf I}-{\bf P} \} f^{P,\epsilon}(t)\langle v\rangle\right\|^2,\\
\mathcal{D}_{f^{P,\epsilon},n,\ell,-\gamma}(t)&=& \mathcal{D}^0_{f^{P,\epsilon},n,\ell,-\gamma}(t),\\
\overline{\mathcal{D}}_{f^{P,\epsilon},n,\ell,-\gamma}(t)&=& \mathcal{D}_{f^{P,\epsilon},n,\ell,-\gamma}(t)
+\left\|\Lambda^{-\varrho}\left[f^{P,\epsilon}(t), E^{P,\epsilon}(t)\right]\right\|^2,\nonumber\\
 \mathcal{D}^{(n,j)}_{f^{P,\epsilon},n,\ell,1}(t)&\equiv& \sum_{|\alpha|+|\beta|=n} \left\|w_{\ell-|\beta|,1}\partial^\alpha_\beta{\bf\{I-P\}}f^{P,\epsilon}(t)\right\|_{\nu}^2  \nonumber\\
 &&+\frac1{(1+t)^{1+\vartheta}}
  \sum_{|\alpha|+|\beta|=n} \left\|w_{\ell-|\beta|,1}\partial^\alpha_\beta{\bf\{I-P\}}f^{P,\epsilon}(t)\langle v\rangle\right\|^2, 1\leq |\beta|=j\leq n, n\geq 1.\nonumber
\end{eqnarray*}
Moreover, for the solution $\left[f^{i,\epsilon}(t,x,v),\right.$ $\left. E^{i,\epsilon}(t,x)\right]$ $(i=1,2,\cdots,m-1)$ of the Cauchy problem
\eqref{f-i-vector}-\eqref{f-i-e-b-compatibility conditions}, the energy functionals $\mathcal{E}_{f^{i,\epsilon},n}(t), \mathcal{E}^k_{f^{i,\epsilon},n}(t), \mathcal{E}_{f^{i,\epsilon},n,\ell,-\gamma}(t),$ $ \mathcal{E}^k_{f^{i,\epsilon},n,\ell,-\gamma}(t), \overline{\mathcal{E}}_{f^{P,\epsilon},n,\ell,-\gamma}(t)$, $\mathcal{E}^{(n,j)}_{f^{i,\epsilon},n,\ell,1}(t)$ and the corresponding energy dissipation rate functionals $\mathcal{D}_{f^{i,\epsilon},n}(t), \mathcal{D}^k_{f^{i,\epsilon},n}(t),$ $\mathcal{D}_{f^{i,\epsilon},n,\ell,-\gamma}(t), \mathcal{D}^k_{f^{i,\epsilon},n,\ell,-\gamma}(t), \overline{\mathcal{D}}_{f^{i,\epsilon},n,\ell,-\gamma}(t)$,  $\mathcal{D}^{(n,j)}_{f^{i,\epsilon},n,\ell,1}(t)$ can be defined in the same way.

With the above preparations in hand, we are now ready to state our main result for cutoff soft potentials.
\begin{theorem}\label{Th1.3}
Assume that
\begin{itemize}
\item $\gamma\in(-3,-1)$, $\varrho\in\left[\frac12,\frac32\right)$ and
 $m>\frac52+\varrho$;
\item  For any given constants vectors $B^P, B^i (i=1,2,\cdots,m-1)$ independent of $\epsilon$, it holds that
\begin{eqnarray}\label{F-Soft-initial-Expansion}
F^\epsilon_0(x,v)&=&\mu+\sqrt{\mu}\left(f^{P,\epsilon}_0(x,v) +\sum_{i=1}^{m-1}\epsilon^{i}f_0^{i,\epsilon}(x,v)
+\epsilon^mf_0^{m,\epsilon}(x,v)\right)\geq 0,\nonumber\\
E^\epsilon_0(x)&=&E^{P,\epsilon}_0(x)+\sum_{i=1}^{m-1}\epsilon^{i}E_0^{i,\epsilon}(x) +\epsilon^{m} E_0^{m,\epsilon}(x),\\
B^\epsilon_0(x)&=&B^{P}+\sum_{i=1}^{m-1}\epsilon^iB^{i}+\epsilon^m B_0^{m,\epsilon}(x);\nonumber
\end{eqnarray}
\item The parameters $l_m^\sharp, l_m^*, l_i^* (i=1,2,\cdots, m-1), l_P^*$ are suitably chosen sufficiently large positive constants whose precise ranges will be specified in the proof of Theorem \ref{Th1.3}.
\end{itemize}
Then if we assume further that there exists a sufficiently small positive constant $c_1$, which is independent of $\epsilon$, such that
\begin{eqnarray}\label{inital-condtions-Total}
\mathbb{Y}_{\textrm{Total},0}&=&\sum_{|\alpha|+|\beta|\leq N^0_m}\left\|\langle v\rangle^{l_m^\sharp-|\beta|}{e^{q\langle v\rangle^2}}\partial^\alpha_\beta f_0^{m,\epsilon}\right\| +\sum_{N_0+1\leq|\alpha|+|\beta|\leq N_m+1}\left\|\langle v\rangle^{l_m^*-|\beta|}{e^{q\langle v\rangle^2}}\partial^\alpha_\beta f_0^{m,\epsilon}\right\|
\nonumber\\
&&+\sum_{|\alpha|+|\beta|\leq N_{i}+1,\atop 1\leq i\leq m-1}\left\|\langle v\rangle^{l_i^*-|\beta|}{e^{q\langle v\rangle^2}}\partial^\alpha_\beta f_0^{i,\epsilon}\right\|
+\sum_{|\alpha|+|\beta|\leq N_P}\left\|\langle v\rangle^{l^*_P-|\beta|}e^{q\langle v\rangle^2}\partial^\alpha_\beta f^{P,\epsilon}_0\right\|\\
&&+\sum_{1\leq i\leq m-1}\left\|\left[f_0^{m,\epsilon}, f_0^{i,\epsilon}, f^{P,\epsilon}_0, E_0^{m,\epsilon}, E_0^{i,\epsilon}, E^{P,\epsilon}_0,B_0^{m,\epsilon}\right]\right\|_{\dot{H}^{-\varrho}}\leq c_1,\nonumber
\end{eqnarray}
we can deduce that
\begin{itemize}
\item For $i=1,2,\cdots, m-1$, the Cauchy problems \eqref{f-P-sign}-\eqref{VPB-sign-IC}, \eqref{f-i-vector}-\eqref{f-i-e-b-compatibility conditions}, and \eqref{f-R-vector}-\eqref{f-r-s-compatibility conditions} admit unique global solutions $\left[f^{P,\epsilon}(t,x,v), E^{P,\epsilon}(t,x)\right]$, $\left[f^{i,\epsilon}(t,x,v), E^{i,\epsilon}(t,x)\right]$ and $\left[f^{m,\epsilon}(t,x,v), E^{m,\epsilon}(t,x), \right.$ $\left.B^{m,\epsilon}(t,x)\right]$, respectively;
\item The Cauchy problem \eqref{VMB}, \eqref{Maxwell}, \eqref{VMB-in}, \eqref{IC-Compatibility} admits a unique global solution $\left[F^\epsilon(t,x,v), E^\epsilon(t,x), B^\epsilon(t,x)\right]$ which satisfies
\begin{eqnarray*}
F^\epsilon_\pm(t,x,v)&=&\mu(v)+\sqrt{\mu(v)}\left(f^{P,\epsilon}_\pm(t,x,v) +\sum_{i=1}^{m-1}\epsilon^i f_{\pm}^{i,\epsilon}(t,x,v)+\epsilon^m f^{m,\epsilon}_{\pm}(t,x,v)\right)\geq 0,\nonumber\\
E^\epsilon(t,x)&=&E^{P,\epsilon}(t,x)+\sum_{i=1}^{m-1}\epsilon^{i}E^{i,\epsilon}(t,x) +\epsilon^{m} E^{m,\epsilon}(t,x),\\
B^\epsilon(t,x)&=&B^P+\sum_{i=1}^{m-1}\epsilon^iB^i+\epsilon^m B^{m,\epsilon}(t,x). \nonumber
\end{eqnarray*}
\end{itemize}
Furthermore, we have
\begin{equation}\label{Th2-2}
\sum_{|\alpha|+|\beta|\leq N_m}\left\|\partial^\alpha_\beta \left\{\frac{F^\epsilon-F^{P,\epsilon}}{\mu^{1/2}}\right\}\right\|^2 +\left\|E^\epsilon-E^{P,\epsilon}\right\|_{H^{N_m+1}_x}^2
+\left\|B^\epsilon -B^P\right\|^2_{H^{N_m+1}_x}\lesssim\epsilon^{2}\mathbb{Y}_{\textrm{Total},0}^2.
\end{equation}
\end{theorem}

\begin{remark} Some remarks concerning Theorem \ref{Th1.3} are listed below:
\begin{itemize}
\item [(i).]Although only the case of $-3<\gamma<-1$ is studied in Theorem \ref{Th1.3}, the case of $-1\leq\gamma<1$ is much more simpler and similar result holds also. In fact in such a case, one can close the desired energy type estimates by employing only one set of weighted energy method with the weight $w_{\ell-|\beta|,|\gamma|}(t,v)=\langle v\rangle ^{|\gamma|(\ell-|\beta|)} e^{\frac{q\langle v\rangle^2}{(1+t)^\vartheta}}$;
\item [(ii).] Under the assumptions imposed in Theorem \ref{Th1.3}, since we do not know whether $F^{P,\epsilon}_0(x,v)=\mu(v)+\sqrt{\mu(v)}f_0^{P,\epsilon}(x,v)$ is nonnegative or not, thus we can not deduce that $F^{P,\epsilon}(t,x,v)=\mu(v)+\sqrt{\mu(v)}f^{P,\epsilon}(t,x,v)$ is nonnegative, but we can really prove that, as the parameter $\epsilon\to 0_+$, $\left[F^{P,\epsilon}(t,x,v), E^{P,\epsilon}(t,x)\right]$ will converge strongly to the unique global-in-time solution $\left[F^\infty(t,x,v), E^\infty(t,x)\right]$ of the Cauchy problem of the Vlasov-Poisson-Boltzmann system \eqref{VPB-Original}-\eqref{Poisson-Rewritten} with prescribed initial data \eqref{VPB-IC-Original}. Here $F^\infty_0(x,v)=\lim\limits_{\epsilon\to 0_+}F^\epsilon_0(x,v)$. Since $F^\epsilon_0(x,v)\geq 0$, we can thus get that $F^\infty_0(x,v)\geq 0$ and from which one can further prove that $F^\infty(t,x,v)\geq 0$;
\item [(iii).] The arguments used in this paper can also be used to study the global-in-time Vlasov-Poisson-Boltzmann limit of the Vlasov-Maxwell-Boltzmann system \eqref{VMB}-\eqref{Maxwell} for the whole range of non-cutoff intermolecular interactions;
\item[(iv).] The parameters $l_m^\sharp, l_m^*, l_i^* (i=1,2,\cdots,m-1), l_P^*$ in Theorem \ref{Th1.3} satisfy the following relations:
\begin{itemize}
\item  $\vartheta$ satisfies \[
    0<\vartheta\leq \min\left\{\frac{\gamma-2}{4\gamma-2}\left(\frac52+\varrho\right)-1,\frac\varrho2-\frac14\right\};
    \]
\item Let $\epsilon_0>0$ be a sufficiently small fixed constant and the constants $N_P, N_i, N_m$, and $\sigma_{n,k}$ are chosen to satisfy
\begin{itemize}
\item $N_m^0\geq 5$ and $N_m+1=2N_m^0-1+[\varrho]$;
\item  $N_{i}\geq N_{i+1}+2$ for $1\leq i\leq m-1$, $N_P\geq N_1+2$;
 \item $\sigma_{n,0}=0$ for $n\leq N_m$ and $\sigma_{N_m+1,0}=\frac{1+\epsilon_0}2$, and
$
\sigma_{n,k}-\sigma_{n,k-1}=\frac{2(1+\gamma)}{\gamma-2}(1+\vartheta),  1\leq k\leq n,1\leq n\leq N_m+1;
$
\end{itemize}
\item There exist properly large constants $\widehat{l}_i, \widehat{l}_m$ and the parameters $l_P,l_i,l_m$ satisifies
\begin{itemize}
\item $\widehat{l}_m>\frac{N_P+\varrho}2$,
  $\widehat{l}_{m-1}\geq \widehat{l}_m, \widehat{l}_{j-1}\geq \widehat{l}_j$ for $2\leq j\leq m-1$, $\widehat{l}^{P}\geq \widehat{l}_1$;
    \item $\widetilde{\ell}_m\geq\frac\gamma2+\frac{(2-4\gamma)\sigma_{N_m,N_m}}{1+\varrho}$, $l_m\geq N_m+\frac12-\frac1\gamma$;
  \item $l_m^*\geq \widetilde{\ell}_m-\gamma l_m+1+\frac\gamma2$,
    $\ell_m\geq -\frac{l^*_m+1}\gamma+\frac12$, $l_m^0\geq \ell_m+\widehat{l}_m$, $l_m^\sharp\geq \widetilde{\ell}_m-\gamma l^0_m+1+\frac\gamma2$;
    \item $\widetilde{\ell}_{m-1}\geq\frac\gamma2-\frac{2\gamma\sigma_{N_{m-1},N_{m-1}}}{1+\varrho}$, $\ell_{m-1}\geq N_m^0+3-\frac{l_m^\sharp+1-N_m^0}\gamma$, $l_{m-1}\geq\ell_{m-1}+\widehat{\ell}_{m-1}$ and $l_{m-1}^*\geq \widetilde{\ell}_{m-1}-\gamma l_{m-1}$;
  \item $\widetilde{\ell}_{j}\geq\frac\gamma2-\frac{2\gamma\sigma_{N_{j},N_{j}}}{1+\varrho}$, $\ell_{j}\geq-\frac{l_{j+1}^*}\gamma+\frac12-\frac 2{\gamma}$, $l_{j}\geq \ell_j+\widehat{l}_j$  and $l_j^*\geq \widetilde{\ell}_{j}-\gamma l_{j}$ for $1\leq j\leq m-2$;
       \item $\widetilde{\ell}^{P}\geq\frac\gamma2-\frac{2\gamma\sigma_{N^{P},N^{P}}}{1+\varrho}$, $\ell_{P}\geq-\frac{l_{1}^*}\gamma+\frac12-\frac 2{\gamma}$, $l_P\geq \ell_P+\widehat{l}_P$  and $l_P^*\geq \widetilde{\ell}_{P}-\gamma l_{P}$.
        \end{itemize}
\end{itemize}
\end{itemize}
\end{remark}

\subsection{Main ideas}
Now we outline our main ideas to deduce our main results, Theorem \ref{Th1.1} and Theorem \ref{Th1.3}. Our method is based on the nonlinear energy method developed recently in \cite{Guo-IUMJ-2004, Liu-Yang-Yu-Physica D-2004, Liu-Yu-CMP-2004} to deal with the global solvability of the Boltzmann equation in the perturbative framework. Such a method is based on the so-called macro-microscopic decompositions of the solution of the Boltzmann equation and the equation itself and it has been proved to be effective to be used to yield the global solvability of some complex kinetic equations, cf. \cite{Duan-Lei-Yang-Zhao-CMP-2017, Duan_Liu-Yang_Zhao-KRM-2013, Duan-Stain-CPAM-2011, Fan-Lei-Liu-Zhao-SCM-2017, Guo-Invent Math-2003, Jang-ARMA-2009, Lei-Zhao-JDE-2016, Li-Yang-Zhong-SIMA-2016, Strain-CMP-2006, Strain-Guo-ARMA-2008} for the  Vlasov-Maxwell-Boltzmann system \eqref{VMB}, \eqref{Maxwell} and
 \cite{Duan-Liu-CMP-2013, Duan-Strain-ARMA-2011, Duan_Yang_Zhao-JDE-2012, Duan_Yang_Zhao-M3AS-2013,  Guo-CPAM-2002, Li-Yang-Zhong-IUMJ-2016, Wang-Xiao-Xiong-Zhao-AMS-2016, Wang-JDE-2013,  Xiao-Xiong-Zhao-SCM-2014, Xiao-Xiong-Zhao-JDE-2013, Xiao-Xiong-Zhao-JFA-2016, Yang-Yu-CMP-2011, Yang-Yu-Zhao-ARMA-2006, Yang-Zhao-CMP-2006} for the Vlasov-Poisson-Boltzmann system \eqref{VPB-Original}, \eqref{Poisson-Rewritten}, etc.

Since the main purpose of this paper is to give a rigorous mathematical justification of the global-in-time limit from the Vlasov-Maxwell-Boltzmann system \eqref{VMB}, \eqref{Maxwell} to the Vlasov-Poisson-Boltzmann system \eqref{VPB-Original},
we need to deduce certain {\it a priori} estimates on the solution $\left[f^\epsilon_+(t,x,v), f^\epsilon_-(t,x,v), E^\epsilon(t,x), \widetilde{B}^\epsilon(t,x)\right]$ of the Cauchy problem of \eqref{f}, \eqref{f-initial}, \eqref{compatibility conditions} which are independent of the parameter $\epsilon=\frac 1c$.

For Theorem \ref{Th1.1}, the main difficulty is due to degeneracy of the dissipative effect of the electromagnetic field $\left[E^\epsilon(t,x),\widetilde{B}^\epsilon(t,x)\right]$ for small $\epsilon=\frac 1c$ and such a fact is reflected in the expression of the energy dissipation rate functional ${\mathcal{D}}_{n,f^\epsilon}(t)$ defined by \eqref{Energy-dissipation-rate-functional-VMB}, i.e. the term
\[
\epsilon^2\left\|\left\{E^\epsilon(t)+\epsilon b^{f^\epsilon}(t)\times \mathfrak{B}\right\}\right\|^2+\epsilon^2\left\|\left[\nabla_xE^\epsilon(t),
\nabla_x\widetilde{B}^\epsilon(t)\right]\right\|_{H^{m-2}_x}^2
\]
in the right hand side of \eqref{Energy-dissipation-rate-functional-VMB}.

Such a degeneracy effect does lead to some difficulties on the energy types estimate related to the term $E^\epsilon\cdot v\mu^{\frac12} q_1$ in \eqref{f gn} induced by $E^\epsilon\cdot\nabla_v F^\epsilon$. In fact, if one performs the energy estimates directly to \eqref{f gn}, such a term does not cause any inconvenience since
$$
-\left(\partial^\alpha\left(E^\epsilon\cdot v\mu^{\frac12}q_1\right), \partial^\alpha f^\epsilon\right)=\frac{d}{dt}\left\{\left\|\partial^\alpha E^\epsilon\right\|^2+\left\|\partial^\alpha \widetilde{B}^\epsilon\right\|^2\right\},$$
while for the corresponding estimates on the terms in which mix-derivative $\partial^\alpha_\beta$ on $\{{\bf I-P}\}f^\epsilon$ is involved, one has to perform the microscopic projection $\{{\bf I-P}\}$ to \eqref{f gn} and noticing that $E^\epsilon\cdot v\mu^{\frac12} q_1$ remains unchanged under such a microscopic projection, thus one has to treat the linear term
\[
\left(\partial^\alpha_\beta\left(E^\epsilon\cdot v\mu^{\frac12}q_1\right), \partial^\alpha _{\beta} \{{\bf I-P}\}f^\epsilon\right)
\]
the argument used in \cite{Duan-Lei-Yang-Zhao-CMP-2017, Duan_Liu-Yang_Zhao-KRM-2013, Duan-Stain-CPAM-2011, Guo-Invent Math-2003, Strain-CMP-2006} can not be applied any more to yield the desired $\epsilon-$independent estimates. To overcome such a difficulty, we note that
\begin{equation*}
\{{\bf I-P }\} \left\{E^\epsilon \cdot v \mu^{1/2}q_1-q_0\epsilon(v\times\mathfrak{B})\cdot\nabla_v {\bf P}f^\epsilon\right\}=\left\{E^\epsilon+\epsilon b^{f^\epsilon}\times \mathfrak{B}\right\} \cdot v \mu^{1/2}q_1,
\end{equation*}
one has from Lemma \ref{E-estimates} that
\begin{eqnarray}\label{intro-1}
  &&\left(\partial_{\beta}^{\alpha}\left\{q_1\left\{E^\epsilon+\epsilon b^{f^\epsilon}\times \mathfrak{B}\right\}\cdot v \mu^{1/2}\right\}, \partial^\alpha_\beta{\{\bf I-P\}}f^\epsilon\right)\nonumber\\
  &=&(-1)^{|\beta|}\sum\limits_{i=1}^3\left(\partial^{\alpha}\left\{E^\epsilon+\epsilon b^{f^\epsilon}\times \mathfrak{B}\right\}_i,\left\langle \partial_{2\beta}\left\{v_i \mu^{1/2}\right\}, \partial^\alpha{\{\bf I_+-P_+\}}f^\epsilon-\partial^\alpha{\{\bf I_--P_-\}}f^\epsilon\right\rangle\right)\\
   &\lesssim&\frac1{4C_\beta}\sum_{i=1}^3\frac{d}{dt}\left\|\left\langle \partial_{2\beta}\left\{v_i \mu^{1/2}\right\}, \partial^\alpha{\{\bf I_+-P_+\}}f^\epsilon(t)-\partial^\alpha{\{\bf I_--P_-\}}f^\epsilon(t)\right\rangle\right\|^2\nonumber\\
  &&+\left\|\nabla^{|\alpha|}\{{\bf I-P}\}f^\epsilon(t)\right\|^2+\left\|\nabla^{|\alpha|+1}f^\epsilon(t)\right\|^2
  +\mathcal{E}_{3, f^\epsilon}(t)\mathcal{D}_{3,f^\epsilon}(t).\nonumber
\end{eqnarray}
and then such a term can be controlled suitably. Here $C_\beta$ are some positive constants and $\mathcal{E}_{3, f^\epsilon}(t), \mathcal{D}_{3,f^\epsilon}(t)$ are defined by \eqref{Energy-functional-VMB} and \eqref{Energy-dissipation-rate-functional-VMB}, respectively.

As for the proof of Theorem \ref{Th1.3}, since our main purpose is to cover the whole range of intermolecular interactions, the dissipative effect induced by the coercivity estimate \eqref{coercive-estimates} of the linearized Boltzmann collision operator $L$ for non-hard sphere model is weaker than the hard sphere model and such a dissipative effect is not sufficient to control the corresponding nonlinear terms induced by the Lorentz forces, which can lead to the velocity growth at the rate of $|v|$. To overcome such a difficulty, a general strategy developed in \cite{Duan-Lei-Yang-Zhao-CMP-2017, Duan_Liu-Yang_Zhao-KRM-2013} is to introduce some weights of the form $w_{\ell-|\beta|,\kappa}(t,v)=\langle v\rangle ^{\kappa(\ell-|\beta|)} e^{\frac{q\langle v\rangle^2}{(1+t)^\vartheta}}$ to compensate the degeneracy coercivity property of the linearized Boltzmann collision operator $Lf$.

As a consequence of the introduction of the weight $w_{\ell-|\beta|,\kappa}(t,v)$, we can now deduce the following two types of dissipative mechanisms for the non-fluid component $\{{\bf I}-{\bf P}\}f^{m,\epsilon}$:
\begin{itemize}
\item The first one is the dissipative term
\begin{equation}\label{dissipative-L}
D^L_{\ell-|\beta|,\kappa}(f^{m,\epsilon})\equiv \left\|w_{\ell-|\beta|,\kappa}\partial^\alpha_\beta\{{\bf I}-{\bf P}\}f^{m,\epsilon}\right\|^2_\nu
\end{equation}
which is due to the coercive estimate of $L$;

\item The second type is the extra dissipative term
\begin{equation}\label{dissipative-weight}
D^W_{\ell-|\beta|,\kappa}(f^{m,\epsilon})\equiv (1+t)^{-1-\vartheta}\left\|\langle v\rangle w_{\ell-|\beta|,\kappa}\partial^\alpha_\beta\{{\bf I}-{\bf P}\}f^{m,\epsilon}\right\|^2
\end{equation}
induced by the exponent factor of the weight $w_{\ell-|\beta|,\kappa}(t,v)$.
\end{itemize}

With the above two types of dissipative effects on the non-fluid component $\{{\bf I}-{\bf P}\}f^{m,\epsilon}$, the nonlinear terms related to the Lorentz force can be estimates suitably as in \cite{Duan-Lei-Yang-Zhao-CMP-2017, Duan_Liu-Yang_Zhao-KRM-2013}
if one can show that the electromagnetic fields $\left[E^{m,\epsilon}(t,x), B^{m,\epsilon}(t,x)\right]$ enjoy certain temporal decay estimates. We note, however, that, since the influence of the light speed, if we employ the interpolation technique introduced in \cite{Guo-CPDE-2012} to yield the desired temporal decay estimates, we can only deduce that the temporal decay estimates of the electromagnetic fields $\left[E^{m,\epsilon}(t,x), B^{m,\epsilon}(t,x)\right]$ do depend on light speed. In fact one can only show that the electric field $E^{m,\epsilon}(t,x)$ enjoy the following temporal decay rates
\begin{equation}\label{intro-decay-1}
  \left\|\nabla^k E^{m,\epsilon}(t)\right\|^2\lesssim \epsilon ^{-2(k+\varrho)}(1+t)^{-(k+\varrho)}
\end{equation}
and such a decay estimate can not guarantee that the corresponding term related to  $E^{m,\epsilon}\cdot v f^{m,\epsilon}$ can be bounded suitably which is due to the appearance of the factor $\epsilon ^{-2(k+\varrho)}$ in \eqref{intro-decay-1}.

Our idea to go on is to use the ansatz \eqref{F-S-Expansion} to expand $\left[F^{\epsilon}(t,x,v), E^{\epsilon}(t,x), B^{\epsilon}(t,x)\right]$ to certain order. Our observation is that if $m$, the order of expansion, is taken suitably large, such an  expansion \eqref{F-S-Expansion} will leads to the appearance of the factor $\epsilon^m$ in front of $E^{m,\epsilon}\cdot v f^{m,\epsilon}$ and such a term now can be estimated as follows:
\begin{eqnarray}\label{intro-4}
&&\left(\epsilon^{m}E^{m,\epsilon}\cdot v\{{\bf I-P}\}f^{m,\epsilon}, w^2_{\ell,\kappa}\{{\bf I-P}\}f^{m,\epsilon}\right)\nonumber\\
&\lesssim&\epsilon^{m}\left\|E^{m,\epsilon}(t)\right\|_{L^\infty_x} \left\|w_{\ell,\kappa}\{{\bf I-P}\}f^{m,\epsilon}(t)\langle v\rangle\right\|^2\nonumber\\
&\lesssim&\epsilon^{m}\epsilon^{-\varrho-\frac32}(1+t)^{-\frac34-\frac\varrho2} \left\|w_{\ell,1}\{{\bf I-P}\}f^{m,\epsilon}(t)\langle v\rangle\right\|^2\\
&\lesssim&\epsilon^{m-\varrho-\frac32}(1+t)^{-1-\vartheta} \left\|w_{\ell,\kappa}\{{\bf I-P}\}f^{m,\epsilon}(t)\langle v\rangle\right\|^2\nonumber\\
&\lesssim&(1+t)^{-1-\vartheta} \left\|w_{\ell,\kappa}\{{\bf I-P}\}f^{m,\epsilon}(t)\langle v\rangle\right\|^2\nonumber
\end{eqnarray}
if $m\geq \rho+\frac 32$. Consequently, the above term can be absorbed by the extra dissipative term $D^W_{\ell-|\beta|,\kappa}(f^{m,\epsilon})$ defined by \eqref{dissipative-weight}, which is induced by the exponential part of the s $w_{\ell-|\beta|,\kappa}(t,v)$.

Although the above analysis has shown the power of the above-mentioned weighted energy method based on the weight functions $w_{\ell-|\beta|,\kappa}(t,v)$ in studying the global solvability of the Cauchy problem
\eqref{f-R-vector}-\eqref{f-r-s-compatibility conditions}, we note, however, that, as pointed out in \cite{Duan-Lei-Yang-Zhao-CMP-2017, Duan_Liu-Yang_Zhao-KRM-2013}, we can not hope to close the desired energy estimates by using a single weight $w_{\ell-|\beta|,\kappa}(t,v)$ with some specifically chosen constant $\kappa$. More precisely, we can not hope to control the terms related both to the nonlinear term induced by the magnetic field $B^{m,\epsilon}(t,x)$
\begin{equation}\label{Nonlinear-term-magnetic-field}
\left(\left(v\times \partial^{\alpha-e_j} B^{m,\epsilon}\right)\cdot\nabla_v\partial^{\alpha+e_j}\{{\bf I-P}\}f^{m,\epsilon}, w_{\ell-|\beta|,\kappa}\partial^\alpha_\beta\{{\bf I-P}\}f^{m,\epsilon}\right)
\end{equation}
and to the linear transport term $v\cdot\nabla_xf^{m,\epsilon}(t,x,v)$
\begin{equation}\label{Transport-term}
\left(\partial^{\alpha}_\beta \left\{v\cdot\nabla_x\{{\bf I-P}\}f^{m,\epsilon}\right\},w_{\ell-|\beta|,\kappa}\partial^\alpha_\beta\{{\bf I-P}\}f^{m,\epsilon}\right)
\end{equation}
simultaneously by using the dissipative effects \eqref{dissipative-L} and \eqref{dissipative-weight} for all $-3<\gamma\leq 1$ and for just one suitably chosen
constant $\kappa$.

In fact, for $-3<\gamma<0$, if we want to control \eqref{Transport-term} suitably first by \eqref{dissipative-L}, one has to choose $\kappa=-\gamma$, but if we try to use \eqref{dissipative-weight} to bound \eqref{Nonlinear-term-magnetic-field} with $\kappa=-\gamma$, we had to assume that $\gamma\geq -1$.

Our main idea, as in \cite{Duan-Lei-Yang-Zhao-CMP-2017, Duan_Liu-Yang_Zhao-KRM-2013}, is to introduce the following two sets of weighted energy estimates:
\begin{itemize}
\item [(i).] The weighted energy method with respect to the weight $w_{\ell-|\beta|,-\gamma}(t,v)$ is used to deduce the necessary temporal decay estimates on $\left[f^{m,\epsilon}(t,x,v), E^{m,\epsilon}(t,x), B^{m,\epsilon}(t,x)\right]$ together with the uniform-in-time weighted energy type estimates;
\item[(ii).] The weighted energy method with respect to the weight $w_{\ell-|\beta|,1}(t,v)$ is used to yield the time increase rates on the weighted energy type estimates on $\left[f^{m,\epsilon}(t,x,v), E^{m,\epsilon}(t,x), B^{m,\epsilon}(t,x)\right]$.
\end{itemize}
By employing these two sets of weighted energy methods and by suitably interpolations between these two sets of weighted norms, we can indeed obtain the desired energy type estimates and show that the solution $\left[f^{m,\epsilon}(t,x,v), E^{m,\epsilon}(t,x), B^{m,\epsilon}(t,x)\right]$ satisfies the desired temporal decay estimates like \eqref{intro-decay-1}.

At last, similar to that of hard sphere model, we had to deal with the estimates related to the term $E^{m,\epsilon}\cdot v\mu^{\frac12} q_1$, which is due to the degeneracy of the dissipative effect of the electromagnetic field $\left[E^{m,\epsilon}(t,x),B^{m,\epsilon}(t,x)\right]$ for small $\epsilon=\frac 1c$.
Compared with the hard sphere model, the difference now is that we had to treat the following weighted estimate:
$$
\left(\partial_{\beta}^{\alpha}\left\{q_1E^{m,\epsilon}\cdot v \mu^{1/2}\right\}, w^2_{\ell-|\beta|,\kappa}\partial^\alpha_\beta{\{\bf I-P\}}f^{m,\epsilon}\right).
$$

To control such a term, we note that
 \begin{equation}\label{intro-6}
\{{\bf I-P }\} \left\{E^{m,\epsilon} \cdot v \mu^{1/2}q_1-q_0\epsilon(v\times\mathfrak{B})\cdot\nabla_v {\bf P}f^{m,\epsilon}\right\}=\left\{E^{m,\epsilon}+\epsilon b^{f^{m,\epsilon}}\times \mathfrak{B}\right\} \cdot v \mu^{1/2}q_1
\end{equation}
and
\begin{equation}\label{intro-7}
E^{m,\epsilon} \cdot v \mu^{1/2}q_1-q_0\epsilon(v\times\mathfrak{B})\cdot\nabla_v {\bf P}f^{m,\epsilon}=\left\{E^{m,\epsilon}+\epsilon b^{f^{m,\epsilon}}\times \mathfrak{B}\right\} \cdot v \mu^{1/2}q_1,
\end{equation}
we can therefore get from Lemma \ref{lemma-nonhard-1} that
\begin{eqnarray}\label{intro-8}
&&\left(\partial_{\beta}^\alpha\left\{q_1\left\{E^{m,\epsilon}+\epsilon b^{f^{m,\epsilon}}\times \left\{B^P
+\sum_{k=1}^{m-1}\epsilon^kB^k\right\}\right\}\cdot v \mu^{1/2}\right\}, w^2_{\ell-|\beta|,\kappa}\partial^\alpha_\beta{\{\bf I-P\}}f^{m,\epsilon}\right)\nonumber\\
&\lesssim&\frac1{4C_\beta}\sum_{i=1}^3\frac{d}{dt}\left\|\left\langle \partial_{\beta}\left\{w^2_{\ell-|\beta|,\kappa}\partial_\beta\left[v_i \mu^{1/2}\right]\right\}, \partial^\alpha{\{\bf I_+-P_+\}}f^{m,\epsilon}(t)- \partial^\alpha{\{\bf I_--P_-\}}f^{m,\epsilon}(t)\right\rangle\right\|^2\\
&&+\left\|\nabla^{|\alpha|}\{{\bf I-P}\}f^{m,\epsilon}(t)\right\|_\nu^2 +\left\|\nabla^{|\alpha|+1}f^{m,\epsilon}(t)\right\|^2_\nu
  +\mathcal{E}_{f^{m,\epsilon},N}(t)\mathcal{D}_{f^{P,\epsilon},N}(t)\nonumber\\
&&+\sum_{j_1+j_2\geq m,\atop 0<j_1,j_2\leq m}\mathcal{E}_{f^{j_1,\epsilon},N}(t) \mathcal{D}_{f^{j_2,\epsilon},N}(t)\nonumber
\end{eqnarray}
and
\begin{eqnarray}\label{intro-9}
&&\left(\partial^{\alpha}\left\{q_1\left\{E^{m,\epsilon}+\epsilon b^{f^{m,\epsilon}}\times \left\{B^P
+\sum_{k=1}^{m-1}\epsilon^kB^k\right\}\right\}\cdot v \mu^{1/2}\right\}, w^2_{\ell-|\beta|,\kappa}\partial^\alpha f^{m,\epsilon}\right)\nonumber\\
   &\lesssim&\frac1{4{C_0}}\frac{d}{dt}\sum_{i=1}^3\left\|\left\langle \left\{w^2_{\ell-|\beta|,\kappa}\left[v_i \mu^{1/2}\right]\right\}, \partial^\alpha  f_+^{m,\epsilon}(t)-\partial^\alpha f_-^{m,\epsilon}(t)\right\rangle\right\|^2\\
   &&+\left\|\nabla^{|\alpha|}f^{m,\epsilon}(t)\right\|_\nu^2 +\left\|\nabla^{|\alpha|+1}f^{m,\epsilon}(t)\right\|^2_\nu
  +\mathcal{E}_{f^{m,\epsilon},N}(t)\mathcal{D}_{f^{P,\epsilon},N}(t)\nonumber\\
&&+\sum_{j_1+j_2\geq m,\atop 0<j_1,j_2\leq m}\mathcal{E}_{f^{j_1,\epsilon},N}(t)\mathcal{D}_{f^{j_2,\epsilon},N}(t),\nonumber
  \end{eqnarray}
which do not include $\left\|\partial^\alpha E^{m,\epsilon}(t)\right\|$ any more,
such that the terms in the right hand side of the above inequalities can be absorbed by the corresponding dissipation functionals. Here $\mathcal{E}_{f^{j,\epsilon},N}(t) (j=1,2,\cdots,m), \mathcal{E}_{f^{P,\epsilon},N}(t)$, and $\mathcal{D}_{f^{j,\epsilon},N}(t) (j=1,2,\cdots, m), \mathcal{D}_{f^{P,\epsilon},N}(t)$ are defined by \eqref{def-E-n-g} and \eqref{def-D-n-g}, respectively.

The rest of this paper is organized as follows.  Section 2 and Section 3 are devoted to the proofs of Theorem \ref{Th1.1} and Theorem \ref{Th1.3} respectively. Some preliminaries and some tedious but standard calculations will be given in Appendix A and Appendix B for brevity.



\section{The proof of Theorem \ref{Th1.1}}
In this section, we prove Theorem \ref{Th1.1}. For this purpose, suppose that $\left[f^\epsilon(t,x,v), E^\epsilon(t,x), \widetilde{B}^\epsilon(t,x)\right]$ is a unique solution of the Cauchy problem (\ref{f}), (\ref{f-initial}), \eqref{compatibility conditions} defined on the strip $\Pi_T=[0,T]\times \mathbb{R}^3\times\mathbb{R}^3$ for some given positive constant $T>0$, we now turn to deduce some energy type {\it a priori} estimates on $\left[f^\epsilon(t,x,v), E^\epsilon(t,x), \widetilde{B}^\epsilon(t,x)\right]$ in terms of the corresponding initial data $\left[f^\epsilon_0(x,v), E^\epsilon_0(x), \widetilde{B}^\epsilon_0(x)\right]$.

To make the presentation easy to follow, we divide this section into several subsections and the first one is on the macro-structure for $\left[f^\epsilon(t,x,v), E^\epsilon(t,x), \widetilde{B}^\epsilon(t,x)\right]$.

\subsection{Macro-structure for $\left[f^\epsilon(t,x,v), E^\epsilon(t,x), \widetilde{B}^\epsilon(t,x)\right]$}

We first list below some identities concerning the macroscopic part of $f^\epsilon(t,x,v)$, which can be obtained by recalling the macro-micro decomposition $(\ref{macro-micro})$ introduced in \cite{Guo-IUMJ-2004}, multiplying (\ref{f}) by some $\mu^{\frac12}-$type terms and by integrating the resulting identities with respect to $v$ over $\mathbb{R}^3$.
\begin{lemma}\label{lemma-macro-1}
We can get that
\begin{eqnarray}\label{Macro-equation}
&&\partial_t\left(\frac{a^{f^\epsilon}_++a^{f^\epsilon}_-}2\right)+\nabla_x\cdot b^{f^\epsilon}=0,\nonumber\\
&&\partial_tb^{f^\epsilon}_i+\partial_i\left(\frac{a^{f^\epsilon}_+ +a^{f^\epsilon}_-}2+2c^{f^\epsilon}\right)
+\frac12\sum\limits_{j=1}^3\partial_j\mathbb{A}_{ij}(\{{\bf I-P}\}{f^\epsilon}\cdot [1,1])=\frac{a^{f^\epsilon}_+-a^{f^\epsilon}_-}{2}E^{\epsilon}_i
+\epsilon\left[G^\epsilon\times \mathfrak{B}\right]_i+\epsilon\left[G^\epsilon\times \widetilde{B}^\epsilon\right]_i,\nonumber\\
&&\partial_tc^{f^\epsilon}+\frac13\nabla_x\cdot b^{f^\epsilon}+\frac56\sum\limits_{i=1}^3\partial_i \mathbb{B}_i\left(\{{\bf I-P}\}{f^\epsilon}\cdot [1,1]\right)=\frac16 G^\epsilon\cdot E^{\epsilon}\\
&&\partial_t\left(a^{f^\epsilon}_+-a^{f^\epsilon}_-\right)+\nabla_x\cdot G^\epsilon=0,\nonumber\\
&&\partial_tG^\epsilon+\nabla_x\left(a^{f^\epsilon}_+-a^{f^\epsilon}_-\right) -2E^{\epsilon}+\nabla_x\cdot \mathbb{A}\left(\{{\bf I-P}\}{f^\epsilon}\cdot q_1\right)\nonumber\\
&&\qquad= E^{\epsilon}\left(a^{f^\epsilon}_++a^{f^\epsilon}_-\right)+2\epsilon b^{f^\epsilon}\times \mathfrak{B}+2\epsilon b^{f^\epsilon}\times \widetilde{B}^\epsilon
+\left\langle [v,-v]\mu^{1/2},L{f^\epsilon}
+{\Gamma}_\pm\left(f^\epsilon,f^\epsilon\right)\right\rangle\nonumber
\end{eqnarray}
and
\begin{eqnarray}\label{Micro-equation1}
&&\frac12\partial_t\mathbb{A}_{ij}\left(\{{\bf I-P}\}{f^\epsilon}\cdot [1,1]\right)+\partial_jb^{f^\epsilon}_i+\partial_ib^{f^\epsilon}_j
-\frac23\delta_{ij}\nabla_x\cdot b^{f^\epsilon}
-\frac53\delta_{ij}\nabla_x\cdot \mathbb{B}\left(\{{\bf I-P}\}{f^\epsilon}\cdot [1,1]\right)\nonumber\\
&=&\frac12\mathbb{A}_{ij}\left(r^\epsilon_++r^\epsilon_-+g^\epsilon_+ +g^\epsilon_-\right)-\frac13\delta_{ij}G^\epsilon\cdot E^{\epsilon},\\
&&\frac12\partial_t \mathbb{B}_{j}\left(\{{\bf I-P}\}{f^\epsilon}\cdot [1,1]\right)+\partial_jc^{f^\epsilon}=\frac12\mathbb{B}_{j}\left(r^\epsilon_++r^\epsilon_- +g^\epsilon_++g^\epsilon_-\right),\nonumber
\end{eqnarray}
\end{lemma}
where
\begin{eqnarray}\label{def-r-g}
 r^\epsilon_\pm&=&- v\cdot\nabla_x\{{\bf I_\pm-P_\pm}\}f^\epsilon-{L}_\pm f^\epsilon,\nonumber\\
g^\epsilon_\pm&=&{\Gamma}_\pm\left(f^\epsilon,f^\epsilon\right)
\pm \frac1 2 E^\epsilon\cdot v f^\epsilon_{\pm}\mp E^\epsilon\cdot\nabla_{ v  }f^\epsilon_{\pm}\mp\epsilon \left(v\times \mathfrak{B}\right)\cdot\nabla_vf^\epsilon_{\pm}\mp\epsilon \left(v\times \widetilde{B}^\epsilon\right)\cdot\nabla_vf^\epsilon_{\pm},\nonumber\\
\mathbb{A}_{mj}(g)&=&\int_{{\mathbb{R}}^3}\left( v_m v_j-1\right)\mu^{1/2}gd v,\\
\mathbb{B}_j(g)&=&\frac{1}{10}\int_{{\mathbb{R}}^3}\left(| v|^2-5\right) v_j\mu^{1/2}gd v,\nonumber\\
G^\epsilon&=&\left\langle v\mu^{1/2},\{{\bf I-P}\}f^\epsilon \cdot q_1 \right\rangle.\nonumber
\end{eqnarray}

The following lemma is concerned with the dissipative effect of the macroscopic quantities such as ${\bf P}f^\epsilon(t,x,v)$ and $a^{f^\epsilon}_+(t,x)-a^{f^\epsilon}_-(t,x)$, which can be proved by repeating the argument used in the proof of {\cite[Lemma 3.2]{Lei-Zhao-JFA-2014}}, we thus omit the details for brevity.
\begin{lemma}\label{lemma-f-epsilon-1}
For $N\geq 2$,
there exists $G_{f^\epsilon}(t)$ satisfying
  \[
  G_{f^\epsilon}(t)\lesssim \sum_{|\alpha|\leq N}\left\|\partial^\alpha f^\epsilon(t)\right\|^2
  \]
such that
  \begin{eqnarray}\label{pf-e}
    &&\frac{d}{dt}G_{f^\epsilon}(t)+\sum_{1\leq|\alpha|\leq N}\left\|\partial^\alpha{\bf P}f^\epsilon(t)\right\|^2 +\left\|a^{f^\epsilon}_+(t)-a^{f^\epsilon}_-(t)\right\|^2\\ \nonumber
    &\lesssim& \sum_{|\alpha|\leq N}\left\|\partial^\alpha\{{\bf I-P}\}f^\epsilon(t)\right\|^2  +\mathcal{E}_{N,f^\epsilon}(t)\mathcal{D}_{N,f^\epsilon}(t)
  \end{eqnarray}
holds for all $0\leq t\leq T$.
\end{lemma}

Our next result focuses on the dissipative effect of the electromagnetic field $\left[E^{\epsilon}(t,x), \widetilde{B}^\epsilon(t,x)\right]$.
\begin{lemma}\label{lemma-E-B-R-1}
For $N\geq2$,
there exists $G_{E^\epsilon,\widetilde{B}^\epsilon}(t)$ satisfying
\[
G_{E^\epsilon,\widetilde{B}^\epsilon}(t)\lesssim\sum_{|\alpha|\leq N}\left\|\partial^\alpha\left[f^\epsilon,E^\epsilon,\widetilde{B}^\epsilon\right](t)\right\|^2
\]
such that the following estimate
\begin{eqnarray}\label{lemma-E-B}
  &&\frac{d}{dt}G_{E^\epsilon,\widetilde{B}^\epsilon}(t)+\epsilon^2\left\|E^\epsilon(t)+\epsilon b^{f^\epsilon}(t)\times \mathfrak{B}\right\|^2+\epsilon^2\sum_{1\leq|\alpha|\leq N-1}\left\|\partial^\alpha \left[E^\epsilon(t),\widetilde{B}^\epsilon(t)\right]\right\|^2\\
&\lesssim&\sum_{|\alpha|\leq N}\left\|\partial^\alpha\{{\bf I-P}\}f^\epsilon(t)\right\|^2 +\mathcal{E}_{N, f^\epsilon}(t)\mathcal{D}_{N,f^\epsilon}(t)\nonumber
\end{eqnarray}
holds for all $0\leq t\leq T$.
\end{lemma}
\begin{proof}
Since
\begin{eqnarray*}
2E^{\epsilon}+2\epsilon b^{f^\epsilon}\times \mathfrak{B}&=&\partial_tG^\epsilon+\nabla_x\left(a^{f^\epsilon}_+ -a^{f^\epsilon}_-\right)+\nabla_x\cdot \mathbb{A}\left(\{{\bf I-P}\}f^\epsilon\cdot q_1\right) -E^\epsilon\left(a^{f^\epsilon}_++a^{f^\epsilon}_-\right)\nonumber\\
&&-2\epsilon b^{f^\epsilon}\times \widetilde{B}^\epsilon-\left\langle [v,-v]\mu^{1/2}, Lf^\epsilon+\Gamma\left(f^\epsilon,f^\epsilon\right)\right\rangle,
\end{eqnarray*}
multiplying the above identity by $E^\epsilon+\epsilon b^{f^\epsilon}\times \mathfrak{B}$ and by integrating the result with respect to $x$ over $\mathbb{R}_x^3$, we can get from Lemma \ref{lemma-nonlinear} that
\begin{eqnarray*}
&&2\epsilon^2\left\|E^\epsilon+\epsilon b^{f^\epsilon}\times \mathfrak{B}\right\|^2=2\epsilon^2\left(\left\{E^\epsilon+\epsilon b^{f^\epsilon}\times \mathfrak{B}\right\},\left\{E^\epsilon+\epsilon b^{f^\epsilon}\times \mathfrak{B}\right\}\right)\\
&\lesssim&\epsilon^2\left(\partial_t G^\epsilon(t),\left\{E^\epsilon+\epsilon b^{f^\epsilon}\times \mathfrak{B}\right\}\right)+\epsilon^2\left(\left\{\nabla_x\left(a^{f^\epsilon}_+-a^{f^\epsilon}_-\right)+\nabla_x\cdot \mathbb{A}\left(\{{\bf I-P}\}f^\epsilon(t)\cdot q_1\right)\right\},\left\{E^\epsilon+\epsilon b^{f^\epsilon}\times \mathfrak{B}\right\}\right)\\
&&-\epsilon^2\left(E^\epsilon\left(a^{f^\epsilon}_++a^{f^\epsilon}_-\right)
+2\epsilon b\times \widetilde{B}^\epsilon+\left\langle [v,-v]\mu^{1/2},L f^\epsilon+\Gamma\left(f^\epsilon,f^\epsilon\right)\right\rangle,\left\{E^\epsilon+\epsilon b^{f^\epsilon}\times \mathfrak{B}\right\}\right)\\
&\lesssim&\epsilon^2\frac{d}{dt}\left(G^\epsilon,\left\{E^\epsilon+\epsilon b^{f^\epsilon}\times \mathfrak{B}\right\}\right)
-\epsilon^2\left(G^\epsilon, \partial_t\left\{E^\epsilon+\epsilon b^{f^\epsilon}\times \mathfrak{B}\right\}\right)+\eta\epsilon^2\left\|E^\epsilon+\epsilon b^{f^\epsilon}\times \mathfrak{B}\right\|^2\\
&&+\left\|\left(a^{f^\epsilon}_+-a^{f^\epsilon}_-\right)\right\|_{\dot{H}^1_x}^2+\left\|\{{\bf I-P}\}f^\epsilon\right\|^2_{H^1_xL^2_\nu} +\mathcal{E}_{N,f^\epsilon}(t)\mathcal{D}_{N,f^\epsilon}(t)\\
&\lesssim&\frac{d}{dt}\left(G^\epsilon(t),\left\{E^\epsilon+\epsilon b^{f^\epsilon}\times \mathfrak{B}\right\}\right)+\eta\epsilon^2\left\|\nabla\times \widetilde{B}^\epsilon\right\|^2+\eta\epsilon^2\left\|E^\epsilon+\epsilon b^{f^\epsilon}\times \mathfrak{B}\right\|^2\\
&&+\left\|\left[a^{f^\epsilon}_+\pm a^{f^\epsilon}_-, c^{f^\epsilon}\right]\right\|_{\dot{H}^1_x}^2+\left\|\{{\bf I-P}\}f^\epsilon\right\|^2_{H^1_xL^2_\nu} +\mathcal{E}_{N,f^\epsilon}(t)\mathcal{D}_{N,f^\epsilon}(t).
\end{eqnarray*}

For $1\leq|\alpha|\leq N-1$, one also has from Lemma \ref{lemma-nonlinear} that
\begin{eqnarray*}
&&\epsilon^2\left\|\partial^\alpha E^\epsilon(t)\right\|^2=\left(\epsilon\partial^\alpha E^\epsilon(t),\epsilon\partial^\alpha E^\epsilon(t)\right)\\
&\lesssim&\left(\epsilon\partial_t\partial^\alpha {G^\epsilon}(t),\epsilon\partial^\alpha E^\epsilon(t)\right)+\left(\epsilon\partial^\alpha\left\{\nabla_x\left(a^{f^\epsilon}_+-a^{f^\epsilon}_-\right)(t)\right\}+\epsilon\partial^\alpha\left\{\nabla_x\cdot \mathbb{A}\left(\{{\bf I-P}\}f^\epsilon(t)\cdot q_1\right)\right\},\epsilon\partial^\alpha E^\epsilon(t)\right)\\
&&-\left(\epsilon\partial^\alpha\left\{E^\epsilon(t)\left(a^{f^\epsilon}_++a^{f^\epsilon}_-\right)(t)
+\left\langle [v,-v]\mu^{1/2},L f^\epsilon(t)+\Gamma\left(f^\epsilon(t),f^\epsilon(t)\right)\right\rangle\right\},\epsilon\partial^\alpha E^\epsilon(t)\right)\\
&&-\left(\epsilon\partial^\alpha\left\{
2\epsilon b^{f^\epsilon}(t)\times \widetilde{B}^\epsilon(t)+2\epsilon b^{f^\epsilon}(t)\times \mathfrak{B}\right\},\epsilon\partial^\alpha E^\epsilon(t)\right)\\
&\lesssim&\frac{d}{dt}\left(\epsilon\partial^\alpha {G^\epsilon}(t),
\epsilon\partial^\alpha E^\epsilon(t)\right)-\left(\epsilon\partial^\alpha {G^\epsilon}(t),\epsilon\partial^\alpha \partial_tE^\epsilon(t)\right)+\eta\epsilon^2\left\|\partial^{\alpha}E^\epsilon(t)\right\|^2
\\
&&+\left\|\partial^{\alpha}\left(a^{f^\epsilon}_+-a^{f^\epsilon}_-\right)(t)\right\|_{\dot{H}^1_x}^2+\left\|\partial^\alpha\{{\bf I-P}\}f^\epsilon(t) \right\|^2_{H^1_x} +\left\|\partial^\alpha b^{f^\epsilon}\right\|^2 +\mathcal{E}_{N,f^\epsilon}(t)\mathcal{D}_{N,f^\epsilon}(t)\\
&\lesssim&\frac{d}{dt}\left(\epsilon\partial^\alpha {G^\epsilon}(t),\epsilon\partial^\alpha E^\epsilon(t)\right)-\left(\epsilon\partial^\alpha {G^\epsilon}(t), \partial^\alpha \left\{\nabla\times \widetilde{B}^\epsilon(t)-\epsilon G^\epsilon(t)\right\}\right)+\eta\epsilon^2\left\|\partial^{\alpha}E^\epsilon(t)\right\|^2
\\
&&+\left\|\partial^{\alpha}\left(a^{f^\epsilon}_+-a^{f^\epsilon}_-\right)(t)\right\|_{\dot{H}^1_x}^2+\left\|\partial^\alpha\{{\bf I-P}\}f^\epsilon(t)\right\|^2_{H^1_x}+\left\|\partial^\alpha b^{f^\epsilon}\right\|^2 +\mathcal{E}_{N,f^\epsilon}(t)\mathcal{D}_{N,f^\epsilon}(t)\\
&\lesssim&\frac{d}{dt}\left(\epsilon\partial^\alpha {G^\epsilon}(t),\epsilon\partial^\alpha E^\epsilon(t)\right)+\eta\epsilon^2\left\|\partial^{\alpha}E^\epsilon(t)\right\|^2 +\left\|\partial^{\alpha}\left(a^{f^\epsilon}_+-a^{f^\epsilon}_-\right)(t)\right\|_{\dot{H}^1_x}^2+\left\|\partial^\alpha\{{\bf I-P}\}f^\epsilon(t)\right\|^2_{H^1_xL^2_v}\\
&&+\min\left\{\left\|\partial^\alpha {G^\epsilon}(t)\right\|^2+\epsilon^2\left\|\partial^\alpha\nabla\times \widetilde{B}^\epsilon(t)\right\|^2,\sum_{i=1}^3\left\|\partial^{\alpha+e_i} G^\epsilon(t)\right\|^2+\epsilon^2\left\|\partial^{\alpha-e_i}\nabla\times \widetilde{B}^\epsilon(t)\right\|^2\right\}\\
&&+\left\|\partial^\alpha b^{f^\epsilon}\right\|^2+\mathcal{E}_{N,f^\epsilon}(t)\mathcal{D}_{N,f^\epsilon}(t).
\end{eqnarray*}
Consequently
\begin{eqnarray}\label{E-R}
&&-\frac{d}{dt}\sum_{|\alpha|\leq N}\left(\epsilon\partial^\alpha {G^\epsilon}(t),\epsilon\partial^\alpha E^\epsilon(t)\right) +\epsilon^2\left\|E^\epsilon+\epsilon b^{f^\epsilon}\times \mathfrak{B}\right\|^2+\epsilon^2\sum_{1\leq|\alpha|\leq N-1}\left\|\partial^\alpha E^\epsilon(t)\right\|^2\nonumber\\
&\lesssim&\eta\epsilon^2\sum_{1\leq|\alpha|\leq N-1}\left\|\partial^\alpha \widetilde{B}^\epsilon(t)\right\|^2
+\sum_{1\leq|\alpha|\leq N}\left\|\partial^\alpha{\bf P}f^\epsilon(t)\right\|^2 +\left\|a^{f^\epsilon}_+(t)-a^{f^\epsilon}_-(t)\right\|^2\\
&&+\left\|\partial^\alpha\{{\bf I-P}\}f^\epsilon(t)\right\|^2_{H^N_x}+\mathcal{E}_{N,f^\epsilon}(t)\mathcal{D}_{N,f^\epsilon}(t).\nonumber
\end{eqnarray}

For $\widetilde{B}^\epsilon$, since
\begin{eqnarray*}
&&\epsilon^2\sum_{i=1}^3\left\|\partial^{\alpha+e_i} \widetilde{B}^\epsilon(t)\right\|^2\\
&=&\epsilon^2\left\|\partial^\alpha \nabla\times \widetilde{B}^\epsilon(t)\right\|^2\\
&=&\left(\epsilon\partial^\alpha \nabla\times \widetilde{B}^\epsilon(t),\epsilon\partial^\alpha \nabla\times \widetilde{B}^\epsilon(t)\right)\\
&=&\left(\epsilon\partial^\alpha\partial_t E^\epsilon(t)+\epsilon \partial^\alpha {G^\epsilon}(t),\epsilon^2\partial^\alpha \nabla\times \widetilde{B}^\epsilon(t)\right)\\
&=&\frac{d}{dt}\left(\epsilon\partial^\alpha E^\epsilon(t),\epsilon^2\partial^\alpha \nabla\times \widetilde{B}^\epsilon(t)\right)-\left(\epsilon\partial^\alpha E^\epsilon(t),\epsilon^2\partial^\alpha \nabla\times \partial_t\widetilde{B}^\epsilon(t)\right)
+\left(\epsilon \partial^\alpha {G^\epsilon}(t),\epsilon^2\partial^\alpha \nabla\times \widetilde{B}^\epsilon(t)\right)\\
&=&\frac{d}{dt}\left(\epsilon\partial^\alpha E^\epsilon(t),\epsilon^2\partial^\alpha \nabla\times \widetilde{B}^\epsilon(t)\right)+\left(\epsilon\partial^\alpha \nabla\times E^\epsilon(t),\epsilon\partial^\alpha \nabla\times E^\epsilon(t)\right)
+\left(\epsilon \partial^\alpha {G^\epsilon}(t),\epsilon^2\partial^\alpha \nabla\times \widetilde{B}^\epsilon(t)\right)\\
&\lesssim&\frac{d}{dt}\left(\epsilon\partial^\alpha E^\epsilon(t),\epsilon^2\partial^\alpha \nabla\times \widetilde{B}^\epsilon(t)\right)+\epsilon^2\left\|\partial^\alpha \nabla\times E^\epsilon(t)\right\|^2
+\eta\epsilon^2\left\|\partial^\alpha \nabla\times \widetilde{B}^\epsilon(t)\right\|^2+\epsilon^2\left\| \partial^\alpha {G^\epsilon}(t)\right\|^2,
\end{eqnarray*}
we can deduce from Lemma \ref{lemma-nonlinear} that
\begin{eqnarray}\label{B-R}
  &&-\frac{d}{dt}\sum_{|\alpha|\leq N-1} \left(\epsilon\partial^\alpha E^\epsilon(t),\epsilon^2\partial^\alpha \nabla\times \widetilde{B}^\epsilon(t)\right)+\epsilon^2\sum_{1\leq |\alpha|\leq N-1}\left\|\partial^\alpha \widetilde{B}^\epsilon(t)\right\|^2\\
  &\lesssim&\epsilon^2\sum_{1\leq |\alpha|\leq N-1}\left\|\partial^\alpha E^\epsilon(t)\right\|^2
  +\epsilon^2\left\|\{{\bf I-P}\}f^\epsilon(t)\right\|_{H^N_xL^2_v}^2.\nonumber
\end{eqnarray}

For sufficiently small $\kappa>0$, $(\ref{E-R})$+$\kappa\times$(\ref{B-R}) gives
\begin{eqnarray}\label{E-B-{m}}
&&\frac{d}{dt}G_{E^\epsilon,\widetilde{B}^\epsilon}(t)+\epsilon^2\left\|E^\epsilon+\epsilon b^{f^\epsilon}\times \mathfrak{B}\right\|^2+\epsilon^2\sum_{1\leq|\alpha|\leq N-1}\left\|\partial^\alpha \left[E^\epsilon(t),\widetilde{B}^\epsilon(t)\right]\right\|^2\nonumber\\
&\lesssim&\sum_{1\leq|\alpha|\leq N}\left\|\partial^\alpha{\bf P}f^\epsilon(t)\right\|^2 +\left\|a^{f^\epsilon}_+(t)-a^{f^\epsilon}_-(t)\right\|^2+\left\|\partial^\alpha\{{\bf I-P}\}f^\epsilon(t)\right\|^2_{H^N_xL^2_v}+\mathcal{E}_{N,f^\epsilon}(t) \mathcal{D}_{N,f^\epsilon}(t).\nonumber
\end{eqnarray}
Here we have set
\[
G_{E^\epsilon,\widetilde{B}^\epsilon}(t)=-\sum_{|\alpha|\leq N}\left(\epsilon\partial^\alpha {G^\epsilon}(t),\epsilon\partial^\alpha E^\epsilon(t)\right)-\kappa\sum_{|\alpha|\leq N-1} \left(\epsilon\partial^\alpha E^\epsilon(t),\epsilon^2\partial^\alpha \nabla\times \widetilde{B}^\epsilon(t)\right).
\]
 This completes the proof of Lemma \ref{lemma-E-B-R-1}.
\end{proof}
Combing Lemma \ref{lemma-f-epsilon-1} with Lemma \ref{lemma-E-B-R-1}, we can get by taking $N=3$ that
\begin{lemma}
There exists $G_{f^\epsilon,E^\epsilon,\widetilde{B}^\epsilon}(t)$ satisfying
\[
G_{f^\epsilon,E^\epsilon,\widetilde{B}^\epsilon}(t)\lesssim \sum_{|\alpha|\leq3}\left\|\partial^\alpha\left[f^\epsilon,E^\epsilon, \widetilde{B}^\epsilon\right](t)\right\|^2
\]
such that
  \begin{eqnarray}\label{f-R-Macro-0}
    &&\frac{d}{dt}G_{f^\epsilon,E^\epsilon,\widetilde{B}^\epsilon}(t)+\sum_{1\leq|\alpha|\leq 3}
    \left\|\partial^\alpha {\bf P}f^\epsilon(t)\right\|^2 +\left\|a^{f^\epsilon}_+(t)-a^{f^\epsilon}_-(t)\right\|^2\nonumber\\
    &&+\epsilon^2\left\|E^\epsilon(t)+\epsilon b^{f^\epsilon}(t)\times \mathfrak{B}\right\|^2
    +\epsilon^2\sum_{1\leq|\alpha|\leq 2}\left\|\partial^\alpha \left[E^\epsilon(t), \widetilde{B}^\epsilon(t)\right]\right\|^2\\
    &\lesssim&\sum_{|\alpha|\leq3}\left\|\partial^\alpha\{{\bf I-P}\}f^\epsilon(t)\right\|^2 +\mathcal{E}_{3,f^\epsilon}(t)\mathcal{D}_{3,f^\epsilon}(t)\nonumber
  \end{eqnarray}
holds for all $0\leq t\leq T$.
\end{lemma}
\subsection{A key estimate related to the electric field $E^\epsilon(t,x)$}

The estimate \eqref{f-R-Macro-0} tells us that the dissipative effect of the electromagnetic field $\left[E^\epsilon(t,x), \widetilde{B}^\epsilon(t,x)\right]$ is degenerate for small $\epsilon=\frac 1c$. As pointed out in the introduction, such a degeneracy will lead to a difficulty in dealing with the term
\[
\left(\partial^\alpha_\beta\left(E^\epsilon\cdot v\mu^{\frac12} q_1\right), \partial^\alpha _{\beta} \{{\bf I-P}\}f^\epsilon\right)
\]
and the main purpose of this subsection is to provide another way to bound such a term.
To this end, we need the following result.
\begin{lemma}\label{Lemma-3.5}
It holds that
\begin{equation}\label{Micro-identity-key}
\{{\bf I-P }\} \left\{E^\epsilon \cdot v \mu^{1/2}q_1-q_0\epsilon(v\times\mathfrak{B})\cdot\nabla_v {\bf P}f^\epsilon\right\}=\left\{E^\epsilon+\epsilon b^{f^\epsilon}\times \mathfrak{B}\right\} \cdot v \mu^{1/2}q_1.
\end{equation}
\end{lemma}
\begin{proof}
From the definition of ${\bf P}$, one has
\begin{eqnarray}\label{micro-key-1}
\{{\bf I-P }\} \left\{E^\epsilon \cdot v \mu^{1/2}q_1\right\}
&=& E^\epsilon \cdot v \mu^{1/2}q_1-{\bf P }\left\{ E^\epsilon \cdot v \mu^{1/2}q_1\right\}\nonumber\\
&=& E^\epsilon \cdot v \mu^{1/2}q_1.
\end{eqnarray}

Similarly, one also has
\begin{eqnarray}\label{micro-key-2}
&&\{{\bf I-P }\} \left\{-q_0(v\times\mathfrak{B})\cdot\nabla_v {\bf P}f^\epsilon\right\}\nonumber\\
&=&\{{\bf I-P }\}\left\{-(v\times\mathfrak{B})\cdot\nabla_v \left\{b^{f^\epsilon}\cdot v \mu^{1/2}\right\}q_1\right\}\nonumber\\
&=&\{{\bf I-P }\}\left\{-(v\times\mathfrak{B})\cdot b^{f^\epsilon} \mu^{1/2}q_1\right\}\nonumber\\
&=&\{{\bf I-P }\}\left\{(\mathfrak{B}\times v)\cdot b^{f^\epsilon} \mu^{1/2}q_1\right\}\\
&=&\{{\bf I-P }\}\left\{\left(b^{f^\epsilon}\times \mathfrak{B}\right)\cdot v \mu^{1/2}q_1\right\}\nonumber\\
&=&\left(b^{f^\epsilon}\times \mathfrak{B}\right)\cdot v \mu^{1/2}q_1-{\bf P}\left\{\left(b^{f^\epsilon}\times \mathfrak{B}\right)\cdot v \mu^{1/2}q_1\right\}\nonumber\\
&=&\left(b^{f^\epsilon}\times \mathfrak{B}\right)\cdot v \mu^{1/2}q_1.\nonumber
\end{eqnarray}
Therefore, one can get \eqref{Micro-identity-key} from \eqref{micro-key-1} and \eqref{micro-key-2}.
\end{proof}

Having obtained Lemma \ref{Lemma-3.5}, we can get that
\begin{lemma}\label{E-estimates}
For $|\alpha|+|\beta|\leq 3$ with $|\beta|\geq 1$, it holds that
\begin{eqnarray}\label{E-R-estimates}
  &&\left(\partial_{\beta}^{\alpha}\left\{q_1\left\{E^\epsilon+\epsilon b^{f^\epsilon}\times \mathfrak{B}\right\}\cdot v \mu^{1/2}\right\}, \partial^\alpha_\beta{\{\bf I-P\}}f^\epsilon\right)\nonumber\\
  &=&(-1)^{|\beta|}\sum_{i=1}^3\left(\partial^{\alpha}\left\{E^\epsilon+\epsilon b^{f^\epsilon}\times \mathfrak{B}\right\}_i,\left\langle \partial_{2\beta}\left\{v_i \mu^{1/2}\right\}, \partial^\alpha{\{\bf I_+-P_+\}}f^\epsilon-\partial^\alpha{\{\bf I_--P_-\}}f^\epsilon\right\rangle\right)\\
   &\lesssim&\frac1{4C_\beta}\frac{d}{dt}\sum_{i=1}^3\left\|\left\langle \partial_{2\beta}\left\{v_i \mu^{1/2}\right\}, \partial^\alpha{\{\bf I_+-P_+\}}f^\epsilon(t)-\partial^\alpha{\{\bf I_--P_-\}}f^\epsilon(t)\right\rangle\right\|^2\nonumber\\
  &&+\left\|\nabla^{|\alpha|}\{{\bf I-P}\}f^\epsilon(t)\right\|^2+\left\|\nabla^{|\alpha|+1}f^\epsilon(t)\right\|^2
  +\mathcal{E}_{3,f^\epsilon}(t)\mathcal{D}_{3,f^\epsilon}(t).\nonumber
\end{eqnarray}
Here
$C_\beta=\int_{\mathbb{R}^3_v}
\left\{\partial_\beta\left(v_i\mu^{\frac12}\right)\right\}^2dv
$ is some positive constant.
\end{lemma}
\begin{proof}
 Applying ${\bf I-P}$ to \eqref{f gn}, one has
 \begin{eqnarray}\label{f-micro}
&&\partial_t\{{\bf I-P}\}f^\epsilon+ v  \cdot\nabla_x\{{\bf I-P}\}f^\epsilon- \left\{E^\epsilon+\epsilon b^{f^\epsilon}\times \mathfrak{B}\right\} \cdot v \mu^{1/2}q_1+{ L} f^\epsilon+ v  \cdot\nabla_x{\bf P}f^\epsilon\nonumber\\
&=&-\epsilon \{{\bf I-P}\}q_0 (v\times \mathfrak{B})\cdot\nabla_v\{{\bf I-P}\}f^\epsilon+ {\bf P}\{v\cdot\nabla_x f^\epsilon\}+\{{\bf I-P}\}{\tilde{G}^\epsilon}
\end{eqnarray}
with
\begin{equation}\label{def-G}
\tilde{G}^\epsilon\equiv\frac1 2 q_0 E^\epsilon\cdot v f^\epsilon-q_0 E^\epsilon\cdot\nabla_{ v  }f^\epsilon-\epsilon q_0 \left(v\times \widetilde{B}^\epsilon\right)\cdot\nabla_vf^\epsilon +{\Gamma}\left(f^\epsilon,f^\epsilon\right).
\end{equation}

Multiplying \eqref{f-micro} by $\partial_{2\beta}\{v_i \mu^{1/2}\}$ and
integrating the resulting identity with respect to $v$ over $\mathbb{R}^3$, one has
  \begin{eqnarray}\label{key1-hard-sphere}
  &&\partial_t\left\langle \partial_{2\beta}\left(v_i\mu^{1/2}\right),\{{\bf I_\pm-P_\pm}\}f^\epsilon\right\rangle+\nabla_x\cdot\left\langle v\partial_{2\beta}\left(v_i\mu^{1/2}\right), \{{\bf I_\pm-P_\pm}\}f^\epsilon \right\rangle\nonumber\\
  &&\mp\underbrace{\left\langle \partial_{2\beta}\left(v_i\mu^{1/2}\right),v_i\mu^{1/2}\right\rangle}_{J_1} \left\{E^\epsilon+\epsilon b^{f^\epsilon}\times \mathfrak{B}\right\}_i\\
&=&\left\langle \partial_{2\beta}\left(v_i\mu^{1/2}\right), {\bf P}_{\pm}\left\{v\cdot\nabla_x f^\epsilon\right\}-{ L} _{\pm}f^\epsilon- v  \cdot\nabla_x{\bf P}_{\pm}f^\epsilon
-\epsilon \{{\bf I_\pm-P_\pm}\}q_0 (v\times \mathfrak{B})\cdot\nabla_v\{{\bf I-P}\}f^\epsilon\right\rangle\nonumber\\ \nonumber
&&+\left\langle \partial_{2\beta}\left(v_i\mu^{1/2}\right), \{{\bf I_{\pm}-P_{\pm}}\}{\tilde{G}^\epsilon}
\right\rangle.
\end{eqnarray}

Noticing that
\[
J_1=(-1)^{|\beta|}\int_{\mathbb{R}^3_v}
\left\{\partial_\beta(v_i\mu^{\frac12})\right\}^2dv\equiv (-1)^{|\beta|}C_{\beta},
\]
we can get from \eqref{Micro-identity-key} that
\begin{eqnarray}\label{key2-hard-sphere}
&&\partial_t\left\langle \partial_{2\beta}\left(v_i\mu^{1/2}\right),\{{\bf I_+-P_+}\}f^\epsilon-\{{\bf I_--P_-}\}f^\epsilon\right\rangle-2(-1)^{|\beta|}C_\beta \left\{E^\epsilon+\epsilon b^{f^\epsilon}\times \mathfrak{B}\right\}_i\nonumber\\
&=&-\nabla_x\cdot\left\langle v\partial_{2\beta}\left(v_i\mu^{1/2}\right),\{{\bf I_+-P_+}\}f^\epsilon-\{{\bf I_--P_-}\}f^\epsilon \right\rangle\\ \nonumber
&&+\left\langle \partial_{2\beta}\left(v_i\mu^{1/2}\right), {\bf
P}_{+}\{v\cdot\nabla_x f^\epsilon\}-{\bf
P}_{-}\{v\cdot\nabla_x f^\epsilon\}-{ L} _{+}f^\epsilon+-{ L} _{-}f^\epsilon- v  \cdot\nabla_x{\bf P}_{+}f^\epsilon+ v  \cdot\nabla_x{\bf P}_{-}f^\epsilon\right\rangle\\ \nonumber
&&+\left\langle \partial_{2\beta}\left(v_i\mu^{1/2}\right),-\epsilon \{{\bf I_+-P_+}\}q_0 (v\times \mathfrak{B})\cdot\nabla_v\{{\bf I-P}\}f^\epsilon+\epsilon \{{\bf I_--P_-}\}q_0 (v\times \mathfrak{B})\cdot\nabla_v\{{\bf I-P}\}f^\epsilon\right\rangle\\ \nonumber
&&+\left\langle \partial_{2\beta}\left(v_i\mu^{1/2}\right),\{{\bf I_{+}-P_{+}}\}{\tilde{G}^\epsilon}-\{{\bf I_{-}-P_{-}}\}{\tilde{G}^\epsilon}
\right\rangle.
\end{eqnarray}
Applying $\partial^\alpha$ to \eqref{key2-hard-sphere}, one has
\begin{eqnarray}\label{key3-hard-sphere}
&&\partial_t\left\langle \partial_{2\beta}\left(v_i\mu^{1/2}\right),\partial^\alpha\{{\bf I_+-P_+}\}f^\epsilon-\partial^\alpha\{{\bf I_--P_-}\}f^\epsilon\right\rangle\nonumber\\
&&-2(-1)^{|\beta|}C_\beta\partial^\alpha \left\{E^\epsilon+\epsilon b^{f^\epsilon}\times \mathfrak{B}\right\}_i\nonumber\\
&=&-\nabla_x\cdot\left\langle v\partial_{2\beta}\left(v_i\mu^{1/2}\right),\partial^\alpha\{{\bf I_+-P_+}\}f^\epsilon-\partial^\alpha\{{\bf I_--P_-}\}f^\epsilon \right\rangle\\ \nonumber
&&+\left\langle \partial_{2\beta}\left(v_i\mu^{1/2}\right), \partial^\alpha{\bf
P}_{+}\{v\cdot\nabla_x f^\epsilon\}-\partial^\alpha{\bf
P}_{-}\{v\cdot\nabla_x f^\epsilon\}\right.\\ \nonumber
&&\left.\ \ \ \ \ \ \ \ \ \ \ \ \ -\partial^\alpha{ L} _{+}f^\epsilon+-\partial^\alpha{ L} _{-}f^\epsilon- v  \cdot\nabla_x\partial^\alpha{\bf P}_{+}f^\epsilon+ v  \cdot\nabla_x\partial^\alpha{\bf P}_{-}f^\epsilon\right\rangle\\ \nonumber
&&+\left\langle \partial_{2\beta}\left(v_i\mu^{1/2}\right),-\epsilon \{{\bf I_+-P_+}\}q_0 (v\times \mathfrak{B})\cdot\nabla_v\partial^\alpha\{{\bf I-P}\}f^\epsilon+\epsilon \{{\bf I_--P_-}\}q_0 (v\times \mathfrak{B})\cdot\nabla_v\partial^\alpha\{{\bf I-P}\}f^\epsilon\right\rangle\\ \nonumber
&&+\left\langle \partial_{2\beta}\left(v_i\mu^{1/2}\right),\partial^\alpha\{{\bf I_{+}-P_{+}}\}{\tilde{G}^\epsilon}-\partial^\alpha\{{\bf I_{-}-P_{-}}\}{\tilde{G}^\epsilon}\right\rangle\nonumber\\
&:=&R^\epsilon_1,\nonumber
\end{eqnarray}
where we use $R^\epsilon_{1}$ to denote all terms in the right hand side of \eqref{key3-hard-sphere}.

Therefore, one can finally get from Lemma \ref{lemma-nonlinear} that
\begin{eqnarray*}
  &&(-1)^{|\beta|}\left(\partial^{\alpha}\left\{E^\epsilon+\epsilon b^{f^\epsilon}\times \mathfrak{B}\right\}_i,\left\langle \partial_{2\beta}\left\{v_i \mu^{1/2}\right\}, \partial^\alpha{\{\bf I_+-P_+\}}f^\epsilon-\partial^\alpha{\{\bf I_--P_-\}}f^\epsilon\right\rangle\right)\\
  &=&\frac1{2C_\beta}\left(\partial_t\left\langle \partial_{2\beta}\left(v_i\mu^{1/2}\right),\partial^\alpha\{{\bf I_+-P_+}\}f^\epsilon-\partial^\alpha\{{\bf I_--P_-}\}f^\epsilon\right\rangle-R^\epsilon_1,\right.\\
  &&\quad\quad \left.\left\langle \partial_{2\beta}\left\{v_i \mu^{1/2}\right\}, \partial^\alpha{\{\bf I_+-P_+\}}f^\epsilon-\partial^\alpha{\{\bf I_--P_-\}}f^\epsilon\right\rangle\right)\\ \nonumber
  &=&\frac1{4C_\beta}\sum_{i=1}^3\frac{d}{dt}\left\|\left\langle \partial_{2\beta}\left\{v_i \mu^{1/2}\right\}, \partial^\alpha{\{\bf I_+-P_+\}}f^\epsilon-\partial^\alpha{\{\bf I_--P_-\}}f^\epsilon\right\rangle\right\|^2\\ \nonumber
  &&+\frac1{2C_\beta}\left(R_1^\epsilon,\left\langle \partial_{2\beta}\left\{v_i \mu^{1/2}\right\}, \partial^\alpha{\{\bf I_+-P_+\}}f^\epsilon-\partial^\alpha{\{\bf I_--P_-\}}f^\epsilon\right\rangle\right)\\ \nonumber
  &\lesssim&\frac1{4C_\beta}\sum_{i=1}^3\frac{d}{dt}\left\|\left\langle \partial_{2\beta}\left\{v_i \mu^{1/2}\right\}, \partial^\alpha{\{\bf I_+-P_+\}}f^\epsilon-\partial^\alpha{\{\bf I_--P_-\}}f^\epsilon\right\rangle\right\|^2\\ \nonumber
  &&+\left\|\nabla^{|\alpha|}\{{\bf I-P}\}f^\epsilon(t)\right\|^2+\left\|\nabla^{|\alpha|+1}f^\epsilon(t)\right\|^2
  +\mathcal{E}_{3,f^\epsilon}(t)\mathcal{D}_{3,f^\epsilon}(t).
\end{eqnarray*}
Thus we have completed the proof of this lemma.
\end{proof}
\subsection{Lyapunov-type inequalities for $\mathcal{E}_{3,f^\epsilon}(t)$}
With the above results in hand, we now turn to deduce the key {\it a priori} estimates on the solution $\left[f^\epsilon(t,x,v), E^\epsilon(t,x), \widetilde{B}^\epsilon(t,x)\right]$ of the Cauchy problem \eqref{f}, \eqref{f-initial}, \eqref{compatibility conditions}.
\begin{lemma}
It holds for all $0\leq t\leq T$ that
   \begin{eqnarray}\label{final-1-0}
     &&\frac{d}{dt}\mathcal{E}_{3,f^\epsilon}(t)+\mathcal{D}_{3,f^\epsilon}(t)
     \lesssim\mathcal{E}_{3,f^\epsilon}(t)\mathcal{D}_{3,f^\epsilon}(t).
   \end{eqnarray}
\end{lemma}
\begin{proof} For the brevity of the presentation, the proof of this lemma is divided into the following two steps:\\

\noindent{\bf Step 1:}\quad Firstly, multiplying \eqref{f gn} by $f^\epsilon$ and integrating the resulting identity over $\mathbb{R}^3_v\times\mathbb{R}^3_x$, one has
\begin{eqnarray}\label{0-order-micro}
   && \frac{d}{dt}\left\|\left[f^\epsilon,E^\epsilon,\widetilde{B}^\epsilon\right](t)\right\|^2+\left\|\{{\bf I-P}\}f^\epsilon(t)\right\|_\nu^2\nonumber\\
   &\lesssim&
  \underbrace{\frac12 \left(q_0 E^\epsilon(t)\cdot v f^\epsilon(t), f^\epsilon(t)\right)}_{I_1}
  +\underbrace{\left( {\Gamma}\left(f^\epsilon(t), f^\epsilon(t)\right), f^\epsilon(t)\right)}_{I_2}.
    \end{eqnarray}

For $I_1$, one can easily get from Sobolev's inequality that
\begin{eqnarray*}
I_1
&\lesssim&\mathcal{E}_{2,f^\epsilon}(t)\mathcal{D}_{2,f^\epsilon}(t)+\eta\mathcal{D}_{2,f^\epsilon}(t),
\end{eqnarray*}
while by employing the estimate \eqref{nonlinear-1} on the nonlinear collision operator $\Gamma_\pm(g_1,g_2)$ obtained in Lemma \ref{lemma-nonlinear}, one can deduce that
\begin{eqnarray*}
  I_2\lesssim \mathcal{E}_{2,f^\epsilon}(t)\mathcal{D}_{2,f^\epsilon}(t)+\eta\mathcal{D}_{2,f^\epsilon}(t).
\end{eqnarray*}

Plugging the above estimates on $I_1\sim I_2$ into \eqref{0-order-micro}, one arrives at
\begin{eqnarray}\label{1-sum}
&& \frac{d}{dt}\left\|\left[f^\epsilon,E^\epsilon,\widetilde{B}^\epsilon\right](t)\right\|^2+\left\|\{{\bf I-P}\}f^\epsilon(t)\right\|_\nu^2\\  \nonumber
&\lesssim&\mathcal{E}_{2,f^\epsilon}(t)\mathcal{D}_{2,f^\epsilon}(t)+\eta\mathcal{D}_{2,f^\epsilon}(t).
\end{eqnarray}

Next applying $\partial^\alpha$ with $1\leq|\alpha|\leq 3$ to \eqref{f gn}, multiplying the resulting identity by $\partial^\alpha f^\epsilon$, one has by integrating the final result with respect to $v$ and $x$ over $\mathbb{R}^3_v\times\mathbb{R}^3_x$ that
 \begin{eqnarray}\label{x-derivatives}
   && \frac{d}{dt}\left\|\partial^\alpha\left[f^\epsilon,E^\epsilon,\widetilde{B}^\epsilon\right](t)\right\|^2
   +\left\|\partial^\alpha\{{\bf I-P}\}f^\epsilon(t)\right\|_\nu^2\nonumber\\
   &\lesssim&
  -\underbrace{\left(\partial^\alpha\left\{ q_0 E^\epsilon(t)\cdot\nabla_{v}f^\epsilon(t)\right\},\partial^\alpha f^\epsilon(t)\right)}_{I_{3}}+\underbrace{\frac12 \left(\partial\left\{q_0 E^\epsilon(t)\cdot v f^\epsilon(t)\right\}, \partial^\alpha f^\epsilon(t)\right)}_{I_4}\nonumber\\
  &&-\underbrace{\left(\partial^\alpha\left\{\epsilon q_0 v\times \widetilde{B}^\epsilon(t)\cdot\nabla_vf^\epsilon(t)\right\},\partial^\alpha f^\epsilon(t)\right)}_{I_5}
+\underbrace{\left(\partial ^\alpha {\Gamma}\left(f^\epsilon(t),f^\epsilon(t)\right),\partial^\alpha f^\epsilon(t)\right)}_{I_6}.\nonumber
    \end{eqnarray}

 For $I_3$, one has from Sobolev's inequality that
 \begin{eqnarray*}
   I_3&\lesssim & \left\|\partial^{e_i}E^\epsilon(t)\right\|_{L^\infty_x}\left\|\partial^{\alpha-e_i}_{e_i}f^\epsilon(t)\right\| \left\|\partial^\alpha f^\epsilon(t)\right\| \\ &&+\left\|\partial^{2e_i}E^\epsilon(t)\right\|_{L^3_x} \left\|\partial^{\alpha-2e_i}_{e_i}f^\epsilon(t)\right\|_{L^6_xL^2_v}\left\|\partial^\alpha f^\epsilon(t)\right\| +\left\|\partial^\alpha E^\epsilon(t)\right\|\left\|\nabla_vf^\epsilon(t)\right\|_{L^\infty_xL^2_v}\left\|\partial^\alpha f^\epsilon(t)\right\|\\
   &\lesssim&\left\|\partial^{e_i}E^\epsilon(t)\right\|^2_{L^\infty_x}\left\|\partial^{\alpha-e_i}_{e_i}f^\epsilon(t)\right\|^2+
   \left\|\partial^{e_i} E^\epsilon(t)\right\|_{L^3_x}^2\left\|\partial^{\alpha-2e_i}_{e_i}f^\epsilon(t)\right\|^2_{L^6_xL^2_v}\\
   &&+\left\|\partial^\alpha E^\epsilon(t)\right\|^2\left\|\nabla_vf^\epsilon(t)\right\|^2_{L^\infty_xL^2_v}
   +\eta\left\|\partial^\alpha f^\epsilon(t)\right\|^2\\
   &\lesssim&\mathcal{E}_{3,f^\epsilon}(t)\mathcal{D}_{3,f^\epsilon}(t)+\eta\mathcal{D}_{3,f^\epsilon}(t).
 \end{eqnarray*}

 Similarly, one can also get from Lemma \ref{lemma-nonlinear} and Sobolev's inequality that
 \begin{equation*}
   I_{4}+I_5+I_6\lesssim\mathcal{E}_{3,f^\epsilon}(t)\mathcal{D}_{3,f^\epsilon}(t)+\eta\mathcal{D}_{3,f^\epsilon}(t).
 \end{equation*}

Substituting the above estimates on $I_3\sim I_{6}$ into \eqref{x-derivatives} and by exploiting the estimate \eqref{1-sum}, one has
 \begin{equation}\label{2-sum}
   \frac{d}{dt}\left\|\partial^\alpha\left[f^\epsilon,E^\epsilon,\widetilde{B}^\epsilon\right](t)\right\|^2+\left\|\partial^\alpha{\bf \{I-P\}}f^\epsilon(t)\right\|^2\lesssim\mathcal{E}_{3,f^\epsilon}(t)\mathcal{D}_{3,f^\epsilon}(t)+\eta\mathcal{D}_{3,f^\epsilon}(t)
 \end{equation}
holds for $1\leq |\alpha|\leq 3$.\\

\noindent{\bf Step 2:}\quad
Applying $\partial^\alpha_\beta$ to \eqref{f-micro}, multiplying the result by $\partial^\alpha_\beta{\{\bf I-P\}}f^\epsilon$ with $|\beta|\geq 1$ and integrating the final resulting identity with respect to $v$ and $x$ over $\mathbb{R}^3_v\times\mathbb{R}^3_x$, one has
 \begin{eqnarray}\label{alpha-beta-e}
   &&\frac{d}{dt}\left\|\partial^\alpha_\beta\{{\bf I-P}\}f^\epsilon(t)\right\|^2
   +\left\|\partial^\alpha_\beta\{{\bf I-P}\}f^\epsilon(t)\right\|_\nu^2\nonumber\\
   &\lesssim&\eta\sum_{|\beta'|<|\beta|}\left\|\partial^\alpha_{\beta'}\{{\bf I-P}\}f^\epsilon(t)\right\|_\nu^2+C_\eta\left\|\partial^{\alpha+e_i}_{\beta-e_i}\{{\bf I-P}\}f^\epsilon(t)\right\|^2
   +\left\|\nabla_x f^\epsilon(t)\right\|_\nu^2\\
   &&+\underbrace{\left(\partial^{\alpha}_\beta\left\{\left\{E^\epsilon+\epsilon b^{f^\epsilon}\times \mathfrak{B}\right\} \cdot v \mu^{1/2}q_1\right\}, \partial^\alpha_\beta{\{\bf I-P\}}f^\epsilon(t)\right)}_{I_{7}}+\underbrace{\left(\partial^{\alpha}_\beta{\tilde{G}^\epsilon(t)}, \partial^\alpha_\beta{\{\bf I-P\}}f^\epsilon(t)\right)}_{I_{8}}.\nonumber
 \end{eqnarray}

 By employing a similar argument used to estimate $I_5$ and $I_{6}$, one has
 \[I_{7}\lesssim
 \mathcal{E}_{3,f^\epsilon}(t)\mathcal{D}_{3,f^\epsilon}(t)+\eta\mathcal{D}_{3,f^\epsilon}(t).\]

To control $I_{8}$, one has by \eqref{E-R-estimates} and Lemma \ref{lemma-nonlinear} that
\begin{eqnarray*}
  I_{8}&=&(-1)^{|\beta|}
  \left(\partial^{\alpha}_{2\beta}\left\{\left\{E^\epsilon+\epsilon b^{f^\epsilon}\times \mathfrak{B}\right\}_i \cdot v_i \mu^{1/2}\right\}, \partial^\alpha_\beta{\{\bf I-P\}}f^\epsilon_+(t)-\partial^\alpha_\beta{\{\bf I-P\}}f^\epsilon_-(t)\right)\\
  &=&(-1)^{|\beta|}\left(\partial^{\alpha}\left\{E^\epsilon+\epsilon b^{f^\epsilon}\times \mathfrak{B}\right\}_i \cdot \partial_{2\beta}\left\{v_i \mu^{1/2}\right\}, \partial^\alpha{\{\bf I-P\}}f^\epsilon_+(t)-\partial^\alpha{\{\bf I-P\}}f^\epsilon_-(t)\right)\\
    &=&(-1)^{|\beta|}\left(\partial^{\alpha}\left\{E^\epsilon+\epsilon b^{f^\epsilon}\times \mathfrak{B}\right\}_i,\left\langle \partial_{2\beta}\left\{v_i \mu^{1/2}\right\}, \partial^\alpha{\{\bf I-P\}}f^\epsilon_+(t)-\partial^\alpha{\{\bf I-P\}}f^\epsilon_-(t)\right\rangle\right)\\
    &\lesssim&\frac1{4C_\beta}\frac{d}{dt}\sum_{i=1}^3\left\|\left\langle \partial_{2\beta}\left\{v_i \mu^{1/2}\right\}, \partial^\alpha{\{\bf I_+-P_+\}}f^\epsilon(t)-\partial^\alpha{\{\bf I_--P_-\}}f^\epsilon(t)\right\rangle\right\|^2\\
  &&+\left\|\nabla^{|\alpha|}\{{\bf I-P}\}f^\epsilon(t)\right\|^2+\left\|\nabla^{|\alpha|+1}f^\epsilon(t)\right\|^2
 +\mathcal{E}_{3,f^\epsilon}(t)\mathcal{D}_{3,f^\epsilon}(t).
\end{eqnarray*}

Plugging the estimates on $I_{7}$ and $I_{8}$ into \eqref{alpha-beta-e} gives
 \begin{eqnarray}\label{alpha-beta-e-1}
   &&\frac{d}{dt}\left\|\partial^\alpha_\beta\{{\bf I-P}\}f^\epsilon(t)\right\|^2
   +\left\|\partial^\alpha_\beta\{{\bf I-P}\}f^\epsilon(t)\right\|_\nu^2\nonumber\\
   &\lesssim&\eta\sum_{|\beta'|<|\beta|}\left\|\partial^\alpha_{\beta'}\{{\bf I-P}\}f^\epsilon(t)\right\|_\nu^2+C_\eta\left\|\partial^{\alpha+e_i}_{\beta-e_i}\{{\bf I-P}\}f^\epsilon(t)\right\|^2
   +\left\|\nabla_x f^\epsilon(t)\right\|_\nu^2\\
   &&+\mathcal{E}_{3,f^\epsilon}(t)\mathcal{D}_{3,f^\epsilon}(t)+\eta\mathcal{D}_{3,f^\epsilon}(t) +\left\|\nabla^{|\alpha|}\{{\bf I-P}\}f^\epsilon(t)\right\|^2
   +\left\|\nabla^{|\alpha|+1}f^\epsilon(t)\right\|^2\nonumber\\
   &&+\frac1{4C_\beta}\frac{d}{dt}\sum_{i=1}^3\left\|\left\langle \partial_{2\beta}\left\{v_i \mu^{1/2}\right\}, \partial^\alpha{\{\bf I_+-P_+\}}f^\epsilon(t)-\partial^\alpha{\{\bf I_--P_-\}}f^\epsilon(t)\right\rangle\right\|^2.\nonumber
   \end{eqnarray}

Taking the summation over
$\{|\beta|=m,|\alpha|+|\beta|\leq 3\}$ for each given $1\leq m\leq3$, and then taking the proper linear
combination of those estimates with properly chosen constants,  one has
 \begin{eqnarray}\label{alpha-beta-sum-0}
   &&\frac{d}{dt}\sum_{|\alpha|+|\beta|\leq 3,|\beta|\geq1}\left\|\partial^\alpha_\beta\{{\bf I-P}\}f^\epsilon(t)\right\|^2+\sum_{|\alpha|+|\beta|\leq 3,|\beta|\geq1}\left\|\partial^\alpha_\beta\{{\bf I-P}\}f^\epsilon(t)\right\|_\nu^2\nonumber\\
   &\lesssim&\sum_{\alpha'\leq 3}\left\|\partial^{\alpha'}\{{\bf I-P}\}f^\epsilon(t)\right\|_\nu^2
   +\left\|\nabla_x {\bf P}f^\epsilon(t)\right\|_{H^2_xL^2_v}^2
  +\mathcal{E}_{3,f^\epsilon}(t)\mathcal{D}_{3,f^\epsilon}(t)+\eta\mathcal{D}_{3,f^\epsilon}(t)\\
   &&+\frac1{4C_\beta}\frac{d}{dt}\sum_{i=1}^3\left\|\left\langle \partial_{2\beta}\left\{v_i \mu^{1/2}\right\}, \partial^\alpha{\{\bf I_+-P_+\}}f^\epsilon(t)-\partial^\alpha{\{\bf I_--P_-\}}f^\epsilon(t)\right\rangle\right\|^2.\nonumber
   \end{eqnarray}

 A proper linear combination of \eqref{f-R-Macro-0}, \eqref{1-sum}, \eqref{2-sum} and \eqref{alpha-beta-sum-0} yields \eqref{final-1-0}. This completes the proof of Lemma 3.7.
\end{proof}

\subsection{The proof of Theorem \ref{Th1.1}}

Now we are ready to deduce the desired {\it a priori} estimates on the solution $\left[f^\epsilon(t,x,v), E^\epsilon(t,x),\right.$ $\left. \widetilde{B}^\epsilon(t,x)\right]$ of the Cauchy problem \eqref{f}, \eqref{f-initial}, \eqref{compatibility conditions} defined on the strip $\Pi_T=[0,T]\times\mathbb{R}^3\times\mathbb{R}^3$ and complete the proof of Theorem \ref{Th1.1}.

Suppose that the Cauchy problem (\ref{f}), (\ref{f-initial}), \eqref{compatibility conditions} admits a unique solution $\left[f^\epsilon(t,x,v),  E^\epsilon(t,x),\right.$ $\left. \widetilde{B}^\epsilon(t,x)\right]$ on the strip $\Pi_T=[0,T]\times\mathbb{R}^3_x\times\mathbb{R}^3_v$ and satisfy the following {\it a priori assumption}
\begin{equation}\label{the-a-priori-estimates}
X(t)=\sum_{|\alpha|+|\beta|\leq 3}\left\|\partial^\alpha_\beta f^\epsilon(t)\right\|^2+\left\|\left[E^\epsilon(t),\widetilde{B}^\epsilon(t)\right]\right\|^2_{H^3_x}\leq M, \quad 0\leq t\leq T
\end{equation}
for some sufficiently small positive constant $M>0$.

Under the {\it a priori} estimates \eqref{the-a-priori-estimates}, from \eqref{final-1-0}, integrating from $0$ to $t$, one has for all $0\leq t\leq T$ that
\begin{equation}\label{A-priori-estimate}
\sup_{0\leq s\leq t}\left\{\mathcal{E}_{3,f^\epsilon}(s)\right\}\lesssim \mathcal{E}_{3,f^\epsilon}(0),
\end{equation}
from which one can deduce \eqref{th1.1-1} immediately.

Having obtained the estimate \eqref{A-priori-estimate}, one can then get the global solvability result for the Cauchy problem  (\ref{f}), (\ref{f-initial}), \eqref{compatibility conditions}  by the continuation argument and as a by-product of such an argument, one can show that the estimate \eqref{A-priori-estimate} holds for all $t\in\mathbb{R}^+$. This completes the proof of Theorem \ref{Th1.1}.

\section{The proof of Theorem \ref{Th1.3}}
This section is devoted to proving Theorem \ref{Th1.3}. To this end, suppose that $\left[f^{m,\epsilon}(t,x,v), E^{m,\epsilon}(t,x), {B}^{m,\epsilon}(t,x)\right]$ is a unique solution of the Cauchy problem \eqref{f-R-vector}-\eqref{f-r-s-compatibility conditions} and for $i=1,2,\cdots, m-1$, let $\left[f^{P,\epsilon}(t,x,v), E^{P,\epsilon}(t,x)\right]$, $\left[f^{i,\epsilon}(t,x,v), E^{i,\epsilon}(t,x)\right]$, and $\left[f^{m,\epsilon}(t,x,v), E^{m,\epsilon}(t,x), B^{m,\epsilon}(t,x)\right]$ be solutions of the Cauchy problems \eqref{f-P-sign}-\eqref{VPB-sign-IC}, \eqref{f-i-vector}-\eqref{f-i-e-b-compatibility conditions}, and \eqref{f-R-vector}-\eqref{f-r-s-compatibility conditions}, respectively, defined on the strip $\Pi_T=[0,T]\times\mathbb{R}^3\times\mathbb{R}^3$ for some positive constant $T>0$, we now turn to deduce certain energy type {\it a priori} estimates on $\left[f^{m,\epsilon}(t,x,v), E^{m,\epsilon}(t,x), \right.$ $\left.B^{m,\epsilon}(t,x)\right]$ in terms of its initial data in the coming subsections. The first result is on the macro-structure of $\left[f^{m,\epsilon}(t,x,v), E^{m,\epsilon}(t,x), B^{m,\epsilon}(t,x)\right]$.

\subsection{Macro-structure for $\left[f^{m,\epsilon}(t,x,v), E^{m,\epsilon}(t,x), B^{m,\epsilon}(t,x)\right]$}
In this section, the first result is concerned with the dissipative effect of the macroscopic quantities such as ${\bf P}f^{m,\epsilon}(t,x,v)$, $a^{f^{m,\epsilon}}_+(t,x)-a^{f^{m,\epsilon}}_-(t,x)$, etc. To do so, we first need the following identities concerning the macroscopic part of $f^{m,\epsilon}(t,x,v)$.
\begin{lemma}
It holds that
\begin{eqnarray}\label{Macro-f-R-equation1}
&&\partial_t\left(\frac{a^{f^{m,\epsilon}}_++a^{f^{m,\epsilon}}_-}2\right)+\nabla_x\cdot b^{f^{m,\epsilon}}=0,\nonumber\\
&&\partial_tb^{f^{m,\epsilon}}_j+\partial_j\left(\frac{a^{f^{m,\epsilon}}_+ +a^{f^{m,\epsilon}}_-}2+2c^{f^{m,\epsilon}}\right)
+\sum\limits_{k=1}^3\frac{\partial_k\mathbb{A}_{jk}(\{{\bf I-P}\}{f^{m,\epsilon}}\cdot [1,1])}2\nonumber\\
&=&\frac{a^{f^{P,\epsilon}}_+-a^{f^{P,\epsilon}}_-}{2}E^{m,\epsilon}_{j}+\frac{a^{f^{m,\epsilon}}_+ -a^{f^{m,\epsilon}}_-}{2}E^{P,\epsilon}_j
+\sum_{j_1+j_2\geq m,\atop 0<j_1,j_2<m}\frac{a^{f^{j_2,\epsilon}}_+-a^{f^{j_2,\epsilon}}_-}{2}E^{j_1,\epsilon}_j \nonumber\\
&&+\epsilon\left[G^{i,\epsilon}\times \left\{B^{P}
+\sum_{j_1=1}^{m-1}\epsilon^{j_1}B^{j_1}\right\}\right]_j,\quad j=1,2,3,\\
&&\partial_tc^{f^{m,\epsilon}}+\frac13\nabla_x\cdot b^{f^{m,\epsilon}}+\frac56\sum\limits_{j=1}^3\partial_j \mathbb{B}_j\left(\{{\bf I-P}\}{f^{i,\epsilon}}\cdot [1,1]\right)\nonumber\\
&=&\frac16 G^{P,\epsilon}\cdot E^{m,\epsilon}+\frac16 G^{m,\epsilon}\cdot E^{P,\epsilon}+\frac16\sum_{j_1+j_2\geq m,\atop 0<j_1,j_2<m} G^{j_2,\epsilon}\cdot E^{j_1,\epsilon},\nonumber\\
&&\partial_t\left(a^{f^{m,\epsilon}}_+-a^{f^{m,\epsilon}}_-\right)+\nabla_x\cdot G^{m,\epsilon}=0,\nonumber\\
&&\partial_tG^{m,\epsilon}+\nabla_x\left(a^{f^{m,\epsilon}}_+-a^{f^{m,\epsilon}}_-\right) -2E^{m,\epsilon}+\nabla_x\cdot \mathbb{A}\left(\{{\bf I-P}\}{f^{m,\epsilon}}\cdot q_1\right)\nonumber\\
&=&E^{P,\epsilon}\left(a^{f^{m,\epsilon}}_++a^{f^{m,\epsilon}}_-\right)+ E^{m,\epsilon}\left(a^{f^{P,\epsilon}}_++a^{f^{P,\epsilon}}_-\right) \nonumber\\
&&+\sum_{j_1+j_2\geq m,\atop 0<j_1,j_2<m}\epsilon^{j_1+j_2-m} E^{j_1,\epsilon}\left(a^{f^{j_2,\epsilon}}_++a^{f^{j_2,\epsilon}}_-\right)
+2\epsilon b^{f^{m,\epsilon}}\times \left\{B^{P}
+\sum_{j_1=1}^{m-1}\epsilon^{j_1}B^{j_1}\right\}\nonumber\\
&&+\left\langle [v,-v]\mu^{1/2},-{ L} f^{m,\epsilon}+{\Gamma}\left(f^{P,\epsilon},f^{m,\epsilon}\right)
  +{\Gamma}\left(f^{m,\epsilon},f^{P,\epsilon}\right)
  +\sum_{j_1+j_2\geq m,\atop 0<j_1,j_2<m} \epsilon^{j_1+j_2-m}\Gamma\left(f^{j_1,\epsilon},f^{j_2,\epsilon}\right) \right\rangle,\nonumber
\end{eqnarray}
and
\begin{eqnarray}\label{Micro-f-i-equation-1}
&&\frac12\partial_t\mathbb{A}_{jk}\left(\{{\bf I-P}\}{f^{m,\epsilon}}\cdot [1,1]\right)+\partial_kb^{f^{m,\epsilon}}_j+\partial_jb^{f^{m,\epsilon}}_k -\frac23\delta_{jk}\nabla_x\cdot b^{f^{m,\epsilon}}
-\frac53\delta_{jk}\nabla_x\cdot \mathbb{B}\left(\{{\bf I-P}\}{f^{m,\epsilon}}\cdot [1,1]\right)\nonumber\\
&=&\frac12\mathbb{A}_{jk}\left(r^{m,\epsilon}_{+}+r^{m,\epsilon}_{-}+g^{m,\epsilon}_{+} +g^{m,\epsilon}_{-}\right),\\
&&\frac12\partial_t \mathbb{B}_{k}\left(\{{\bf I-P}\}{f^{m,\epsilon}}\cdot [1,1]\right)+\partial_kc^{f^{m,\epsilon}}=\frac12\mathbb{B}_{k}\left(r^{m,\epsilon}_{+} +r^{m,\epsilon}_{-}+g^{m,\epsilon}_{+}+g^{m,\epsilon}_{-}\right),\nonumber
\end{eqnarray}
where
\begin{eqnarray}\label{def-f-R-r-g}
r^{m,\epsilon}_{\pm}&=&- v\cdot\nabla_x\{{\bf I_\pm-P_\pm}\}f^{m,\epsilon} -{ L}_\pm f^{m,\epsilon},\nonumber\\
g^{m,\epsilon}_{\pm}&=&\pm\frac12  E^{m,\epsilon}\cdot v f^{P,\epsilon}_{\pm}
\mp E^{m,\epsilon}\cdot\nabla_{ v  }f^{P,\epsilon}_{\pm}
\pm \frac12 E^{P,\epsilon}\cdot v f^{m,\epsilon}_{\pm}
\mp E^{P,\epsilon}\cdot\nabla_{ v  }f^{m,\epsilon}_{\pm}\nonumber\\
&&\pm\frac12\sum_{j_1+j_2\geq m,\atop 0<j_1,j_2<m} \epsilon^{j_1+j_2-m}E^{j_1,\epsilon}\cdot vf^{j_2,\epsilon}_{\pm}\mp\sum_{j_1+j_2\geq m,\atop 0<j_1,j_2<m} \epsilon^{j_1+j_2-m}E^{j_1,\epsilon}\cdot\nabla_vf^{j_2,\epsilon}_{\pm}
\pm\frac12\sum_{0<j_1< m} \epsilon^{j_1}E^{j_1,\epsilon}
\cdot vf^{m,\epsilon}_{\pm}\nonumber\\
&&\mp\sum_{0<j_1< m}\epsilon^{j_1}E^{j_1,\epsilon}
\cdot\nabla_vf^{m,\epsilon}_{\pm}\pm\frac12\sum_{0<j_1< m} \epsilon^{j_1}E^{m,\epsilon}
\cdot vf^{j_1,\epsilon}_{\pm}\mp\sum_{0<j_1< m} \epsilon^{j_1}E^{m,\epsilon}\cdot\nabla_vf^{j_1,\epsilon}_{\pm}\nonumber\\
&&\mp\epsilon \left(v\times B^{m,\epsilon}\right)\cdot \nabla_v \left\{f^{P,\epsilon}_{\pm}+\sum_{i=1}^{m-1}\epsilon^i f^{i,\epsilon}_{\pm}\right\}\mp \epsilon \left\{v\times \left(B^{P}
+\sum_{i=1}^{m-1}\epsilon^{i}B^{i}\right)\right\}\cdot\nabla_vf^{m,\epsilon}_{\pm}\\
&&\pm \frac {\epsilon^{m}} 2 E^{m,\epsilon}\cdot v f^{m,\epsilon}_{\pm}\mp \epsilon^{m} E^{m,\epsilon}\cdot\nabla_{ v  }f^{m,\epsilon}_{\pm}\mp\epsilon^{m+1} \left(v\times B^{m,\epsilon}\right)\cdot\nabla_vf^{m,\epsilon}_{\pm}\nonumber\\
  &&+{\Gamma}_\pm\left(f^{P,\epsilon},f^{m,\epsilon}\right)
  +{\Gamma}_\pm\left(f^{m,\epsilon},f^{P,\epsilon}\right)
  +\epsilon^m{\Gamma}_\pm\left(f^{m,\epsilon},f^{m,\epsilon}\right)+\sum_{j_1+j_2\geq m,\atop 0<j_1,j_2<m}\epsilon^{j_1+j_2-m}\Gamma_\pm\left(f^{j_1,\epsilon},f^{j_2,\epsilon}\right),
  \nonumber\\
  G^{j,\epsilon}&=&\left\langle v\mu^{1/2},\{{\bf I-P}\}f^{j,\epsilon} \cdot q_1 \right\rangle,\quad j=1,2,\cdots,m.\nonumber
\end{eqnarray}
\end{lemma}

Based on the above lemma, one has by repeating the argument used in the proof of {\cite[Lemma 3.2]{Lei-Zhao-JFA-2014}} that
\begin{lemma}
There exists $G_{f^{m,\epsilon},E^{m,\epsilon},B^{m,\epsilon}}(t)$ satisfying
\[
G_{f^{m,\epsilon},E^{m,\epsilon},B^{m,\epsilon}}(t)\lesssim \sum_{|\alpha|\leq N}\left\|\partial^\alpha f^{m,\epsilon}(t)\right\|^2+\sum_{|\alpha|\leq N-1}\left\|\partial^\alpha\left[ E^{m,\epsilon}(t), B^{m,\epsilon}(t)\right]\right\|^2
\]
such that
  \begin{eqnarray}\label{mac-dis-f-R}
    &&\frac{d}{dt}G_{f^{m,\epsilon},E^{m,\epsilon},B^{m,\epsilon}}(t)
  +\sum_{1\leq|\alpha|\leq N}
    \left\|\partial^\alpha {\bf P}f^{m,\epsilon}(t)\right\|^2\nonumber\\ &&+\epsilon^2\left\|E^{m,\epsilon}(t)+\epsilon b^{f^{m,\epsilon}}(t)\times \left\{B^{P}
+\sum_{j=1}^{m-1}\epsilon^{j}B^{j}\right\}\right\|^2
+\epsilon^2\sum_{1\leq|\alpha|\leq N-1}\left\|\partial^\alpha \left [E^{m,\epsilon}(t), B^{m,\epsilon}(t)\right]\right\|^2\\
    &\lesssim&\sum_{|\alpha|\leq N}\left\|\partial^\alpha\{{\bf I-P}\}f^{m,\epsilon}(t)\right\|_\nu^2+\mathcal{E}_{f^{P,\epsilon},N}(t) \mathcal{D}_{f^{m,\epsilon},N}(t)+\sum_{j_1+j_2\geq m,\atop 0<j_1,j_2\leq m}\mathcal{E}_{f^{j_2,\epsilon},N}(t)\mathcal{D}_{f^{j_1,\epsilon},N}(t)\nonumber
  \end{eqnarray}
holds for all $0\leq t\leq T$.
\end{lemma}

\subsection{Some estimates related to the electromagnetic field and the transport term}
The main purpose of this section is to derive certain estimates related to the electromagnetic field $\left[E^{m,\epsilon}(t,x), B^{m,\epsilon}(t,x)\right]$ and the linear transport term $v\cdot\nabla_xf^{m,\epsilon}(t,x,v)$. As pointed out before, compared with the problem on global solvability near Maxwellians for fixed light speed studied in \cite{Duan-Lei-Yang-Zhao-CMP-2017}, the difficulty we encountered here is due to the degeneracy of the dissipative effect of the electromagnetic field  $\left[E^{m,\epsilon}(t,x), B^{m,\epsilon}(t,x)\right]$ for large light speed, cf. the small parameter $\epsilon$ in front of terms related to the electromagnetic field  $\left[E^{m,\epsilon}(t,x), B^{m,\epsilon}(t,x)\right]$ in the definitions of the
 energy dissipation rate functionals $\mathcal{D}_{f^{m,\epsilon},n}(t)$, $\mathcal{D}_{f^{m,\epsilon},n,\ell,-\gamma}(t)$, and $\overline{\mathcal{D}}_{f^{m,\epsilon},n,\ell,-\gamma}(t)$ given by \eqref{def-D-n-g}.
 To deduce an $\epsilon-$independent estimates on $f^{m,\epsilon}(t,x,v)$, unlike \cite{Duan-Lei-Yang-Zhao-CMP-2017}, we had to use a different way to deal with all those terms related to the electromagnetic field  $\left[E^{m,\epsilon}(t,x), B^{m,\epsilon}(t,x)\right]$. Moreover, since weighted energy estimates will be used, we also need to control the terms related to the linear transport term $v\cdot\nabla_x f^{m,\epsilon}(t,,x,v)$ suitably.

\subsubsection{Some linear estimates related to the electromegnetic field and the transport term}
This section focuses on dealing with some estimates on some linear terms related to the electromagnetic field  $\left[E^{m,\epsilon}(t,x), B^{m,\epsilon}(t,x)\right]$ and the linear transport term $v\cdot\nabla_x f^{m,\epsilon}(t,,x,v)$.

Our first result is on some linear terms related to the electromagnetic field  $\left[E^{m,\epsilon}(t,x), B^{m,\epsilon}(t,x)\right]$.
\begin{lemma}\label{lemma-nonhard-1}
We have the following estimates:
\begin{itemize}
\item [(i).] For $|\alpha|+|\beta|\leq N$ with $|\beta|\geq1$ or $|\alpha|=|\beta|=0$,
it holds that
\begin{eqnarray}\label{E-R-nonhard-1}
&&\left(\partial_{\beta}^{\alpha}\left\{q_1\left\{E^{m,\epsilon}+\epsilon b^{f^{m,\epsilon}}\times \left\{B^{P}
+\sum_{j_1=1}^{m-1}\epsilon^{j_1}B^{j_1}\right\}\right\}\cdot v \mu^{1/2}\right\}, w^2_{\ell-|\beta|,\kappa}\partial^\alpha_\beta{\{\bf I-P\}}f^{m,\epsilon}\right)\nonumber\\
&=&(-1)^{|\beta|}\sum\limits_{i=1}^3\left(\left\{\partial^{\alpha}E^{m,\epsilon}+\epsilon b^{f^{m,\epsilon}}\times \left\{B^{P}
+\sum_{j_1=1}^{m-1}\epsilon^{j_1}B^{j_1}\right\}\right\}_i,\right.\nonumber\\
&&\left.\left\langle \partial_{\beta}\left\{w^2_{\ell-|\beta|,\kappa}\partial_\beta\left[v_i \mu^{1/2}\right]\right\}, \partial^\alpha{\{\bf I_+-P_+\}}f^{m,\epsilon}-\partial^\alpha{\{\bf I_--P_-\}}f^{m,\epsilon}\right\rangle\right)\\
&\lesssim&\frac 1{4C_\beta}\sum_{i=1}^3\frac{d}{dt}\left\|\left\langle \partial_{\beta}\left\{w^2_{\ell-|\beta|,\kappa}\partial_\beta\left[v_i \mu^{1/2}\right]\right\}, \partial^\alpha{\{\bf I_+-P_+\}}f^{m,\epsilon}(t)-\partial^\alpha{\{\bf I_--P_-\}}f^{m,\epsilon}(t)\right\rangle\right\|^2\nonumber\\
&&+\left\|\nabla^{|\alpha|}\{{\bf I-P}\}f^{m,\epsilon}(t)\right\|_\nu^2 +\left\|\nabla^{|\alpha|+1}f^{m,\epsilon}(t)\right\|^2_\nu
+\mathcal{E}_{f^{m,\epsilon},N}(t)\mathcal{D}_{f^{P,\epsilon},N}(t)+\sum_{j_1+j_2\geq m,\atop 0<j_1,j_2\leq m}\mathcal{E}_{f^{j_1,\epsilon},N}(t)\mathcal{D}_{f^{j_2,\epsilon},N}(t).\nonumber
\end{eqnarray}
Here
$$
C_\beta(t)=\int_{\mathbb{R}^3_v}
\left\{w_{\ell-|\beta|,\kappa}\partial_\beta\left(v_i\mu^{\frac12}\right)\right\}^2dv
$$
depends only on the time variable $t$ but satisfies $C_\beta\sim 1$ for $t\geq 0$;

\item[(ii).] For $|\alpha|\leq N-1$, it also holds that
\begin{eqnarray}\label{E-R-nonhard-2}
&&\left(\partial^{\alpha}\left\{q_1\left\{E^{m,\epsilon}+\epsilon b^{f^{m,\epsilon}}\times \left\{B^{P}
+\sum_{j_1=1}^{m-1}\epsilon^{j_1}B^{j_1}\right\}\right\}\cdot v \mu^{1/2}\right\}, w^2_{\ell-|\beta|,\kappa}\partial^\alpha f^{m,\epsilon}\right)\nonumber\\
  &=&\sum\limits_{i=1}^3\left(\partial^{\alpha}\left\{E^{m,\epsilon}+\epsilon b^{f^{m,\epsilon}}\times \left\{B^{P}
+\sum_{j_1=1}^{m-1}\epsilon^{j_1}B^{j_1}\right\}\right\}_i,\left\langle \left\{w^2_{\ell-|\beta|,\kappa}\left[v_i \mu^{1/2}\right]\right\}, \partial^\alpha f_{m,+}-\partial^\alpha f_{m,-}\right\rangle\right)\\ \nonumber
&\lesssim&\frac1{4{C_0}}\frac{d}{dt}\sum_{i=1}^3\left\|\left\langle \left\{w^2_{\ell-|\beta|,\kappa}\left[v_i \mu^{1/2}\right]\right\}, \partial^\alpha  f_{m,+}(t)-\partial^\alpha f_{m,-}(t)\right\rangle\right\|^2\\ \nonumber
&&+\left\|\nabla^{|\alpha|}f^{m,\epsilon}(t)\right\|_\nu^2 +\left\|\nabla^{|\alpha|+1}f^{m,\epsilon}(t)\right\|^2_\nu
+\mathcal{E}_{f^{m,\epsilon},N}(t)\mathcal{D}_{f^{P,\epsilon},N}(t) +\sum_{j_1+j_2\geq m,\atop 0<j_1,j_2\leq m}\mathcal{E}_{f^{j_1,\epsilon},N}(t)\mathcal{D}_{f^{j_2,\epsilon},N}(t).
\end{eqnarray}
Here again
$C_0(t)=\int_{\mathbb{R}^3_v}
\left\{w_{\ell-|\beta|,\kappa}\left(v_i\mu^{\frac12}\right)\right\}^2dv
$ depends only on $t$ and satisfies $C_0(t)\sim 1$.
\end{itemize}
 \end{lemma}
\begin{proof} To prove Lemma \ref{lemma-nonhard-1}, we can get first the following equation for ${\bf \{I-P\}}f^{m,\epsilon}$ by applying ${\bf \{I-P\}}$ to \eqref{f-i-vector}:
 \begin{eqnarray} \label{f-R-vect-I-P}
&&\partial_t{\bf \{I-P\}}f^{m,\epsilon}+ v\cdot\nabla_x{\bf \{I-P\}}f^{m,\epsilon}- \left\{E^{m,\epsilon}+\epsilon b^{f^{m,\epsilon}}\times \left\{B^{P}
+\sum_{j_1=1}^{m-1}\epsilon^{j_1}B^{j_1}\right\}\right\} \cdot v \mu^{1/2}q_1\nonumber\\
&&+v\cdot\nabla_x{\bf P}f-{\bf P}(v\cdot\nabla_x f)+{ L} f^{m,\epsilon}+\epsilon {\bf \{I-P\}}q_0\left\{v\times \left(B^{P}
+\sum_{j_1=1}^{m-1}\epsilon^{j_1}B^{j_1}\right)\right\}\cdot\nabla_v{\bf \{I-P\}}f^{m,\epsilon}\nonumber\\
&=& \frac12{\bf \{I-P\}}\left(q_0 E^{m,\epsilon}\cdot v f^{P,\epsilon}\right)-{\bf \{I-P\}}\left(q_0  E^{m,\epsilon}\cdot\nabla_vf^{P,\epsilon}\right)\nonumber\\
&&+\frac12{\bf \{I-P\}}\left(q_0\sum_{0<j_1< m} \epsilon^{j_1}E^{m,\epsilon}
\cdot vf^{j_1,\epsilon}\right)-{\bf \{I-P\}}\left(
q_0\sum_{0<j_1< m} \epsilon^{j_1}E^{m,\epsilon}\cdot\nabla_vf^{j_1,\epsilon}\right)\nonumber\\
&&-{\bf \{I-P\}}\left(
q_0\epsilon \left(v\times B^{m,\epsilon}\right)\cdot \nabla_v\left\{f^{P,\epsilon}+\sum_{i=1}^{m-1}\epsilon^if^{i,\epsilon}\right\}\right)\nonumber\\
&&+
\frac12{\bf \{I-P\}} \left(q_0\sum_{j_1+j_2\geq m,\atop 0<j_1,j_2<m} \epsilon^{j_1+j_2-m}E^{j_1,\epsilon}\cdot vf^{j_2,\epsilon}\right) -{\bf \{I-P\}}\left(q_0\sum_{j_1+j_2\geq m,\atop 0<j_1,j_2<m} \epsilon^{j_1+j_2-m}E^{j_1,\epsilon}\cdot \nabla_vf^{j_2,\epsilon}\right)\nonumber\\
&&+\frac12{\bf \{I-P\}}\left(q_0 E^{P,\epsilon}\cdot v f^{m,\epsilon}\right)+{\bf \{I-P\}}\left(q_0 E^{P,\epsilon}\cdot\nabla_{ v  }f^{m,\epsilon}\right)\\
&&+\frac12{\bf \{I-P\}}\left(q_0\sum_{0<j_1< m} \epsilon^{j_1}E^{j_1,\epsilon}
\cdot vf^{m,\epsilon}\right)+{\bf \{I-P\}}\left(q_0\sum_{0<j_1< m} \epsilon^{j_1}E^{j_1,\epsilon}
\cdot\nabla_vf^{m,\epsilon}\right)\nonumber\\
&&+ {\bf \{I-P\}}\left(\frac {\epsilon^{m}} 2q_0 E^{m,\epsilon}\cdot v f^{m,\epsilon}\right)+ {\bf \{I-P\}}\left(q_0\epsilon^{m}\left\{ E^{m,\epsilon}+ \epsilon v\times B^{m,\epsilon}\right\}\cdot\nabla_vf^{m,\epsilon}\right)\nonumber\\
&&+{\Gamma}\left(f^{P,\epsilon},f^{m,\epsilon}\right)
  +{\Gamma}\left(f^{m,\epsilon},f^{P,\epsilon}\right)
  +\sum_{j_1+j_2\geq m,\atop 0<j_1,j_2<m}\epsilon^{j_1+j_2-m}\Gamma\left(f^{j_1,\epsilon},f^{j_2,\epsilon}\right)\nonumber\\
  &&
  +\sum_{1\leq j_1\leq m-1}\epsilon^{j_1}
  \left\{\Gamma\left(f^{j_1,\epsilon},f^{m,\epsilon}\right)
  +\Gamma\left(f^{m,\epsilon},f^{j_1,\epsilon}\right)\right\}
  +\epsilon^m{\Gamma}\left(f^{m,\epsilon},f^{m,\epsilon}\right)\nonumber\\
  &:=&\left[I^{m,\epsilon,+}_{mic,mix}(t),I^{m,\epsilon,-}_{mic,mix}(t)\right],\nonumber
  \end{eqnarray}
where we $[I^{m,\epsilon,+}_{mic,mix}(t),I^{m,\epsilon,-}_{mic,mix}(t)]$ to denote the all terms in the right-hand side of the above vector equality for brevity.

Multiplying \eqref{f-R-vect-I-P} by $\partial_{\beta}\left\{w^2_{\ell-|\beta|,\kappa}\partial_\beta\left[v_i\mu^{1/2}\right]\right\}$ and integrating the resulting identity with respect to $v$ over $\mathbb{R}^3$, one has
  \begin{eqnarray}\label{key1}
  &&\partial_t\left\langle \partial_{\beta}\left\{w^2_{\ell-|\beta|,\kappa}\partial_\beta\left[v_i\mu^{1/2}\right]\right\},\{{\bf I_\pm-P_\pm}\}f^{m,\epsilon}\right\rangle+\nabla_x\cdot\left\langle v\partial_{\beta}\left\{w^2_{\ell-|\beta|,\kappa}\partial_\beta\left[v_i\mu^{1/2}\right]\right\}, \{{\bf I_\pm-P_\pm}\}f^{m,\epsilon} \right\rangle\nonumber\\
  &&\mp\underbrace{\left\langle \partial_{\beta}\left\{w^2_{\ell-|\beta|,\kappa}\partial_\beta\left[v_i \mu^{1/2}\right]\right\},v_i\mu^{1/2}\right\rangle}_{{J}_1} \left\{E^{m,\epsilon}+\epsilon b^{f^{m,\epsilon}}\times \left\{B^{P}
+\sum_{j_1=1}^{m-1}\epsilon^{j_1}B^{j_1}\right\}\right\}_i\\
&=&\left\langle \partial_t\partial_{\beta}\left\{w^2_{\ell-|\beta|,\kappa}\partial_\beta\left[v_i\mu^{1/2}\right]\right\},\{{\bf I_\pm-P_\pm}\}f^{m,\epsilon}\right\rangle\nonumber\\
&&+\left\langle \partial_{\beta}\left\{w^2_{\ell-|\beta|,\kappa}\partial_\beta\left[v_i \mu^{1/2}\right]\right\}, {\bf P}_{\pm}\left\{v\cdot\nabla_x f^{m,\epsilon}\right\}
-{ L} _{\pm}f^{m,\epsilon}- v  \cdot\nabla_x{\bf P}_{\pm}f^{m,\epsilon}
\right\rangle \nonumber\\
&&+\left\langle \partial_{\beta}\left\{w^2_{\ell-|\beta|,\kappa}\partial_\beta\left[v_i \mu^{1/2}\right]\right\},
-\epsilon {\bf \{I_{\pm}-P_{\pm}\}}q_0\left\{v\times \left(B^{P}
+\sum_{j_1=1}^{m-1}\epsilon^{j_1}B^{j_1}\right)\right\}\cdot\nabla_v{\bf \{I-P\}}f^{m,\epsilon}\right\rangle \nonumber\\
&&+\left\langle \partial_{\beta}\left\{w^2_{\ell-|\beta|,\kappa}\partial_\beta\left[v_i \mu^{1/2}\right]\right\}, I^{m,\epsilon,\pm}_{mic,mix}(t)
\right\rangle \nonumber.
\end{eqnarray}

Noticing that
\[
{J}_1=(-1)^{|\beta|}\int_{\mathbb{R}^3_v}
\left\{w_{\ell-|\beta|,\kappa}\partial_\beta(v_i\mu^{\frac12})\right\}^2dv\equiv (-1)^{|\beta|}{C}_{\beta},
\]
we can get from \eqref{key1} that
\begin{eqnarray}\label{key2}
 &&\partial_t\left\langle \partial_{\beta}\left\{w^2_{\ell-|\beta|,\kappa}\partial_\beta\left[v_i \mu^{1/2}\right]\right\},\{{\bf I_+-P_+}\}f^{m,\epsilon}-\{{\bf I_--P_-}\}f^{m,\epsilon}\right\rangle\nonumber\\
 &&-2(-1)^{|\beta|}C_\beta \left\{E^{m,\epsilon}+\epsilon b^{f^{m,\epsilon}}\times \left\{B^{P}
+\sum_{j_1=1}^{m-1}\epsilon^{j_1}B^{j_1}\right\}\right\}_i\\ \nonumber
&=&\left\langle \partial_t\partial_{\beta}\left\{w^2_{\ell-|\beta|,\kappa}\partial_\beta\left[v_i\mu^{1/2}\right]\right\},\{{\bf I_+-P_+}\}f^{m,\epsilon}-\{{\bf I_--P_-}\}f^{m,\epsilon}\right\rangle\\ \nonumber
&&-\nabla_x\cdot\left\langle v\partial_{\beta}\left\{w^2_{\ell-|\beta|,\kappa}\partial_\beta\left[v_i \mu^{1/2}\right]\right\},\{{\bf I_+-P_+}\}f^{m,\epsilon}-\{{\bf I_--P_-}\}f^{m,\epsilon} \right\rangle\\ \nonumber
&&+\left\langle \partial_{\beta}\left\{w^2_{\ell-|\beta|,\kappa}\partial_\beta\left[v_i \mu^{1/2}\right]\right\}, \{{\bf
P}_{+}-{\bf P}_-\}\{v\cdot\nabla_x f^{m,\epsilon}\}
-\{{ L} _{+}-L_-\}f^{m,\epsilon}- v  \cdot\nabla_x{\bf P}_{+}f^{m,\epsilon}+ v  \cdot\nabla_x{\bf P}_{-}f^{m,\epsilon}\right\rangle\\ \nonumber
&&+\left\langle \partial_{\beta}\left\{w^2_{\ell-|\beta|,\kappa}\partial_\beta\left[v_i \mu^{1/2}\right]\right\},-\epsilon {\bf \{I_+-P_+\}}q_0\left\{v\times \left(B^{P}
+\sum_{j_1=1}^{m-1}\epsilon^{j_1}B^{j_1}\right)\right\}\cdot\nabla_v{\bf \{I-P\}}f^{m,\epsilon}\right\rangle\\ \nonumber
&&+\left\langle \partial_{\beta}\left\{w^2_{\ell-|\beta|,\kappa}\partial_\beta\left[v_i \mu^{1/2}\right]\right\},\epsilon {\bf \{I_--P_-\}}q_0\left\{v\times \left(B^{P}
+\sum_{j_1=1}^{m-1}\epsilon^{j_1}B^{j_1}\right)\right\}\cdot\nabla_v{\bf \{I-P\}}f^{m,\epsilon}\right\rangle\\ \nonumber
&&+\left\langle \partial_{\beta}\left\{w^2_{\ell-|\beta|,\kappa}\partial_\beta\left[v_i \mu^{1/2}\right]\right\},I^{m,\epsilon,+}_{mic,mix}(t)-I^{m,\epsilon,-}_{mic,mix}(t)
\right\rangle.
\end{eqnarray}
Applying $\partial^\alpha$ to \eqref{key2}, one has
\begin{eqnarray}\label{key3}
&&\partial_t\left\langle \partial_{\beta}\left\{w^2_{\ell-|\beta|,\kappa}\partial_\beta\left[v_i \mu^{1/2}\right]\right\},\partial^\alpha\{{\bf I_+-P_+}\}f^{m,\epsilon}-\partial^\alpha\{{\bf I_--P_-}\}f^{m,\epsilon}\right\rangle\nonumber\\
&&-2(-1)^{|\beta|}C_\beta\partial^\alpha \left\{E^{m,\epsilon}+\epsilon b^{f^{m,\epsilon}}\times \left\{B^{P}
+\sum_{j_1=1}^{m-1}\epsilon^{j_1}B^{j_1}\right\}\right\}_i\\ \nonumber
&=&\left\langle \partial_t\partial_{\beta}\left\{w^2_{\ell-|\beta|,\kappa}\partial_\beta\left[v_i\mu^{1/2}\right]\right\},\partial^\alpha\{{\bf I_+-P_+}\}f^{m,\epsilon}-\partial^\alpha\{{\bf I_--P_-}\}f^{m,\epsilon}\right\rangle\\ \nonumber
&&-\nabla_x\cdot\left\langle v\partial_{\beta}\left\{w^2_{\ell-|\beta|,\kappa}\partial_\beta\left[v_i \mu^{1/2}\right]\right\},\partial^\alpha\{{\bf I_+-P_+}\}f^{m,\epsilon}-\partial^\alpha\{{\bf I_--P_-}\}f^{m,\epsilon} \right\rangle\\ \nonumber
&&+\left\langle \partial_{\beta}\left\{w^2_{\ell-|\beta|,\kappa}\partial_\beta\left[v_i \mu^{1/2}\right]\right\}, \partial^\alpha{\bf
P}_{+}\{v\cdot\nabla_x f^{m,\epsilon}\}-\partial^\alpha{\bf
P}_{-}\{v\cdot\nabla_x f^{m,\epsilon}\}\right.\\[2mm] \nonumber
&&\left.\ \ \ \ \ \ \ \ \ \ \ \ \ \ \ \ \ -\partial^\alpha{ L} _{+}f^{m,\epsilon}+\partial^\alpha{ L} _{-}f^{m,\epsilon}- v  \cdot\nabla_x\partial^\alpha{\bf P}_{+}f^{m,\epsilon}+ v  \cdot\nabla_x\partial^\alpha{\bf P}_{-}f^{m,\epsilon}\right\rangle\\ \nonumber
&&+\left\langle \partial_{\beta}\left\{w^2_{\ell-|\beta|,\kappa}\partial_\beta\left[v_i \mu^{1/2}\right]\right\},-\epsilon {\bf \{I_+-P_+\}}q_0\left\{v\times \left(B^{P}
+\sum_{j_1=1}^{m-1}\epsilon^{j_1}B^{j_1}\right)\right\}\cdot\nabla_v\partial^\alpha{\bf \{I-P\}}f^{m,\epsilon}\right\rangle\\ \nonumber
&&+\left\langle \partial_{\beta}\left\{w^2_{\ell-|\beta|,\kappa}\partial_\beta\left[v_i \mu^{1/2}\right]\right\},\epsilon {\bf \{I_--P_-\}}q_0\left\{v\times \left(B^{P}
+\sum_{j_1=1}^{m-1}\epsilon^{j_1}B^{j_1}\right)\right\}\cdot\nabla_v\partial^\alpha{\bf \{I-P\}}f^{m,\epsilon}\right\rangle\\ \nonumber
&&+\left\langle \partial_{\beta}\left\{w^2_{\ell-|\beta|,\kappa}\partial_\beta\left[v_i \mu^{1/2}\right]\right\},\partial^\alpha I^{m,\epsilon,+}_{mic,mix}(t)-\partial^\alpha I^{m,\epsilon,-}_{mic,mix}(t)
\right\rangle\\
&:=&\mathcal{{RHS}}_{1},\nonumber
\end{eqnarray}
where we use $\mathcal{{RHS}}_{1}$ to denote all terms in the right hand side of \eqref{key3}.

Therefore, we can get from Lemma \ref{lemma-nonlinear}, Lemma \ref{Lemma-5.3}, and Lemma \ref{Lemma-5.4} that
\begin{eqnarray*}
  &&(-1)^{|\beta|}\left(\partial^{\alpha}\left\{E^{m,\epsilon}+\epsilon b^{f^{m,\epsilon}}\times \left\{B^{P}
+\sum_{j_1=1}^{m-1}\epsilon^{j_1}B^{j_1}\right\}\right\}_i,\right.\\
&&\left.\quad\quad\quad\quad\quad\left\langle \partial_{\beta}\left\{w^2_{\ell-|\beta|,\kappa}\partial_\beta\left[v_i \mu^{1/2}\right]\right\}, \partial^\alpha{\{\bf I_+-P_+\}}f^{m,\epsilon}-\partial^\alpha{\{\bf I_--P_-\}}f^{m,\epsilon}\right\rangle\right)\\
&=&\frac1{2C_\beta}\left(\partial_t\left\langle \partial_{\beta}\left\{w^2_{\ell-|\beta|,\kappa}\partial_\beta\left[v_i \mu^{1/2}\right]\right\},\partial^\alpha\{{\bf I_+-P_+}\}f^{m,\epsilon}-\partial^\alpha\{{\bf I_--P_-}\}f^{m,\epsilon}\right\rangle-\mathcal{RHS}_1,\right.\\
  &&\qquad\qquad\quad \left.\left\langle \partial_{\beta}\left\{w^2_{\ell-|\beta|,\kappa}\partial_\beta\left[v_i \mu^{1/2}\right]\right\}, \partial^\alpha{\{\bf I_+-P_+\}}f^{m,\epsilon}-\partial^\alpha{\{\bf I_--P_-\}}f^{m,\epsilon}\right\rangle\right)\\ \nonumber
  &=&\frac1{4C_\beta}\sum_{i=1}^3\frac{d}{dt}\left\|\left\langle \partial_{\beta}\left\{w^2_{\ell-|\beta|,\kappa}\partial_\beta\left[v_i \mu^{1/2}\right]\right\}, \partial^\alpha{\{\bf I_+-P_+\}}f^{m,\epsilon}-\partial^\alpha{\{\bf I_--P_-\}}f^{m,\epsilon}\right\rangle\right\|^2\\ \nonumber
  &&+\frac1{2C_\beta}\left(\mathcal{RHS}_1,\left\langle \partial_{\beta}\left\{w^2_{\ell-|\beta|,\kappa}\partial_\beta\left[v_i \mu^{1/2}\right]\right\}, \partial^\alpha{\{\bf I_+-P_+\}}f^{m,\epsilon}-\partial^\alpha{\{\bf I_--P_-\}}f^{m,\epsilon}\right\rangle\right)\\ \nonumber
  &\lesssim&\frac1{4C_\beta}\sum_{i=1}^3\frac{d}{dt}\left\|\left\langle \partial_{\beta}\left\{w^2_{\ell-|\beta|,\kappa}\partial_\beta\left[v_i \mu^{1/2}\right]\right\}, \partial^\alpha{\{\bf I_+-P_+\}}f^{m,\epsilon}-\partial^\alpha{\{\bf I_--P_-\}}f^{m,\epsilon}\right\rangle\right\|^2\\ \nonumber
  &&+\left\|\nabla^{|\alpha|}\{{\bf I-P}\}f^{m,\epsilon}\right\|_\nu^2+\left\|\nabla^{|\alpha|+1}f^{m,\epsilon}\right\|^2_\nu
 +\mathcal{E}_{f^{m,\epsilon},N}(t)\mathcal{D}_{f^{P,\epsilon},N}(t)+\sum_{j_1+j_2\geq m,\atop 0<j_1,j_2\leq m}\mathcal{E}_{f^{j_1,\epsilon},N}(t)\mathcal{D}_{f^{j_2,\epsilon},N}(t),\nonumber
 \end{eqnarray*}
which is \eqref{E-R-nonhard-1},
and \eqref{E-R-nonhard-2} can be obtained in a similar way. Thus we have completed the proof of this lemma.
\end{proof}
For the corresponding term related to the linear transport term $v\cdot\nabla_x f^{m,\epsilon}(t,x,v)$, we have the following result.
\begin{lemma}
  Assume $|\alpha|+|\beta|\leq N_m+1$ with $|\beta|\geq 1$,and
  \begin{equation}\label{sigma-relation}
    \sigma_{n,k}-\sigma_{n,k-1}=\frac{2(1+\gamma)}{\gamma-2}(1+\vartheta),\ \ 1\leq k\leq n.
  \end{equation}
 one has
  \begin{eqnarray}\label{linear-fi-high-R}
    &&(1+t)^{-\sigma_{n,|\beta|}}\left(\partial^\alpha_\beta(v\cdot \nabla_x{\bf\{I-P\}}f^{m,\epsilon}), w^2_{\ell_m^*-|\beta|,1}\partial^\alpha_\beta{\bf\{I-P\}}f^{m,\epsilon}\right)\nonumber\\
    &\lesssim&(1+t)^{-\sigma_{n,|\beta|-1}}\mathcal{D}^{(n,|\beta|-1)}_{f^{m,\epsilon},n,\ell_m^*,1}(t)+
    \eta(1+t)^{-\sigma_{n,|\beta|}}\left\|w_{\ell_m^*-|\beta|,1} \partial^\alpha_\beta{\bf\{I-P\}}f^{m,\epsilon}(t) \langle v\rangle^{\frac\gamma2}\right\|^2.
  \end{eqnarray}
\end{lemma}
\begin{proof}
By Cauchy's inequality, one has
    \begin{eqnarray*}
    &&\left(\partial^\alpha_\beta(v\cdot \nabla_x{\bf\{I-P\}}f^{m,\epsilon}), w^2_{\ell_m^*-|\beta|,1}\partial^\alpha_\beta{\bf\{I-P\}}f^{m,\epsilon}\right)\nonumber\\
    &\lesssim&\left\|w_{\ell_m^*-|\beta-e_j|,1}\partial^{\alpha+e_j}_{\beta-e_j}{\bf\{I-P\}} f^{m,\epsilon}\langle v\rangle^{-\frac\gamma2-1}\right\|^2 +\eta\left\|w_{\ell_m^*-|\beta|,1}\partial^\alpha_\beta{\bf\{I-P\}}f^{m,\epsilon}(t)\langle v\rangle^{\frac\gamma2}\right\|^2.
  \end{eqnarray*}

  To control $\left\|w_{\ell_m^*-|\beta-e_j|,1}\partial^{\alpha+e_j}_{\beta-e_j}{\bf\{I-P\}}f^{m,\epsilon}\langle v\rangle^{-\frac\gamma2-1}\right\|^2$, one applies different time-rates
  factors to deduce
 \begin{eqnarray}\label{linear-fi-compute-R}
&&\sum_{|\alpha|+|\beta|=n}(1+t)^{-\sigma_{n,|\beta|}}
\left\|w_{\ell_m^*-|\beta-e_j|,1}\partial^{\alpha+e_j}_{\beta-e_j} \{{\bf I-P}\} f^{m,\epsilon}\langle v\rangle^{-\frac\gamma2-1}\right\|
^2\\ \nonumber
&\lesssim&\sum_{|\alpha|+|\beta|=n}(1+t)^{-\sigma_{n,|\beta|-1}-\frac{2(1+\gamma)}{\gamma-2}(1+\vartheta)}
\left\|w_{\ell_m^*-|\beta-e_j|,1}\partial^{\alpha+e_j}_{\beta-e_j} \{{\bf I-P}\} f^{m,\epsilon}\langle v\rangle\right\|^{\frac{4(1+\gamma)}{\gamma-2}}\\ \nonumber
&&\times \left\|w_{\ell_m^*-|\beta-e_j|,1}\partial^{\alpha+e_j}_{\beta-e_j} \{{\bf I-P}\} f^{m,\epsilon}\langle v\rangle^\frac\gamma 2\right\|^{\frac{-2\gamma-8}{\gamma-2}}\\ \nonumber
&\lesssim&\sum_{|\alpha|+|\beta|=n}\left\{(1+t)^{-\sigma_{n,|\beta|-1}-1-\vartheta}
\left\|w_{\ell_m^*-|\beta-e_j|,1} \partial^{\alpha+e_j}_{\beta-e_j}\{{\bf I-P}\} f^{m,\epsilon}\langle v\rangle\right\|^{2}\right.\\
&&\left.+(1+t)^{-\sigma_{n,|\beta|-1}}\left\|w_{\ell_m^*-|\beta-e_j|,1} \partial^{\alpha+e_j}_{\beta-e_j}\{{\bf I-P}\} f^{m,\epsilon}\right\|^2_\nu\right\},\nonumber
\end{eqnarray}
where we used \eqref{sigma-relation},
then one has \eqref{linear-fi-high-R} by \eqref{linear-fi-compute-R}.
\end{proof}
\subsubsection{Some estimates on the interactions between $\left[f^{m,\epsilon}(t,x,v), E^{m,\epsilon}(t,x), B^{m,\epsilon}(t,x)\right]$, $\left[f^{P,\epsilon}(t,x,v), E^{P,\epsilon}(t,x)\right]$, and  $\left[f^{i,\epsilon}(t,x,v), E^{i,\epsilon}(t,x)\right]$ $(i=1,2,3,\cdots,m-1)$}

In this subsection, we will focus on some estimates on the interactions between $\left[f^{m,\epsilon}(t,x,v), E^{m,\epsilon}(t,x), B^{m,\epsilon}(t,x)\right]$, $\left[f^{P,\epsilon}(t,x,v), E^{P,\epsilon}(t,x)\right]$, and  $\left[f^{i,\epsilon}(t,x,v), E^{i,\epsilon}(t,x)\right]$ $(i=1,\cdots,m-1)$, in which the electromagnetic $\left[E^{m,\epsilon}(t,x), B^{m,\epsilon}(t,x)\right]$ are involved, with respect to the weight $w_{l_m-|\beta|,\kappa}(t,v)$ for both $\kappa=-\gamma$ and $\kappa=1$.

First for estimates on the interactions between $E^{m,\epsilon}(t,x)$ and $\left[f^{P,\epsilon}(t,x,v), f^{m,\epsilon}(t,x,v)\right]$ with respect to the weight $w_{l_m-|\beta|,\kappa}(t,v)$ for both $\kappa=-\gamma$ and $\kappa=1$, we have the following result.
\begin{lemma} Take $n\geq 2$ and $w_{l_m-|\beta|,\kappa}(t,v)
=\langle v\rangle^{\kappa(l_m-|\beta|)}e^{\frac{q\langle v\rangle^2}{(1+t)^\vartheta}}$, one has the following estimates:
\begin{itemize}
\item [i)]For $1\leq |\alpha|\leq n$, one has
\begin{eqnarray}\label{lemma-f-R-1}
  &&\left|\left(\partial^\alpha\left\{\frac12 q_0E^{m,\epsilon}\cdot v f^{P,\epsilon}\right\}, \partial^\alpha f^{m,\epsilon}\right)\right|\\
  &\lesssim&\left\|E^{m,\epsilon}(t)\right\|_{H^n_x}^2\left\| \nabla_xf^{P,\epsilon}(t)\right\|_{H^{n-1}_xL^2_\nu}^2+\left\| E^{m,\epsilon}(t)\right\|_{H^{n}_x}^2\left\| {\bf \{I-P\}}f^{P,\epsilon}(t)\langle v\rangle^{1-\frac\gamma2}\right\|^2_{H^{n}_x}+\eta\left\|\partial^\alpha f^{m,\epsilon}(t)\langle v\rangle^{\frac\gamma2}\right\|^2\nonumber\\
  &\lesssim&\mathcal{E}_{f^{m,\epsilon},n}(t)\mathcal{D}_{f^{P,\epsilon},n,n+1-\frac1\gamma,-\gamma}(t)
  +\eta\mathcal{D}_{f^{m,\epsilon},n}(t).\nonumber
\end{eqnarray}
\item [ii)]For $1\leq |\alpha|\leq n$, one has
\begin{eqnarray}\label{lemma-f-R-2}
  &&\left|\left(\partial^\alpha\left\{\frac12 q_0E^{m,\epsilon}\cdot v f^{P,\epsilon}\right\},w^2_{l_m,\kappa}\partial^\alpha f^{m,\epsilon}\right)\right|\\
 &\lesssim&\mathcal{E}_{f^{m,\epsilon},n}(t)\mathcal{D}_{f^{P,\epsilon},n,-\frac{\kappa l_m}{\gamma}+n+1+\frac{\kappa n-1}\gamma,-\gamma}(t)+\eta\left\|w_{l_m,\kappa}\partial^\alpha f^{m,\epsilon}(t)\langle v\rangle^{\frac\gamma2}\right\|^2\nonumber.
\end{eqnarray}
\item [iii)] For $|\alpha|+|\beta|\leq n$ with $|\beta|\geq 1$ or $|\alpha|=|\beta|=0$, one has
  \begin{eqnarray}\label{lemma-f-R-3}
    &&\left|\left(\partial^\alpha_\beta\left\{{\bf \{I-P\}}\left(\frac\epsilon2 q_0E^{m,\epsilon}\cdot v f^{P,\epsilon}\right)\right\},w^2_{l_m-|\beta|,\kappa}\partial^\alpha_\beta{\bf \{I-P\}}f^{m,\epsilon}\right)\right|\\
    &\lesssim&\mathcal{E}_{f^{m,\epsilon},n}(t)\mathcal{D}_{f^{P,\epsilon},n,-\frac{\kappa l_m}{\gamma}+n+1+\frac{\kappa n-1}\gamma,-\gamma}(t)+\eta\left\|w_{l_m-|\beta|,\kappa}\partial^\alpha_\beta{\bf \{I-P\}}f^{m,\epsilon}(t)\langle v\rangle^{\frac\gamma2}\right\|^2.\nonumber
  \end{eqnarray}
  \end{itemize}
\end{lemma}
\begin{proof}
  We only prove \eqref{lemma-f-R-3} for simplicity. For this purpose, we can deduce from the triangle inequality that
   \begin{eqnarray}\label{J-alpha-beta}
    &&\left|\left(\partial^\alpha_\beta\left\{{\bf \{I-P\}}\left(\frac12 q_0E^{m,\epsilon}\cdot v f^{P,\epsilon}\right)\right\},w^2_{l_m-|\beta|,\kappa}\partial^\alpha_\beta{\bf \{I-P\}}f^{m,\epsilon}\right)\right|\nonumber\\
    &\lesssim&\underbrace{\left|\left(\partial^\alpha_\beta\left\{\left(\frac12 q_0E^{m,\epsilon}\cdot v {\bf \{I-P\}} f^{P,\epsilon}\right)\right\},w^2_{l_m-|\beta|,\kappa}\partial^\alpha_\beta{\bf \{I-P\}}f^{m,\epsilon}\right)\right|}_{I_{\alpha,\beta,1}}\\
    &&+\underbrace{\left|\left(\partial^\alpha_\beta\left\{\frac12 q_0E^{m,\epsilon}\cdot v {\bf P} f^{P,\epsilon}\right\},w^2_{l_m-|\beta|,\kappa}\partial^\alpha_\beta{\bf \{I-P\}}f^{m,\epsilon}\right)\right|}_{I_{\alpha,\beta,2}}\nonumber\\
    &&+\underbrace{\left|\left(\partial^\alpha_\beta\left\{{\bf P}\left(\frac12 q_0E^{m,\epsilon}\cdot v f^{P,\epsilon}\right)\right\},w^2_{l_m-|\beta|,\kappa}\partial^\alpha_\beta{\bf \{I-P\}}f^{m,\epsilon}\right)\right|}_{I_{\alpha,\beta,3}}.\nonumber
  \end{eqnarray}
$I_{\alpha,\beta,1}$ can be bounded by
  \begin{eqnarray*}
    I_{\alpha,\beta,1}&\lesssim&\left|\left(\frac12 q_0E^{m,\epsilon}\cdot v \partial^\alpha_\beta{\bf \{I-P\}}f^{P,\epsilon},w^2_{l_m-|\beta|,\kappa}\partial^\alpha_\beta{\bf \{I-P\}}f^{m,\epsilon}\right)\right|\\
    &&+\left|\left(\frac12 q_0E^{m,\epsilon}\cdot \partial_{e_j}v \partial^\alpha_{\beta-e_j}{\bf \{I-P\}}f^{P,\epsilon},w^2_{l_m-|\beta|,\kappa}\partial^\alpha_\beta{\bf \{I-P\}}f^{m,\epsilon}\right)\right|\\
    &&+\sum_{0<\alpha_1\leq \alpha}\left|\left(\frac12 q_0\partial^{\alpha_1}E^{m,\epsilon}\cdot v \partial^{\alpha-\alpha_1}_\beta{\bf \{I-P\}}f^{P,\epsilon},w^2_{l_m-|\beta|,\kappa}\partial^\alpha_\beta{\bf \{I-P\}}f^{m,\epsilon}\right)\right|.
      \end{eqnarray*}
  By $L^2-L^\infty-L^2$, $L^2-L^\infty-L^2$, $L^3-L^6-L^2$ or $L^6-L^3-L^2$ type Sobolev  inequalities, one can deduce
  \begin{eqnarray}
  I_{\alpha,\beta,1}
  &\lesssim&\sum_{|\alpha|+|\beta|\leq n,\atop |\beta|\geq 1}\left\| \nabla_xE^{m,\epsilon}\right\|^2_{H^{n-2}_x}\left\|w_{l_m-|\beta|,\kappa}\partial^\alpha_\beta {\bf \{I-P\}}f^{P,\epsilon}\langle v\rangle^{1-\frac\gamma2}\right\|^2\nonumber\\
  &&+\eta\left\|w_{l_m-|\beta|,\kappa}\partial^\alpha_\beta{\bf \{I-P\}}f^{m,\epsilon}\langle v\rangle^{\frac\gamma2}\right\|^2\nonumber\\
  &\lesssim&\mathcal{E}_{f^{m,\epsilon},n}(t)\mathcal{D}_{f^{P,\epsilon},n,-\frac{\kappa l_m}{\gamma}+n+1+\frac{\kappa n-1}\gamma,-\gamma}(t)+\eta\left\|w_{l_m-|\beta|,\kappa}\partial^\alpha_\beta{\bf \{I-P\}}f^{m,\epsilon}\langle v\rangle^{\frac\gamma2}\right\|^2\nonumber.
  \end{eqnarray}
  For $I_{\alpha,\beta,2}$ and $I_{\alpha,\beta,3}$, one has
  \begin{eqnarray*}
    I_{\alpha,\beta,2}+I_{\alpha,\beta,3}&\lesssim&
    \left\|\nabla_xE^{m,\epsilon}\right\|^2_{H^{n-2}_x}\|\nabla_x f^{P,\epsilon}\|^2_{H^{n-1}_xL^2_\nu} +\eta\left\|\partial^\alpha_\beta{\bf \{I-P\}}f^{m,\epsilon}\langle v\rangle^{\frac\gamma2}\right\|^2\nonumber\\
    &\lesssim&\mathcal{E}_{f^{m,\epsilon},n}(t)\mathcal{D}_{f^{P,\epsilon},n}(t)+\eta\left\|\partial^\alpha_\beta{\bf \{I-P\}}f^{m,\epsilon}(t)\langle v\rangle^{\frac\gamma2}\right\|^2.
  \end{eqnarray*}
  Thus \eqref{lemma-f-R-3} follows by collecting the estimates on $I_{\alpha,\beta,1}$, $I_{\alpha,\beta,2}$ and $I_{\alpha,\beta,3}$ into \eqref{J-alpha-beta}.
\end{proof}

Similarly, one can also deduce the following estimates on the interactions between $\left[E^{m,\epsilon}(t,x), B^{m,\epsilon}(t,x)\right]$ with $\left[f^{P,\epsilon}(t,x,v), f^{1,\epsilon}(t,x,v), \cdots, f^{m,\epsilon}(t,x,v)\right]$ with respect to the weight $w_{l_m-|\beta|,\kappa}(t,v)$ for both $\kappa=-\gamma$ and $\kappa=1$:
\begin{lemma} Take $n\geq 2$ and $w_{l_m-|\beta|,\kappa}
=\langle v\rangle^{\kappa(l_i-|\beta|)}e^{\frac{q\langle v\rangle^2}{(1+t)^\vartheta}}$, one has the following estimates:
\begin{itemize}
\item [i)] For $1\leq |\alpha|\leq n$ and $0<j_1<m$, one has
\begin{eqnarray}\label{0-order-E-R-B-R}
  &&\left|\left(\partial^\alpha \left\{\frac12q_0\epsilon^{j_1}E^{m,\epsilon}\cdot f^{j_1,\epsilon}-q_0\epsilon^{j_1}E^{m,\epsilon}\cdot \nabla_v f^{{j_1},\epsilon}-q_0E^{m,\epsilon}\cdot\nabla_vf^{P,\epsilon}\right\}, \partial^\alpha f^{m,\epsilon}\right)\right|\\
   &\lesssim&\mathcal{E}_{f^{m,\epsilon},n}(t)\left\{\mathcal{D}_{f^{j_1,\epsilon},n+1,n+2-\frac1\gamma,-\gamma}(t)+\mathcal{D}_{f^{P,\epsilon},n+1,n+2-\frac1\gamma,-\gamma}(t)\right\}
  +\eta\mathcal{D}_{f^{m,\epsilon},n}(t)\nonumber
\end{eqnarray}
and one also has for $j_1+j_2\geq m$ with $0<j_1,j_2<m$ that
\begin{eqnarray}\label{0-order-E-R-12}
 &&\left(\partial^\alpha\left\{\frac12 q_0 \epsilon^{j_1+j_2-m} E^{j_1,\epsilon}\cdot vf^{j_2,\epsilon}-q_0\epsilon ^{j_1+j_2-m}E^{j_1,\epsilon}\nabla_vf^{j_2,\epsilon}\right.\right.\nonumber\\
 &&\left.\left.\quad\quad\quad\quad\quad\quad\quad-q_0\epsilon \left(v\times B^{m,\epsilon}\right)\cdot\nabla_v\left\{f^{P,\epsilon}+\sum_{i=1}^{m-1}\epsilon^if^{i,\epsilon}\right\}\right\}, \partial^\alpha f^{m,\epsilon}\right)\\
   &\lesssim&
  \mathcal{E}_{f^{m,\epsilon},n}(t)\left\{\mathcal{D}_{f^{P,\epsilon},n+1,n+2-\frac1\gamma,-\gamma}(t)+\sum_{i=1}^{m-1}\mathcal{D}_{f^{i,\epsilon},n+1,n+2-\frac1\gamma,-\gamma}(t)\right\}\nonumber\\
  &&+\mathcal{E}_{f^{j_1,\epsilon},n}(t) \mathcal{D}_{f^{j_2,\epsilon},n+1,n+2-\frac1\gamma,-\gamma}(t) +\eta\mathcal{D}_{f^{m,\epsilon},n}(t).\nonumber
\end{eqnarray}
\item[ii)]
For $1\leq |\alpha|\leq n$ with $0<j_1<m$, one has
\begin{eqnarray}\label{i-order-E-R-B}
  &&\left(\partial^\alpha\left\{\frac12q_0\epsilon^{j_1}E^{m,\epsilon}\cdot f^{j_1,\epsilon}-q_0\epsilon^{j_1}E^{m,\epsilon}\cdot \nabla_v f^{{j_1},\epsilon}-q_0E^{m,\epsilon}\cdot\nabla_vf^{P,\epsilon}\right\},w^2_{l_m,\kappa}\partial^\alpha f^{m,\epsilon}\right)\nonumber\\
    &\lesssim&\mathcal{E}_{f^{m,\epsilon},n}(t)\left\{\mathcal{D}_{f^{P,\epsilon},n+1,-\frac{\kappa l_m}{\gamma}+n+2+\frac{\kappa n-1}\gamma,-\gamma}(t)+\mathcal{D}_{f^{j_1,\epsilon},n+1,-\frac{\kappa l_m}{\gamma}+n+2+\frac{\kappa n-1}\gamma,-\gamma}(t)\right\}\\
  &&
+\eta\left\|w_{l_m,\kappa}\partial^\alpha f^{m,\epsilon}(t)\right\|_\nu^2\nonumber
\end{eqnarray}
and one also has for $j_1+j_2\geq m$ with $0<j_1,j_2<m$ that
\begin{eqnarray}\label{i-order-E-R-12}
 &&\left(\partial^\alpha\left\{\frac12 q_0 \epsilon^{j_1+j_2-m} E^{j_1,\epsilon}\cdot vf^{j_2,\epsilon}-q_0\epsilon ^{j_1+j_2-m}E^{j_1,\epsilon}\nabla_vf^{j_2,\epsilon}\right.\right.\nonumber\\
 &&\left.\left.\quad\quad\quad\quad\quad\quad\quad-q_0\epsilon \left(v\times B^{m,\epsilon}\right)\cdot\nabla_v\left(f^{P,\epsilon}+\sum_{i=1}^{m-1}\epsilon^if^{i,\epsilon}\right)\right\}, \partial^\alpha f^{m,\epsilon}\right)\\
   &\lesssim&
  \mathcal{E}_{f^{m,\epsilon},n}(t)\left\{\mathcal{D}_{f^{P,\epsilon},n+1,-\frac{\kappa l_m}{\gamma}+n+2+\frac{\kappa n-1}\gamma,-\gamma}(t)+\sum_{i=1}^{m-1}\mathcal{D}_{f^{i,\epsilon},n+1,-\frac{\kappa l_m}{\gamma}+n+2+\frac{\kappa n-1}\gamma,-\gamma}(t)\right\}\nonumber\\
&&+\mathcal{E}_{f^{j_1,\epsilon},n}(t)\mathcal{D}_{f^{j_2,\epsilon},n+1,-\frac{\kappa l_m}{\gamma}+n+2+\frac{\kappa n-1}\gamma,-\gamma}(t)+\eta\left\|w_{l_m,\kappa}\partial^\alpha f^{m,\epsilon}(t)\right\|_\nu^2.\nonumber
\end{eqnarray}
\item [iii)] For $|\alpha|+|\beta|\leq n,\ |\beta|\geq 1$ or $\alpha=\beta=0$, one has
  \begin{eqnarray}\label{micro-weight-E-R}
   &&\left(\partial^\alpha_\beta\left\{{\bf\{I-P\}}\left\{\frac12q_0\epsilon^{j_1}E^{m,\epsilon}\cdot f^{j_1,\epsilon}-q_0\epsilon^{j_1}E^{m,\epsilon}\cdot \nabla_v f^{{j_1},\epsilon}-q_0E^{m,\epsilon}\cdot \nabla_vf^{P,\epsilon}\right\}\right\},\right.\nonumber\\
 &&\left.\quad\quad\quad\quad\quad\quad\quad\quad\quad\quad\quad\quad\quad\quad\quad\quad\quad \quad\quad\quad\quad\quad\quad\quad\quad\quad\quad w^2_{l_m-|\beta|,\kappa}\partial^\alpha_\beta {\bf\{I-P\}}f^{m,\epsilon}\right)\\
    &\lesssim&\mathcal{E}_{f^{m,\epsilon},n}(t) \left\{\mathcal{D}_{f^{j_1,\epsilon},n+1,n+2-\frac1\gamma,-\gamma}(t) +\mathcal{D}_{f^{P,\epsilon},n+1,n+2-\frac1\gamma,-\gamma}(t)\right\}\nonumber\\
    &&
    +\eta\left\|w_{l_m-|\beta|,\kappa}\partial^\alpha_\beta{\bf \{I-P\}}f^{m,\epsilon}(t)\right\|_\nu^2\nonumber
    \end{eqnarray}
    and one also has for $j_1+j_2\geq m$ with $0<j_1,j_2<m$ that
\begin{eqnarray}\label{micro-weight-E-R-12}
   &&\left|\left(\partial^\alpha_\beta\left\{{\bf\{I-P\}}\left[\frac12 q_0 \epsilon^{j_1+j_2-m} E^{j_1,\epsilon}\cdot vf^{j_2,\epsilon}-q_0\epsilon ^{j_1+j_2-m}E^{j_1,\epsilon}\nabla_vf^{j_2,\epsilon}\right.\right.\right.\right.\nonumber\\
    &&\left.\left.\left.\left.\quad\quad\quad\quad\quad\quad-q_0\epsilon \left(v\times B^{m,\epsilon}\right)\cdot\nabla_v\left(f^{P,\epsilon}+\sum_{i=1}^{m-1}\epsilon^if^{i,\epsilon}\right)\right]\right\},w^2_{l_m-|\beta|,\kappa}\partial^\alpha_\beta{\bf \{I-P\}}f^{m,\epsilon}\right)\right|\\
   &\lesssim&
  \mathcal{E}_{f^{m,\epsilon},n}(t)\left\{\mathcal{D}_{f^{P,\epsilon},n+1,-\frac{\kappa l_m}{\gamma}+n+2+\frac{\kappa n-1}\gamma,-\gamma}(t)+\sum_{i=1}^{m-1}\mathcal{D}_{f^{i,\epsilon},n+1,-\frac{\kappa l_m}{\gamma}+n+2+\frac{\kappa n-1}\gamma,-\gamma}(t)\right\}\nonumber\\
&&+\mathcal{E}_{f^{j_1,\epsilon},n}(t)\mathcal{D}_{f^{j_2,\epsilon},n+1,-\frac{\kappa l_m}{\gamma}+n+2+\frac{\kappa n-1}\gamma,-\gamma}(t)+\eta\left\|w_{l_m-|\beta|,\kappa}\partial^\alpha_\beta{\bf \{I-P\}}f^{m,\epsilon}(t)\right\|_\nu^2.\nonumber
\end{eqnarray}
  \end{itemize}
\end{lemma}

\subsubsection{Some nonlinear estimates related to the electromagnetic field with respect to the weight $w_{\ell-|\beta|, -\gamma}(t,v)$} In this section, we try to deduce certain weighted nonlinear estimates related to the electromagnetic field  $\left[E^{m,\epsilon}(t,x), B^{m,\epsilon}(t,x)\right]$ with respect to the weight $w_{\ell-|\beta|, -\gamma}(t,v)$. For result in this direction, we have
 \begin{lemma}
 Recall the definitions in \eqref{def-D-n-g},
take
\[n\geq 3, \ \widetilde{\ell}_m\geq 1-\frac{3\gamma} 2,\ {\theta}=\frac{2-4\gamma}{2\widetilde{\ell}_m-\gamma},\]
we have the following estimates:
\begin{itemize}
  \item [i)]For $1\leq|\alpha|\leq n$, one has
\begin{eqnarray}\label{1-typical-low-non-R}
   &&\left|\left(\partial^\alpha\left\{
   \epsilon^{m}q_0\left(E^{m,\epsilon}+\epsilon v\times B^{m,\epsilon}\right)\cdot\nabla_v f^{m,\epsilon}\right\},\partial^\alpha f^{m,\epsilon}\right)\right|\nonumber\\
   &\lesssim&\mathcal{E}_{f^{m,\epsilon},n-1,n-\frac12-\frac1\gamma,-\gamma}(t)
 \mathcal{D}_{f^{m,\epsilon},n}(t)+\mathcal{E}_{f^{m,\epsilon},n}(t) \mathcal{D}_{f^{m,\epsilon},3,4-\frac1\gamma,-\gamma}(t)\\
  &&+\left\|\epsilon^m\nabla_x\left[E^{m,\epsilon}(t),B^{m,\epsilon}(t)\right]\right\|^2_{L^\infty_x} \left\|\nabla_v\nabla^{n-1}_x{\bf\{I-P\}}f^{m,\epsilon}(t)\langle v\rangle^{1-\frac\gamma2}\right\|^2+\eta\mathcal{D}_{f^{m,\epsilon},n}(t)\nonumber
\end{eqnarray}
and
\begin{eqnarray}\label{1-typical-low-non-1-R}
   &&\left|\left(\partial^\alpha\left\{\frac12q_0\epsilon^mE^{m,\epsilon}\cdot vf^{m,\epsilon}\right\},\partial^\alpha f^{m,\epsilon}\right)\right|\nonumber\\
   &\lesssim&\mathcal{E}_{f^{m,\epsilon},n-1,n-\frac12-\frac1\gamma,-\gamma}(t)
 \mathcal{D}_{f^{m,\epsilon},n}(t)+\mathcal{E}_{f^{m,\epsilon},n}(t) \mathcal{D}_{f^{m,\epsilon},3,4-\frac1\gamma,-\gamma}(t)\\
  &&+\left\|\epsilon^m E^{m,\epsilon}(t)\right\|^2_{L^\infty_x} \left\|\nabla^{n}_x{\bf\{I-P\}}f^{m,\epsilon}(t)\langle v\rangle^{1-\frac\gamma2}\right\|^2+\eta\mathcal{D}_{f^{m,\epsilon},n}(t);\nonumber
\end{eqnarray}
\item [ii)]For $1\leq|\alpha|\leq n$, one has
\begin{eqnarray}\label{1-w-typical-low-non-R}
   &&\left|\left(\partial^\alpha\left\{
   \epsilon^{m}q_0\left(E^{m,\epsilon}+\epsilon v\times B^{m,\epsilon}\right)\cdot\nabla_v f^{m,\epsilon}\right\},w^2_{l_m,-\gamma}\partial^\alpha f^{m,\epsilon}\right)\right|\nonumber\\
&\lesssim&\mathcal{E}_{f^{m,\epsilon},n}(t)\mathcal{D}_{f^{m,\epsilon},n}(t) +\epsilon^{m}\left\| E^{m,\epsilon}(t)\right\|_{L^\infty_x} \left\|w_{l_m,-\gamma}\partial^\alpha f^{m,\epsilon}(t)\langle v\rangle^{\frac12}\right\|^2\\
&&+\left\{\epsilon^{2m}\left\|\nabla_x
\left[E^{m,\epsilon}(t),\epsilon B^{m,\epsilon}(t)\right]\right\|^2_{H^{N_m^0-1}_x}
\right\}^{\frac{1}{\theta}}\sum_{1\leq j\leq n}\mathcal{D}_{f^{m,\epsilon},j,\widetilde{\ell}_m-\gamma l_m+1+\frac\gamma2,1}^{(j,1)}(t)\nonumber\\
  &&+\chi_{|\alpha|\geq N_m^0+1}\mathcal{E}_{f^{m,\epsilon},N_m}(t) \mathcal{E}^1_{f^{m,\epsilon},N^0_m,l_m+\frac32-\frac1\gamma,1}(t) +\eta\mathcal{D}_{f^{m,\epsilon},n,l_m,-\gamma}(t)\nonumber
\end{eqnarray}
and
\begin{eqnarray}\label{1-w-typical-low-non-1-R}
   &&\left|\left(\partial^\alpha\left\{\frac12q_0\epsilon^mE^{m,\epsilon}\cdot vf^{m,\epsilon}\right\},w^2_{l_m,-\gamma}\partial^\alpha f^{m,\epsilon}\right)\right|\nonumber\\
 &\lesssim&\mathcal{E}_{f^{m,\epsilon},n}(t)\mathcal{D}_{f^{m,\epsilon},n}(t) +\epsilon^{m}\left\| E^{m,\epsilon}(t)\right\|_{L^\infty_x} \left\|w_{l_m,-\gamma}\partial^\alpha f^{m,\epsilon}(t)\langle v\rangle^{\frac12}\right\|^2\\
  &&+\left\{\epsilon^{2m}\left\|\nabla_x
E^{m,\epsilon}(t)\right\|^2_{H^{N_m^0-1}_x}
\right\}^{\frac{1}{\theta}}\sum_{0\leq j\leq n}\mathcal{D}_{f^{m,\epsilon},j,\widetilde{\ell}_m-\gamma l_m+1+\frac\gamma2,1}^{(j,0)}(t)\nonumber\\
  &&+\chi_{|\alpha|\geq N_m^0+1}\mathcal{E}_{f^{m,\epsilon},N_m}(t) \mathcal{E}^1_{f^{m,\epsilon},N^0_m,l_m+\frac32-\frac1\gamma,1}(t) +\eta\mathcal{D}_{f^{m,\epsilon},n,l_m,-\gamma}(t);\nonumber
\end{eqnarray}
\item [iii)] For $|\alpha|+|\beta|= n$ with $|\beta|\geq 1$ or $|\alpha|=|\beta|=0$, one has
\begin{eqnarray}\label{mic-non-I-P-1}
 &&\left|\left(\partial^\alpha_\beta\left\{{\bf\{I-P\}}\left(
   \epsilon^{m}q_0\left(E^{m,\epsilon}+\epsilon v\times B^{m,\epsilon}\right)\cdot\nabla_v f^{m,\epsilon}\right)\right\},w^2_{l_m-|\beta|,-\gamma}\partial^\alpha_\beta{\bf\{I-P\}} f^{m,\epsilon}\right)\right|\nonumber\\
 &\lesssim&\mathcal{E}_{f^{m,\epsilon},n}(t)\mathcal{D}_{f^{m,\epsilon},n}(t) +\epsilon ^m\left\|E^{m,\epsilon}(t)\right\|_{L^\infty_x} \left\|w_{l_m-|\beta|,-\gamma}\partial^\alpha_\beta{\bf\{I-P\}} f^{m,\epsilon}(t)\langle v\rangle^{\frac12}\right\|^2\\
 &&+\chi_{|\beta|\geq 1}\left\{\epsilon^{2m}\left\|\nabla_x
\left[E^{m,\epsilon}(t),\epsilon B^{m,\epsilon}(t)\right]\right\|^2_{H^{N_m^0-1}_x}
\right\}^{\frac{1}{\theta}}\sum_{|\beta|\leq k\leq \min\{|\beta|+1,|\alpha|+|\beta|\},\atop k\leq j\leq |\alpha|}\mathcal{D}^{(j,k)}_{f^{m,\epsilon},j,\widetilde{\ell}_m-\gamma l_m+1+\frac\gamma2,1}(t)\nonumber\\
  &&+\chi_{|\alpha|+|\beta|\geq N_m^0+1}\mathcal{E}_{f^{m,\epsilon},N_m}(t) \mathcal{E}^1_{f^{m,\epsilon},N^0_m,l_m+\frac32-\frac1\gamma,1}(t) +\eta\mathcal{D}_{f^{m,\epsilon},n,l_m,-\gamma}(t)\nonumber
\end{eqnarray}
and
\begin{eqnarray}\label{mic-non-I-P-2}
  &&\left|\left(\partial^\alpha_\beta\left\{{\bf\{I-P\}}\left\{\frac12q_0\epsilon^mE^{m,\epsilon}\cdot vf^{m,\epsilon}\right\}\right\},w^2_{l_m-|\beta|,-\gamma}\partial^\alpha_\beta{\bf\{I-P\}} f^{m,\epsilon}\right)\right|\nonumber\\
  &\lesssim&\mathcal{E}_{f^{m,\epsilon},n}(t)\mathcal{D}_{f^{m,\epsilon},n}(t)+\epsilon ^m\left\| E^{m,\epsilon}(t)\right\|_{L^\infty_x} \left\|w_{l_m-|\beta|,-\gamma}\partial^\alpha_\beta{\bf\{I-P\}} f^{m,\epsilon}(t)\langle v\rangle^{\frac12}\right\|^2\\
  &&+\chi_{|\beta|\geq 1}\left\{\epsilon^{2m}\left\|\nabla_x
E^{m,\epsilon}(t)\right\|^2_{H^{N_m^0-1}_x}
\right\}^{\frac{1}{\theta}}\sum_{|\beta|\leq k\leq \min\{|\beta|+1,|\alpha|+|\beta|\},\atop k\leq j\leq |\alpha|}\mathcal{D}^{(j,k)}_{f^{m,\epsilon},j,\widetilde{\ell}_m-\gamma l_m+1+\frac\gamma2,1}(t)\nonumber\\
  &&+\chi_{|\alpha|+|\beta|\geq N_m^0+1}\mathcal{E}_{f^{m,\epsilon},N_m}(t) \mathcal{E}^1_{f^{m,\epsilon},N^0_m,l_m+\frac32-\frac1\gamma,1}(t) +\eta\mathcal{D}_{f^{m,\epsilon},n,l_m,-\gamma}(t).\nonumber
\end{eqnarray}
\end{itemize}
 \end{lemma}

\begin{proof}
We only prove \eqref{1-w-typical-low-non-R} since the rest can be done similarly. To this end, since
 \begin{eqnarray}
&&\left(\partial^\alpha\left\{
  \epsilon^{m}q_0\left(E^{m,\epsilon}+\epsilon v\times B^{m,\epsilon}\right)\cdot\nabla_v f^{m,\epsilon}\right\},w^2_{l_m,-\gamma}\partial^\alpha f^{m,\epsilon}\right)\nonumber\\
&=&\underbrace{\left(
   q_0\epsilon^{m}\left(E^{m,\epsilon}+\epsilon v\times B^{m,\epsilon}\right)\cdot\nabla_v\partial^\alpha f^{m,\epsilon},w^2_{l_m,-\gamma}\partial^\alpha f^{m,\epsilon}\right)}_{{J}_{\alpha,1}}\nonumber\\
   &&+\underbrace{\sum_{0<\alpha_1\leq\alpha}\left( q_0\epsilon^{m}\partial^{\alpha_1}\left(E^{m,\epsilon}+\epsilon v\times B^{m,\epsilon}\right)\cdot\nabla_v\partial^{\alpha-\alpha_1 }{\bf P}f^{m,\epsilon},w^2_{l_m,-\gamma}\partial^\alpha f^{m,\epsilon}\right)}_{{J}_{\alpha,2}}\\
&&+\underbrace{\sum_{0<\alpha_1\leq\alpha}\left( q_0\epsilon^{m}\partial^{\alpha_1}\left(E^{m,\epsilon}+\epsilon v\times B^{m,\epsilon}\right)\cdot\nabla_v\partial^{\alpha-\alpha_1 }{\bf\{I-P\} }f^{m,\epsilon},w^2_{l_m,-\gamma}\partial^\alpha f^{m,\epsilon}\right)}_{{J}_{\alpha,3}},\nonumber
\end{eqnarray}
we can get by using integration by parts that
\begin{eqnarray}
{J}_{\alpha,1}&\lesssim&\epsilon^{m}\left\{\left\|E^{m,\epsilon}\right\|_{L^\infty_x}+\|\epsilon B^{m,\epsilon}\|_{L^\infty_x}\right\}\left\|w_{l_m,-\gamma}\partial^\alpha f^{m,\epsilon}\langle v\rangle \right\|^2.\nonumber
\end{eqnarray}

For ${J}_{\alpha,2}$, one has by Sobolev's inequalities that
\[
{J}_{\alpha,2}\lesssim\epsilon^{m}\left\{\left\|E^{m,\epsilon}\right\|_{H^{n}_x}+\|\epsilon B^{m,\epsilon}\|_{H^{n}_x}\right\}\|\nabla_x f^{m,\epsilon}\|^2_{H^{n-2}_xL^2_\nu}+\eta\left\|\partial^\alpha f^{m,\epsilon}\langle v\rangle^{\frac\gamma2}\right\|^2
\]
and ${J}_{\alpha,3}$ can be estimated as follows
 \begin{eqnarray*}
{J}_{\alpha,3}
&\lesssim&\underbrace{\epsilon^{2m}\sum_{1\leq|\alpha_1|\leq N^0_{R}-2}\left\{\left\|\partial^{\alpha_1}
E^{m,\epsilon}\right\|^2_{L^\infty_x}
+\left\|\epsilon\partial^{\alpha_1}B^{m,\epsilon}\right\|^2_{L^\infty_x}
\right\}
\left\|w_{l_m-1,-\gamma}\partial^{\alpha-\alpha_1}_{e_j}{\bf \{I-P\}}f^{m,\epsilon}\langle v\rangle^{1-\frac{3\gamma}2}\right\|^2}_{{J}^{(1)}_{\alpha,3}}\\{}
&&+\underbrace{\epsilon^{2m}\sum_{|\alpha_1|= N^0_{R}-1}\left\{\left\|\partial^{\alpha_1}E^{m,\epsilon}\right\|
^2_{L^6_x}
+\left\|\epsilon\partial^{\alpha_1}B^{m,\epsilon}\right\|^2_{L^6_x}
\right\}
\left\|w_{l_m-1,-\gamma}
\partial^{\alpha-\alpha_1}_{e_j}{\bf \{I-P\}}f^{m,\epsilon}\langle v\rangle^{1-\frac{3\gamma}2}\right\|^2_{L^2_vL^3_x}}_{{J}^{(2)}_{\alpha,3}}\\{}
&&+\underbrace{\epsilon^{2m}\sum_{|\alpha_1|=N^0_{R}}\left\{\left\|\partial^{\alpha_1}E^{m,\epsilon}\right\|^2
+\left\|\epsilon\partial^{\alpha_1}B^{m,\epsilon}\right\|^2
\right\}
\left\|w_{l_m-1,-\gamma}\partial^{\alpha-\alpha_1}_{e_j}{\bf \{I-P\}}f^{m,\epsilon}\langle v\rangle^{1-\frac{3\gamma}2}\right\|^2_{L^\infty_xL^2_v}}_{{J}^{(3)}_{\alpha,3}}\\{}
&&+\chi_{|\alpha|\geq N^0_{R}+1}\underbrace{\epsilon^{2m}\sum_{N^0_{R}+1\leq|\alpha_1|\leq N_m-2}\left\{\left\|\partial^{\alpha_1}\left[E^{m,\epsilon},\epsilon B^{m,\epsilon}\right]\right\|^2_{L^\infty_x}\right\}
\left\|w_{l_m-1,-\gamma}\partial^{\alpha-\alpha_1}_{e_j}{\bf \{I-P\}}f^{m,\epsilon}\langle v\rangle^{1-\frac{3\gamma}2}\right\|^2}_{{J}^{(4)}_{\alpha,3}}\\{}
&&+\chi_{|\alpha|\geq N_{R}-1}\underbrace{\epsilon^{2m}\sum_{|\alpha_1|= N_m-1 }\left\{\left\|\partial^{\alpha_1}E^{m,\epsilon}\right\|_{L^6_x}^2
+\left\|\epsilon\partial^{\alpha_1}B^{m,\epsilon}\right\|_{L^6_x}^2
\right\}
\left\|w_{l_m-1,-\gamma}\partial^{\alpha-\alpha_1}_{e_j}{\bf \{I-P\}}f^{m,\epsilon}\langle v\rangle^{1-\frac{3\gamma}2}\right\|^2_{L^3_xL^2_v}}_{{J}^{(5)}_{\alpha,3}}\\{}
&&+\chi_{|\alpha|= N_{R}}\underbrace{\epsilon^{2m}\sum_{|\alpha_1|= N_m }\left\{\left\|\partial^{\alpha_1}E^{m,\epsilon}\right\|^2
+\left\|\epsilon\partial^{\alpha_1}B^{m,\epsilon}\right\|^2
\right\}
\left\|w_{l_m-1,-\gamma}\partial_{e_j}{\bf \{I-P\}}f^{m,\epsilon}\langle v\rangle^{1-\frac{3\gamma}2}\right\|^2_{L^\infty_xL^2_v}}_{{J}^{(6)}_{\alpha,3}}\\{}
&&+\eta\left\|w_{l_m,-\gamma}\partial^\alpha f^{m,\epsilon}\right\|_\nu^2.
\end{eqnarray*}

To estimate ${J}^{(i)}_{\alpha,3}$ $(i=1,2,3,4,5,6)$, one has for ${J}^{(1)}_{\alpha,3}$ that
\begin{eqnarray*}
{J}^{(1)}_{\alpha,3}&\lesssim&\sum_{1\leq|\alpha_1|\leq N^0_{R}-2}\epsilon^{2m}\left\{\left\|\partial^{\alpha_1}
E^{m,\epsilon}\right\|^2_{L^\infty_x}
+\left\|\epsilon\partial^{\alpha_1}B^{m,\epsilon}\right\|^2_{L^\infty_x}
\right\}\left\|w_{l_m-1,-\gamma}\partial^{\alpha-\alpha_1}_{e_j}{\bf \{I-P\}}f^{m,\epsilon}\langle v\rangle^{\widetilde{\ell}_m}\right\|^{2{\theta}}\\
&&\times
\left\|w_{l_m-1,-\gamma}\partial^{\alpha-\alpha_1}_{e_j}{\bf \{I-P\}}f^{m,\epsilon}\langle v\rangle^{\frac{\gamma}2}\right\|^{2(1-{\theta})}
\\
&\lesssim&\sum_{1\leq|\alpha_1|\leq N^0_{R}-2}\epsilon^{2m}\left\{\left\|\partial^{\alpha_1}
E^{m,\epsilon}\right\|^2_{L^\infty_x}
+\left\|\epsilon\partial^{\alpha_1}B^{m,\epsilon}\right\|^2_{L^\infty_x}
\right\}\left\|w_{l_m-1,-\gamma}\partial^{\alpha-\alpha_1}_{e_j}{\bf \{I-P\}}f^{m,\epsilon}\langle v\rangle^{\widetilde{\ell}_m}\right\|^{2{\theta}}\\
&&\times
\left\|w_{l_m-1,-\gamma}\partial^{\alpha-\alpha_1}_{e_j}{\bf \{I-P\}}f^{m,\epsilon}\langle v\rangle^{\frac{\gamma}2}\right\|^{2(1-{\theta})}\nonumber\\
&\lesssim&\sum_{1\leq|\alpha_1|\leq N^0_{R}-2}\left\{\epsilon^{2m}\left\|\partial^{\alpha_1}
E^{m,\epsilon}\right\|^2_{L^\infty_x}
+\epsilon^{2m}\left\|\epsilon\partial^{\alpha_1}B^{m,\epsilon}\right\|^2_{L^\infty_x}
\right\}^{\frac{1}{\theta}}\left\|w_{l_m-1,-\gamma}\partial^{\alpha-\alpha_1}_{e_j}{\bf \{I-P\}}f^{m,\epsilon}\langle v\rangle^{\widetilde{\ell}_m}\right\|^{2}\\
&&+\eta
\left\|w_{l_m-1,-\gamma}\partial^{\alpha-\alpha_1}_{e_j}{\bf \{I-P\}}f^{m,\epsilon}\langle v\rangle^{\frac{\gamma}2}\right\|^{2}\nonumber\\
&\lesssim&\sum_{1\leq|\alpha_1|\leq N^0_{R}-2}\left\{\epsilon^{2m}\left\|\partial^{\alpha_1}
E^{m,\epsilon}\right\|^2_{L^\infty_x}
+\epsilon^{2m}\left\|\epsilon\partial^{\alpha_1}B^{m,\epsilon}\right\|^2_{L^\infty_x}
\right\}^{\frac{1}{\theta}}\nonumber\\
&&\times\sum_{1\leq j\leq n}\mathcal{D}_{f^{m,\epsilon},j,\widetilde{\ell}_m-\gamma l_m+1+\frac\gamma2,1}^{(j,1)}(t)+\eta\mathcal{D}_{f^{m,\epsilon},n,l_m,-\gamma}(t),
\end{eqnarray*}
where ${\theta}$ satisfies that
$1-\frac32\gamma=\frac\gamma2(1-{\theta})+\widetilde{\ell}_m{\theta}$ which yields that ${\theta}=\frac{2-4\gamma}{2\widetilde{\ell}_m-\gamma}$.
${J}^{(2)}_{\alpha,3}$ and ${J}^{(3)}_{\alpha,3}$ can be controlled in a similar way, while for the last three terms, one can deduce that
\begin{eqnarray*}
  &&{J}^{(4)}_{\alpha,3}+{J}^{(5)}_{\alpha,3}
  +{J}^{(6)}_{\alpha,3}\lesssim\mathcal{E}_{f^{m,\epsilon},N_m}(t) \mathcal{E}^1_{f^{m,\epsilon},N^0_m,l_m+\frac32-\frac1\gamma,1}(t).
\end{eqnarray*}
Thus we complete the proof of this lemma.
\end{proof}

For the estimates on the interactions of $f^{m,\epsilon}(t,x,v)$ with $\left[E^{P,\epsilon}(t,x), E^{1,\epsilon}(t,x),\cdots, E^{m-1,\epsilon}(t,x)\right]$ with respect to the weight $w_{\ell^*_m-|\beta|,-\gamma}(t,v)$, we can get similarly that
\begin{lemma}
Assume
\[\widetilde{\ell}_m\geq 1-\frac{3\gamma} 2,\ {\theta}=\frac{2-4\gamma}{2\widetilde{\ell}_m-\gamma},\]
we have the following estimates:
\begin{itemize}
\item [i)] For $0\leq|\alpha|\leq n$, one has
      \begin{eqnarray}\label{0-w-typical-R}
       &&\left|\left(\partial^\alpha\left\{\frac12q_0 E^{P,\epsilon}\cdot v f^{m,\epsilon}-q_0 E^{P,\epsilon}\cdot\nabla_{ v  }f^{m,\epsilon}\right\},\partial^\alpha f^{m,\epsilon}\right)\right|\nonumber\\
   &\lesssim&\mathcal{E}_{f^{P,\epsilon},n}(t)\mathcal{D}_{f^{m,\epsilon},n}(t)
   +\mathcal{E}_{f^{m,\epsilon},n-1,n-\frac12-\frac1\gamma,-\gamma}(t)
  \mathcal{D}_{f^{P,\epsilon},n}(t)\nonumber\\
  &&+\left\|\nabla_xE^{P,\epsilon}(t)\right\|^2_{L^\infty_x} \left\|\nabla_v\nabla^{n-1}_x{\bf\{I-P\}}f^{m,\epsilon}(t)\langle v\rangle^{1-\frac\gamma2}\right\|^2\\
   &&+\left\|E^{P,\epsilon}(t)\right\|^2_{L^\infty_x} \left\|\nabla^{n}_x{\bf\{I-P\}}f^{m,\epsilon}(t)\langle v\rangle^{1-\frac\gamma2}\right\|^2+\eta\mathcal{D}_{f^{m,\epsilon},n}(t)\nonumber
\end{eqnarray}
  and  for $0<j_1<m$, one can get that
  \begin{eqnarray}\label{1-w-typical-R}
       &&\left|\left(\partial^\alpha\left\{\frac12q_0 E^{j_1,\epsilon}\cdot v f^{m,\epsilon}-q_0 E^{j_1,\epsilon}\cdot\nabla_{ v  }f^{m,\epsilon}\right\},\partial^\alpha f^{m,\epsilon}\right)\right|\nonumber\\
   &\lesssim&\mathcal{E}_{f^{j_1,\epsilon},n}(t)\mathcal{D}_{f^{m,\epsilon},n}(t)
   +\mathcal{E}_{f^{m,\epsilon},n-1,n-\frac12-\frac1\gamma,-\gamma}(t)
  \mathcal{D}_{f^{j_1,\epsilon},n}(t)\nonumber\\
  &&+\|\nabla_xE^{j_1,\epsilon}(t)\|^2_{L^\infty_x} \left\|\nabla_v\nabla^{n-1}_x{\bf\{I-P\}}f^{m,\epsilon}(t)\langle v\rangle^{1-\frac\gamma2}\right\|^2\\
   &&+\|E^{j_1,\epsilon}(t)\|^2_{L^\infty_x}\|\nabla^{n}_x{\bf\{I-P\}}f^{m,\epsilon}(t)\langle v\rangle^{1-\frac\gamma2}\|^2+\eta\mathcal{D}_{f^{m,\epsilon},n}(t).\nonumber
\end{eqnarray}
\item [ii)]For $1\leq|\alpha|\leq n$, one has
\begin{eqnarray}\label{1-w-typical-low-R}
   &&\left|\left(\partial^\alpha\left\{\frac12q_0 E^{P,\epsilon}\cdot v f^{m,\epsilon}-q_0 E^{P,\epsilon}\cdot\nabla_{ v  }f^{m,\epsilon}\right\},w^2_{l_m,-\gamma}\partial^\alpha f^{m,\epsilon}\right)\right|\\
&\lesssim&\mathcal{E}_{f^{P,\epsilon},n}(t)\mathcal{D}_{f^{m,\epsilon},n}(t)+\left\| E^{P,\epsilon}(t)\right\|_{L^\infty_x}\left\|w_{l_m,-\gamma}\partial^\alpha f^{m,\epsilon}(t)\langle v\rangle^{\frac12}\right\|^2\nonumber\\
  &&+\left\|\nabla_xE^{P,\epsilon}(t)\right\|_{H^{N_m}_x}^{\frac2{{\theta}}}\sum_{0\leq k\leq 1,\atop k\leq j\leq |\alpha|}\mathcal{D}^{(j,k)}_{f^{m,\epsilon},j,-\gamma l_m+\widetilde{l}_m,1}(t)+\eta\mathcal{D}_{f^{m,\epsilon},n,l_m,-\gamma}(t)\nonumber
  \end{eqnarray}
and we also have
\begin{eqnarray}\label{1-w-typical-low-1-R}
   &&\left|\left(\partial^\alpha\left\{\frac12q_0 E^{j_1,\epsilon}\cdot v f^{m,\epsilon}-q_0 E^{j_1,\epsilon}\cdot\nabla_{ v  }f^{m,\epsilon}\right\},w^2_{l_m,-\gamma}\partial^\alpha f^{m,\epsilon}\right)\right|\\
&\lesssim&\mathcal{E}_{f^{j_1,\epsilon},n}(t)\mathcal{D}_{f^{m,\epsilon},n}(t)+\left\| E^{j_1,\epsilon}(t)\right\|_{L^\infty_x}\left\|w_{l_m,-\gamma}\partial^\alpha f^{m,\epsilon}(t)\langle v\rangle^{\frac12}\right\|^2\nonumber\\
  &&+\left\|\nabla_xE^{j_1,\epsilon}(t)\right\|_{H^{N_m}_x}^{\frac2{{\theta}}}\sum_{0\leq k\leq 1,\atop k\leq j\leq |\alpha|}\mathcal{D}^{(j,k)}_{f^{m,\epsilon},j,-\gamma l_m+\widetilde{l}_m,1}(t)+\eta\mathcal{D}_{f^{m,\epsilon},n,l_m,-\gamma}(t).\nonumber
\end{eqnarray}
\item [iii)] For $|\alpha|+|\beta|\leq n$ with $|\beta|\geq 1$, one has
\begin{eqnarray}\label{2-w-typical-low-R}
 &&\left|\left(\partial^\alpha_\beta\left\{{\bf\{I-P\}}\left\{\frac12q_0 E^{P,\epsilon}\cdot v f^{m,\epsilon}-q_0 E^{P,\epsilon}\cdot\nabla_{ v  }f^{m,\epsilon}\right\}\right\},w^2_{l_m-|\beta|,-\gamma}\partial^\alpha_\beta{\bf\{I-P\}} f^{m,\epsilon}\right)\right|\nonumber\\
   &\lesssim&\mathcal{E}_{f^{P,\epsilon},n}(t)\mathcal{D}_{f^{m,\epsilon},n}(t)+\left\|  E^{P,\epsilon}(t)\right\|_{L^\infty_x}\left\|w_{l_m-|\beta|,-\gamma} \partial^\alpha_\beta{\bf\{I-P\}} f^{m,\epsilon}(t)\langle v\rangle^{\frac12}\right\|^2\\
  &&+\left\|\nabla_xE^{P,\epsilon}(t)\right\|_{H^{N_m}_x}^{\frac2{{\theta}}}\sum_{|\beta|-1\leq k\leq \min\{|\beta|+1,|\alpha|+|\beta|\},\atop k\leq j\leq |\alpha|+|\beta|}\mathcal{D}^{(j,k)}_{f^{m,\epsilon},j,-\gamma l_m+\widetilde{l}_m,1}(t)
+\eta\mathcal{D}_{f^{m,\epsilon},N_m,l_m,-\gamma}(t)\nonumber
\end{eqnarray}
and we also have
\begin{eqnarray}\label{2-w-typical-low-1-R}
  &&\left|\left(\partial^\alpha_\beta\left\{{\bf\{I-P\}}\left\{ \frac12q_0 E^{j_1,\epsilon}\cdot v f^{m,\epsilon}-q_0 E^{j_1,\epsilon}\cdot\nabla_{ v  }f^{m,\epsilon}\right\}\right\},w^2_{l_m-|\beta|,-\gamma}\partial^\alpha_\beta{\bf\{I-P\}} f^{m,\epsilon}\right)\right|\nonumber\\
   &\lesssim&\mathcal{E}_{f^{j_1,\epsilon},n}(t)\mathcal{D}_{f^{m,\epsilon},n}(t) +\left\|  E^{j_1,\epsilon}(t)\right\|_{L^\infty_x} \left\|w_{l_m-|\beta|,-\gamma}\partial^\alpha_\beta{\bf\{I-P\}} f^{m,\epsilon}(t)\langle v\rangle^{\frac12}\right\|^2\\
   &&+\left\|\nabla_xE^{j_1,\epsilon}(t)\right\|_{H^{N_m}_x}^{\frac2{{\theta}}} \sum_{|\beta|-1\leq k\leq \min\{|\beta|+1,|\alpha|+|\beta|\},\atop k\leq j\leq |\alpha|+|\beta|}\mathcal{D}^{(j,k)}_{f^{m,\epsilon},j,-\gamma l_m+\widetilde{l}_m,1}(t)
+\eta\mathcal{D}_{f^{m,\epsilon},N_m,l_m,-\gamma}(t).\nonumber
\end{eqnarray}
\end{itemize}
 \end{lemma}
\subsubsection{Some nonlinear estimates related to the electromagnetic field with respect to the weight $w_{\ell-|\beta|,1}(t,v)$} As we mentioned above, two sets of energy estimates with respect to weights $w_{\ell-|\beta|,-\gamma}(t,v)$ and $w_{\ell-|\beta|,1}(t,v)$ respectively are introduced to close the desired energy estimates, the main purpose of this section is to deduce certain nonlinear estimates related to the electromagnetic field $\left[E^{m,\epsilon}(t,x), B^{m,\epsilon}(tx)\right]$ with respect to the weight $w_{\ell-|\beta|,1}(t,v)$.
 \begin{lemma}\label{lemma-high-E-i}
For $2\leq n\leq N_m+1$, we have the following estimates:
\begin{itemize}
\item [i)]For $1\leq|\alpha|\leq n$, one has
\begin{eqnarray}\label{1-w-typical-high-non-R}
  &&\left|\left(\partial^\alpha\left\{
   q_0\epsilon^{m}(E^{m,\epsilon}+\epsilon v\times B^{m,\epsilon})\cdot\nabla_v f^{m,\epsilon}\right\},w^2_{\ell^*_m,1}\partial^\alpha f^{m,\epsilon}\right)\right|\nonumber\\
   &\lesssim&\mathcal{E}_{f^{P,\epsilon},n}(t)\mathcal{D}_{f^{m,\epsilon},n}(t) +\left\|\epsilon^m E^{m,\epsilon}(t)\right\|_{L^\infty_x}\left\|w_{\ell^*_m,1}\partial^\alpha f^{m,\epsilon}(t)\langle v\rangle^{\frac12}\right\|^2\nonumber\\
 &&+(1+t)^{2+2\vartheta}\sum_{1\leq|\alpha_1|\leq N^0_{R}-2}\left\|\partial^{\alpha_1}\left[\epsilon^m
 E^{m,\epsilon}(t),\epsilon^{m+1} B^{m,\epsilon}(t)\right]\right\|^2_{L^\infty_x}\sum_{1\leq j\leq |\alpha|}\mathcal{D}_{f^{m,\epsilon},j,\ell_m^*,1}^{(j,1)}(t)\\
   &&+(1+t)^{2+2\vartheta}\left\|\nabla^{N^0_{R}-1}\left[\epsilon^m E^{m,\epsilon}(t), \epsilon^{m+1}B^{m,\epsilon}(t)\right]\right\|_{H^1_x}^2\sum_{1\leq j\leq |\alpha|}\mathcal{D}_{f^{m,\epsilon},j,\ell_m^*,1}^{(j,1)}(t)\nonumber\\
  &&+\chi_{|\alpha|\geq N_m^0+1}\mathcal{E}_{f^{m,\epsilon},N_m}(t)\mathcal{E}^1_{f^{m,\epsilon},N^0_m,\ell^*_m+1-\frac\gamma2,1}(t)+\eta\left\|\partial^\alpha f^{m,\epsilon}(t)\right\|_\nu^2+\eta(1+t)^{-1-\vartheta}
\left\|w_{\ell^*_m,1}\partial^\alpha f^{m,\epsilon}(t)\langle v\rangle^{\frac{1}2}\right\|^2\nonumber
\end{eqnarray}
and
\begin{eqnarray}\label{1-w-typical-high-non-1-R}
  &&\left|\left(\partial^\alpha\left\{q_0 \epsilon^{m} E^{m,\epsilon}\cdot v f^{m,\epsilon}\right\},w^2_{\ell^*_m,1}\partial^\alpha f^{m,\epsilon}\right)\right|\nonumber\\
   &\lesssim&\mathcal{E}_{f^{P,\epsilon},n}(t)\mathcal{D}_{f^{m,\epsilon},n}(t)+\left\|\epsilon^m E^{m,\epsilon}(t)\right\|_{L^\infty_x}\left\|w_{\ell^*_m,1}\partial^\alpha f^{m,\epsilon}(t)\langle v\rangle^{\frac12}\right\|^2\nonumber\\
 &&+(1+t)^{2+2\vartheta}\sum_{1\leq|\alpha_1|\leq N^0_{R}-2}\left\|\partial^{\alpha_1}\left[\epsilon^m
 E^{m,\epsilon}(t),\epsilon^{m+1} B^{m,\epsilon}(t)\right]\right\|^2_{L^\infty_x}\sum_{0\leq j\leq |\alpha|}\mathcal{D}_{f^{m,\epsilon},j,\ell_m^*,1}^{(j,1)}(t)\\
   &&+(1+t)^{2+2\vartheta}\left\|\nabla^{N^0_{R}-1}\left[\epsilon^m E^{m,\epsilon}(t),\epsilon^{m+1}B^{m,\epsilon}(t)\right]\right\|_{H^1_x}^2\sum_{0\leq j\leq |\alpha|}\mathcal{D}_{f^{m,\epsilon},0,\ell_m^*,1}^{(j,1)}(t)\nonumber\\
  &&+\chi_{|\alpha|\geq N_m^0+1}\mathcal{E}_{f^{m,\epsilon},N_m}(t)\mathcal{E}^1_{f^{m,\epsilon},N^0_m,\ell^*_m+1-\frac\gamma2,1}(t)+\eta\left\|\partial^\alpha f^{m,\epsilon}(t)\right\|_\nu^2+\eta(1+t)^{-1-\vartheta}
\left\|w_{\ell^*_m,1}\partial^\alpha f^{m,\epsilon}(t)\langle v\rangle^{\frac{1}2}\right\|^2;\nonumber
\end{eqnarray}
\item [ii)] For $|\alpha|+|\beta|\leq n$, one has
\begin{eqnarray}\label{2-w-typical-high-non-R}
&&\left|\left(\partial^\alpha_\beta{\bf\{I-P\}}\left\{
   q_0\epsilon^{m}(E^{m,\epsilon}+\epsilon v\times B^{m,\epsilon})\cdot\nabla_v f^{m,\epsilon}\right\},w^2_{\ell^*_m-|\beta|,1}\partial^\alpha_\beta{\bf\{I-P\}} f^{m,\epsilon}\right)\right|\nonumber\\
 &\lesssim&\mathcal{E}_{f^{P,\epsilon},n}(t)\mathcal{D}_{f^{m,\epsilon},n}(t)+\left\|\epsilon^m E^{m,\epsilon}(t)\right\|_{L^\infty_x}\left\|w_{\ell^*_m-|\beta|,1}\partial^\alpha_\beta{\bf \{I-P\}} f^{m,\epsilon}(t)\langle v\rangle^{\frac{1}2}\right\|^2\\
 &&+(1+t)^{2+2\vartheta}\sum_{1\leq|\alpha_1|\leq N^0_{R}-2}\left\|\partial^{\alpha_1}\left[\epsilon^m
 E^{m,\epsilon}(t),\epsilon^{m+1} B^{m,\epsilon}(t)\right]\right\|^2_{L^\infty_x}\sum_{|\beta|\leq k\leq\min\{|\beta|+1,|\alpha|+|\beta|\} \atop k\leq j\leq |\alpha|+|\beta|}\mathcal{D}_{f^{m,\epsilon},j,\ell_m^*,1}^{(j,k)}(t)\nonumber\\
   &&+(1+t)^{2+2\vartheta}\left\|\epsilon\nabla^{N^0_{R}-1}\left[\epsilon^m
 E^{m,\epsilon}(t),\epsilon^{m+1} B^{m,\epsilon}(t)\right]\right\|_{H^1_x}^2\sum_{|\beta|\leq k\leq\min\{|\beta|+1,|\alpha|+|\beta|\} \atop k\leq j\leq |\alpha|+|\beta|}\mathcal{D}_{f^{m,\epsilon},j,\ell_m^*,1}^{(j,k)}(t)\nonumber\\
  &&+\chi_{|\alpha|+|\beta|\geq N_m^0+1}\mathcal{E}_{f^{m,\epsilon},N_m}(t) \mathcal{E}^1_{f^{m,\epsilon},N^0_m,\ell^*_m+1-\frac\gamma2,1}(t)\nonumber\\
&&+\eta\left\|\partial^\alpha {\bf \{I-P\}} f^{m,\epsilon}(t)\right\|_\nu^2+\eta(1+t)^{-1-\vartheta}
\left\|w_{\ell^*_m-|\beta|,1}\partial^\alpha_\beta{\bf \{I-P\}} f^{m,\epsilon}(t)\langle v\rangle^{\frac{1}2}\right\|^2\nonumber
\end{eqnarray}
and
\begin{eqnarray}\label{2-w-typical-high-non-1-R}
&&\left|\left(\partial^\alpha_\beta\left\{q_0 \epsilon^{m} E^{m,\epsilon}\cdot v f^{m,\epsilon}\right\},w^2_{\ell^*_m-|\beta|,1}\partial^\alpha_\beta{\bf\{I-P\}} f^{m,\epsilon}\right)\right|\nonumber\\
 &\lesssim&\mathcal{E}_{f^{P,\epsilon},n}(t)\mathcal{D}_{f^{m,\epsilon},n}(t)+\left\|\epsilon^m E^{m,\epsilon}(t)\right\|_{L^\infty_x}\left\|w_{\ell^*_m-|\beta|,1}\partial^\alpha_\beta{\bf \{I-P\}} f^{m,\epsilon}(t)\langle v\rangle^{\frac{1}2}\right\|^2\nonumber\\
 &&+(1+t)^{2+2\vartheta}\sum_{1\leq|\alpha_1|\leq N^0_{R}-2}\left\|\partial^{\alpha_1}\left[\epsilon^m
 E^{m,\epsilon}(t),\epsilon^{m+1} B^{m,\epsilon}(t)\right]\right\|^2_{L^\infty_x}\sum_{|\beta|-1\leq k\leq|\beta| \atop k\leq j\leq |\alpha|+|\beta|}\mathcal{D}_{f^{m,\epsilon},j,\ell_m^*,1}^{(j,k)}(t)\\
   &&+(1+t)^{2+2\vartheta}\left\|\epsilon\nabla^{N^0_{R}-1}\left[\epsilon^m
 E^{m,\epsilon}(t),\epsilon^{m+1} B^{m,\epsilon}(t)\right]\right\|_{H^1_x}^2 \sum_{|\beta|-1\leq k\leq|\beta| \atop k\leq j\leq |\alpha|+|\beta|}\mathcal{D}_{f^{m,\epsilon},j,\ell_m^*,1}^{(j,k)}(t)\nonumber\\
  &&+\chi_{|\alpha|+|\beta|\geq N_m^0+1}\mathcal{E}_{f^{m,\epsilon},N_m}(t) \mathcal{E}^1_{f^{m,\epsilon},N^0_m,\ell^*_m+1-\frac\gamma2,1}(t)\nonumber\\
&&+\eta\left\|\partial^\alpha {\bf \{I-P\}} f^{m,\epsilon}(t)\right\|_\nu^2+\eta(1+t)^{-1-\vartheta}
\left\|w_{\ell^*_m-|\beta|,1}\partial^\alpha_\beta{\bf \{I-P\}} f^{m,\epsilon}(t)\langle v\rangle^{\frac{1}2}\right\|^2.\nonumber
\end{eqnarray}
\end{itemize}
 \end{lemma}
\begin{proof} To simplify the presentation, we only prove \eqref{1-w-typical-high-non-R} since the other estimates can be obtained similarly. For this purpose, since
 \begin{eqnarray}
&&\left(\partial^\alpha \left\{q_0 \epsilon^m E^{m,\epsilon}\cdot\nabla_{ v  }f^{m,\epsilon}\right\},w^2_{\ell^*_m,1}\partial^\alpha f^{m,\epsilon}\right)\nonumber\\
&=&\underbrace{\left( q_0 \epsilon^m E^{m,\epsilon}\cdot\nabla_{ v  }\partial^\alpha f^{m,\epsilon},w^2_{\ell^*_m,1}\partial^\alpha f^{m,\epsilon}\right)}_{\tilde{H}_{\alpha,1}}
+\underbrace{\sum_{0<\alpha_1\leq\alpha}\left( q_0 \epsilon^m \partial^{\alpha_1} E^{m,\epsilon}\cdot\nabla_{ v  }\partial^{\alpha-\alpha_1}{\bf P}f^{m,\epsilon},w^2_{\ell^*_m,1}\partial^\alpha f^{m,\epsilon}\right)}_{\tilde{H}_{\alpha,2}}\\
&&+\underbrace{\sum_{0<\alpha_1\leq\alpha}\left( q_0 \epsilon^m \partial^{\alpha_1} E^{m,\epsilon}\cdot\nabla_{ v  }\partial^{\alpha-\alpha_1}{\bf \{I-P\}}f^{m,\epsilon},w^2_{\ell^*_m,1}\partial^\alpha f^{m,\epsilon}\right)}_{\tilde{H}_{\alpha,3}},\nonumber
\end{eqnarray}
we can get by using integrations by parts and Holder's inequality that
\[
\tilde{H}_{\alpha,1}\lesssim\left\|\epsilon^m E^{m,\epsilon}\right\|_{L^\infty_x}\left\|w_{\ell^*_m,1}\partial^\alpha f^{m,\epsilon}\langle v\rangle^{\frac12}\right\|^2
\]
and
\[\tilde{H}_{\alpha,2}\lesssim
\left\| \epsilon^m E^{m,\epsilon}\right\|^2_{H^{N_m+1}_x}\|\nabla_x f^{m,\epsilon}\|^2_{H^{N_m}_x}+\eta\left\|\partial^\alpha f^{m,\epsilon}\langle v\rangle^{\frac\gamma2}\right\|^2.\]

For $\tilde{H}_{\alpha,3}$, one has
 \begin{eqnarray*}
\tilde{H}_{\alpha,3}&=&\sum_{0<\alpha_1\leq\alpha}\left( q_0 \epsilon^m \partial^{\alpha_1} E^{m,\epsilon}\cdot\nabla_{ v  }\partial^{\alpha-\alpha_1}{\bf \{I-P\}}f^{m,\epsilon},w^2_{\ell^*_m,1}\partial^\alpha f^{m,\epsilon}\right)\\
&\lesssim&\underbrace{\sum_{1\leq|\alpha_1|\leq N^0_{R}-2}\left\|\epsilon^m\partial^{\alpha_1}E^{m,\epsilon}\right\|_{L^\infty_x}
\left\|w_{\ell^*_m-1,1}\partial^{\alpha-\alpha_1}_{e_j}{\bf \{I-P\}}f^{m,\epsilon}\langle v\rangle^{\frac{1}2}\right\|
\left\|w_{\ell^*_m,1}\partial^\alpha f^{m,\epsilon}\langle v\rangle^{\frac{1}2}\right\|}_{\tilde{H}^{(1)}_{\alpha,3}}\\{}
&&+\underbrace{\sum_{|\alpha_1|= N^0_m-1}\left\|\epsilon^m\partial^{\alpha_1}E^{m,\epsilon}\right\|_{L^6_x}
\left\|w_{\ell^*_m-1,1}\partial^{\alpha-\alpha_1}_{e_j}{\bf \{I-P\}}f^{m,\epsilon}\langle v\rangle^{\frac{1}2}\right\|_{L^2_vL^3_x}
\left\|w_{\ell^*_m,1}\partial^\alpha f^{m,\epsilon}\langle v\rangle^{\frac{1}2}\right\|}_{\tilde{H}^{(2)}_{\alpha,3}}\\{}
&&+\underbrace{\sum_{|\alpha_1|= N^0_m}\left\|\epsilon^m\partial^{\alpha_1}E^{m,\epsilon}\right\|
\left\|w_{\ell^*_m-1,1}\partial_{e_j}{\bf \{I-P\}}f^{m,\epsilon}\langle v\rangle^{\frac{1}2}\right\|_{L^2_vL^\infty_x}
\left\|w_{\ell^*_m,1}\partial^\alpha f^{m,\epsilon}\langle v\rangle^{\frac{1}2}\right\|}_{\tilde{H}^{(3)}_{\alpha,3}}\\{}
&&+\chi_{|\alpha|\geq N_m^0+1}\nonumber\\
&&\times\underbrace{\sum_{|\alpha_1|\geq N^0_m+1}\int_{\mathbb{R}^3_x\times \mathbb{R}^3_v}\left|\epsilon^m\partial^{\alpha_1}E^{m,\epsilon}\right|
\left|w_{\ell^*_m-1,1}
\partial^{\alpha-\alpha_1}_{e_j}{\bf \{I-P\}}f^{m,\epsilon}
\langle v\rangle^{1-\frac\gamma2}\right|
\left|w_{\ell^*_m,1}\partial^\alpha f^{m,\epsilon}\langle v\rangle^{\frac{\gamma}2}\right|dvdx}_{\tilde{H}^{(4)}_{\alpha,3}}.
\end{eqnarray*}
To control the terms $\tilde{H}^{(j)}_{\alpha,3}$ for $j=1,2,3,4$, we first deal with
$\tilde{H}^{(3)}_{\alpha,3}$ as follows
\begin{eqnarray}
  \tilde{H}^{(3)}_{\alpha,3}  &\lesssim&(1+t)^{1+\vartheta}\sum_{1\leq|\alpha_1|\leq N^0_{R}-2}\left\|\epsilon^m\partial^{\alpha_1}E^{m,\epsilon}\right\|^2_{L^\infty_x}
\left\|w_{\ell^*_m-1,1}\partial^{\alpha-\alpha_1}_{e_j}{\bf \{I-P\}}f^{m,\epsilon}\langle v\rangle^{\frac{1}2}\right\|^2\nonumber\\
&&+\eta(1+t)^{-1-\vartheta}
\left\|w_{\ell^*_m,1}\partial^\alpha f^{m,\epsilon}\langle v\rangle^{\frac{1}2}\right\|^2\nonumber\\
&\lesssim&(1+t)^{2+2\vartheta}\sum_{1\leq|\alpha_1|\leq N^0_{R}-2}\left\|\epsilon^m\partial^{\alpha_1}E^{m,\epsilon}\right\|^2_{L^\infty_x}\sum_{1\leq j\leq |\alpha|}\mathcal{D}_{f^{m,\epsilon},j,\ell_m^*,1}^{(j,1)}(t)\nonumber\\
&&+\eta(1+t)^{-1-\vartheta}
\left\|w_{\ell^*_m,1}\partial^\alpha f^{m,\epsilon}\langle v\rangle^{\frac{1}2}\right\|^2\nonumber
\end{eqnarray}
and $\tilde{H}^{(1)}_{\alpha,3}$ and $\tilde{H}^{(2)}_{\alpha,3}$ can be estimated similarly.

As for $\tilde{H}^{(4)}_{\alpha,3}$,
one has by Sobolev's inequalities that
\[
\tilde{H}^{(4)}_{\alpha,3}\lesssim \mathcal{E}_{f^{m,\epsilon},N_m}(t)\mathcal{E}^1_{f^{m,\epsilon},N^0_m,\ell^*_m+1-\frac\gamma2,1}(t)
+\eta(1+t)^{-1-\vartheta}
\left\|w_{\ell^*_m,1}\partial^\alpha f^{m,\epsilon}\langle v\rangle^{\frac{1}2}\right\|^2\nonumber.\]

Collecting the above estimates, one can deduce \eqref{1-w-typical-high-non-R} immediately.
\end{proof}
Similarly, for the estimates on the interactions of $f^{m,\epsilon}(t,x,v)$ with $\left[E^{P,\epsilon}(t,x), E^{1,\epsilon}(t,x),\cdots, E^{m-1,\epsilon}(t,x)\right]$ with respect to the weight $w_{\ell^*_m-|\beta|,1}(t,v)$, we also have the following lemma, whose proofs will be omitted for brevity.
\begin{lemma}\label{lemma-high-E-R}
For $2\leq n\leq N_m+1$, we have the following estimates:
\begin{itemize}
\item [i)]For $1\leq |\alpha|\leq n$, one has
\begin{eqnarray}\label{1-w-typical-high-R}
  &&\left|\left(\partial^\alpha\left\{q_0 E^{P,\epsilon}\cdot\nabla_{ v  }f^{m,\epsilon}- \frac12q_0 E^{P,\epsilon}\cdot v f^{m,\epsilon}\right\},w^2_{\ell^*_m,1}\partial^\alpha f^{m,\epsilon}\right)\right|\nonumber\\
   &\lesssim&\mathcal{E}_{f^{P,\epsilon},n}(t)\mathcal{D}_{f^{m,\epsilon},n}(t) +\left\| E^{P,\epsilon}(t)\right\|_{L^\infty_x}\left\|w_{\ell^*_m,1}\partial^\alpha f^{m,\epsilon}(t)\langle v\rangle^{\frac12}\right\|^2\nonumber\\
  &&+(1+t)^{2+2\vartheta}\left\{\left\|\nabla_xE^{P,\epsilon}(t)\right\|^2_{L^\infty_x}
  +\left\|\nabla^3_xE^{P,\epsilon}(t)\right\|^2_{H^{N_m-2}_x}\right\}\sum_{0\leq k\leq 1,\atop k\leq j\leq |\alpha|}\mathcal{D}^{(j,k)}_{f^{m,\epsilon},j,\ell^*_m,1}(t)\\
  &&+\eta\left\|\partial^\alpha f^{m,\epsilon}(t)\right\|_\nu^2 +\eta(1+t)^{-1-\vartheta}
\left\|w_{\ell^*_m,1}\partial^\alpha f^{m,\epsilon}(t)\langle v\rangle^{\frac{1}2}\right\|^2\nonumber
\end{eqnarray}
and
\begin{eqnarray}\label{1-w-typical-high-R-1}
&&\left|\left(\partial^\alpha\left\{q_0 E^{j_1,\epsilon}\cdot\nabla_{ v  }f^{m,\epsilon}- \frac12q_0 E^{j_1,\epsilon}\cdot v f^{m,\epsilon}\right\},w^2_{\ell^*_m,1}\partial^\alpha f^{m,\epsilon}\right)\right|\nonumber\\
&\lesssim&\mathcal{E}_{f^{j_1,\epsilon},n}(t)\mathcal{D}_{f^{m,\epsilon},n}(t) +\left\| E^{j_1,\epsilon}(t)\right\|_{L^\infty_x}\left\|w_{\ell^*_m,1}\partial^\alpha f^{m,\epsilon}(t)\langle v\rangle^{\frac12}\right\|^2\nonumber\\
  &&+(1+t)^{2+2\vartheta}\left\{\left\|\nabla_xE^{j_1,\epsilon}(t)\right\|^2_{L^\infty_x}
  +\left\|\nabla^3_xE^{j_1,\epsilon}(t)\right\|^2_{H^{N_m-2}_x}\right\}\sum_{0\leq j\leq |\alpha|}\mathcal{D}^{(j,0)}_{f^{m,\epsilon},j,\ell^*_m,1}(t)\\
  &&+\eta\left\|\partial^\alpha f^{m,\epsilon}(t)\right\|_\nu^2 +\eta(1+t)^{-1-\vartheta}
\left\|w_{\ell^*_m,1}\partial^\alpha f^{m,\epsilon}(t)\langle v\rangle^{\frac{1}2}\right\|^2;\nonumber
\end{eqnarray}
\item [ii)] For $|\alpha|+|\beta|\leq n$ with $|\beta|\geq 1$ or $|\alpha|=|\beta|=0$, one has
\begin{eqnarray}\label{2-w-typical-high}
  &&\left|\left(\partial^\alpha_\beta\left\{{\bf\{I-P\}}\left(q_0 E^{P,\epsilon}\cdot\nabla_{ v  }f^{m,\epsilon}- \frac12q_0 E^{P,\epsilon}\cdot v f^{m,\epsilon}\right)\right\},w^2_{\ell^*_m-|\beta|,1}\partial^\alpha_\beta{\bf\{I-P\}} f^{m,\epsilon}\right)\right|\nonumber\\
 &\lesssim&\mathcal{E}_{f^{P,\epsilon},n}(t)\mathcal{D}_{f^{m,\epsilon},n}(t)+\left\| E^{P,\epsilon}(t)\right\|_{L^\infty_x}\left\|w_{\ell^*_m-|\beta|,1}\partial^\alpha_\beta{\bf\{I-P\}} f^{m,\epsilon}(t)\langle v\rangle^{\frac12}\right\|^2\\
 &&+(1+t)^{2+2\vartheta}\left\{\left\|\nabla_xE^{P,\epsilon}(t)\right\|^2_{L^\infty_x}
  +\left\|\nabla^3_xE^{P,\epsilon}(t)\right\|^2_{H^{N_m-2}_x}\right\}\sum_{|\beta|-1\leq k\leq \min\{|\beta|+1,|\alpha|+|\beta|\}\atop k\leq j\leq |\alpha|+|\beta|}\mathcal{D}^{(j,k)}_{f^{m,\epsilon},j,\ell^*_m,1}(t)\nonumber\\
&&+\eta\left\|\partial^\alpha {\bf \{I-P\}} f^{i,\epsilon}(t)\right\|_\nu^2+\eta(1+t)^{-1-\vartheta}
\left\|w_{\ell^*_m-|\beta|,1}\partial^\alpha_\beta{\bf \{I-P\}}f^{m,\epsilon}(t)\langle v\rangle^{\frac{1}2}\right\|^2\nonumber
\end{eqnarray}
and
\begin{eqnarray}\label{2-w-typical-high-1}
  &&\left|\left(\partial^\alpha_\beta\left\{{\bf\{I-P\}}\left(q_0 E^{j_1,\epsilon}\cdot\nabla_{ v  }f^{m,\epsilon}- \frac12q_0 E^{j_1,\epsilon}\cdot v f^{m,\epsilon}\right)\right\},w^2_{\ell^*_m-|\beta|,1}\partial^\alpha_\beta{\bf\{I-P\}} f^{m,\epsilon}\right)\right|\nonumber\\
 &\lesssim&\mathcal{E}_{f^{j_1,\epsilon},n}(t)\mathcal{D}_{f^{m,\epsilon},n}(t)+\left\| E^{j_1,\epsilon}(t)\right\|_{L^\infty_x}\left\|w_{\ell^*_m-|\beta|,1}\partial^\alpha_\beta{\bf\{I-P\}} f^{m,\epsilon}(t)\langle v\rangle^{\frac12}\right\|^2\\
 &&+(1+t)^{2+2\vartheta}\left\{\left\|\nabla_xE^{j_1,\epsilon}(t)\right\|^2_{L^\infty_x}
  +\left\|\nabla^3_xE^{j_1,\epsilon}(t)\right\|^2_{H^{N_m-2}_x}\right\}\sum_{|\beta|-1\leq k\leq \min\{|\beta|+1,|\alpha|+|\beta|\}\atop k\leq j\leq |\alpha|+|\beta|}\mathcal{D}^{(j,k)}_{f^{m,\epsilon},j,\ell^*_m,1}(t)\nonumber\\
&&+\eta\left\|\partial^\alpha {\bf \{I-P\}} f^{i,\epsilon}(t)\right\|_\nu^2+\eta(1+t)^{-1-\vartheta}
\left\|w_{\ell^*_m-|\beta|,1}\partial^\alpha_\beta{\bf \{I-P\}}f^{m,\epsilon}(t)\langle v\rangle^{\frac{1}2}\right\|^2.\nonumber
\end{eqnarray}
\end{itemize}
 \end{lemma}

\subsection{Lyapunov-type inequalities for $\mathcal{E}_{f^{m,\epsilon},N_m,l_m,-\gamma}(t)$, $\mathcal{E}_{f^{m,\epsilon},N^0_m,l^0_m,-\gamma}(t)$, and $\mathcal{E}_{f^{m,\epsilon},N_m+1}(t)$}
Now we are going to derive certain  Lyapunov-type inequalities for $\mathcal{E}_{f^{m,\epsilon},N_m,l_m,-\gamma}(t)$, $\mathcal{E}_{f^{m,\epsilon},N^0_m,l^0_m,-\gamma}(t)$, and $\mathcal{E}_{f^{m,\epsilon},N_m+1}(t)$. Our first result is to deduce the Lyapunov-type inequalities for $\mathcal{E}_{f^{m,\epsilon},N_m,l_m,-\gamma}(t)$.
\subsubsection{Lyapunov-type inequalities for $\mathcal{E}_{f^{m,\epsilon},N_m,l_m,-\gamma}(t)$} For result on this problem, we can get that
\begin{lemma}\label{lemma-f-m}
Take
\[\widetilde{\ell}_m\geq 1-\frac{3\gamma} 2,\ {\theta}=\frac{2-4\gamma}{2\widetilde{\ell}_m-\gamma},\]
and assume that
\begin{eqnarray}\label{Assume-f-R-N-R-low}
 \max&&\left\{\sup_{0\leq\tau\leq t}(1+\tau)^{1+\vartheta}\left\|\left[E^{P,\epsilon},E^{j_1,\epsilon},\epsilon^m E^{m,\epsilon}\right](\tau)\right\|_{L^\infty_x},  \sup_{0\leq\tau\leq t}\mathcal{E}_{f^{P,\epsilon},N_m+1,l_m+2-\frac1\gamma,-\gamma}(\tau), \right.\nonumber\\
 &&\ \ \ \ \ \ \ \ \ \ \ \ \ \ \ \ \ \ \ \ \ \ \ \ \ \ \left.\sup_{0\leq\tau\leq t}\mathcal{E}_{f^{j_1,\epsilon},N_m+1,l_m+2-\frac1\gamma,-\gamma}(\tau),\sup_{0\leq\tau\leq t}\mathcal{E}_{f^{m,\epsilon},N_m,l_m,-\gamma}(\tau)\right\}
\end{eqnarray}
is sufficiently small,
then one has
\begin{eqnarray}\label{lemma-f-R-N-R}
&&\frac{d}{dt}\mathcal{E}_{f^{m,\epsilon},N_m,l_m,-\gamma}(t)+\mathcal{D}_{f^{m,\epsilon},N_m,l_m,-\gamma}(t)\nonumber\\
 &\lesssim&\left\{\sum_{1\leq j_1\leq m-1}\left\|\nabla_x\left[E^{P,\epsilon}(t), E^{j_1,\epsilon}(t)\right]\right\|_{H^{N_m-1}_x}^{2}+\epsilon^{2m}\left\|\nabla_x
\left[E^{m,\epsilon}(t), \epsilon B^{m,\epsilon}(t)\right]\right\|^2_{H^{N_m^0-1}_x}
\right\}^{\frac{1}{\theta}}\nonumber\\
&&\times\sum_{0\leq k\leq N_m,\atop k\leq j\leq N_m}\mathcal{D}^{(j,k)}_{f^{m,\epsilon},j,\widetilde{\ell}_m-\gamma l_m+1+\frac\gamma2,1}(t)\\
 &&+\sum_{j_1+j_2\geq m,\atop 0<j_1,j_2<m}\mathcal{E}_{f^{j_1,\epsilon},N_m,l_m,-\gamma}(t)\mathcal{D}_{f^{j_2,\epsilon},N_m+1, l_m+2-\frac{1}\gamma,-\gamma}(t)\nonumber\\
   &&+
  \mathcal{E}_{f^{m,\epsilon},N_m,l_m,-\gamma}(t)\left\{\mathcal{D}_{f^{P,\epsilon},N_m+1,l_m+2-\frac{1}\gamma,-\gamma}(t)+\sum_{i=1}^{m-1}\mathcal{D}_{f^{i,\epsilon},N_m+1,l_m+2-\frac{1}\gamma,-\gamma}(t)\right\}\nonumber\\
  &&+\mathcal{E}_{f^{m,\epsilon},N_m}(t) \mathcal{E}^1_{f^{m,\epsilon},N^0_m,l_m+\frac32-\frac1\gamma,1}(t) +\left\|\nabla^{N_m+1}f^{m,\epsilon}(t)\right\|_\nu^2.\nonumber
  \end{eqnarray}
\end{lemma}
\begin{proof}
  \noindent{\bf Step 1:}\quad Applying $\partial^\alpha$ with $|\alpha|\leq N_m$ to \eqref{f-R-vector}, multiplying the resulting identity by $\partial^\alpha f^{m,\epsilon}$ with $|\alpha|\leq N_m$, and integrating the final result with respect to $v$ and $x$ over $\mathbb{R}^3_v\times\mathbb{R}^3_x$, one has from \eqref{lemma-f-R-1}, \eqref{0-order-E-R-B-R}, \eqref{0-order-E-R-12}, \eqref{0-order-gamma-R}, \eqref{0-w-typical-R}, \eqref{1-w-typical-R},\eqref{1-typical-low-non-R}, \eqref{1-typical-low-non-1-R}, and \eqref{1-typical-low-non-gamma-R} that
  \begin{eqnarray}\label{f-R-spatial-x}
&&\frac{d}{dt}\left\|\left[\partial^\alpha f^{m,\epsilon},\partial^\alpha E^{m,\epsilon},\partial^\alpha B^{m,\epsilon}\right](t)\right\|^2+\left\|\partial^\alpha\{{\bf I-P}\}f^{m,\epsilon}(t)\right\|_\nu^2\nonumber\\
  &\lesssim&
  \mathcal{E}_{f^{m,\epsilon},N_m-1,N_m-\frac12-\frac1\gamma,-\gamma}(t)
  \left\{\mathcal{D}_{f^{P,\epsilon},N_m}(t)+\sum_{1\leq j_1\leq m-1}\mathcal{D}_{f^{j_1,\epsilon},N_m}(t)\right\}\nonumber\\
   &&+
  \mathcal{E}_{f^{m,\epsilon},N_m}(t)\nonumber\\
&&\times\left\{\mathcal{D}_{f^{P,\epsilon},N_m+1,N_m+2-\frac1\gamma,-\gamma}(t)
  +\sum_{i=1}^{m-1}\mathcal{D}_{f^{i,\epsilon},N_m+1,N_m+2-\frac1\gamma,-\gamma}(t)
  +\mathcal{D}_{f^{m,\epsilon},3,4-\frac1\gamma,-\gamma}(t)\right\}\nonumber\\
  &&+\left\{\mathcal{E}_{f^{P,\epsilon},N_m,N_m,-\gamma}(t)+\sum_{1\leq j_1\leq m-1}\mathcal{E}_{f^{j_1,\epsilon},N_m-1,N_m-1,-\gamma}(t)
  +\mathcal{E}_{f^{m,\epsilon},N_m-1,N_m-\frac12-\frac1\gamma,-\gamma}(t)\right\}
\nonumber\\
&&\times \mathcal{D}_{f^{m,\epsilon},N_m}(t)\nonumber\\
 &&+\sum_{1\leq j_1\leq m-1}\left\|\nabla_x\left[\epsilon E^{j_1,\epsilon},E^{P,\epsilon},\epsilon^m E^{m,\epsilon},\epsilon^{m+1} B^{m,\epsilon}\right]\right\|^2_{L^\infty_x} \left\|\nabla_v\nabla^{N_m-1}_x{\bf\{I-P\}}f^{m,\epsilon}\langle v\rangle^{1-\frac\gamma2}\right\|^2\\
  &&+\sum_{1\leq j_1\leq m-1}\left\|\left[\epsilon ^m E^{m,\epsilon}, E^{j_1,\epsilon},E^{P,\epsilon}\right]\right\|^2_{L^\infty_x} \left\|\nabla^{N_m}_x{\bf\{I-P\}}f^{m,\epsilon}\langle v\rangle^{1-\frac\gamma2}\right\|^2\nonumber\\
  &&+\sum_{j_1+j_2\geq m,\atop 0<j_1,j_2<m}\left\{\mathcal{E}_{f^{j_1,\epsilon},N_m,N_m,-\gamma}(t)
  \mathcal{D}_{f^{j_2,\epsilon},N_m+1,N_m+2-\frac1\gamma,-\gamma}(t)\right\}
  +\eta\mathcal{D}_{f^{m,\epsilon},N_m}(t).\nonumber
    \end{eqnarray}

 \noindent{\bf Step 2:}
For simplicity in presentation, we use the notations ${I}_{mix,mic}(t)$ and $I_{pure,mic}(t)$ to denote
\begin{eqnarray}\label{def-J-mic-mix-R}
  {I}_{mix,mic}(t)\equiv
&& \frac12{\bf \{I-P\}}q_0 E^{m,\epsilon}\cdot v f^{P,\epsilon}-{\bf \{I-P\}}q_0  E^{m,\epsilon}\cdot\nabla_vf^{P,\epsilon}\nonumber\\
&&+\frac12{\bf \{I-P\}}q_0\sum_{0<j_1< m} \epsilon^{j_1}E^{m,\epsilon}
\cdot vf^{j_1,\epsilon}-{\bf \{I-P\}}
q_0\sum_{0<j_1< m} \epsilon^{j_1}E^{m,\epsilon}\cdot\nabla_vf^{j_1,\epsilon}\\
&&-{\bf \{I-P\}}
q_0\epsilon \left(v\times B^{m,\epsilon}\right)\cdot \nabla_v\left\{f^{P,\epsilon}+\sum_{i=1}^{m-1}\epsilon^if^{i,\epsilon}\right\}\nonumber\\
&&+
\frac12{\bf \{I-P\}}q_0\sum_{j_1+j_2\geq m,\atop 0<j_1,j_2<m} \epsilon^{j_1+j_2-m}E^{j_1,\epsilon}\cdot vf^{j_2,\epsilon}\nonumber\\
&&-{\bf \{I-P\}}q_0\sum_{j_1+j_2\geq m,\atop 0<j_1,j_2<m} \epsilon^{j_1+j_2-m}E^{j_1,\epsilon}\cdot \nabla_vf^{j_2,\epsilon}\nonumber\\
&&+{\Gamma}\left(f^{P,\epsilon},f^{m,\epsilon}\right)
  +{\Gamma}\left(f^{m,\epsilon},f^{P,\epsilon}\right)+\sum_{j_1+j_2\geq m,\atop 0<j_1,j_2<m}\epsilon^{j_1+j_2-m}\Gamma\left(f^{j_1,\epsilon},f^{j_2,\epsilon}\right)
  \nonumber\\
&&+\sum_{1\leq j_1\leq m-1}\epsilon^{j_1}\left\{\Gamma\left(f^{j_1,\epsilon},f^{m,\epsilon}\right)+\Gamma\left(f^{m,\epsilon},f^{j_1,\epsilon}\right)\right\}\nonumber
\end{eqnarray}
and
\begin{eqnarray}\label{def-J-mic-pure-R}
  I_{pure,mic}(t)&\equiv&
  \frac12{\bf \{I-P\}}q_0 E^{P,\epsilon}\cdot v f^{m,\epsilon}-{\bf \{I-P\}}q_0 E^{P,\epsilon}\cdot\nabla_{ v  }f^{m,\epsilon}\nonumber\\
&&+\frac12{\bf \{I-P\}}q_0\sum_{0<j_1< m} \epsilon^{j_1}E^{j_1,\epsilon}
\cdot vf^{m,\epsilon}-{\bf \{I-P\}}q_0\sum_{0<j_1< m} \epsilon^{j_1}E^{j_1,\epsilon}
\cdot\nabla_vf^{m,\epsilon}\\
&&+ \frac 1 2{\bf \{I-P\}}q_0\epsilon^{m}p E^{m,\epsilon}\cdot v f^{m,\epsilon}- {\bf \{I-P\}}q_0\epsilon^{m}\left\{ E^{m,\epsilon}+ \epsilon v\times B^{m,\epsilon}\right\}\cdot\nabla_vf^{m,\epsilon}\nonumber\\
&&+\epsilon^m\Gamma(f^{m,\epsilon},f^{m,\epsilon}),\nonumber
\end{eqnarray}
then \eqref{f-R-vect-I-P} can be rewritten as
\begin{eqnarray}\label{mic-I-P-equation}
  &&\partial_t{\bf \{I-P\}}f^{m,\epsilon}+ {\bf \{I-P\}}v\cdot\nabla_xf^{m,\epsilon}- \left\{E^{m,\epsilon}+\epsilon b^{f^{m,\epsilon}}\times \left\{B^{P}
+\sum_{j_1=1}^{m-1}\epsilon^{j_1}B^{j_1}\right\}\right\} \cdot v \mu^{1/2}q_1\nonumber\\
  &&+\epsilon q_0\left\{\left(v\times \left(B^{P}
+\sum_{j_1=1}^{m-1}\epsilon^{j_1}B^{j_1}\right)\right)\right\}\cdot\nabla_v{\bf \{I-P\}}f^{m,\epsilon}+Lf^{m,\epsilon}\\
&\equiv&
  I_{mix,mic}(t)+I_{pure,mic}(t).\nonumber
\end{eqnarray}

Multiplying \eqref{mic-I-P-equation} by $w^2_{l_m,-\gamma}{\bf \{I-P\}}f^{m,\epsilon}$ and integrating the resulting identity over $\mathbb{R}^3_v\times\mathbb{R}^3_x$, one has
\begin{eqnarray}\label{00-weight-I-P-R}
&&  \frac{d}{dt}\|w_{l_m,-\gamma}{\bf \{I-P\}}f^{m,\epsilon}\|^2+\|w_{l_m,-\gamma}{\bf \{I-P\}}f^{m,\epsilon}\|^2_{\nu}\nonumber\\
&&+\frac{q\vartheta}{(1+t)^{1+\vartheta}}\|w_{l_m,-\gamma}{\bf \{I-P\}}f^{m,\epsilon}\langle v\rangle\|^2\nonumber\\
&\lesssim&\|{\bf\{I-P\}}f^{m,\epsilon}\|_\nu^2+\underbrace{\left|\left(-{\bf \{I-P\}}v\cdot\nabla_xf^{m,\epsilon},w^2_{l_m,-\gamma}{\bf \{I-P\}}f^{m,\epsilon}\right)\right|
}_{\mathcal{S}_1}\nonumber\\
&&+\underbrace{\left|\left(\left\{E^{m,\epsilon}+\epsilon b^{f^{m,\epsilon}}\times \left\{B^{P}
+\sum_{j_1=1}^{m-1}\epsilon^{j_1}B^{j_1}\right\}\right\} \cdot v \mu^{1/2}q_1,w^2_{l_m,-\gamma}{\bf \{I-P\}}f^{m,\epsilon}\right)\right|
}_{\mathcal{S}_2}\nonumber\\
&&
+\underbrace{\left|\left(I_{mix,mic}(t),w^2_{l_m,-\gamma}{\bf \{I-P\}}f^{m,\epsilon}\right)\right|}_{\mathcal{S}_3}\nonumber\\
&&+\underbrace{\left|\left(I_{pure,mic}(t),w^2_{l_m,-\gamma}{\bf \{I-P\}}f^{m,\epsilon}\right)\right|}_{\mathcal{S}_4}\nonumber\\
&&+\underbrace{\left|\left(\epsilon q_0\left\{v\times \left\{B^{P}
+\sum_{j_1=1}^{m-1}\epsilon^{j_1}B^{j_1}\right\}\right\}\cdot\nabla_v{\bf \{I-P\}}f^{m,\epsilon},w^2_{l_m,-\gamma}{\bf \{I-P\}}f^{m,\epsilon}\right)\right|}_{\mathcal{S}_5}.
\end{eqnarray}
For $\mathcal{S}_1$, one has
\begin{eqnarray*}
  \mathcal{S}_1&\lesssim&\left|\left(v\cdot\nabla_x{\bf \{I-P\}}f^{m,\epsilon},w^2_{l_m,-\gamma}{\bf \{I-P\}}f^{m,\epsilon}\right)\right|+\left|\left(v\cdot\nabla_x{\bf P}f^{m,\epsilon},w^2_{l_m,-\gamma}{\bf \{I-P\}}f^{m,\epsilon}\right)\right|\\
  &&+\left|\left({\bf P}\left\{v\cdot\nabla_xf^{m,\epsilon}\right\},w^2_{l_m,-\gamma}{\bf \{I-P\}}f^{m,\epsilon}\right)\right|\\
  &\lesssim&\eta\|\nabla_x f^{m,\epsilon}\|_\nu^2+\|{\bf\{I-P\}}f^{m,\epsilon}\|_\nu^2.
\end{eqnarray*}
\eqref{E-R-nonhard-1} yields
\begin{eqnarray}
  \mathcal{S}_2 &\lesssim&\frac1{4C_\beta}\sum_{i=1}^3\frac{d}{dt}\left\|\left\langle \left\{w^2_{l_m,-\gamma}\left[v_i \mu^{1/2}\right]\right\}, {\{\bf I_+-P_+\}}f^{m,\epsilon}-{\{\bf I_--P_-\}}f^{m,\epsilon}\right\rangle\right\|^2\\ \nonumber
  &&+\left\|f^{m,\epsilon}\right\|_\nu^2+\left\|\nabla_xf^{m,\epsilon}\right\|^2_\nu
  +\mathcal{E}_{f^{m,\epsilon},2}(t)\left\{\mathcal{D}_{f^{P,\epsilon},2}(t)+
  \sum_{1\leq j_1\leq m-1}\mathcal{D}_{f^{j_1,\epsilon},2}(t)+\mathcal{D}_{f^{m,\epsilon},2}(t)\right\}\\ \nonumber
  &&+\sum_{j_1+j_2\geq m,\atop 0<j_1,j_2<m}\mathcal{E}_{f^{j_1,\epsilon},2}(t)\mathcal{D}_{f^{j_2,\epsilon},2}(t)
  +\sum_{1\leq j_1\leq m-1}\mathcal{E}_{f^{j_1,\epsilon},2}(t)\mathcal{D}_{f^{m,\epsilon},2}(t).
\end{eqnarray}
Using \eqref{lemma-f-R-3}, \eqref{micro-weight-E-R}, \eqref{micro-weight-E-R-12}, and \eqref{micro-weight-gamma-low-R} with $\alpha=\beta=0$,
one has
\begin{eqnarray}
  \mathcal{S}_3
     &\lesssim&\sum_{j_1+j_2\geq m,\atop 0<j_1,j_2<m}\mathcal{E}_{f^{j_1,\epsilon},2,l_m,-\gamma}(t)\mathcal{D}_{f^{j_2,\epsilon},3,l_m+2-\frac{1}\gamma,-\gamma}(t)\nonumber\\
   &&+
  \mathcal{E}_{f^{m,\epsilon},2,l_m,-\gamma}(t)\left\{\mathcal{D}_{f^{P,\epsilon},3,l_m+2-\frac{1}\gamma,-\gamma}(t)+\sum_{i=1}^{m-1}\mathcal{D}_{f^{i,\epsilon},3,l_m+2-\frac{1}\gamma,-\gamma}(t)\right\}\nonumber\\
 &&+\left\{\mathcal{E}_{f^{P,\epsilon},2,l_m,-\gamma}(t)+\sum_{1\leq j_1\leq m-1}\mathcal{E}_{f^{j_1,\epsilon},2,l_m,-\gamma}(t)\right\}\mathcal{D}_{f^{m,\epsilon},2,l_m,-\gamma}(t)\\
  && +\eta\left\|w_{l_m,-\gamma}{\bf \{I-P\}} f^{m,\epsilon}\right\|_\nu^2.\nonumber
  \end{eqnarray}
  As for $\mathcal{S}_4$, one has from \eqref{2-w-typical-low-R}, \eqref{2-w-typical-low-1-R}, \eqref{mic-non-I-P-1}, \eqref{mic-non-I-P-2} and \eqref{2-w-typical-low-non-gamma-R} with $|\alpha|=|\beta|=0$
  \begin{eqnarray}\label{estimate-H-4}
    \mathcal{S}_4
   &\lesssim&\sum_{1\leq j_1\leq m-1}\left\|  [E^{P,\epsilon},E^{j_1,\epsilon},\epsilon^m E^{m,\epsilon}]\right\|_{L^\infty_x}\left\|w_{l_m,-\gamma}{\bf\{I-P\}} f^{m,\epsilon}\langle v\rangle^{\frac12}\right\|^2\\
&&+\left\{\mathcal{E}_{f^{P,\epsilon},2}(t)+\sum_{1\leq j_1\leq m-1}\mathcal{E}_{f^{j_1,\epsilon},2}(t)\right\}\mathcal{D}_{f^{m,\epsilon},2}(t)
\nonumber\\
&&+\mathcal{E}_{f^{m,\epsilon},2,l_m,-\gamma}(t)\mathcal{D}_{f^{m,\epsilon},2,l_m,-\gamma}(t)
  +\eta\mathcal{D}_{f^{m,\epsilon},2,l_m,-\gamma}(t).\nonumber
    \end{eqnarray}
Finally, it is straightforward to compute that
\[ \mathcal{S}_5\lesssim\eta\|f^{m,\epsilon}\|_\nu+\|\{{\bf I-P}\}f^{m,\epsilon}\|_\nu^2.\]

Collecting the estimates on $\mathcal{S}_1-\mathcal{S}_5$ into \eqref{00-weight-I-P-R} implies that
\begin{eqnarray}\label{00-weight-I-P-1-R}
&&  \frac{d}{dt}\|w_{l_m,-\gamma}{\bf \{I-P\}}f^{m,\epsilon}\|^2+\|w_{l_m,-\gamma}{\bf \{I-P\}}f^{m,\epsilon}\|^2_{\nu}\nonumber\\
&&+\frac{q\vartheta}{(1+t)^{1+\vartheta}}\|w_{l_m,-\gamma}{\bf \{I-P\}}f^{m,\epsilon}\langle v\rangle\|^2\nonumber\\
&\lesssim&\frac1{4C_\beta}\sum_{i=1}^3\frac{d}{dt}\left\|\left\langle \left\{w^2_{l_m,-\gamma}\left[v_i \mu^{1/2}\right]\right\}, {\{\bf I_+-P_+\}}f^{m,\epsilon}-{\{\bf I_--P_-\}}f^{m,\epsilon}\right\rangle\right\|^2\\ \nonumber
   &&+\sum_{j_1+j_2\geq m,\atop 0<j_1,j_2<m}\mathcal{E}_{f^{j_1,\epsilon},2,l_m,-\gamma}(t)\mathcal{D}_{f^{j_2,\epsilon},3,l_m+2-\frac{1}\gamma,-\gamma}(t)\nonumber\\
   &&+
  \mathcal{E}_{f^{m,\epsilon},2,l_m,-\gamma}(t)\left\{\mathcal{D}_{f^{P,\epsilon},3,l_m+2-\frac{1}\gamma,-\gamma}(t)+\sum_{i=1}^{m-1}\mathcal{D}_{f^{i,\epsilon},3,l_m+2-\frac{1}\gamma,-\gamma}(t)\right\}\nonumber\\
 &&+\left\{\mathcal{E}_{f^{P,\epsilon},2,l_m,-\gamma}(t)+\sum_{1\leq j_1\leq m-1}\mathcal{E}_{f^{j_1,\epsilon},2,l_m,-\gamma}(t)\right\}\mathcal{D}_{f^{m,\epsilon},2,l_m,-\gamma}(t)\nonumber\\
 &&+\mathcal{E}_{f^{m,\epsilon},2,l_m,-\gamma}(t)\mathcal{D}_{f^{m,\epsilon},2,l_m,-\gamma}(t)+\eta\|f^{m,\epsilon}\|_\nu^2+\mathcal{D}_{f^{m,\epsilon},2}(t)+\eta\mathcal{D}_{f^{m,\epsilon},2,l_m,-\gamma}(t)\nonumber\\
&&+ \sum_{1\leq j_1\leq m-1}\left\|  [E^{P,\epsilon},E^{j_1,\epsilon},\epsilon^m E^{m,\epsilon}]\right\|_{L^\infty_x}\left\|w_{l_m,-\gamma}{\bf\{I-P\}} f^{m,\epsilon}\langle v\rangle^{\frac12}\right\|^2.\nonumber
\end{eqnarray}

Secondly, applying $\partial^\alpha$ with $1\leq|\alpha|\leq N_m$ to \eqref{f-R-vector},
multiplying the resulting identity by $w^2_{l_m,-\gamma}\partial^\alpha f^{m,\epsilon}$,  and integrating the final result with respect to $v$ and $x$ over $\mathbb{R}^3_v\times\mathbb{R}^3_x$, one has

\begin{eqnarray}\label{0-weight-alpha-R}
&&  \frac{d}{dt}\|w_{l_m,-\gamma}\partial^\alpha f^{m,\epsilon}\|^2+\|w_{l_m,-\gamma}\partial^\alpha f^{m,\epsilon}\|^2_{\nu}+\frac{q\vartheta}{(1+t)^{1+\vartheta}}\|w_{l_m,-\gamma}\partial^\alpha f^{m,\epsilon}\langle v\rangle\|^2\nonumber\\
&\lesssim&\|\partial^\alpha f^{m,\epsilon}\|_\nu^2+\underbrace{\left|\left(\partial^\alpha\left\{E^{m,\epsilon} \cdot v \mu^{1/2}q_1\right\},w^2_{l_m,-\gamma}\partial^\alpha f^{m,\epsilon}\right)\right|
}_{\mathcal{S}_6}\nonumber\\
&&+\underbrace{\left|\left(\partial^\alpha I_{mix}(t),w^2_{l_m,-\gamma}\partial^\alpha f^{m,\epsilon}\right)\right|}_{\mathcal{S}_7}+\underbrace{\left|\left(\partial^\alpha I_{pure}(t),w^2_{l_m,-\gamma}\partial^\alpha f^{m,\epsilon}\right)\right|}_{\mathcal{S}_8}\nonumber\\
&&+\underbrace{\left|\left(\partial^\alpha\epsilon q_0\left\{v\times \left\{B^{P}
+\sum_{j_1=1}^{m-1}\epsilon^{j_1}B^{j_1}\right\}\right\}\cdot\nabla_vf^{m,\epsilon},w^2_{l_m,-\gamma}\partial^\alpha f^{m,\epsilon}\right)\right|}_{\mathcal{S}_9},
\end{eqnarray}
where
\begin{eqnarray*}\label{def-J-mix-R}
  {I}_{mix}(t)\equiv
 && \frac12q_0 E^{m,\epsilon}\cdot v f^{P,\epsilon}-q_0  E^{m,\epsilon}\cdot\nabla_vf^{P,\epsilon}\nonumber\\
&&+\frac12q_0\sum_{0<j_1< m} \epsilon^{j_1}E^{m,\epsilon}
\cdot vf^{j_1,\epsilon}-
q_0\sum_{0<j_1< m} \epsilon^{j_1}E^{m,\epsilon}\cdot\nabla_vf^{j_1,\epsilon}\\
&&-
q_0\epsilon \left(v\times B^{m,\epsilon}\right)\cdot \nabla_v\left\{f^{P,\epsilon}+\sum_{i=1}^{m-1}\epsilon^if^{i,\epsilon}\right\}+
\frac12q_0\sum_{j_1+j_2\geq m,\atop 0<j_1,j_2<m} \epsilon^{j_1+j_2-m}E^{j_1,\epsilon}\cdot vf^{j_2,\epsilon}\nonumber\\
&&-q_0\sum_{j_1+j_2\geq m,\atop 0<j_1,j_2<m} \epsilon^{j_1+j_2-m}E^{j_1,\epsilon}\cdot \nabla_vf^{j_2,\epsilon}+{\Gamma}\left(f^{P,\epsilon},f^{m,\epsilon}\right)
  +{\Gamma}\left(f^{m,\epsilon},f^{P,\epsilon}\right)\nonumber\\
&&+\sum_{j_1+j_2\geq m,\atop 0<j_1,j_2<m}\epsilon^{j_1+j_2-m}\Gamma\left(f^{j_1,\epsilon},f^{j_2,\epsilon}\right)
  +\sum_{1\leq j_1\leq m-1}\epsilon^{j_1}\left\{\Gamma\left(f^{j_1,\epsilon},f^{m,\epsilon}\right)+\Gamma\left(f^{m,\epsilon},f^{j_1,\epsilon}\right)\right\}\nonumber
\end{eqnarray*}
and
\begin{eqnarray}\label{def-J-pure-R}
  I_{pure}(t)&\equiv&
\frac12q_0 E^{P,\epsilon}\cdot v f^{m,\epsilon}-q_0 E^{P,\epsilon}\cdot\nabla_{ v  }f^{m,\epsilon}\nonumber\\
&&+\frac12q_0\sum_{0<j_1< m} \epsilon^{j_1}E^{j_1,\epsilon}
\cdot vf^{m,\epsilon}-q_0\sum_{0<j_1< m} \epsilon^{j_1}E^{j_1,\epsilon}
\cdot\nabla_vf^{m,\epsilon}\nonumber\\
&&+ \frac {\epsilon^{m}} 2q_0 E^{m,\epsilon}\cdot v f^{m,\epsilon}- q_0\epsilon^{m}\left\{ E^{m,\epsilon}+ \epsilon v\times B^{m,\epsilon}\right\}\cdot\nabla_vf^{m,\epsilon}+\epsilon^m\Gamma(f^{m,\epsilon},f^{m,\epsilon}).
\end{eqnarray}
Thus we can written \eqref{f-R-vector} as
\begin{eqnarray}\label{mic-R-equation}
  &&\partial_tf^{m,\epsilon}+ v\cdot\nabla_xf^{m,\epsilon}-  E^{m,\epsilon} \cdot v \mu^{1/2}q_1\nonumber\\
  &&\quad\quad\quad\quad\quad\quad+\epsilon q_0\left\{v\times \left\{B^{P}
+\sum_{j_1=1}^{m-1}\epsilon^{j_1}B^{j_1}\right\}\right\}\cdot\nabla_vf^{m,\epsilon}+Lf^{m,\epsilon}=
  I_{mix}(t)+I_{pure}(t).
\end{eqnarray}
Using \eqref{E-R-nonhard-2} yields
\begin{eqnarray}
\mathcal{S}_6&\lesssim&\frac1{4\tilde{C}}\frac{d}{dt}\sum_{i=1}^3\left\|\left\langle \left\{w^2_{l_m,-\gamma}\left[v_i \mu^{1/2}\right]\right\}, \partial^\alpha  f_{m,+}-\partial^\alpha f_{m,-}\right\rangle\right\|^2+\left\|\nabla^{|\alpha|}f^{m,\epsilon}\right\|_\nu^2\\ \nonumber
   &&+\left\|\nabla^{|\alpha|+1}f^{m,\epsilon}\right\|^2_\nu
  +\mathcal{E}_{f^{m,\epsilon},N}(t)\left\{\mathcal{D}_{f^{P,\epsilon},N}(t)+
  \sum_{1\leq j_1\leq m-1}\mathcal{D}_{f^{j_1,\epsilon},N}(t)+\mathcal{D}_{f^{m,\epsilon},N}(t)\right\}\\ \nonumber
  &&+\sum_{j_1+j_2\geq m,\atop 0<j_1,j_2<m}\mathcal{E}_{f^{j_1,\epsilon},N}(t)\mathcal{D}_{f^{j_2,\epsilon},N}(t)
  +\sum_{1\leq j_1\leq m-1}\mathcal{E}_{f^{j_1,\epsilon},N}(t)\mathcal{D}_{f^{m,\epsilon},N}(t).
\end{eqnarray}
For $ \mathcal{S}_7$, we can get by applying \eqref{lemma-f-R-2}, \eqref{i-order-E-R-B}, \eqref{i-order-E-R-12}, and \eqref{i-weight-gamma-low-R} that
\begin{eqnarray}
  \mathcal{S}_7
 &\lesssim&
  \mathcal{E}_{f^{m,\epsilon},N_m,l_m,-\gamma}(t)\left\{\mathcal{D}_{f^{P,\epsilon},N_m+1, l_m+2-\frac{1}\gamma,-\gamma}(t)+\sum_{i=1}^{m-1}\mathcal{D}_{f^{i,\epsilon},N_m+1,l_m+2-\frac{1}\gamma,-\gamma}(t)\right\}\nonumber\\
  &&+\left\{\mathcal{E}_{f^{P,\epsilon},N_m,l_m,-\gamma}(t)+\sum_{1\leq j_1\leq m-1}\mathcal{E}_{f^{j_1,\epsilon},N_m,l_m,-\gamma}(t)\right\}\mathcal{D}_{f^{m,\epsilon},N_m,l_m,-\gamma}(t)\nonumber\\
    &&+\sum_{j_1+j_2\geq m,\atop 0<j_1,j_2<m}\mathcal{E}_{f^{j_1,\epsilon},N_m,l_m,-\gamma}(t)\mathcal{D}_{f^{j_2,\epsilon},N_m+1,l_m+2-\frac{1}\gamma,-\gamma}(t)
   +\eta\left\|w_{l_m,-\gamma}\partial^\alpha f^{m,\epsilon}\right\|_\nu^2.\nonumber
  \end{eqnarray}
For $\mathcal{S}_8$, one has from \eqref{1-w-typical-low-R}, \eqref{1-w-typical-low-1-R} ,\eqref{1-w-typical-low-non-R}, \eqref{1-w-typical-low-non-1-R}, and \eqref{1-w-typical-low-non-gamma-R} that
\begin{eqnarray}
\mathcal{S}_8
&\lesssim&\sum_{1\leq j_1\leq m-1}\left\| [E^{P,\epsilon},E^{j_1,\epsilon},\epsilon^m E^{m,\epsilon}]\right\|_{L^\infty_x}\left\|w_{l_m,-\gamma}\partial^\alpha f^{m,\epsilon}\langle v\rangle^{\frac12}\right\|^2\nonumber\\
 &&+\left\{\sum_{1\leq j_1\leq m-1}\left\|\nabla_x[E^{P,\epsilon},E^{j_1,\epsilon}]\right\|_{H^{N_m-1}_x}^{2}+\epsilon^{2m}\left\|\nabla_x
[E^{m,\epsilon},\epsilon B^{m,\epsilon}]\right\|^2_{H^{N_m^0-1}_x}
\right\}^{\frac{1}{\theta}}\nonumber\\
&&\times\sum_{0\leq k\leq 1,\atop k\leq j\leq |\alpha|}\mathcal{D}^{(j,k)}_{f^{m,\epsilon},j,\widetilde{\ell}_m-\gamma l_m+1+\frac\gamma2,1}(t)\nonumber\\
&&+\left\{\mathcal{E}_{f^{P,\epsilon},N_m}(t)+\sum_{1\leq j_1\leq m-1}\mathcal{E}_{f^{j_1,\epsilon},N_m}(t)\right\}
\mathcal{D}_{f^{m,\epsilon},N_m}(t)\nonumber\\
&&+\chi_{|\alpha|\geq N_m^0+1}\mathcal{E}_{f^{m,\epsilon},N_m}(t)\mathcal{E}^1_{f^{m,\epsilon},N^0_m,l_m +\frac32-\frac1\gamma,1}(t)\\
&&+\mathcal{E}_{f^{m,\epsilon},N_m,l_m,-\gamma}(t)\mathcal{D}_{f^{m,\epsilon},N_m,l_m,-\gamma}(t)
  +\eta\mathcal{D}_{f^{m,\epsilon},N_m,l_m,-\gamma}(t).\nonumber
  \end{eqnarray}
Finally, one has by integrations by parts that
\begin{eqnarray}
  \mathcal{S}_9=\left|\left(\epsilon q_0\left\{v\times \left\{B^{P}
+\sum_{j_1=1}^{m-1}\epsilon^{j_1}B^{j_1}\right\}\right\}\cdot\nabla_v\partial^\alpha f^{m,\epsilon},w^2_{l_m,-\gamma}\partial^\alpha f^{m,\epsilon}\right)\right|=0.
\end{eqnarray}
Plugging the estimates on $\mathcal{S}_6-\mathcal{S}_9$ into \eqref{0-weight-alpha-R} yields that
\begin{eqnarray}\label{0-weight-alpha-1}
&&  \frac{d}{dt}\|w_{l_m,-\gamma}\partial^\alpha f^{m,\epsilon}\|^2+\|w_{l_m,-\gamma}\partial^\alpha f^{m,\epsilon}\|^2_{\nu}+\frac{q\vartheta}{(1+t)^{1+\vartheta}}\|w_{l_m,-\gamma}\partial^\alpha f^{m,\epsilon}\langle v\rangle\|^2\nonumber\\
&\lesssim&\frac1{4\tilde{C}}\frac{d}{dt}\sum_{i=1}^3\left\|\left\langle \left\{w^2_{l_m,-\gamma}\left[v_i \mu^{1/2}\right]\right\}, \partial^\alpha  f_{m,+}-\partial^\alpha f_{m,-}\right\rangle\right\|^2\nonumber\\
 &&+\sum_{1\leq j_1\leq m-1}\left\| \left[E^{P,\epsilon},E^{j_1,\epsilon},\epsilon^m E^{m,\epsilon}\right]\right\|_{L^\infty_x}\left\|w_{l_m,-\gamma}\partial^\alpha f^{m,\epsilon}\langle v\rangle^{\frac12}\right\|^2\nonumber\\
 &&+\left\{\sum_{1\leq j_1\leq m-1}\left\|\nabla_x[E^{P,\epsilon},E^{j_1,\epsilon}]\right\|_{H^{N_m-1}_x}^{2}+\epsilon^{2m}\left\|\nabla_x
[E^{m,\epsilon},\epsilon B^{m,\epsilon}]\right\|^2_{H^{N_m^0-1}_x}
\right\}^{\frac{1}{\theta}}\nonumber\\
&&\times\sum_{0\leq k\leq 1,\atop k\leq j\leq |\alpha|}\mathcal{D}^{(j,k)}_{f^{m,\epsilon},j,\widetilde{\ell}_m-\gamma l_m+1+\frac\gamma2,1}(t)\nonumber\\
&&+
  \mathcal{E}_{f^{m,\epsilon},N_m,l_m,-\gamma}(t)\left\{\mathcal{D}_{f^{P,\epsilon},N_m+1, l_m+2-\frac{1}\gamma,-\gamma}(t)+\sum_{i=1}^{m-1}\mathcal{D}_{f^{i,\epsilon},N_m+1,l_m+2-\frac{1}\gamma,-\gamma}(t)\right\}\nonumber\\
  &&+\left\{\mathcal{E}_{f^{P,\epsilon},N_m,l_m,-\gamma}(t)+\sum_{1\leq j_1\leq m-1}\mathcal{E}_{f^{j_1,\epsilon},N_m,l_m,-\gamma}(t)\right\}\mathcal{D}_{f^{m,\epsilon},N_m,l_m,-\gamma}(t)\nonumber\\
    &&+\sum_{j_1+j_2\geq m,\atop 0<j_1,j_2<m}\mathcal{E}_{f^{j_1,\epsilon},N_m,l_m,-\gamma}(t)\mathcal{D}_{f^{j_2,\epsilon},N_m+1,l_m+2-\frac{1}\gamma,-\gamma}(t)
   \nonumber\\
  &&+\chi_{|\alpha|\geq N_m^0+1}\mathcal{E}_{f^{m,\epsilon},N_m}(t)\mathcal{E}^1_{f^{m,\epsilon},N^0_m,l_m +\frac32-\frac1\gamma,1}(t)+\mathcal{E}_{f^{m,\epsilon},N_m,l_m,-\gamma}(t)\mathcal{D}_{f^{m,\epsilon},N_m,l_m,-\gamma}(t)\nonumber\\
  &&
  +\left\|\nabla^{N_m+1}f^{m,\epsilon}\right\|_\nu^2 +\mathcal{D}_{f^{m,\epsilon},N_m}(t)+\eta\mathcal{D}_{f^{m,\epsilon},N_m,l_m,-\gamma}(t)\nonumber.
  \end{eqnarray}
\noindent{\bf Step 3:} Applying $\partial^\alpha_\beta$ with $|\alpha|+|\beta|\leq N_m,\ |\beta|\geq1$ to \eqref{mic-I-P-equation},
multiplying the resulting identity by $w^2_{l_m-|\beta|,-\gamma}$ $\partial^\alpha_\beta {\bf\{I-P\}}f^{m,\epsilon}$, and integrating the final result with respect to $v$ and $x$ over $\mathbb{R}^3_v\times\mathbb{R}^3_x$, one has
\begin{eqnarray}\label{R-weight-I-P}
&&  \frac{d}{dt}\|w_{l_m-|\beta|,-\gamma}\partial^\alpha_\beta{\bf \{I-P\}}f^{m,\epsilon}\|^2+\|w_{l_m-|\beta|,-\gamma}\partial^\alpha_\beta{\bf \{I-P\}}f^{m,\epsilon}\|^2_{\nu}\nonumber\\
&&+\frac{q\vartheta}{(1+t)^{1+\vartheta}}\|w_{l_m-|\beta|,-\gamma}\partial^\alpha_\beta{\bf \{I-P\}}f^{m,\epsilon}\langle v\rangle\|^2\nonumber\\
&\lesssim&\eta\|w_{l_m,-\gamma}\partial^\alpha{\bf\{I-P\}}f^{m,\epsilon}\|_\nu^2+\|\partial^\alpha{\bf\{I-P\}}f^{m,\epsilon}\|_\nu^2\nonumber\\
&&+\underbrace{\left|\left(\partial^\alpha_\beta\{-{\bf \{I-P\}}v\cdot\nabla_xf^{m,\epsilon}\},w^2_{l_m-|\beta|,-\gamma}\partial^\alpha_\beta{\bf \{I-P\}}f^{m,\epsilon}\right)\right|
}_{\mathcal{S}_{10}}\nonumber\\
&&+\underbrace{\left|\left(\partial^\alpha_\beta\left\{ \left\{E^{m,\epsilon}+\epsilon b^{f^{m,\epsilon}}\times \left\{B^{P}
+\sum_{j_1=1}^{m-1}\epsilon^{j_1}B^{j_1}\right\}\right\} \cdot v \mu^{1/2}q_1\right\},w^2_{l_m-|\beta|,-\gamma}\partial^\alpha_\beta{\bf \{I-P\}}f^{m,\epsilon}\right)\right|}_{\mathcal{S}_{11}}\nonumber\\
&&
+\underbrace{\left|\left(\partial^\alpha_\beta I_{mix,mic}(t),w^2_{l_m-|\beta|,-\gamma}\partial^\alpha_\beta{\bf \{I-P\}}f^{m,\epsilon}\right)\right|}_{\mathcal{S}_{12}}\nonumber\\
&&
+\underbrace{\left|\left(\partial^\alpha_\beta I_{pure,mic}(t),w^2_{l_m-|\beta|,-\gamma}\partial^\alpha_\beta{\bf \{I-P\}}f^{m,\epsilon}\right)\right|}_{\mathcal{S}_{13}}\nonumber\\
&&
+\underbrace{\left|\left(\partial^\alpha_\beta \left\{\epsilon q_0\left\{v\times \left\{B^{P}
+\sum_{j_1=1}^{m-1}\epsilon^{j_1}B^{j_1}\right\}\right\}\cdot\nabla_v\{{\bf I-P}\}f^{m,\epsilon}\right\},w^2_{l_m-|\beta|,-\gamma}\partial^\alpha_\beta{\bf \{I-P\}}f^{m,\epsilon}\right)\right|}_{\mathcal{S}_{14}}.
\end{eqnarray}
Here $I_{mix,mic}(t)$ and $I_{pure,mic}(t)$ are defined by \eqref{def-J-mic-mix-R} and \eqref{def-J-mic-pure-R}.

It is straightforward to compute that
\begin{eqnarray}
\mathcal{S}_{10}&\lesssim&\eta\left\|\nabla^{|\alpha|+1}_x f^{m,\epsilon}\right\|_\nu^2+\left\|\partial^{\alpha}{\bf\{I-P\}}f^{m,\epsilon}\right\|_\nu^2\nonumber\\
&&+\left\|w_{l_m,-\gamma}(\alpha+e_j,\beta-e_j)\partial^{\alpha+e_j}_{\beta-e_j}{\bf \{I-P\}}f^{m,\epsilon}\right\|^2_\nu+\eta
\left\|w_{l_m-|\beta|,-\gamma}\partial^{\alpha}_{\beta}{\bf \{I-P\}}f^{m,\epsilon}\right\|^2_\nu.
\end{eqnarray}
Using \eqref{E-R-nonhard-1} yields
\begin{eqnarray*}
  \mathcal{S}_{11} &\lesssim&\frac1{4C_\beta}\sum_{i=1}^3\frac{d}{dt}\left\|\left\langle \partial_{\beta}\left\{w^2_{l_m-|\beta|,-\gamma}\partial_\beta\left[v_i \mu^{1/2}\right]\right\}, \partial^\alpha{\{\bf I_+-P_+\}}f^{m,\epsilon}-\partial^\alpha{\{\bf I_--P_-\}}f^{m,\epsilon}\right\rangle\right\|^2\\ \nonumber
  &&+\left\|\nabla^{|\alpha|}{\bf\{I-P\}}f^{m,\epsilon}\right\|_\nu^2
  +\left\|\nabla^{|\alpha|+1}f^{m,\epsilon}\right\|^2_\nu\nonumber\\
  &&
  +\mathcal{E}_{f^{m,\epsilon},N_m}(t)\left\{\mathcal{D}_{f^{P,\epsilon},N_m}(t)+
  \sum_{1\leq j_1\leq m-1}\mathcal{D}_{f^{j_1,\epsilon},N_m}(t)+\mathcal{D}_{f^{m,\epsilon},N_m}(t)\right\}\\ \nonumber
  &&+\sum_{j_1+j_2\geq m,\atop 0<j_1,j_2<m}\mathcal{E}_{f^{j_1,\epsilon},N_m}(t)\mathcal{D}_{f^{j_2,\epsilon},N_m}(t)
  +\sum_{1\leq j_1\leq m-1}\mathcal{E}_{f^{j_1,\epsilon},N_m}(t)\mathcal{D}_{f^{m,\epsilon},N_m}(t).
\end{eqnarray*}
By employing  \eqref{lemma-f-R-3}, \eqref{micro-weight-E-R}, \eqref{micro-weight-E-R-12}, \eqref{2-w-typical-low-R}, \eqref{2-w-typical-low-1-R} and \eqref{micro-weight-gamma-low-R}, one has
\begin{eqnarray}
   \mathcal{S}_{12}
     &\lesssim&\sum_{j_1+j_2\geq m,\atop 0<j_1,j_2<m}\mathcal{E}_{f^{j_1,\epsilon},N_m,l_m,-\gamma}(t)\mathcal{D}_{f^{j_2,\epsilon},N_m+1, l_m+2-\frac{1}\gamma,-\gamma}(t)\nonumber\\
   &&+
  \mathcal{E}_{f^{m,\epsilon},N_m,l_m,-\gamma}(t)\left\{\mathcal{D}_{f^{P,\epsilon},N_m+1,l_m+2-\frac{1}\gamma,-\gamma}(t)+\sum_{i=1}^{m-1}\mathcal{D}_{f^{i,\epsilon},N_m+1,l_m+2-\frac{1}\gamma,-\gamma}(t)\right\}\nonumber\\
 &&+\left\{\mathcal{E}_{f^{P,\epsilon},N_m,l_m,-\gamma}(t)+\sum_{1\leq j_1\leq m-1}\mathcal{E}_{f^{j_1,\epsilon},N_m,l_m,-\gamma}(t)\right\}\mathcal{D}_{f^{m,\epsilon},N_m,l_m,-\gamma}(t)\\
  && +\eta\left\|w_{l_m-|\beta|,-\gamma}\partial^\alpha_\beta{\bf \{I-P\}} f^{m,\epsilon}\langle v\rangle^{\frac\gamma2}\right\|^2.\nonumber
\end{eqnarray}
As for $\mathcal{S}_{13}$, one has from \eqref{2-w-typical-low-R}, \eqref{2-w-typical-low-1-R}, \eqref{mic-non-I-P-1}, \eqref{mic-non-I-P-2}, and \eqref{2-w-typical-low-non-gamma-R} that
\begin{eqnarray}
  \mathcal{S}_{13}
    &\lesssim&\sum_{1\leq j_1\leq m-1}\left\|  [E^{P,\epsilon},E^{j_1,\epsilon},\epsilon^m E^{m,\epsilon}]\right\|_{L^\infty_x} \left\|w_{l_m-|\beta|,-\gamma}\partial^\alpha_\beta{\bf\{I-P\}} f^{m,\epsilon}\langle v\rangle^{\frac12}\right\|^2\nonumber\\
  &&+\chi_{|\beta|\geq 1}\sum_{|\beta|-1\leq k\leq \min\{|\beta|+1,|\alpha|+|\beta|\},\atop k\leq j\leq |\alpha|+|\beta|}\\ \nonumber
   &&\times\left\{\sum_{1\leq j_1\leq m-1}\left\|\nabla_x[E^{P,\epsilon},E^{j_1,\epsilon}]\right\|_{H^{N_m-1}_x}^{2}+\epsilon^{2m}\left\|\nabla_x
[E^{m,\epsilon},\epsilon B^{m,\epsilon}]\right\|^2_{H^{N_m^0-1}_x}
\right\}^{\frac{1}{\theta}}\nonumber\\
&&\times\mathcal{D}^{(j,k)}_{f^{m,\epsilon},j,\widetilde{\ell}_m-\gamma l_m+1+\frac\gamma2,1}(t)+\left\{\mathcal{E}_{f^{P,\epsilon},N_m}(t)+\sum_{1\leq j_1\leq m-1}\mathcal{E}_{f^{j_1,\epsilon},N_m}(t)\right\}\mathcal{D}_{f^{m,\epsilon},N_m}(t)\nonumber\\
&&+\chi_{|\alpha|+|\beta|\geq N_m^0+1}\mathcal{E}_{f^{m,\epsilon},N_m}(t) \mathcal{E}^1_{f^{m,\epsilon},N^0_m,l_m+\frac32-\frac1\gamma,1}(t)
\nonumber\\
&&+\mathcal{E}_{f^{m,\epsilon},N_m,l_m,-\gamma}(t)\mathcal{D}_{f^{m,\epsilon},N_m,l_m,-\gamma}(t)
  +\eta\mathcal{D}_{f^{m,\epsilon},N_m,l_m,-\gamma}(t)\nonumber.
  \end{eqnarray}
Finally, it is straightforward to compute that
  \begin{eqnarray}
    \mathcal{S}_{14}=0.
  \end{eqnarray}

Collecting the estimates on $\mathcal{S}_{10}-\mathcal{S}_{14}$ into \eqref{R-weight-I-P} yields that
\begin{eqnarray}\label{R-weight-I-P-1}
  &&  \frac{d}{dt}\|w_{l_m-|\beta|,-\gamma}\partial^\alpha_\beta{\bf \{I-P\}}f^{m,\epsilon}\|^2+\|w_{l_m-|\beta|,-\gamma}\partial^\alpha_\beta{\bf \{I-P\}}f^{m,\epsilon}\|^2_{\nu}\nonumber\\
&&+\frac{q\vartheta}{(1+t)^{1+\vartheta}}\|w_{l_m-|\beta|,-\gamma}\partial^\alpha_\beta{\bf \{I-P\}}f^{m,\epsilon}\langle v\rangle\|^2\nonumber\\
  &\lesssim&\frac1{4C_\beta}\sum_{i=1}^3\frac{d}{dt}\left\|\left\langle \partial_{\beta}\left\{w^2_{l_m-|\beta|,-\gamma}\partial_\beta\left[v_i \mu^{1/2}\right]\right\}, \partial^\alpha{\{\bf I_+-P_+\}}f^{m,\epsilon}-\partial^\alpha{\{\bf I_--P_-\}}f^{m,\epsilon}\right\rangle\right\|^2\nonumber\\
    &&+\sum_{1\leq j_1\leq m-1}\left\|  [E^{P,\epsilon},E^{j_1,\epsilon},\epsilon^m E^{m,\epsilon}]\right\|_{L^\infty_x}\left\|w_{l_m-|\beta|,-\gamma}\partial^\alpha_\beta{\bf\{I-P\}} f^{m,\epsilon}\langle v\rangle^{\frac12}\right\|^2\\
  &&+\chi_{|\beta|\geq 1}\left\{\sum_{1\leq j_1\leq m-1}\left\|\nabla_x[E^{P,\epsilon},E^{j_1,\epsilon}]\right\|_{H^{N_m-1}_x}^{2}+\epsilon^{2m}\left\|\nabla_x
[E^{m,\epsilon},\epsilon B^{m,\epsilon}]\right\|^2_{H^{N_m^0-1}_x}
\right\}^{\frac{1}{\theta}}\nonumber\\
&&\times\sum_{|\beta|-1\leq k\leq \min\{|\beta|+1,|\alpha|+|\beta|\},\atop k\leq j\leq |\alpha|+|\beta|}\mathcal{D}^{(j,k)}_{f^{m,\epsilon},j,\widetilde{\ell}_m-\gamma l_m+1+\frac\gamma2,1}(t)\nonumber\\
   &&+\sum_{j_1+j_2\geq m,\atop 0<j_1,j_2<m}\mathcal{E}_{f^{j_1,\epsilon},N_m,l_m,-\gamma}(t)\mathcal{D}_{f^{j_2,\epsilon},N_m+1, l_m+2-\frac{1}\gamma,-\gamma}(t)\nonumber\\
   &&+
  \mathcal{E}_{f^{m,\epsilon},N_m,l_m,-\gamma}(t)\left\{\mathcal{D}_{f^{P,\epsilon},N_m+1,l_m+2-\frac{1}\gamma,-\gamma}(t)+\sum_{i=1}^{m-1}\mathcal{D}_{f^{i,\epsilon},N_m+1,l_m+2-\frac{1}\gamma,-\gamma}(t)\right\}\nonumber\\
 &&+\left\{\mathcal{E}_{f^{P,\epsilon},N_m,l_m,-\gamma}(t)+\sum_{1\leq j_1\leq m-1}\mathcal{E}_{f^{j_1,\epsilon},N_m,l_m,-\gamma}(t)\right\}\mathcal{D}_{f^{m,\epsilon},N_m,l_m,-\gamma}(t) \nonumber\\
   &&+\mathcal{E}_{f^{m,\epsilon},N_m,l_m,-\gamma}(t)\mathcal{D}_{f^{m,\epsilon},N_m,l_m,-\gamma}(t)+\chi_{|\alpha|+|\beta|\geq N_m^0+1}\mathcal{E}_{f^{m,\epsilon},N_m}(t) \mathcal{E}^1_{f^{m,\epsilon},N^0_m,l_m+\frac32-\frac1\gamma,1}(t)
  \nonumber\\
  &&+\eta\|f^{m,\epsilon}\|_\nu^2+\mathcal{D}_{f^{m,\epsilon},N_m}(t) +\eta\mathcal{D}_{f^{m,\epsilon},N_m,l_m,-\gamma}(t)\nonumber.
  \end{eqnarray}

A proper linear combination of \eqref{mac-dis-f-R}, \eqref{f-R-spatial-x}, \eqref{00-weight-I-P-1-R}, \eqref{0-weight-alpha-1}, and \eqref{R-weight-I-P-1} implies that
\begin{eqnarray}
  &&  \frac{d}{dt}\sum_{|\alpha|+|\beta|\leq N_m}\|w_{l_m-|\beta|,-\gamma}\partial^\alpha_\beta f^{m,\epsilon}\|^2+\|\nabla_x{\bf P}f^{m,\epsilon}\|^2_{H^{N_m-1}_x}+\sum_{|\alpha|\leq N_m}\|\partial^\alpha{\bf\{I-P\}}f\|^2_{\nu}\nonumber\\
  &&+\sum_{|\alpha|+|\beta|\leq N_m}\left\{\|w_{l_m-|\beta|,-\gamma}\partial^\alpha_\beta{\bf \{I-P\}}f^{m,\epsilon}\|^2_{\nu}+\frac{q\vartheta}{(1+t)^{1+\vartheta}}\|w_{l_m-|\beta|,-\gamma}\partial^\alpha_\beta{\bf \{I-P\}}f^{m,\epsilon}\langle v\rangle\|^2\right\}\nonumber\\
   &\lesssim&\left\{\sum_{1\leq j_1\leq m-1}\left\|\nabla_x[E^{P,\epsilon},E^{j_1,\epsilon}]\right\|_{H^{N_m-1}_x}^{2}+\epsilon^{2m}\left\|\nabla_x
[E^{m,\epsilon},\epsilon B^{m,\epsilon}]\right\|^2_{H^{N_m^0-1}_x}
\right\}^{\frac{1}{\theta}}\nonumber\\
  &&\times\sum_{0\leq k\leq N_m,\atop k\leq j\leq N_m}\mathcal{D}^{(j,k)}_{f^{m,\epsilon},j,\widetilde{\ell}_m-\gamma l_m+1+\frac\gamma2,1}(t)\nonumber\\
 &&+\sum_{j_1+j_2\geq m,\atop 0<j_1,j_2<m}\mathcal{E}_{f^{j_1,\epsilon},N_m,l_m,-\gamma}(t)\mathcal{D}_{f^{j_2,\epsilon},N_m+1, l_m+2-\frac{1}\gamma,-\gamma}(t)\nonumber\\
   &&+
  \mathcal{E}_{f^{m,\epsilon},N_m,l_m,-\gamma}(t)\left\{\mathcal{D}_{f^{P,\epsilon},N_m+1,l_m+2-\frac{1}\gamma,-\gamma}(t)+\sum_{i=1}^{m-1}\mathcal{D}_{f^{i,\epsilon},N_m+1,l_m+2-\frac{1}\gamma,-\gamma}(t)\right\}\nonumber\\
  &&+\mathcal{E}_{f^{m,\epsilon},N_m}(t)\mathcal{E}^1_{f^{m,\epsilon},N^0_m,l_m+\frac32-\frac1\gamma}(t)
  +\|\nabla^{N_m+1}f^{m,\epsilon}\|_\nu^2,
  \end{eqnarray}
where we used the smallness of \eqref{Assume-f-R-N-R-low}.
Thus one gets \eqref{lemma-f-R-N-R}.
\end{proof}
\begin{remark}
Since $\left\|\nabla^{N_m+1}f^{m,\epsilon}(t)\right\|_\nu^2$  is absent in the dissipation rate functional $\mathcal{D}_{f^{m,\epsilon},N_m,l_m,-\gamma}(t)$, thus we attempt to use $\mathcal{D}_{f^{m,\epsilon},N_m+1}(t)$ to control $\left\|\nabla^{N_m+1}f^{m,\epsilon}(t)\right\|_\nu^2$  in the right hand side of \eqref{lemma-f-R-N-R}.
\end{remark}
\subsubsection{Lyapunov-type inequalities for $\mathcal{E}_{f^{m,\epsilon},N_m+1}(t)$}
For the estimate of $\mathcal{E}_{f^{m,\epsilon},N_m+1}(t)$, we can get that
\begin{lemma}\label{lemma-f-R-high-spatial}
Assume that
\begin{eqnarray}\label{assumption-2}
\max\left\{\sup_{0\leq\tau\leq t}\mathcal{E}_{f^{P,\epsilon},N_m+1,N_m+1,-\gamma}(\tau),\sup_{0\leq\tau\leq t}\mathcal{E}_{f^{j_1,\epsilon},N_m,N_m,-\gamma}(\tau),\sup_{0\leq\tau\leq t}\mathcal{E}_{f^{m,\epsilon},N_m,N_m+\frac12-\frac1\gamma,-\gamma}(\tau)\right\}
\end{eqnarray}
 is sufficiently small, where $1\leq j_1\leq m-1$, then one has
  \begin{eqnarray}\label{lemma-spatial-f-R-N+1-R}
    &&\frac{d}{dt}\mathcal{E}_{f^{m,\epsilon},N_m+1}(t)+\mathcal{D}_{f^{m,\epsilon},N_m+1}(t)\nonumber\\
  &\lesssim&
  \sum_{1\leq j_1\leq m-1}\left\|\nabla_x\left[E^{P,\epsilon}(t), E^{j_1,\epsilon}(t), \epsilon^m E^{m,\epsilon}(t),\epsilon^{m+1} B^{m,\epsilon}(t)\right]\right\|^2_{L^\infty_x} \left\|\nabla_v\nabla^{N_m}_x{\bf\{I-P\}}f^{m,\epsilon}(t)\langle v\rangle^{1-\frac\gamma2}\right\|^2\\
  &&+\sum_{1\leq j_1\leq m-1}\left\|\left[\epsilon ^m E^{m,\epsilon}(t), E^{j_1,\epsilon}(t), E^{P,\epsilon}(t)\right]\right\|^2_{L^\infty_x} \left\|\nabla^{N_m+1}_x{\bf\{I-P\}}f^{m,\epsilon}(t)\langle v\rangle^{1-\frac\gamma2}\right\|^2\nonumber\\
  &&+\sum_{j_1+j_2\geq m,\atop 0<j_1,j_2<m}\left\{\mathcal{E}_{f^{j_1,\epsilon},N_m+1,N_m+1,-\gamma}(t)
  \mathcal{D}_{f^{j_2,\epsilon},N_m+2,N_m+3-\frac1\gamma,-\gamma}(t)\right\}
  \nonumber\\
   &&+\mathcal{E}_{f^{m,\epsilon},N_m,N_m+\frac12-\frac1\gamma,-\gamma}(t)
  \left\{\mathcal{D}_{f^{P,\epsilon},N_m+1}(t)+\sum_{1\leq j_1\leq m-1}\mathcal{D}_{f^{j_1,\epsilon},N_m+1}(t)\right\}\nonumber\\
   &&+
  \mathcal{E}_{f^{m,\epsilon},N_m+1}(t)\left\{\mathcal{D}_{f^{P,\epsilon},N_m+2,N_m+3-\frac1\gamma,-\gamma}(t)
  +\sum_{i=1}^{m-1}\mathcal{D}_{f^{i,\epsilon},N_m+2,N_m+3-\frac1\gamma,-\gamma}(t)
  +\mathcal{D}_{f^{m,\epsilon},3,4-\frac1\gamma,-\gamma}(t)\right\}\nonumber.
    \end{eqnarray}
\end{lemma}
\begin{proof} In fact, if we replace $N_m$ in the estimate \eqref{f-R-spatial-x} by $N_m+1$, one can deduce that
 \begin{eqnarray}\label{spatial-x-f-R-N+1}
   &&\frac{d}{dt}\sum_{|\alpha|\leq N_m+1}\left\|\left[\partial^\alpha f^{m,\epsilon}, \partial^\alpha E^{m,\epsilon},\partial^\alpha B^{m,\epsilon}\right](t)\right\|^2+\sum_{|\alpha|\leq N_m+1}\left\|\partial^\alpha\{{\bf I-P}\}f^{m,\epsilon}(t)\right\|_\nu^2\nonumber\\
   &\lesssim&
  \sum_{j_1+j_2\geq m,\atop 0<j_1,j_2<m}\left\{\mathcal{E}_{f^{j_1,\epsilon},N_m+1,N_m+1,-\gamma}(t)
  \mathcal{D}_{f^{j_2,\epsilon},N_m+2,N_m+3-\frac1\gamma,-\gamma}(t)\right\}
  \nonumber\\
   &&+\mathcal{E}_{f^{m,\epsilon},N_m,N_m+\frac12-\frac1\gamma,-\gamma}(t)
  \left\{\mathcal{D}_{f^{P,\epsilon},N_m+1}(t)+\sum_{1\leq j_1\leq m-1}\mathcal{D}_{f^{j_1,\epsilon},N_m+1}(t)\right\}\nonumber\\
   &&+
  \mathcal{E}_{f^{m,\epsilon},N_m+1}(t)\nonumber\\
  &&\times\left\{\mathcal{D}_{f^{P,\epsilon},N_m+2,N_m+3-\frac1\gamma,-\gamma}(t)
  +\sum_{i=1}^{m-1}\mathcal{D}_{f^{i,\epsilon},N_m+2,N_m+3-\frac1\gamma,-\gamma}(t)
  +\mathcal{D}_{f^{m,\epsilon},3,4-\frac1\gamma,-\gamma}(t)\right\}\nonumber\\
  &&+\left\{\mathcal{E}_{f^{P,\epsilon},N_m+1,N_m+1,-\gamma}(t)+\sum_{1\leq j_1\leq m-1}\mathcal{E}_{f^{j_1,\epsilon},N_m,N_m,-\gamma}(t)
  +\mathcal{E}_{f^{m,\epsilon},N_m,N_m+\frac12-\frac1\gamma,-\gamma}(t)\right\}\nonumber\\
  &&\times
 \mathcal{D}_{f^{m,\epsilon},N_m+1}(t)\nonumber\\
 &&+ \sum_{1\leq j_1\leq m-1}\|\nabla_x[ E^{j_1,\epsilon},E^{P,\epsilon},\epsilon^m E^{m,\epsilon},\epsilon^{m+1} B^{m,\epsilon}]\|^2_{L^\infty_x}\|\nabla_v\nabla^{N_m}_x{\bf\{I-P\}}f^{m,\epsilon}\langle v\rangle^{1-\frac\gamma2}\|^2\\
  &&+\sum_{1\leq j_1\leq m-1}\| [\epsilon ^m E^{m,\epsilon}, E^{j_1,\epsilon},E^{P,\epsilon}]\|^2_{L^\infty_x}\|\nabla^{N_m+1}_x{\bf\{I-P\}}f^{m,\epsilon}\langle v\rangle^{1-\frac\gamma2}\|^2+\eta\mathcal{D}_{f^{m,\epsilon},N_m+1}(t),\nonumber
\end{eqnarray}
where we use the smallness of \eqref{assumption-2}.

We can deduce \eqref{lemma-spatial-f-R-N+1-R} by a proper linear combination of \eqref{spatial-x-f-R-N+1} and \eqref{mac-dis-f-R} with $N=N_m+1$, thus the proof is complete.
  \end{proof}
 \subsection{Lyapunov-type inequalities for $\overline{\mathcal{E}}_{f^{m,\epsilon},N^0_m,\ell,-\gamma}(t) $} To yield the desired Lyapunov-type inequalities for the energy functional $\overline{\mathcal{E}}_{f^{m,\epsilon},N^0_m,\ell,-\gamma}(t)$, we first need first to perform some energy estimates with negative Sobolev space.

  \subsubsection{Energy estimates with negative Sobolev space}
By applying similar argument used to prove Lemma 3.2 in \cite[page 3727]{Lei-Zhao-JFA-2014} and Lemma 3.3 in \cite[page 3731]{Lei-Zhao-JFA-2014}, one has
\begin{lemma}\label{lemma3.3-f-r}
Let $\varrho\in \left[\frac12, \frac32\right)$, there exists an interactive functional $G^{-\varrho}_{f^{m,\epsilon},E^{m,\epsilon},B^{m,\epsilon}}(t)$ satisfying
\begin{equation}\label{G_{E,B}}
G^{-\varrho}_{f^{m,\epsilon},E^{m,\epsilon},B^{m,\epsilon}}(t)\lesssim \left\|\Lambda^{1-\varrho}\left[f^{m,\epsilon}(t), E^{m,\epsilon}(t), B^{m,\epsilon}(t)\right]\right\|^2
+\left\|\Lambda^{-\varrho}\left[f^{m,\epsilon}(t), E^{m,\epsilon}(t), B^{m,\epsilon}(t)\right]\right\|^2
\end{equation}
such that
\begin{eqnarray}\label{EB-s-f-R}
&&\frac{d}{dt}G^{-\varrho}_{f^{m,\epsilon},E^{m,\epsilon},B^{m,\epsilon}}(t)
+\left\|\Lambda^{1-\varrho}\left[{\bf P}f^{m,\epsilon}(t), \epsilon E^{m,\epsilon}(t),\epsilon B^{m,\epsilon}(t)\right]\right\|^2 +\left\|\Lambda^{-\varrho}\left(a^{f^{m,\epsilon}}_+(t) -a^{f^{m,\epsilon}}_-(t)\right)\right\|^2\nonumber\\
&\lesssim&\mathcal{\overline{E}}_{f^{P,\epsilon},2}(t)\mathcal{\overline{D}}_{f^{m,\epsilon},2}(t)+\mathcal{\overline{E}}_{f^{m,\epsilon},2}(t)\mathcal{\overline{D}}_{f^{m,\epsilon},2}(t)
+\sum_{1\leq j_1\leq m-1}\mathcal{\overline{E}}_{f^{j_1,\epsilon},2}(t)\mathcal{\overline{D}}_{f^{m,\epsilon},2}(t)\nonumber\\
&&+\sum_{j_1+j_2\geq m,\atop 0<j_1,j_2<m}\mathcal{\overline{E}}_{f^{j_1,\epsilon},2}(t) \mathcal{\overline{D}}_{f^{j_2,\epsilon},2}(t) +\left\|\Lambda^{-\varrho}\{{\bf I-P}\}f^{m,\epsilon}(t)\right\|_{H^1_xL^2_\nu}^2
\end{eqnarray}
holds for any $0\leq t\leq T$.
\end{lemma}
Moreover, by applying the argument used to prove Lemma 3.1 in \cite[page 3724]{Lei-Zhao-JFA-2014}, one has
\begin{lemma}\label{Lemma4.1-f-R} Under the assumptions stated above, we have that
\begin{eqnarray}\label{Lemma4.1-1-f-R}
&&\frac{d}{dt}\left\|\Lambda^{-\varrho}(f^{m,\epsilon},E^{m,\epsilon},B^{m,\epsilon})\right\|^2+\left\|\Lambda^{-\varrho}\{{\bf I-P}\}f^{m,\epsilon}\right\|_{\nu}^2\nonumber\\
&\lesssim&\frac d{dt}\left\{\frac12\mathfrak{R}\left(\mathcal{F}\left[q_0G_m\cdot vf^{P,\epsilon}\right]\mid |\xi|^{-2{\varrho}}{\mathcal{F}[{\bf P}f^{m,\epsilon}]}\right)-\mathfrak{R}\left(\mathcal{F}[q_0G_m\cdot\nabla_{ v  }f^{P,\epsilon}]\mid |\xi|^{-2{\varrho}}{\mathcal{F}({\bf P}f^{m,\epsilon})}\right)\right\}\nonumber\\
 &&+\mathcal{{E}}_{f^{P,\epsilon},2,3-\frac1\gamma,-\gamma}(t)\mathcal{\overline{D}}_{f^{m,\epsilon},2,3-\frac1\gamma,-\gamma}(t)+\mathcal{\overline{E}}_{f^{m,\epsilon},2,3-\frac1\gamma,-\gamma}(t)\mathcal{\overline{D}}_{f^{P,\epsilon},2,3-\frac1\gamma,-\gamma}(t)
\\
&&+\sum_{1\leq j_1\leq m-1}\left\{\mathcal{\overline{E}}_{f^{j_1,\epsilon},2,3-\frac1\gamma,-\gamma}(t)\mathcal{\overline{D}}_{f^{m,\epsilon},2,3-\frac1\gamma,-\gamma}(t)+
\mathcal{\overline{E}}_{f^{m,\epsilon},2,3-\frac1\gamma,-\gamma}(t)\mathcal{\overline{D}}_{f^{j_1,\epsilon},2,3-\frac1\gamma,-\gamma}(t)\right\}\nonumber\\
&&+\sum_{j_1+j_2\geq m,\atop 0<j_1,j_2<m}\mathcal{\overline{E}}_{f^{j_1,\epsilon},2,3-\frac1\gamma,-\gamma}(t)\mathcal{\overline{D}}_{f^{j_2,\epsilon},2,3-\frac1\gamma,-\gamma}(t)
+\mathcal{\overline{E}}_{f^{m,\epsilon},2,3-\frac1\gamma,-\gamma}(t) \mathcal{\overline{D}}_{f^{m,\epsilon},2,3-\frac1\gamma,-\gamma}(t).\nonumber
\end{eqnarray}
\end{lemma}
\subsubsection{Lyapunov-type inequalities for $\overline{\mathcal{E}}_{f^{m,\epsilon},N^0_m,\ell,-\gamma}(t) $}

By employing the argument used to prove Lemma \ref{lemma-f-m} and by applying \eqref{EB-s-f-R} and \eqref{Lemma4.1-1-f-R}, we can get that

\begin{lemma}\label{lemma-f-m-low}
Assume that
\begin{eqnarray}\label{Assume-f-R-N-R-0-low}
 \max&&\left\{\sup_{0\leq\tau\leq t}(1+\tau)^{1+\vartheta}\left\|\left[\epsilon E^{m,\epsilon},E^{j_1,\epsilon},E^{P,\epsilon}\right](\tau)\right\|_{L^\infty_x},  \mathcal{\overline{E}}_{f^{P,\epsilon},N^0_m+1,l^0_m+\frac32-\frac1\gamma,-\gamma}(\tau), \right.\nonumber\\
 &&\ \ \ \ \ \ \ \ \ \ \ \ \ \ \ \ \ \ \ \ \ \ \ \ \ \ \ \ \
 \left.\sup_{0\leq\tau\leq t}\sum_{0<j<m}\mathcal{\overline{E}}_{f_{j},N^0_m+1,l^0_m+\frac32-\frac1\gamma,-\gamma}(\tau),
 \sup_{0\leq\tau\leq t}\mathcal{\overline{E}}_{f^{m,\epsilon},N^0_m,l^0_m,-\gamma}(\tau)\right\}
\end{eqnarray}
is sufficiently small, one has
\begin{eqnarray}\label{lemma-f-m-N-R-0-low-1}
&&\frac{d}{dt}\mathcal{\overline{E}}_{f^{m,\epsilon},N^0_m,l^0_m,-\gamma}(t)
+\mathcal{\overline{D}}_{f^{m,\epsilon},N^0_m,l^0_m,-\gamma}(t)\nonumber\\
 &\lesssim&\left\{\sum_{1\leq j_1\leq m-1}\left\|\nabla_x\left[E^{P,\epsilon}(t), E^{j_1,\epsilon}(t)\right]\right\|_{H^{N_m-1}_x}^{2}+\epsilon^{2m}\left\|\nabla_x
\left[E^{m,\epsilon}(t), \epsilon B^{m,\epsilon}(t)\right]\right\|^2_{H^{N_m^0-1}_x}
\right\}^{\frac{1}{\theta}}\nonumber\\
  &&\times\sum_{0\leq k\leq N^0_m,\atop k\leq j\leq N^0_m}\mathcal{D}^{(j,k)}_{f^{m,\epsilon},j,\widetilde{\ell}_m-\gamma l^0_m+1+\frac\gamma2,1}(t)\\
   &&+
  \overline{\mathcal{E}}_{f^{m,\epsilon},N^0_m,l^0_m,-\gamma}(t)\left\{\overline{\mathcal{D}}_{f^{P,\epsilon},N^0_m+1,l^0_m+2-\frac{1}\gamma,-\gamma}(t)+\sum_{i=1}^{m-1}\overline{\mathcal{D}}_{f^{i,\epsilon},N^0_m+1,l^0_m+2-\frac{1}\gamma,-\gamma}(t)\right\}\nonumber\\
  &&+\sum_{j_1+j_2\geq m,\atop 0<j_1,j_2<m}\overline{\mathcal{E}}_{f^{j_1,\epsilon},N^0_m,l^0_m,-\gamma}(t)\overline{\mathcal{D}}_{f^{j_2,\epsilon},N^0_m+1, l^0_m+2-\frac{1}\gamma,-\gamma}(t) +\left\|\nabla^{N^0_m+1}f^{m,\epsilon}(t)\right\|_\nu^2.\nonumber
 \end{eqnarray}
\end{lemma}

\subsection{Temporal decay rates}
To control
\begin{equation}\label{gamma-weight-1}
\begin{split}
&\left\{\sum_{1\leq j_1\leq m-1}\left\|\nabla_x\left[E^{P,\epsilon}(t), E^{j_1,\epsilon}(t)\right]\right\|_{H^{N_m-1}_x}^{2}+\epsilon^{2m}\left\|\nabla_x
\left[E^{m,\epsilon}(t),\epsilon B^{m,\epsilon}(t)\right]\right\|^2_{H^{N_m^0-1}_x}
\right\}^{\frac{1}{\theta}}\\
&\times\sum_{0\leq k\leq N^0_m,\atop k\leq j\leq N^0_m}\mathcal{D}^{(j,k)}_{f^{m,\epsilon},j,\widetilde{\ell}_m-\gamma l^0_m+1+\frac\gamma2,1}(t),
\end{split}
\end{equation}
\begin{equation}\label{gamma-weight-2}
\begin{split}
&\left\{\sum_{1\leq j_1\leq m-1}\left\|\nabla_x\left[E^{P,\epsilon}(t), E^{j_1,\epsilon}(t)\right]\right\|_{H^{N_m-1}_x}^{2}+\epsilon^{2m}\left\|\nabla_x
\left[E^{m,\epsilon}(t),\epsilon B^{m,\epsilon}(t)\right]\right\|^2_{H^{N_m^0-1}_x}
\right\}^{\frac{1}{\theta}}\\
&\times\sum_{0\leq k\leq N_m,\atop k\leq j\leq N_m}\mathcal{D}^{(j,k)}_{f^{m,\epsilon},j,\widetilde{\ell}_m-\gamma l_m+1+\frac\gamma2,1}(t),
\end{split}
\end{equation}
\begin{equation}\label{no-weight-1}
  \sum_{1\leq j_1\leq m-1}\left\|\nabla_x\left[ E^{P,\epsilon}(t), E^{j_1,\epsilon}(t),\epsilon^m E^{m,\epsilon}(t),\epsilon^{m+1} B^{m,\epsilon}(t)\right]\right\|^2_{L^\infty_x} \left\|\nabla_v\nabla^{N_m}_x{\bf\{I-P\}}f^{m,\epsilon}(t)\langle v\rangle^{1-\frac\gamma2}\right\|^2,
\end{equation}
and
\begin{equation}\label{no-weight-2}
\sum_{1\leq j_1\leq m-1}\left\|\left[\epsilon ^m E^{m,\epsilon}(t,x), E^{j_1,\epsilon}(t,x), E^{P,\epsilon}(t,x)\right]\right\|^2_{L^\infty_x} \left\|\nabla^{N_m+1}_x{\bf\{I-P\}}f^{m,\epsilon}(t)\langle v\rangle^{1-\frac\gamma2}\right\|^2
\end{equation}
in the right hand side of \eqref{lemma-f-R-N-R}, \eqref{lemma-f-m-N-R-0-low-1} and \eqref{lemma-spatial-f-R-N+1-R}, we need the time decay rates property of
 \[\sum_{1\leq j_1\leq m-1}\left\|\left[E^{P,\epsilon}(t), E^{j_1,\epsilon}(t)\right]\right\|_{H^{N_m}_x}^{2},\ \ \left\|\left[E^{m,\epsilon}(t), B^{m,\epsilon}(t)\right]\right\|^2_{H^{N_m^0}_x}.\]
For result in this direction, we have
  \begin{lemma}\label{Lemma1-R}
Assume
\begin{itemize}
  \item $N_m^0\geq 5$ and $N_m+1=2N_m^0-1+[\varrho]$;
  \item $N_{m-1}\geq N_m+2$, $N_j\geq N_{j+1}+1$ for $0<j<m-1$, $N_P\geq N_1+1$;
  \item there exist $\widehat{l}_m>\frac{N_P+\varrho}2$, and
  $\widehat{l}_{m-1}\geq \widehat{l}_m, \widehat{l}_{j-1}\geq \widehat{l}_j$ for $2\leq j\leq m-1$, $\widehat{l}^{P}\geq \widehat{l}_1$
\end{itemize}
 such that if
\begin{eqnarray}\label{def-Xi-R}
&&\Xi_m(t)\nonumber\\
&\equiv&\max\left\{
\sup_{0\leq\tau\leq t}\mathcal{\overline{E}}_{f^{m,\epsilon},N^0_m,N^0_m+\widehat{l}_m,-\gamma}(\tau),\sup_{0\leq\tau\leq t}\mathcal{E}_{f^{m,\epsilon},N_m,N_m,-\gamma}(\tau),\sup_{0\leq\tau\leq t}\mathcal{E}_{f^{m,\epsilon},N_m+1}(\tau),\right.\nonumber\\
&&\left.\ \ \ \ \ \ \ \ \ \sup_{0\leq\tau\leq t}\mathcal{\overline{E}}_{f^{P,\epsilon},N_P,N_P+\widehat{l}_P,-\gamma}(\tau),\sum_{1\leq j\leq m-1}\sup_{0\leq\tau\leq t}\mathcal{\overline{E}}_{f^{j,\epsilon},N_j,N_j+\widehat{l}_{j},-\gamma}(\tau)
\right\}
\end{eqnarray}
is sufficiently small,
then one has for $k=0,1,2,\cdots, N^0_m-2$ that
\begin{eqnarray}\label{lemma-decay-f-R}
\mathcal{E}^{k}_{f^{m,\epsilon},N^0_m}(t)+\mathcal{E}^{k}_{f^{P,\epsilon},N_P}(t)+\sum_{1\leq j\leq m-1}\mathcal{E}^{k}_{f^{j,\epsilon},N_j}(t)
\lesssim\epsilon^{-2(k+\varrho)}(1+t)^{-(k+\varrho)}\Xi_{R}(t)
\end{eqnarray}
holds for $0\leq t\leq T$.
\end{lemma}
\begin{proof} First notice that under the smallness assumption imposed on $\Xi_m(t)$, by applying similar strategy as Lemma 4.1 in \cite[page 3745]{Lei-Zhao-JFA-2014},  one can deduce that
\begin{eqnarray}\label{Lemma1-1-R-1}
\frac{d}{dt}\left\{\mathcal{E}^{k}_{f^{P,\epsilon},N_P}(t)+\sum_{1\leq j\leq m-1}\mathcal{E}^{k}_{f^{j,\epsilon},N_j}(t)\right\}+\mathcal{D}^{k}_{f^{P,\epsilon},N_P}(t)+\sum_{1\leq j\leq m-1}\mathcal{D}^{k}_{f^{j,\epsilon},N_j}(t)\leq 0,\quad 0\leq t\leq T.
\end{eqnarray}

Combing with the above inequality, one has by a similar argument that
\begin{eqnarray}\label{Lemma1-1-R-1}
&&\frac{d}{dt}\left\{\mathcal{E}^{k}_{f^{m,\epsilon},N^0_m}(t)+\mathcal{E}^{k}_{f^{P,\epsilon},N_P}(t)+\sum_{1\leq j\leq m-1}\mathcal{E}^{k}_{f^{j,\epsilon},N_j}(t)\right\}\nonumber\\
&&+\mathcal{D}^{k}_{f^{m,\epsilon},N^0_m}(t)+\mathcal{D}^{k}_{f^{P,\epsilon},N_P}(t)+\sum_{1\leq j\leq m-1}\mathcal{D}^{k}_{f^{j,\epsilon},N_j}(t)\leq 0,\quad 0\leq t\leq T.
\end{eqnarray}

Now we turn to compare the difference between the energy functionals $$\mathcal{E}^{k}_{f^{m,\epsilon},N^0_m}(t)+\mathcal{E}^{k}_{f^{P,\epsilon},N_P}(t)+\sum_{1\leq j\leq m-1}\mathcal{E}^{k}_{f^{j,\epsilon},N_j}(t)$$ and the dissipation functionals $$\mathcal{D}^{k}_{f^{m,\epsilon},N^0_m}(t)+\mathcal{D}^{k}_{f^{P,\epsilon},N_P}(t)+\sum_{1\leq j\leq m-1}\mathcal{D}^{k}_{f^{j,\epsilon},N_j}(t).$$
To this end, since
\[
\left\|\nabla^k\left[
E^{m,\epsilon},B^{m,\epsilon}\right]\right\|\leq \left\{\left\|\nabla^{k+1}\epsilon\left[E^{m,\epsilon}, B^{m,\epsilon}\right]\right\|\right\}^{\frac{k+\varrho}{k+\varrho+1}}
\left\{\epsilon^{-k-\varrho}\left\|\Lambda^{-\varrho}\left[E^{m,\epsilon},B^{m,\epsilon}\right] \right\|\right\}^{\frac{1}{k+\varrho+1}},
\]

\[
\left\|\nabla^{N^R_0}[E^{m,\epsilon},B^{m,\epsilon}]\right\|\lesssim\left\{\epsilon\left\|\nabla^{N^R_0-1}[E^{m,\epsilon},B^{m,\epsilon}]\right\|
\right\}^\frac{k+\varrho}{k+\varrho+1}
\left\{\epsilon^{-k-\varrho}\left\|\nabla^{N^R_0+k+\varrho}[E^{m,\epsilon},B^{m,\epsilon}]\right\|\right\}^\frac{1}{k+\varrho+1},
\]
and for $0<j<m$
\begin{eqnarray*}
 &&\left\|\nabla^k\left[{\bf P}f^{m,\epsilon},{\bf P}f^{P,\epsilon},{\bf P}f^{j,\epsilon},E^{P,\epsilon},E^{j,\epsilon}\right]\right\|\nonumber\\
 &\leq& \left\{\epsilon\left\|\nabla^{k+1}\left[{\bf P}f^{m,\epsilon},{\bf P}f^{P,\epsilon},{\bf P}f^{j,\epsilon},E^{P,\epsilon},E^{j,\epsilon}\right]\right\|\right\}^{\frac{k+\varrho}{k+\varrho+1}} \left\{\epsilon^{-k-\varrho}
\left\|\Lambda^{-\varrho}\left[{\bf P}f^{m,\epsilon},{\bf P}f^{P,\epsilon},{\bf P}f^{j,\epsilon},E^{P,\epsilon},E^{j,\epsilon}\right]\right\|\right\}^{\frac{1}{k+\varrho+1}}
\end{eqnarray*}
holds for the macroscopic component, while for the microscopic component, we have by employing the H$\ddot{o}$lder inequality that
\begin{eqnarray}
&&\left\|\partial^\alpha {\bf\{I-P\}}\left[f^{P,\epsilon},f^{j,\epsilon}, f^{m,\epsilon}\right]\right\|\nonumber\\
&\leq & \left\|\partial^\alpha {\bf\{I-P\}}\left[f^{P,\epsilon},f^{j,\epsilon}, f^{m,\epsilon}\right]\langle v\rangle^{\frac{\gamma}2}\right\|^{\frac{k+\varrho}{k+\varrho+1}}
\left\|\partial^\alpha {\bf\{I-P\}}\left[f^{P,\epsilon},f^{j,\epsilon}, f^{m,\epsilon}\right]\langle v\rangle^{-\frac{\gamma(k+\varrho)}{2}}\right\|^{\frac{1}{k+\varrho+1}}\\
&\leq&\left\|\partial^\alpha {\bf\{I-P\}}\left[f^{P,\epsilon},f^{j,\epsilon}, f^{m,\epsilon}\right]\right\|_\nu^{\frac{k+\varrho}{k+\varrho+1}}
\left\|w_{\frac{k+\varrho}{2},-\gamma}\partial^\alpha {\bf\{I-P\}}\left[f^{P,\epsilon},f^{j,\epsilon}, f^{m,\epsilon}\right]\right\|^{\frac{1}{k+\varrho+1}}.\nonumber
\end{eqnarray}
Therefore, we arrive at by collecting the interpolation estimates that
\begin{equation*}
\begin{aligned}
&\mathcal{E}^{k}_{f^{m,\epsilon},N^0_m}(t)+\mathcal{E}^{k}_{f^{P,\epsilon},N_P}(t)+\sum_{1\leq j\leq m-1}\mathcal{E}^{k}_{f^{j,\epsilon},N_j}(t)\\
\leq &\left\{\mathcal{D}^{k}_{f^{m,\epsilon},N^0_m}(t)+\mathcal{D}^{k}_{f^{P,\epsilon},N_P}(t)+\sum_{1\leq j\leq m-1}\mathcal{D}^{k}_{f^{j,\epsilon},N_j}(t)\right\}^\frac{k+\varrho}{k+\varrho+1}\times\left\{\Xi_{R}(t)\right\}^\frac{1}{k+\varrho+1},
\end{aligned}
\end{equation*}
which combing with \eqref{Lemma1-1-R-1} yields that
\begin{equation*}
\begin{split}\
&\frac{d}{dt}\left\{\mathcal{E}^{k}_{f^{m,\epsilon},N^0_m}(t)+\mathcal{E}^{k}_{f^{P,\epsilon},N_P}(t)+\sum_{1\leq j\leq m-1}\mathcal{E}^{k}_{f^{j,\epsilon},N_j}(t)\right\}\\
&+\left\{\mathcal{E}^{k}_{f^{m,\epsilon},N^0_m}(t)+\mathcal{E}^{k}_{f^{P,\epsilon},N_P}(t)+\sum_{1\leq j\leq m-1}\mathcal{E}^{k}_{f^{j,\epsilon},N_j}(t)\right\}^{1+\frac{1}{k+\varrho}}
\times\left\{\Xi_{R}(t)\right\}^{-\frac{1}{k+\varrho}}
\leq 0.
\end{split}
\end{equation*}

Solving the above inequality directly gives \eqref{lemma-decay-f-R},
this completes the proof of Lemma \ref{Lemma1-R}.
\end{proof}

Based on the above lemma and by applying Lemma 3.6 of \cite{Duan-Lei-Yang-Zhao-CMP-2017} with straightforward modification, we can further obtain the following temporal decay rates of $\mathcal{E}^k_{f^{m,\epsilon},N_m^0,\ell_m,-\gamma}(t)$.
\begin{lemma}\label{lemma2-R}
Under the assumption of the above lemma, let $\ell_m\geq N_P$ and
  $\ell_{m-1}\geq \ell_m, \ell_{j-1}\geq \ell_j$ for $2\leq j\leq m-1$, $\ell^{P}\geq \ell_1$, furthermore, if
\begin{eqnarray*}
&&\Xi_{R,weight}(t)\nonumber\\
&\equiv&\max\left\{
\sup_{0\leq\tau\leq t}\mathcal{\overline{E}}_{f^{m,\epsilon},N^0_m,\ell_m+\widehat{l}_m,-\gamma}(\tau),\sup_{0\leq\tau\leq t}\mathcal{E}_{f^{m,\epsilon},N_m,\ell_m,-\gamma}(\tau),\sup_{0\leq\tau\leq t}\mathcal{E}_{f^{m,\epsilon},N_m+1}(\tau),\right.\nonumber\\
&&\left.\sup_{0\leq\tau\leq t}\mathcal{\overline{E}}_{f^{P,\epsilon},N_P,\ell_P+\widehat{l}_P,-\gamma}(\tau),\sum_{1\leq j\leq m-1}\sup_{0\leq\tau\leq t}\mathcal{\overline{E}}_{f^{j,\epsilon},N_j,\ell_j+\widehat{l}_j,-\gamma}(\tau)
\right\}
\end{eqnarray*}
is sufficiently small, then
one has for $k\leq N_m^0-3$,
\begin{equation}\label{lemma3-1-R}
 \begin{split}
\mathcal{E}^k_{f^{m,\epsilon},N^0_m,{\ell_{R}},-\gamma}(t)+\sum_{0<j<m}\mathcal{E}^k_{f^{j,\epsilon},N_j,{\ell_{j}},-\gamma}(t)
+\mathcal{E}^k_{f^{P,\epsilon},N_P,\ell^{P},-\gamma}(t)
\lesssim& \epsilon^{-2(k+\varrho)}(1+t)^{-(k+\varrho)}\Xi_{R,weight}(t).
\end{split}
\end{equation}
\end{lemma}

\subsection{Lyapunov-type inequalities for some weighted estimates with the weight $w_{l_m^*-|\beta|,1}(t,v)$}
Now we are ready to get some Lyapunov-type inequalities for some weighted estimates with the weight $w_{l_m^*-|\beta|,1}(t,v)$. It is worth to pointing out that, unlike the estimates obtained in Lemma
\ref{lemma-f-m-low} which lead to the increase of the weights in the corresponding terms in the right hand side of the estimate \eqref{lemma-f-m-N-R-0-low-1}, the estimates we obtained in the coming two lemmas will not lead to the growth of the weights.
\begin{lemma}\label{lemma-f-i-high}
Assume that
\begin{equation}\label{Assume-f-R-high}
\sup_{0\leq \tau\leq t}\left\{(1+\tau)^{1+\vartheta}\sum_{1\leq j_1\leq m-1}\left\| \left[E^{P,\epsilon}, E^{j_1,\epsilon},\epsilon^m E^{m,\epsilon}\right](\tau)\right\|_{L^\infty_x}\right\}
\end{equation}
is sufficiently small,
then one has
\begin{eqnarray}\label{lemma-f-R-high}
 &&\frac{d}{dt}\left\{\sum_{N_m^0+1\leq n\leq N_m+1,\atop 0\leq|\beta|\leq n }(1+t)^{-\sigma_{n,|\beta|}}\mathcal{E}^{(n,|\beta|)}_{f^{m,\epsilon},n,l_m^*,1}(t)\right\}\nonumber\\
&&  +\sum_{N_m^0+1\leq n\leq N_m+1,\atop 0\leq|\beta|\leq n }\left\{(1+t)^{-\sigma_{n,|\beta|}}\mathcal{E}^{(n,|\beta|)}_{f^{m,\epsilon},n,l_m^*,1}(t)
+(1+t)^{-\sigma_{n,|\beta|}}\mathcal{D}^{(n,|\beta|)}_{f^{m,\epsilon},n,l_m^*,1}(t)\right\}\\
&\lesssim&\sum_{N_m^0+1\leq n\leq N_m+1,\atop |\alpha|+|\beta|=n}\frac1{4C_\beta}\sum_{i=1}^3\frac{d}{dt}\left\|\left\langle \partial_{\beta}\left\{w^2_{l^*_m-|\beta|,1}\partial_\beta\left[v_i \mu^{1/2}\right]\right\}, \partial^\alpha{\{\bf I_+-P_+\}}f^{m,\epsilon}(t)-\partial^\alpha{\{\bf I_--P_-\}}f^{m,\epsilon}(t)\right\rangle\right\|^2 \nonumber\\
&&+\sum_{N_m^0+1\leq n\leq N_m+1,\atop 0\leq|\beta|\leq n }(1+t)^{-\sigma_{N_m+1,|\beta|}+2+2\vartheta}\nonumber\\
&&\times\left\{\left\|\nabla_x\left[E^{P,\epsilon},E^{j_1,\epsilon},\epsilon^m
 E^{m,\epsilon},\epsilon^{m+1} B^{m,\epsilon}\right](t)\right\|^2_{L^\infty_x}
  +\left\|\nabla^3_x\left[E^{P,\epsilon},E^{j_1,\epsilon},\epsilon^m
 E^{m,\epsilon},\epsilon^{m+1} B^{m,\epsilon}\right](t)\right\|^2_{H^{N_m-3}_x}\right\}\nonumber\\
&&\times\sum_{0\leq k\leq \min\{|\beta|+1,n\},\atop k\leq j\leq n}\mathcal{D}^{(j,k)}_{f^{m,\epsilon},j,l^*_m,1}(t)\nonumber\\
&&+\left\{\mathcal{E}_{f^{P,\epsilon},N_m+1,-\frac{l_m^*}\gamma,-\gamma}(t)
+\sum_{1\leq j_1\leq m-1}\mathcal{E}_{f^{j_1,\epsilon},N_m+1,-\frac{l_m^*}\gamma,-\gamma}(t)\right\}\mathcal{D}_{f^{m,\epsilon},N_m+1}(t)\nonumber\\
 &&+
  \mathcal{E}_{f^{m,\epsilon},N_m+1}(t)\left\{\mathcal{D}_{f^{P,\epsilon},N_m+2,N_m+3-\frac{l^*_m+1-N_m}\gamma,-\gamma}(t)+\sum_{i=1}^{m-1}\mathcal{D}_{f^{i,\epsilon},N_m+2,N_m+2-\frac{l^*_m+1-N_m}\gamma,-\gamma}(t)\right\}\nonumber\\
 &&+\left\{\mathcal{E}_{f^{P,\epsilon},N_m+1,-\frac{l_m^*}\gamma,-\gamma}(t)
   +\sum_{1\leq j_1\leq m-1}\mathcal{E}_{f^{j_1,\epsilon},N_m+1,-\frac{l_m^*}\gamma,-\gamma}(t)\right\}
\sum_{0\leq k\leq |\beta|\atop k\leq j\leq n}\mathcal{D}^{(j,k)}_{f^{m,\epsilon},j,l_m^*,1}(t)\nonumber\\
  &&+\sum_{N_m^0+1\leq n\leq N_m+1,\atop 0\leq|\beta|\leq n }(1+t)^{-\sigma_{N_m+1,|\beta|}}\nonumber\\
  &&\times\sum_{0\leq k\leq |\beta|\atop k\leq j\leq n}\mathcal{E}^{(j,k)}_{f^{m,\epsilon},j,l_m^*,1}(t)\left\{\mathcal{D}_{f^{P,\epsilon},N_m+1,-\frac{l_m^*}\gamma,-\gamma}(t)
  +\sum_{1\leq j_1\leq m-1}
  \mathcal{D}_{f^{j_1,\epsilon},N_m+1,-\frac{l_m^*}\gamma,-\gamma}(t)\right\}\nonumber\\
   &&+\sum_{j_1+j_2\geq m,\atop 0<j_1,j_2<m}\left\{\mathcal{E}_{f^{j_2,\epsilon},N_m+2,-\frac{l_m^*+2}\gamma+\frac12,-\gamma}(t)
    \mathcal{D}_{f^{j_1,\epsilon},N_m+1,-\frac{l_m^*+2}\gamma+\frac12,-\gamma}(t)\right\}\nonumber\\
&&+\sum_{N_m^0+1\leq n\leq N_m+1,\atop 0\leq|\beta|\leq n }(1+t)^{-\sigma_{N_m+1,|\beta|}}\mathcal{E}_{f^{m,\epsilon},N_m,N_m,-\gamma}(t)\sum_{0\leq k\leq |\beta|\atop k\leq j\leq N_m+1}\mathcal{D}^{(j,k)}_{f^{m,\epsilon},j,l_m^*,1}(t)\nonumber\\
  &&+\sum_{N_m^0+1\leq n\leq N_m+1,\atop 0\leq|\beta|\leq n }(1+t)^{-\sigma_{N_m+1,|\beta|}}\sum_{k_1+k_2\leq |\beta|}\left\{\sum_{0\leq j_1\leq n}\mathcal{E}^{(j_1,k_1)}_{f^{m,\epsilon},j_1,l_m^*,1}(t)
  \times\sum_{0\leq j_2\leq n}\mathcal{D}^{(j_2,k_1)}_{f^{m,\epsilon},j_2,l_m^*,1}(t)\right\}\nonumber\\
 &&+\mathcal{E}_{f^{m,\epsilon},N_m+1}(t)\mathcal{E}^1_{f^{m,\epsilon},N^0_m,l^*_m+1-\frac\gamma2,1}(t)+\mathcal{E}_{f^{m,\epsilon},N_m,N_m,-\gamma}(t)\mathcal{D}_{f^{m,\epsilon},N_m+1}(t)
 +\mathcal{D}_{f^{m,\epsilon},N_m+1}(t).\nonumber
\end{eqnarray}
\end{lemma}
\begin{proof}

\noindent{\bf Step 1:}
Firstly, performing $w^2_{l^*_m,1}\partial^\alpha f^{m,\epsilon}\times\partial^\alpha\eqref{f-R-vector}$ with $|\alpha|= N_m+1$ and integrating the resulting identity over $\mathbb{R}^3_v\times\mathbb{R}^3_x$, one has

\begin{eqnarray}\label{0-weight-alpha-high-R}
&&  \frac{d}{dt}\|w_{l^*_m,1}\partial^\alpha f^{m,\epsilon}\|^2+\|w_{l^*_m,1}\partial^\alpha f^{m,\epsilon}\|^2_{\nu}+\frac{q\vartheta}{(1+t)^{1+\vartheta}}\|w_{l^*_m,1}\partial^\alpha f^{m,\epsilon}\langle v\rangle\|^2\nonumber\\
&\lesssim&\|\partial^\alpha f^{m,\epsilon}\|_\nu^2+\underbrace{\left|\left(\partial^\alpha\left\{ E^{m,\epsilon} \cdot v \mu^{1/2}q_1\right\},w^2_{l^*_m,1}(\alpha,0)\partial^\alpha f^{m,\epsilon}\right)\right|
}_{\mathcal{W}_1}\nonumber\\
&&+\underbrace{\left|\left(\partial^\alpha \left(I_{mix}(t)+I_{pure}(t)\right),w^2_{l^*_m,1}(\alpha,0)\partial^\alpha f^{m,\epsilon}\right)\right|}_{\mathcal{W}_{2}}\nonumber\\
&&+\underbrace{\left|\left(\partial^\alpha\left(\epsilon q_0\left\{v\times \left\{B^{P}
+\sum_{j_1=1}^{m-1}\epsilon^{j_1}B^{j_1}\right\}\right\}\cdot\nabla_vf^{m,\epsilon}\right),w^2_{l^*_m,1}(\alpha,0)\partial^\alpha f^{m,\epsilon}\right)\right|}_{\mathcal{W}_{3}}.
\end{eqnarray}
For $\mathcal{W}_1$, multiplying a time factor $(1+t)^{-\frac{1+\epsilon_0}2}$ implies
\begin{eqnarray}
(1+t)^{-\frac{1+\epsilon_0}2}\mathcal{W}_1&\lesssim&(1+t)^{-\frac{1+\epsilon_0}2}\left|\left(\partial^\alpha\left\{ E^{m,\epsilon} \cdot v \mu^{1/2}q_1\right\},w^2_{l^*_m,1}(\alpha,0)\partial^\alpha f^{m,\epsilon}\right)\right|\nonumber\\
&\lesssim&\eta(1+t)^{-1-\epsilon_0}\| \nabla^{N_m+1}E^{m,\epsilon}\|^2+\|\nabla^{N_m+1} f^{m,\epsilon}\|^2_\nu.
\end{eqnarray}
For $\mathcal{W}_2$, setting \eqref{lemma-f-R-2}, \eqref{i-order-E-R-B},\eqref{i-order-E-R-12} with $\kappa=1$, which combining with \eqref{1-w-typical-high-R}, \eqref{1-w-typical-high-R-1}, \eqref{1-w-typical-high-non-R}, \eqref{1-w-typical-high-non-1-R}, and \eqref{i-weight-gamma-non-R} imply that
\begin{eqnarray}
  &&\mathcal{W}_{2}\nonumber\\
   &\lesssim&\sum_{1\leq j_1\leq m-1}\left\| [E^{P,\epsilon},E^{j_1,\epsilon},\epsilon^mE^{m,\epsilon}]\right\|_{L^\infty_x}\left\|w_{l^*_m,1}\partial^\alpha f^{m,\epsilon}\langle v\rangle^{\frac12}\right\|^2\nonumber\\
  &&+(1+t)^{2+2\vartheta}\nonumber\\
  &&\times\left\{\left\|\nabla_x[E^{P,\epsilon},E^{j_1,\epsilon},\epsilon^m
 E^{m,\epsilon},\epsilon^{m+1} B^{m,\epsilon}]\right\|^2_{L^\infty_x}
  +\left\|\nabla^3_x[E^{P,\epsilon},E^{j_1,\epsilon},\epsilon^m
 E^{m,\epsilon},\epsilon^{m+1} B^{m,\epsilon}]\right\|^2_{H^{N_m-3}_x}\right\}\nonumber\\
&&\times\sum_{0\leq k\leq 1,\atop k\leq j\leq |\alpha|}\mathcal{D}^{(j,k)}_{f^{m,\epsilon},j,l^*_m,1}(t)
\nonumber\\
  &&+\left\{\mathcal{E}_{f^{P,\epsilon},N_m+1,-\frac{l_m^*}\gamma,-\gamma}(t)
+\sum_{1\leq j_1\leq m-1}\mathcal{E}_{f^{j_1,\epsilon},N_m+1,-\frac{l_m^*}\gamma,-\gamma}(t)\right\}\mathcal{D}_{f^{m,\epsilon},N_m+1}(t)\nonumber\\
 &&+
  \mathcal{E}_{f^{m,\epsilon},N_m+1}(t)\left\{\mathcal{D}_{f^{P,\epsilon},N_m+2,N_m+3-\frac{l^*_m+1-N_m}\gamma,-\gamma}(t)+\sum_{i=1}^{m-1}\mathcal{D}_{f^{i,\epsilon},N_m+2,N_m+2-\frac{l^*_m+1-N_m}\gamma,-\gamma}(t)\right\}\nonumber\\
 &&+\left\{\mathcal{E}_{f^{P,\epsilon},N_m+1,-\frac{l_m^*}\gamma,-\gamma}(t)
   +\sum_{1\leq j_1\leq m-1}\mathcal{E}_{f^{j_1,\epsilon},N_m+1,-\frac{l_m^*}\gamma,-\gamma}(t)\right\}
 \sum_{0\leq j\leq |\alpha|}\mathcal{D}^{(j,0)}_{f^{m,\epsilon},j,l_m^*,1}(t)\nonumber\\
  &&+\sum_{0\leq j\leq |\alpha|}\mathcal{E}^{(j,0)}_{f^{m,\epsilon},j,\ell_m^*,1}(t)\left\{\mathcal{D}_{f^{P,\epsilon},N_m+1,-\frac{l_m^*}\gamma,-\gamma}(t)
  +\sum_{1\leq j_1\leq m-1}
  \mathcal{D}_{f^{j_1,\epsilon},N_m+1,-\frac{l_m^*}\gamma,-\gamma}(t)\right\}\nonumber\\
   &&+\sum_{j_1+j_2\geq m,\atop 0<j_1,j_2<m}\left\{\mathcal{E}_{f^{j_2,\epsilon},N_m+2,-\frac{l_m^*+2}\gamma+\frac12,-\gamma}(t)
    \mathcal{D}_{f^{j_1,\epsilon},N_m+1,-\frac{l_m^*+2}\gamma+\frac12,-\gamma}(t)\right\}\\
  &&+\mathcal{E}_{f^{m,\epsilon},N_m+1}(t)\mathcal{E}^1_{f^{m,\epsilon},N^0_m,l^*_m+1-\frac\gamma2,1}(t)+\eta\left\|\partial^\alpha f^{m,\epsilon}\right\|_\nu^2+\eta(1+t)^{-1-\vartheta}
\left\|w_{l^*_m,1}\partial^\alpha f^{m,\epsilon}\langle v\rangle^{\frac{1}2}\right\|^2\nonumber\\
&&+\mathcal{E}_{f^{m,\epsilon},N_m,N_m,-\gamma}(t)\left\{\mathcal{D}_{f^{m,\epsilon},N_m+1}(t)+\sum_{0\leq j\leq |\alpha|}\mathcal{D}^{(j,0)}_{f^{m,\epsilon},j,l^*_m,1}(t)\right\}\nonumber\\
   &&+\chi_{|\alpha|=1}\sum_{1\leq j\leq 2}\mathcal{E}^{(j,0)}_{f^{m,\epsilon},j,l_m^*,1}(t)\sum_{1\leq j\leq 2}\mathcal{D}^{(j,0)}_{f^{m,\epsilon},j,l^*_m,1}(t)\nonumber\\
  &&+\chi_{|\alpha|\geq 2}\sum_{0\leq j_1\leq |\alpha|}\mathcal{E}^{(j_1,0)}_{f^{m,\epsilon},j_1,l_m^*,1}(t)
  \times\sum_{0\leq j_2\leq |\alpha|}\mathcal{D}^{(j_2,0)}_{f^{m,\epsilon},j_2,l_m^*,1}(t).\nonumber
    \end{eqnarray}
  For $\mathcal{W}_3$, it is easy to see that
  \[\mathcal{W}_3=0.\]
For later simplicity in presentation, we use
\begin{equation}
  \sigma_{N_m+1,0}=\frac{1+\epsilon_0}2,
\end{equation}
thus we get
\begin{eqnarray}\label{0-weight-alpha-high-1-R}
&&  \frac{d}{dt}\left\{(1+t)^{-\sigma_{N_m+1,0}}
\|w_{l^*_m,1}\partial^\alpha f^{m,\epsilon}\|^2\right\}
+(1+t)^{-1-\sigma_{N_m+1,0}}
\|w_{l^*_m,1}\partial^\alpha f^{m,\epsilon}\|^2\nonumber\\
&&+(1+t)^{-\sigma_{N_m+1,0}}\|w_{l^*_m,1}\partial^\alpha f^{m,\epsilon}\|^2_{\nu}+(1+t)^{-\sigma_{N_m+1,0}}\frac{q\vartheta}{(1+t)^{1+\vartheta}}\|w_{l^*_m,1}\partial^\alpha f^{m,\epsilon}\langle v\rangle\|^2\nonumber\\
 &\lesssim&(1+t)^{-\sigma_{N_m+1,0}}(1+t)^{2+2\vartheta}\nonumber\\
 &&\times\left\{\left\|\nabla_x[E^{P,\epsilon},E^{j_1,\epsilon},\epsilon^m
 E^{m,\epsilon},\epsilon^{m+1} B^{m,\epsilon}]\right\|^2_{L^\infty_x}
  +\left\|\nabla^3_x[E^{P,\epsilon},E^{j_1,\epsilon},\epsilon^m
 E^{m,\epsilon},\epsilon^{m+1} B^{m,\epsilon}]\right\|^2_{H^{N_m-3}_x}\right\}\nonumber\\
&&\times\sum_{0\leq k\leq 1,\atop k\leq j\leq |\alpha|}\mathcal{D}^{(j,k)}_{f^{m,\epsilon},j,\ell^*_m,1}(t)
\nonumber\\
  &&+\left\{\mathcal{E}_{f^{P,\epsilon},N_m+1,-\frac{l_m^*}\gamma,-\gamma}(t)
+\sum_{1\leq j_1\leq m-1}\mathcal{E}_{f^{j_1,\epsilon},N_m+1,-\frac{l_m^*}\gamma,-\gamma}(t)\right\}\mathcal{D}_{f^{m,\epsilon},N_m+1}(t)\nonumber\\
 &&+
  \mathcal{E}_{f^{m,\epsilon},N_m+1}(t)\left\{\mathcal{D}_{f^{P,\epsilon},N_m+2,N_m+3-\frac{l^*_m+1-N_m}\gamma,-\gamma}(t)+\sum_{i=1}^{m-1}\mathcal{D}_{f^{i,\epsilon},N_m+2,N_m+2-\frac{l^*_m+1-N_m}\gamma,-\gamma}(t)\right\}\nonumber\\
 &&+(1+t)^{-\sigma_{N_m+1,0}}\nonumber\\
  &&\left\{\mathcal{E}_{f^{P,\epsilon},N_m+1,-\frac{l_m^*}\gamma,-\gamma}(t)
   +\sum_{1\leq j_1\leq m-1}\mathcal{E}_{f^{j_1,\epsilon},N_m+1,-\frac{l_m^*}\gamma,-\gamma}(t)\right\}
 \sum_{0\leq j\leq |\alpha|}\mathcal{D}^{(j,0)}_{f^{m,\epsilon},j,l_m^*,1}(t)\nonumber\\
  &&+(1+t)^{-\sigma_{N_m+1,0}}\nonumber\\
  &&\times\sum_{0\leq j\leq |\alpha|}\mathcal{E}^{(j,0)}_{f^{m,\epsilon},j,\ell_m^*,1}(t)\left\{\mathcal{D}_{f^{P,\epsilon},N_m+1,-\frac{l_m^*}\gamma,-\gamma}(t)
  +\sum_{1\leq j_1\leq m-1}
  \mathcal{D}_{f^{j_1,\epsilon},N_m+1,-\frac{l_m^*}\gamma,-\gamma}(t)\right\}\nonumber\\
   &&+\sum_{j_1+j_2\geq m,\atop 0<j_1,j_2<m}\left\{\mathcal{E}_{f^{j_2,\epsilon},N_m+2,-\frac{l_m^*+2}\gamma+\frac12,-\gamma}(t)
    \mathcal{D}_{f^{j_1,\epsilon},N_m+1,-\frac{l_m^*+2}\gamma+\frac12,-\gamma}(t)\right\}\\
&&+(1+t)^{-\sigma_{N_m+1,0}}\mathcal{E}_{f^{m,\epsilon},N_m,N_m,-\gamma}(t)\left\{\mathcal{D}_{f^{m,\epsilon},N_m+1}(t)+\sum_{0\leq j\leq |\alpha|}\mathcal{D}^{(j,0)}_{f^{m,\epsilon},j,l^*_m,1}(t)\right\}\nonumber\\
   &&+(1+t)^{-\sigma_{N_m+1,0}}\chi_{|\alpha|=1}\sum_{1\leq j\leq 2}\mathcal{E}^{(j,0)}_{f^{m,\epsilon},j,l_m^*,1}(t)\sum_{1\leq j\leq 2}\mathcal{D}^{(j,0)}_{f^{m,\epsilon},j,l^*_m,1}(t)\nonumber\\
  &&+(1+t)^{-\sigma_{N_m+1,0}}\chi_{|\alpha|\geq 2}\sum_{0\leq j_1\leq |\alpha|}\mathcal{E}^{(j_1,0)}_{f^{m,\epsilon},j_1,l_m^*,1}(t)
  \times\sum_{0\leq j_2\leq |\alpha|}\mathcal{D}^{(j_2,0)}_{f^{m,\epsilon},j_2,l_m^*,1}(t)\nonumber\\
 &&+\mathcal{E}_{f^{m,\epsilon},N_m}(t)\mathcal{E}^1_{f^{m,\epsilon},N^0_m,l^*_m+1-\frac\gamma2,1}(t)+\eta(1+t)^{-1-\epsilon_0}\| \nabla^{N_m+1}E^{m,\epsilon}\|^2+\|\nabla^{N_m+1} f^{m,\epsilon}\|^2_\nu,\nonumber
\end{eqnarray}
where we have used the assumption \eqref{Assume-f-R-high}.

Secondly, applying $\partial^\alpha_\beta$ with $|\alpha|+|\beta|= N_m+1,\ |\beta|\geq1$ to \eqref{f-R-vect-I-P},
multiplying the resulting identity by $w^2_{l^*_m-|\beta|,1}\partial^\alpha_\beta{\bf\{I-P\}} f^{m,\epsilon}$,  and integrating the final result over $\mathbb{R}^3_v\times\mathbb{R}^3_x$, one has
\begin{eqnarray}\label{i-weight-I-P-1}
&&  \frac{d}{dt}\|w_{l^*_m-|\beta|,1}\partial^\alpha_\beta{\bf \{I-P\}}f^{m,\epsilon}\|^2+\|w_{l^*_m-|\beta|,1}\partial^\alpha_\beta{\bf \{I-P\}}f^{m,\epsilon}\|^2_{\nu}\nonumber\\
&&+\frac{q\vartheta}{(1+t)^{1+\vartheta}}\|w_{l^*_m-|\beta|,1}\partial^\alpha_\beta{\bf \{I-P\}}f^{m,\epsilon}\langle v\rangle\|^2\nonumber\\
&\lesssim&\eta\|w_{l^*_m,1}\partial^\alpha{\bf\{I-P\}}f^{m,\epsilon}\|_\nu^2
+\|\partial^\alpha{\bf\{I-P\}}f^{m,\epsilon}\|_\nu^2
\nonumber\\
&&+\underbrace{\left|\left(\partial^\alpha_\beta\{-{\bf \{I-P\}}v\cdot\nabla_xf^{m,\epsilon}\},w^2_{l^*_m-|\beta|,1}\partial^\alpha_\beta{\bf \{I-P\}}f^{m,\epsilon}\right)\right|
}_{\mathcal{W}_{4}}\\
&&+\underbrace{\left|(\partial^\alpha_\beta\left\{\left\{E^{m,\epsilon}+\epsilon b^{f^{m,\epsilon}}\times \left\{B^{P}
+\sum_{j_1=1}^{m-1}\epsilon^{j_1}B^{j_1}\right\}\right\}  \cdot v \mu^{1/2}q_1\right\},w^2_{l^*_m-|\beta|,1}\partial^\alpha_\beta{\bf \{I-P\}}f^{m,\epsilon})\right|
}_{\mathcal{W}_{5}}\nonumber\\
&&+\underbrace{\left|\left(\partial^\alpha_\beta \left(I_{mix,mic}(t)+I_{pure,mic}(t)\right),w^2_{l^*_m-|\beta|,1}\partial^\alpha_\beta{\bf \{I-P\}}f^{m,\epsilon}\right)\right|}_{\mathcal{W}_{6}}\nonumber\\
&&+\underbrace{\left|\left(\partial^\alpha_\beta
\left\{ \left(\epsilon q_0\left\{v\times \left\{B^{P}
+\sum_{j_1=1}^{m-1}\epsilon^{j_1}B^{j_1}\right\}\right\}\cdot\nabla_v{\bf\{I-P\}}f^{m,\epsilon}\right)\right\},w^2_{l^*_m-|\beta|,1}\partial^\alpha_\beta{\bf \{I-P\}}f^{m,\epsilon}\right)\right|}_{\mathcal{W}_{7}}.\nonumber
\end{eqnarray}
By the macro-micro decomposition \eqref{macro-micro}, one has
\begin{eqnarray*}
  \mathcal{W}_{4}&\lesssim&\underbrace{\left|\left(\partial^\alpha_\beta\{v\cdot\nabla_x{\bf \{I-P\}}f^{m,\epsilon}\},w^2_{l^*_m-|\beta|,1}\partial^\alpha_\beta{\bf \{I-P\}}f^{m,\epsilon}\right)\right|}_{\mathcal{W}_{4,1}}\\
  &&+\underbrace{\left|\left(\partial^\alpha_\beta\{v\cdot\nabla_x{\bf P}f^{m,\epsilon}\},w^2_{l^*_m-|\beta|,1}\partial^\alpha_\beta{\bf \{I-P\}}f^{m,\epsilon}\right)\right|}_{\mathcal{W}_{4,2}}\\
  &&+\underbrace{\left|\left(\partial^\alpha_\beta{\bf P}\left\{v\cdot\nabla_xf^{m,\epsilon}\right\},w^2_{l^*_m-|\beta|,1}\partial^\alpha_\beta{\bf \{I-P\}}f^{m,\epsilon}\right)\right|}_{\mathcal{W}_{4,3}}.
\end{eqnarray*}
For $\mathcal{W}_{4,1}$, multiplying time factors $(1+t)^{-\sigma_{N_m+1,1}}$, applying \eqref{linear-fi-high-R}, one has
\begin{eqnarray*}
    (1+t)^{-\sigma_{N_m+1,1}}\mathcal{W}_{4,1}
    &\lesssim&(1+t)^{-\sigma_{N_m+1,0}}\mathcal{D}^{(N_m+1,0)}_{f^{m,\epsilon},N_m+1,l_m^*,1}(t)+
    \eta(1+t)^{-\sigma_{N_m+1,1}}\|w_{l_m^*-|\beta|,1}\partial^\alpha_\beta{\bf\{I-P\}}f^{m,\epsilon}\|_\nu^2.
  \end{eqnarray*}
It is straightforward to compute that
\[\mathcal{W}_{4,2}+\mathcal{W}_{4,3}\lesssim\eta\|\nabla^{|\alpha|+1}_x f^{m,\epsilon}\|_\nu^2+\|\partial^{\alpha}{\bf\{I-P\}}f^{m,\epsilon}\|_\nu^2.\]
Using \eqref{E-R-nonhard-1} yields
\begin{eqnarray*}
\mathcal{W}_{5} &\lesssim&\frac1{4C_\beta}\sum_{i=1}^3\frac{d}{dt}\left\|\left\langle \partial_{\beta}\left\{w^2_{l^*_m-|\beta|,1}\partial_\beta\left[v_i \mu^{1/2}\right]\right\}, \partial^\alpha{\{\bf I_+-P_+\}}f^{m,\epsilon}-\partial^\alpha{\{\bf I_--P_-\}}f^{m,\epsilon}\right\rangle\right\|^2\\ \nonumber
  &&+\left\|\nabla^{|\alpha|}f^{m,\epsilon}\right\|_\nu^2+\left\|\nabla^{|\alpha|+1}f^{m,\epsilon}\right\|^2_\nu
  \nonumber\\
  &&+\mathcal{E}_{f^{m,\epsilon},N_m}(t)\left\{\mathcal{D}_{f^{P,\epsilon},N_m+1}(t)+
  \sum_{1\leq j_1\leq m-1}\mathcal{D}_{f^{j_1,\epsilon},N_m+1}(t)+\mathcal{D}_{f^{m,\epsilon},N_m+1}(t)\right\}\\ \nonumber
  &&+\sum_{j_1+j_2\geq m,\atop 0<j_1,j_2<m}\mathcal{E}_{f^{j_1,\epsilon},N_m+1}(t)\mathcal{D}_{f^{j_2,\epsilon},N_m+1}(t)
  +\sum_{1\leq j_1\leq m-1}\mathcal{E}_{f^{j_1,\epsilon},N_m+1}(t)\mathcal{D}_{f^{m,\epsilon},N_m+1}(t).
\end{eqnarray*}
As for $\mathcal{W}_{6}$, one has
\begin{eqnarray}
&&\mathcal{W}_{6}\nonumber\\
   &\lesssim&\sum_{1\leq j_1\leq m-1}\left\| [E^{P,\epsilon},E^{j_1,\epsilon},\epsilon^mE^{m,\epsilon}]\right\|_{L^\infty_x}\left\| w_{l_m^*-|\beta|,1}\partial^\alpha_\beta{\bf\{I-P\}}f^{m,\epsilon}\langle v\rangle^{\frac12}\right\|^2\nonumber\\
  &&+(1+t)^{2+2\vartheta}\nonumber\\
  &&\left\{\left\|\nabla_x[E^{P,\epsilon},E^{j_1,\epsilon},\epsilon^m
 E^{m,\epsilon},\epsilon^{m+1} B^{m,\epsilon}]\right\|^2_{L^\infty_x}
  +\left\|\nabla^3_x[E^{P,\epsilon},E^{j_1,\epsilon},\epsilon^m
 E^{m,\epsilon},\epsilon^{m+1} B^{m,\epsilon}]\right\|^2_{H^{N_m-3}_x}\right\}\nonumber\\
&&\times\sum_{0\leq k\leq 2,\atop k\leq j\leq |\alpha|}\mathcal{D}^{(j,k)}_{f^{m,\epsilon},j,l^*_m,1}(t)
\nonumber\\
  &&+\left\{\mathcal{E}_{f^{P,\epsilon},N_m+1,-\frac{l_m^*}\gamma,-\gamma}(t)
+\sum_{1\leq j_1\leq m-1}\mathcal{E}_{f^{j_1,\epsilon},N_m+1,-\frac{l_m^*}\gamma,-\gamma}(t)\right\}\mathcal{D}_{f^{m,\epsilon},N_m+1}(t)\nonumber\\
 &&+
  \mathcal{E}_{f^{m,\epsilon},N_m+1}(t)\left\{\mathcal{D}_{f^{P,\epsilon},N_m+2,N_m+3-\frac{l^*_m+1-N_m}\gamma,-\gamma}(t)+\sum_{i=1}^{m-1}\mathcal{D}_{f^{i,\epsilon},N_m+2,N_m+2-\frac{l^*_m+1-N_m}\gamma,-\gamma}(t)\right\}\nonumber\\
 &&+\left\{\mathcal{E}_{f^{P,\epsilon},N_m+1,-\frac{l_m^*}\gamma,-\gamma}(t)
   +\sum_{1\leq j_1\leq m-1}\mathcal{E}_{f^{j_1,\epsilon},N_m+1,-\frac{l_m^*}\gamma,-\gamma}(t)\right\}
\sum_{0\leq k\leq 1\atop k\leq j\leq N_m+1}\mathcal{D}^{(j,k)}_{f^{m,\epsilon},j,l_m^*,1}(t)\nonumber\\
  &&+\sum_{0\leq k\leq 1\atop k\leq j\leq N_m+1}\mathcal{E}^{(j,k)}_{f^{m,\epsilon},j,l_m^*,1}(t)\left\{\mathcal{D}_{f^{P,\epsilon},N_m+1,-\frac{l_m^*}\gamma,-\gamma}(t)
  +\sum_{1\leq j_1\leq m-1}
  \mathcal{D}_{f^{j_1,\epsilon},N_m+1,-\frac{l_m^*}\gamma,-\gamma}(t)\right\}\nonumber\\
   &&+\sum_{j_1+j_2\geq m,\atop 0<j_1,j_2<m}\left\{\mathcal{E}_{f^{j_2,\epsilon},N_m+2,-\frac{l_m^*+2}\gamma+\frac12,-\gamma}(t)
    \mathcal{D}_{f^{j_1,\epsilon},N_m+1,-\frac{l_m^*+2}\gamma+\frac12,-\gamma}(t)\right\}\\
  &&+\mathcal{E}_{f^{m,\epsilon},N_m,N_m,-\gamma}(t)\mathcal{D}_{f^{m,\epsilon},N_m+1}(t)+\mathcal{E}_{f^{m,\epsilon},N_m,N_m,-\gamma}(t)\sum_{0\leq k\leq 1\atop k\leq j\leq N_m+1}\mathcal{D}^{(j,k)}_{f^{m,\epsilon},j,l_m^*,1}(t)\nonumber\\
  &&+\sum_{k_1+k_2\leq 1}\left\{\sum_{0\leq j_1\leq N_m+1}\mathcal{E}^{(j_1,k_1)}_{f^{m,\epsilon},j_1,l_m^*,1}(t)
  \times\sum_{0\leq j_2\leq N_m+1}\mathcal{D}^{(j_2,k_1)}_{f^{m,\epsilon},j_2,l_m^*,1}(t)\right\}\nonumber\\
  &&
  +\mathcal{E}_{f^{m,\epsilon},N_m+1}(t)\mathcal{E}^1_{f^{m,\epsilon},N^0_m,l^*_m+1-\frac\gamma2,1}(t)\nonumber\\
  &&+\left\|\partial^\alpha{\bf\{I-P\}}f^{m,\epsilon}\right\|_\nu^2+\eta(1+t)^{-1-\vartheta}
\left\| w_{l_m^*-|\beta|,1}\partial^\alpha_\beta{\bf\{I-P\}}f^{m,\epsilon}\langle v\rangle^{\frac12}\right\|^2.\nonumber
\end{eqnarray}
Finally, it is straightforward to compute that
\[\mathcal{W}_7=0.\]

Collecting the estimates on $\mathcal{W}_4\sim\mathcal{W}_7$ into \eqref{i-weight-I-P-1} yields that
\begin{eqnarray}\label{f-R-weight-I-P-1-High}
&&\frac{d}{dt}\left\{(1+t)^{-\sigma_{N_m+1,1}}\|w_{l^*_m-|\beta|,1}\partial^\alpha_\beta{\bf \{I-P\}}f^{m,\epsilon}\|^2\right\}+(1+t)^{-1-\sigma_{N_m+1,1}}\|w_{l^*_m-|\beta|,1}\partial^\alpha_\beta{\bf \{I-P\}}f^{m,\epsilon}\|^2\nonumber\\
&&+(1+t)^{-\sigma_{N_m+1,1}}\|w_{l^*_m-|\beta|,1}\partial^\alpha_\beta{\bf \{I-P\}}f^{m,\epsilon}\|^2_{\nu}\nonumber\\
&&+(1+t)^{-\sigma_{N_m+1,1}}\frac{q\vartheta}{(1+t)^{1+\vartheta}}\|w_{l^*_m-|\beta|,1}\partial^\alpha_\beta{\bf \{I-P\}}f^{m,\epsilon}\langle v\rangle\|^2\nonumber\\
&\lesssim&\frac1{4C_\beta}\sum_i\frac{d}{dt}\left\|\left\langle \partial_{\beta}\left\{w^2_{l^*_m-|\beta|,1}\partial_\beta\left[v_i \mu^{1/2}\right]\right\}, \partial^\alpha{\{\bf I_+-P_+\}}f^{m,\epsilon}-\partial^\alpha{\{\bf I_--P_-\}}f^{m,\epsilon}\right\rangle\right\|^2 \nonumber\\
&&+(1+t)^{-\sigma_{N_m+1,0}}\mathcal{D}^{(N_m+1,0)}_{f^{m,\epsilon},N_m+1,l_m^*,1}(t)+
\mathcal{D}_{f^{m,\epsilon},N_m+1}(t)\\
&&+(1+t)^{-\sigma_{N_m+1,1}+2+2\vartheta}\nonumber\\
&&\times\left\{\left\|\nabla_x[E^{P,\epsilon},E^{j_1,\epsilon},\epsilon^m
 E^{m,\epsilon},\epsilon^{m+1} B^{m,\epsilon}]\right\|^2_{L^\infty_x}
  +\left\|\nabla^3_x[E^{P,\epsilon},E^{j_1,\epsilon},\epsilon^m
 E^{m,\epsilon},\epsilon^{m+1} B^{m,\epsilon}]\right\|^2_{H^{N_m-3}_x}\right\}\nonumber\\
&&\times\sum_{0\leq k\leq 2,\atop k\leq j\leq |\alpha|}\mathcal{D}^{(j,k)}_{f^{m,\epsilon},j,l^*_m,1}(t)
\nonumber\\
  &&+\left\{\mathcal{E}_{f^{P,\epsilon},N_m+1,-\frac{l_m^*}\gamma,-\gamma}(t)
+\sum_{1\leq j_1\leq m-1}\mathcal{E}_{f^{j_1,\epsilon},N_m+1,-\frac{l_m^*}\gamma,-\gamma}(t)\right\}\mathcal{D}_{f^{m,\epsilon},N_m+1}(t)\nonumber\\
 &&+
  \mathcal{E}_{f^{m,\epsilon},N_m+1}(t)\left\{\mathcal{D}_{f^{P,\epsilon},N_m+2,N_m+3-\frac{l^*_m+1-N_m}\gamma,-\gamma}(t)+\sum_{i=1}^{m-1}\mathcal{D}_{f^{i,\epsilon},N_m+2,N_m+2-\frac{l^*_m+1-N_m}\gamma,-\gamma}(t)\right\}\nonumber\\
 &&+\left\{\mathcal{E}_{f^{P,\epsilon},N_m+1,-\frac{l_m^*}\gamma,-\gamma}(t)
   +\sum_{1\leq j_1\leq m-1}\mathcal{E}_{f^{j_1,\epsilon},N_m+1,-\frac{l_m^*}\gamma,-\gamma}(t)\right\}
\sum_{0\leq k\leq 1\atop k\leq j\leq N_m+1}\mathcal{D}^{(j,k)}_{f^{m,\epsilon},j,l_m^*,1}(t)\nonumber\\
  &&+(1+t)^{-\sigma_{N_m+1,1}}\nonumber\\
  &&\times\sum_{0\leq k\leq 1\atop k\leq j\leq N_m+1}\mathcal{E}^{(j,k)}_{f^{m,\epsilon},j,l_m^*,1}(t)\left\{\mathcal{D}_{f^{P,\epsilon},N_m+1,-\frac{l_m^*}\gamma,-\gamma}(t)
  +\sum_{1\leq j_1\leq m-1}
  \mathcal{D}_{f^{j_1,\epsilon},N_m+1,-\frac{l_m^*}\gamma,-\gamma}(t)\right\}\nonumber\\
   &&+\sum_{j_1+j_2\geq m,\atop 0<j_1,j_2<m}\left\{\mathcal{E}_{f^{j_2,\epsilon},N_m+2,-\frac{l_m^*+2}\gamma+\frac12,-\gamma}(t)
    \mathcal{D}_{f^{j_1,\epsilon},N_m+1,-\frac{l_m^*+2}\gamma+\frac12,-\gamma}(t)\right\}\nonumber\\
  &&+\mathcal{E}_{f^{m,\epsilon},N_m+1}(t)\mathcal{E}^1_{f^{m,\epsilon},N^0_m,l^*_m+1-\frac\gamma2,1}(t)+\mathcal{E}_{f^{m,\epsilon},N_m,N_m,-\gamma}(t)\mathcal{D}_{f^{m,\epsilon},N_m+1}(t)\nonumber\\
&&+(1+t)^{-\sigma_{N_m+1,1}}\mathcal{E}_{f^{m,\epsilon},N_m,N_m,-\gamma}(t)\sum_{0\leq k\leq 1\atop k\leq j\leq N_m+1}\mathcal{D}^{(j,k)}_{f^{m,\epsilon},j,l_m^*,1}(t)\nonumber\\
  &&+(1+t)^{-\sigma_{N_m+1,1}}\sum_{k_1+k_2\leq 1}\left\{\sum_{0\leq j_1\leq N_m+1}\mathcal{E}^{(j_1,k_1)}_{f^{m,\epsilon},j_1,l_m^*,1}(t)
  \times\sum_{0\leq j_2\leq N_m+1}\mathcal{D}^{(j_2,k_1)}_{f^{m,\epsilon},j_2,l_m^*,1}(t)\right\}.\nonumber
\end{eqnarray}
For $|\alpha|+|\beta|=N_m+1$ with $2\leq |\beta|\leq N_m+1$, one has
\begin{eqnarray}\label{f-R-weight-I-P-1-High-i}
&&\frac{d}{dt}\left\{(1+t)^{-\sigma_{N_m+1,|\beta|}}\|w_{l^*_m-|\beta|,1}\partial^\alpha_\beta{\bf \{I-P\}}f^{m,\epsilon}\|^2\right\}\nonumber\\
  &&+(1+t)^{-1-\sigma_{N_m+1,|\beta|}}\|w_{l^*_m-|\beta|,1}\partial^\alpha_\beta{\bf \{I-P\}}f^{m,\epsilon}\|^2\nonumber\\
&&+(1+t)^{-\sigma_{N_m+1,|\beta|}}\|w_{l^*_m-|\beta|,1}\partial^\alpha_\beta{\bf \{I-P\}}f^{m,\epsilon}\|^2_{\nu}\nonumber\\
&&+(1+t)^{-\sigma_{N_m+1,|\beta|}}\frac{q\vartheta}{(1+t)^{1+\vartheta}}\|w_{l^*_m,-\gamma}(\alpha,\beta)\partial^\alpha_\beta{\bf \{I-P\}}f^{m,\epsilon}\langle v\rangle\|^2\nonumber\\
&\lesssim&\frac1{4C_\beta}\sum_{i=1}^3\frac{d}{dt}\left\|\left\langle \partial_{\beta}\left\{w^2_{l^*_m-|\beta|,1}\partial_\beta\left[v_i \mu^{1/2}\right]\right\}, \partial^\alpha{\{\bf I_+-P_+\}}f^{m,\epsilon}-\partial^\alpha{\{\bf I_--P_-\}}f^{m,\epsilon}\right\rangle\right\|^2 \nonumber\\
&&+(1+t)^{-\sigma_{N_m+1,|\beta|-1}}\mathcal{D}^{(N_m+1,|\beta|-1)}_{f^{m,\epsilon},N_m+1,l_m^*,1}(t)+
\mathcal{D}_{f^{m,\epsilon},N_m+1}(t)\nonumber\\
&&+(1+t)^{-\sigma_{N_m+1,|\beta|}+2+2\vartheta}\nonumber\\
  &&\times\left\{\left\|\nabla_x[E^{P,\epsilon},E^{j_1,\epsilon},\epsilon^m
 E^{m,\epsilon},\epsilon^{m+1} B^{m,\epsilon}]\right\|^2_{L^\infty_x}
  +\left\|\nabla^3_x[E^{P,\epsilon},E^{j_1,\epsilon},\epsilon^m
 E^{m,\epsilon},\epsilon^{m+1} B^{m,\epsilon}]\right\|^2_{H^{N_m-3}_x}\right\}\nonumber\\
&&\times\sum_{0\leq k\leq \min\{|\beta|+1,|\alpha|+|\beta|\},\atop k\leq j\leq |\alpha|}\mathcal{D}^{(j,k)}_{f^{m,\epsilon},j,l^*_m,1}(t)
\nonumber\\
  &&+\left\{\mathcal{E}_{f^{P,\epsilon},N_m+1,-\frac{l_m^*}\gamma,-\gamma}(t)
+\sum_{1\leq j_1\leq m-1}\mathcal{E}_{f^{j_1,\epsilon},N_m+1,-\frac{l_m^*}\gamma,-\gamma}(t)\right\}\mathcal{D}_{f^{m,\epsilon},N_m+1}(t)\nonumber\\
 &&+
  \mathcal{E}_{f^{m,\epsilon},N_m+1}(t)\left\{\mathcal{D}_{f^{P,\epsilon},N_m+2,N_m+3-\frac{l^*_m+1-N_m}\gamma,-\gamma}(t)+\sum_{i=1}^{m-1}\mathcal{D}_{f^{i,\epsilon},N_m+2,N_m+2-\frac{l^*_m+1-N_m}\gamma,-\gamma}(t)\right\}\nonumber\\
 &&+\left\{\mathcal{E}_{f^{P,\epsilon},N_m+1,-\frac{l_m^*}\gamma,-\gamma}(t)
   +\sum_{1\leq j_1\leq m-1}\mathcal{E}_{f^{j_1,\epsilon},N_m+1,-\frac{l_m^*}\gamma,-\gamma}(t)\right\}
\sum_{0\leq k\leq |\beta|\atop k\leq j\leq N_m+1}\mathcal{D}^{(j,k)}_{f^{m,\epsilon},j,l_m^*,1}(t)\nonumber\\
  &&+(1+t)^{-\sigma_{N_m+1,|\beta|}}\nonumber\\
  &&\sum_{0\leq k\leq |\beta|\atop k\leq j\leq N_m+1}\mathcal{E}^{(j,k)}_{f^{m,\epsilon},j,l_m^*,1}(t)\left\{\mathcal{D}_{f^{P,\epsilon},N_m+1,-\frac{l_m^*}\gamma,-\gamma}(t)
  +\sum_{1\leq j_1\leq m-1}
  \mathcal{D}_{f^{j_1,\epsilon},N_m+1,-\frac{l_m^*}\gamma,-\gamma}(t)\right\}\nonumber\\
   &&+\sum_{j_1+j_2\geq m,\atop 0<j_1,j_2<m}\left\{\mathcal{E}_{f^{j_2,\epsilon},N_m+2,-\frac{l_m^*+2}\gamma+\frac12,-\gamma}(t)
    \mathcal{D}_{f^{j_1,\epsilon},N_m+1,-\frac{l_m^*+2}\gamma+\frac12,-\gamma}(t)\right\}\\
  &&+\mathcal{E}_{f^{m,\epsilon},N_m+1}(t)\mathcal{E}^1_{f^{m,\epsilon},N^0_m,l^*_m+1-\frac\gamma2,1}(t)+\mathcal{E}_{f^{m,\epsilon},N_m,N_m,-\gamma}(t)\mathcal{D}_{f^{m,\epsilon},N_m+1}(t)\nonumber\\
&&+(1+t)^{-\sigma_{N_m+1,|\beta|}}\mathcal{E}_{f^{m,\epsilon},N_m,N_m,-\gamma}(t)\sum_{0\leq k\leq |\beta|\atop k\leq j\leq N_m+1}\mathcal{D}^{(j,k)}_{f^{m,\epsilon},j,l_m^*,1}(t)\nonumber\\
  &&+(1+t)^{-\sigma_{N_m+1,|\beta|}}\sum_{k_1+k_2\leq |\beta|}\left\{\sum_{0\leq j_1\leq N_m+1}\mathcal{E}^{(j_1,k_1)}_{f^{m,\epsilon},j_1,l_m^*,1}(t)
  \times\sum_{0\leq j_2\leq N_m+1}\mathcal{D}^{(j_2,k_1)}_{f^{m,\epsilon},j_2,l_m^*,1}(t)\right\}.\nonumber
\end{eqnarray}

Therefore, a proper linear combination of \eqref{0-weight-alpha-high-1-R}, \eqref{f-R-weight-I-P-1-High} and \eqref{f-R-weight-I-P-1-High-i} over $|\alpha|+|\beta|=N_m+1, 0\leq|\beta|\leq N_m+1$, one has
\begin{eqnarray}\label{f-i-N_m+1-high}
  &&\frac{d}{dt}\left\{\sum_{0\leq|\beta|\leq N_m+1}(1+t)^{-\sigma_{N_m+1,|\beta|}}\mathcal{E}^{(N_m+1,|\beta|)}_{f^{m,\epsilon},N_m+1,l_m^*,1}(t)\right\}\nonumber\\
&&  +\left\{\sum_{0\leq|\beta|\leq N_m+1}(1+t)^{-\sigma_{N_m+1,|\beta|}}\mathcal{E}^{(N_m+1,|\beta|)}_{f^{m,\epsilon},N_m+1,l_m^*,1}(t)
+(1+t)^{-\sigma_{N_m+1,|\beta|}}\mathcal{D}^{(N_m+1,|\beta|)}_{f^{m,\epsilon},N_m+1,l_m^*,1}(t)\right\}\nonumber\\
&\lesssim&\sum_{0\leq|\beta|\leq N_m+1,\atop |\alpha|+|\beta|=N_m+1}\frac1{4C_\beta}\sum_{i=1}^3\frac{d}{dt}\left\|\left\langle \partial_{\beta}\left\{w^2_{l^*_m-|\beta|,1}\partial_\beta\left[v_i \mu^{1/2}\right]\right\}, \partial^\alpha{\{\bf I_+-P_+\}}f^{m,\epsilon}-\partial^\alpha{\{\bf I_--P_-\}}f^{m,\epsilon}\right\rangle\right\|^2 \nonumber\\
&&+\sum_{0\leq k\leq \min\{|\beta|+1,|\alpha|+|\beta|\},\atop 0\leq|\beta|\leq N_m+1,k\leq j\leq |\alpha|}(1+t)^{-\sigma_{N_m+1,|\beta|}+2+2\vartheta}\nonumber\\
&&\times\left\{\left\|\nabla_x(E^{P,\epsilon},E^{j_1,\epsilon},\epsilon^m
 E^{m,\epsilon},\epsilon^{m+1} B^{m,\epsilon})\right\|^2_{L^\infty_x}
  +\left\|\nabla^3_x(E^{P,\epsilon},E^{j_1,\epsilon},\epsilon^m
 E^{m,\epsilon},\epsilon^{m+1} B^{m,\epsilon})\right\|^2_{H^{N_m-3}_x}\right\}\nonumber\\
  &&\times\mathcal{D}^{(j,k)}_{f^{m,\epsilon},j,l^*_m,1}(t)
\nonumber\\
&&+\left\{\mathcal{E}_{f^{P,\epsilon},N_m+1,-\frac{l_m^*}\gamma,-\gamma}(t)
+\sum_{1\leq j_1\leq m-1}\mathcal{E}_{f^{j_1,\epsilon},N_m+1,-\frac{l_m^*}\gamma,-\gamma}(t)\right\}\mathcal{D}_{f^{m,\epsilon},N_m+1}(t)\nonumber\\
 &&+
  \mathcal{E}_{f^{m,\epsilon},N_m+1}(t)\left\{\mathcal{D}_{f^{P,\epsilon},N_m+2,N_m+3-\frac{l^*_m+1-N_m}\gamma,-\gamma}(t)+\sum_{i=1}^{m-1}\mathcal{D}_{f^{i,\epsilon},N_m+2,N_m+2-\frac{l^*_m+1-N_m}\gamma,-\gamma}(t)\right\}\nonumber\\
 &&+\left\{\mathcal{E}_{f^{P,\epsilon},N_m+1,-\frac{l_m^*}\gamma,-\gamma}(t)
   +\sum_{1\leq j_1\leq m-1}\mathcal{E}_{f^{j_1,\epsilon},N_m+1,-\frac{l_m^*}\gamma,-\gamma}(t)\right\}
\sum_{0\leq k\leq |\beta|\atop k\leq j\leq N_m+1}\mathcal{D}^{(j,k)}_{f^{m,\epsilon},j,l_m^*,1}(t)\nonumber\\
  &&+\sum_{0\leq|\beta|\leq N_m+1,0\leq k\leq |\beta|\atop k\leq j\leq N_m+1}(1+t)^{-\sigma_{N_m+1,|\beta|}}\mathcal{E}^{(j,k)}_{f^{m,\epsilon},j,l_m^*,1}(t)\nonumber\\
  &&\times\left\{\mathcal{D}_{f^{P,\epsilon},N_m+1,-\frac{l_m^*}\gamma,-\gamma}(t)
  +\sum_{1\leq j_1\leq m-1}
  \mathcal{D}_{f^{j_1,\epsilon},N_m+1,-\frac{l_m^*}\gamma,-\gamma}(t)\right\}\nonumber\\
   &&+\sum_{j_1+j_2\geq m,\atop 0<j_1,j_2<m}\left\{\mathcal{E}_{f^{j_2,\epsilon},N_m+2,-\frac{l_m^*+2}\gamma+\frac12,-\gamma}(t)
    \mathcal{D}_{f^{j_1,\epsilon},N_m+1,-\frac{l_m^*+2}\gamma+\frac12,-\gamma}(t)\right\}\\
&&+\sum_{0\leq|\beta|\leq N_m+1}(1+t)^{-\sigma_{N_m+1,|\beta|}}\mathcal{E}_{f^{m,\epsilon},N_m,N_m,-\gamma}(t)\sum_{0\leq k\leq |\beta|\atop k\leq j\leq N_m+1}\mathcal{D}^{(j,k)}_{f^{m,\epsilon},j,l_m^*,1}(t)\nonumber\\
  &&+\sum_{0\leq|\beta|\leq N_m+1}(1+t)^{-\sigma_{N_m+1,|\beta|}}\nonumber\\
  &&\times\sum_{k_1+k_2\leq |\beta|}\left\{\sum_{0\leq j_1\leq N_m+1}\mathcal{E}^{(j_1,k_1)}_{f^{m,\epsilon},j_1,l_m^*,1}(t)
  \times\sum_{0\leq j_2\leq N_m+1}\mathcal{D}^{(j_2,k_1)}_{f^{m,\epsilon},j_2,l_m^*,1}(t)\right\}\nonumber\\
&&+\mathcal{E}_{f^{m,\epsilon},N_m+1}(t)\mathcal{E}^1_{f^{m,\epsilon},N^0_m,l^*_m+1-\frac\gamma2,1}(t)+\mathcal{E}_{f^{m,\epsilon},N_m,N_m,-\gamma}(t)\mathcal{D}_{f^{m,\epsilon},N_m+1}(t)\nonumber\\
  &&+\mathcal{D}_{f^{m,\epsilon},N_m+1}(t).\nonumber
\end{eqnarray}

\noindent{\bf Step 2:} For $N_m^0+1\leq|\alpha|+|\beta|=n\leq N_m$, similar to that of  {\bf Step 1}, one can deduce that
\begin{eqnarray}\label{f-R-N_m-less-high}
  &&\frac{d}{dt}\left\{\sum_{0\leq|\beta|\leq n}(1+t)^{-\sigma_{n,|\beta|}}\mathcal{E}^{(n,|\beta|)}_{f^{m,\epsilon},n,l_m^*,1}(t)\right\}\nonumber\\
&&  +\left\{\sum_{0\leq|\beta|\leq n}(1+t)^{-\sigma_{n,|\beta|}}\mathcal{E}^{(n,|\beta|)}_{f^{m,\epsilon},n,l_m^*,1}(t)
+(1+t)^{-\sigma_{n,|\beta|}}\mathcal{D}^{(n,|\beta|)}_{f^{m,\epsilon},n,l_m^*,1}(t)\right\}\nonumber\\
&\lesssim&\sum_{0\leq|\beta|\leq n,\atop |\alpha|+|\beta|=n}\frac1{4C_\beta}\sum_{i=1}^3\frac{d}{dt}\left\|\left\langle \partial_{\beta}\left\{w^2_{l^*_m-|\beta|,1}\partial_\beta\left[v_i \mu^{1/2}\right]\right\}, \partial^\alpha{\{\bf I_+-P_+\}}f^{m,\epsilon}-\partial^\alpha{\{\bf I_--P_-\}}f^{m,\epsilon}\right\rangle\right\|^2 \nonumber\\
&&+\sum_{0\leq k\leq \min\{|\beta|+1,n\},\atop 0\leq|\beta|\leq n,k\leq j\leq n}(1+t)^{-\sigma_{N_m+1,|\beta|}+2+2\vartheta}\nonumber\\
&&\times\left\{\left\|\nabla_x[E^{P,\epsilon},E^{j_1,\epsilon},\epsilon^m
 E^{m,\epsilon},\epsilon^{m+1} B^{m,\epsilon}]\right\|^2_{L^\infty_x}
  +\left\|\nabla^3_x[E^{P,\epsilon},E^{j_1,\epsilon},\epsilon^m
 E^{m,\epsilon},\epsilon^{m+1} B^{m,\epsilon}]\right\|^2_{H^{N_m-3}_x}\right\}\nonumber\\
  &&\mathcal{D}^{(j,k)}_{f^{m,\epsilon},j,l^*_m,1}(t)
\nonumber\\
&&+\left\{\mathcal{E}_{f^{P,\epsilon},N_m+1,-\frac{l_m^*}\gamma,-\gamma}(t)
+\sum_{1\leq j_1\leq m-1}\mathcal{E}_{f^{j_1,\epsilon},N_m+1,-\frac{l_m^*}\gamma,-\gamma}(t)\right\}\mathcal{D}_{f^{m,\epsilon},N_m+1}(t)\nonumber\\
 &&+
  \mathcal{E}_{f^{m,\epsilon},N_m+1}(t)\left\{\mathcal{D}_{f^{P,\epsilon},N_m+2,N_m+3-\frac{l^*_m+1-N_m}\gamma,-\gamma}(t)+\sum_{i=1}^{m-1}\mathcal{D}_{f^{i,\epsilon},N_m+2,N_m+2-\frac{l^*_m+1-N_m}\gamma,-\gamma}(t)\right\}\nonumber\\
 &&+\left\{\mathcal{E}_{f^{P,\epsilon},N_m+1,-\frac{l_m^*}\gamma,-\gamma}(t)
   +\sum_{1\leq j_1\leq m-1}\mathcal{E}_{f^{j_1,\epsilon},N_m+1,-\frac{l_m^*}\gamma,-\gamma}(t)\right\}
\sum_{0\leq k\leq |\beta|\atop k\leq j\leq n}\mathcal{D}^{(j,k)}_{f^{m,\epsilon},j,l_m^*,1}(t)\nonumber\\
  &&+\sum_{0\leq|\beta|\leq n,0\leq k\leq |\beta|\atop k\leq j\leq n}(1+t)^{-\sigma_{N_m+1,|\beta|}}\mathcal{E}^{(j,k)}_{f^{m,\epsilon},j,l_m^*,1}(t)\nonumber\\
  &&\times\left\{\mathcal{D}_{f^{P,\epsilon},N_m+1,-\frac{l_m^*}\gamma,-\gamma}(t)
  +\sum_{1\leq j_1\leq m-1}
  \mathcal{D}_{f^{j_1,\epsilon},N_m+1,-\frac{l_m^*}\gamma,-\gamma}(t)\right\}\nonumber\\
   &&+\sum_{j_1+j_2\geq m,\atop 0<j_1,j_2<m}\left\{\mathcal{E}_{f^{j_2,\epsilon},N_m+2,-\frac{l_m^*+2}\gamma+\frac12,-\gamma}(t)
    \mathcal{D}_{f^{j_1,\epsilon},N_m+1,-\frac{l_m^*+2}\gamma+\frac12,-\gamma}(t)\right\}\\
&&+\sum_{0\leq|\beta|\leq n}(1+t)^{-\sigma_{N_m+1,|\beta|}}\mathcal{E}_{f^{m,\epsilon},N_m,N_m,-\gamma}(t)\sum_{0\leq k\leq |\beta|\atop k\leq j\leq N_m+1}\mathcal{D}^{(j,k)}_{f^{m,\epsilon},j,l_m^*,1}(t)\nonumber\\
  &&+\sum_{0\leq|\beta|\leq n}(1+t)^{-\sigma_{N_m+1,|\beta|}}\sum_{k_1+k_2\leq |\beta|}\left\{\sum_{0\leq j_1\leq n}\mathcal{E}^{(j_1,k_1)}_{f^{m,\epsilon},j_1,l_m^*,1}(t)
  \times\sum_{0\leq j_2\leq n}\mathcal{D}^{(j_2,k_1)}_{f^{m,\epsilon},j_2,l_m^*,1}(t)\right\}\nonumber\\
&&+\mathcal{E}_{f^{m,\epsilon},N_m+1}(t)\mathcal{E}^1_{f^{m,\epsilon},N^0_m,l^*_m+1-\frac\gamma2,1}(t)+\mathcal{E}_{f^{m,\epsilon},N_m,N_m,-\gamma}(t)\mathcal{D}_{f^{m,\epsilon},N_m+1}(t)
\nonumber\\
  &&+\mathcal{D}_{f^{m,\epsilon},n}(t)+\|\nabla^{n+1}f^{m,\epsilon}\|^2_\nu.\nonumber
\end{eqnarray}

\noindent{\bf Step 3:} Based on the above two steps, one has \eqref{lemma-f-R-high} by a proper linear combination of \eqref{f-i-N_m+1-high} and \eqref{f-R-N_m-less-high},
which completes the proof of this lemma.
\end{proof}
Moreover, we can deduce by applying a similar strategy that
\begin{lemma}
Under the assumptions in the above lemma, then
 \begin{eqnarray}\label{lemma-f-R-low-1}
 &&\frac{d}{dt}\left\{\sum_{0\leq n\leq N^0_m,\atop 0\leq|\beta|\leq n }(1+t)^{-\sigma_{n,|\beta|}}\mathcal{E}^{(n,|\beta|)}_{f^{m,\epsilon},n,l_m^\sharp,1}(t)\right\}\nonumber\\
&&  +\sum_{0\leq n\leq N^0_m,\atop 0\leq|\beta|\leq n }\left\{(1+t)^{-\sigma_{n,|\beta|}}\mathcal{E}^{(n,|\beta|)}_{f^{m,\epsilon},n,l_m^\sharp,1}(t)
+(1+t)^{-\sigma_{n,|\beta|}}\mathcal{D}^{(n,|\beta|)}_{f^{m,\epsilon},n,l_m^\sharp,1}(t)\right\}\\
&\lesssim&\sum_{0\leq n\leq N^0_m,\atop |\alpha|+|\beta|=n}\frac1{4C_\beta}\sum_{i=1}^3\frac{d}{dt}\left\|\left\langle \partial_{\beta}\left\{w^2_{l^*_m-|\beta|,1}\partial_\beta\left[v_i \mu^{1/2}\right]\right\}, \partial^\alpha{\{\bf I_+-P_+\}}f^{m,\epsilon}(t)-\partial^\alpha{\{\bf I_--P_-\}}f^{m,\epsilon}(t)\right\rangle\right\|^2 \nonumber\\
&&+\sum_{0\leq n\leq N^0_m,\atop 0\leq|\beta|\leq n }\sum_{0\leq k\leq \min\{|\beta|+1,n\},\atop k\leq j\leq n}(1+t)^{-\sigma_{n,|\beta|}+2+2\vartheta}\nonumber\\
&&\times\left\{\left\|\nabla_x\left[E^{P,\epsilon},E^{j_1,\epsilon},\epsilon^m
 E^{m,\epsilon},\epsilon^{m+1} B^{m,\epsilon}\right](t)\right\|^2_{L^\infty_x}
  +\left\|\nabla^3_x\left[E^{P,\epsilon},E^{j_1,\epsilon},\epsilon^m
 E^{m,\epsilon},\epsilon^{m+1} B^{m,\epsilon}\right](t)\right\|^2_{H^{N_m-3}_x}\right\}\nonumber\\
  &&\times\mathcal{D}^{(j,k)}_{f^{m,\epsilon},j,l_m^\sharp,1}(t)\nonumber\\
&&
+\left\{\mathcal{E}_{f^{P,\epsilon},N_m+1,-\frac{l_m^\sharp}\gamma,-\gamma}(t)
+\sum_{1\leq j_1\leq m-1}\mathcal{E}_{f^{j_1,\epsilon},N_m+1,-\frac{l_m^\sharp}\gamma,-\gamma}(t)\right\}\mathcal{D}_{f^{m,\epsilon},N_m+1}(t)\nonumber\\
 &&+
  \mathcal{E}_{f^{m,\epsilon},N_m+1}(t)\left\{\mathcal{D}_{f^{P,\epsilon},N_m+2,N_m+3-\frac{l_m^\sharp+1-N_m}\gamma,-\gamma}(t)+\sum_{i=1}^{m-1}\mathcal{D}_{f^{i,\epsilon},N_m+2,N_m+2-\frac{l_m^\sharp+1-N_m}\gamma,-\gamma}(t)\right\}\nonumber\\
 &&+\left\{\mathcal{E}_{f^{P,\epsilon},N_m+1,-\frac{l_m^\sharp}\gamma,-\gamma}(t)
   +\sum_{1\leq j_1\leq m-1}\mathcal{E}_{f^{j_1,\epsilon},N_m+1,-\frac{l_m^\sharp}\gamma,-\gamma}(t)\right\}
\sum_{0\leq k\leq |\beta|\atop k\leq j\leq n}\mathcal{D}^{(j,k)}_{f^{m,\epsilon},j,l_m^\sharp,1}(t)\nonumber\\
  &&+\sum_{0\leq n\leq N^0_m,\atop 0\leq|\beta|\leq n }(1+t)^{-\sigma_{n,|\beta|}}\sum_{0\leq k\leq |\beta|\atop k\leq j\leq n}\mathcal{E}^{(j,k)}_{f^{m,\epsilon},j,l_m^\sharp,1}(t)\nonumber\\
  &&\times\left\{\mathcal{D}_{f^{P,\epsilon},N_m+1,-\frac{l_m^\sharp}\gamma,-\gamma}(t)
  +\sum_{1\leq j_1\leq m-1}
  \mathcal{D}_{f^{j_1,\epsilon},N_m+1,-\frac{l_m^\sharp}\gamma,-\gamma}(t)\right\}\nonumber\\
   &&+\sum_{j_1+j_2\geq m,\atop 0<j_1,j_2<m}\left\{\mathcal{E}_{f^{j_2,\epsilon},N_m+2,-\frac{l_m^\sharp+2}\gamma+\frac12,-\gamma}(t)
    \mathcal{D}_{f^{j_1,\epsilon},N_m+1,-\frac{l_m^\sharp+2}\gamma+\frac12,-\gamma}(t)\right\}\nonumber\\
&&+\sum_{0\leq n\leq N^0_m,\atop 0\leq|\beta|\leq n }(1+t)^{-\sigma_{n,|\beta|}}\mathcal{E}_{f^{m,\epsilon},N_m,N_m,-\gamma}(t)\sum_{0\leq k\leq |\beta|\atop k\leq j\leq N_m+1}\mathcal{D}^{(j,k)}_{f^{m,\epsilon},j,l_m^\sharp,1}(t)\nonumber\\
  &&+\sum_{0\leq n\leq N^0_m,\atop 0\leq|\beta|\leq n }(1+t)^{-\sigma_{n,|\beta|}}\nonumber\\
  &&\times\sum_{k_1+k_2\leq |\beta|}\left\{\sum_{0\leq j_1\leq \max\{n,2\}}\mathcal{E}^{(j_1,k_1)}_{f^{m,\epsilon},j_1,l_m^\sharp,1}(t)
  \times\sum_{0\leq j_2\leq \max\{n,2\}}\mathcal{D}^{(j_2,k_1)}_{f^{m,\epsilon},j_2,l_m^\sharp,1}(t)\right\}\nonumber\\
  &&+\mathcal{E}_{f^{m,\epsilon},N^0_m,N^0_m,-\gamma}(t)\mathcal{D}_{f^{m,\epsilon},N^0_m}(t)+\mathcal{D}_{f^{m,\epsilon},N^0_m+1}(t).\nonumber
\end{eqnarray}
\end{lemma}

\subsection{The proof of Theorem \ref{Th1.3}}
Now we are ready to close the {\it a priori} estimates

\begin{eqnarray}\label{def-priori-total-estimates}
  &&\Xi_{Total}(t)\nonumber\\
 &\equiv&\sup_{
 0\leq\tau\leq t}\left\{  \sum_{0\leq n\leq N^0_m,\atop 0\leq k\leq n}(1+t)^{-\sigma_{n,k}}\mathcal{E}^{(n,k)}_{f^{m,\epsilon},n,l_m^\sharp,1}(\tau)+\sum_{N^0_m+1\leq n\leq N_m+1,\atop 0\leq k\leq n}(1+t)^{-\sigma_{n,k}}\mathcal{E}^{(n,k)}_{f^{m,\epsilon},n,l_m^*,1}(\tau)\right.\nonumber\\[3mm]
 &&\left.\quad\quad\quad\quad\quad+\mathcal{E}_{f^{m,\epsilon},N_m+1}(\tau)+\mathcal{E}_{f^{m,\epsilon},N_m,l_m,-\gamma}(\tau)+\overline{\mathcal{E}}_{f^{m,\epsilon},N^0_m,l^0_m,-\gamma}(\tau)\right.\\[3mm]
 &&\left.\quad\quad\quad\quad\quad+\sum_{0<i<m}\sum_{0\leq n\leq N_i,\atop 0\leq k\leq n}(1+\tau)^{-{\sigma}_{n,k}}\mathcal{E}^{(n,k)}_{f^{i,\epsilon},n,l_i^*,1}(\tau)
 +\sum_{0<i<m}\overline{\mathcal{E}}_{f^{i,\epsilon},N_i,l_i,-\gamma}(\tau)\right.\nonumber\\[3mm]
 &&\left.\quad\quad\quad\quad\quad
 +\sum_{0\leq n\leq N_P,\atop 0\leq k\leq n}(1+\tau)^{-{\sigma}_{n,k}}{\mathcal{E}}^{(n,k)}_{f^{P,\epsilon},N_P,l^*_P,1}(\tau)+\overline{\mathcal{E}}_{f^{P,\epsilon},N_P,l_P,-\gamma}(\tau)\right\}\lesssim M, \nonumber
\end{eqnarray}
where all parameters satisfy the assumptions listed in Lemma \ref{lemma9} and $M$ is a sufficiently small positive constant independent of $\epsilon$.

In fact, we can get that
\begin{lemma}\label{lemma9}
Under the a priori estimates \eqref{def-priori-total-estimates} and the assumptions listed in Theorem \ref{Th1.3}, then it holds that
\begin{eqnarray}\label{end}
\Xi_{\textrm{Total}}(t)\lesssim& \mathbb{Y}_{\textrm{Total}, 0}^2
\end{eqnarray}
for all $0\leq t\leq T$. Here $\mathbb{Y}_{\textrm{Total}, 0}$ is defined in \eqref{inital-condtions-Total}.
\end{lemma}
\begin{proof}
In fact, firstly, we have from Proposition \ref{Th1.4} and Proposition \ref{lemma-fi-end} under the {\it a priori} estimates \eqref{def-priori-total-estimates} that
\begin{eqnarray}\label{f-R-end-0}
  &&\frac{d}{dt}\mathcal{\overline{E}}_{f^{P,\epsilon},N_P,l_P,-\gamma}(t)+\mathcal{\overline{D}}_{f^{P,\epsilon},N_P,l_P,-\gamma}(t)\lesssim 0,\\
  &&\frac{d}{dt}\mathcal{\overline{E}}_{f^{j_1,\epsilon},N_{j_1},l_{j_1},-\gamma}(t)+\mathcal{\overline{D}}_{f^{j_1,\epsilon},N_{j_1},l_{j_1},-\gamma}(t)\lesssim 0,\  1\leq j\leq m-1,\nonumber
\end{eqnarray}
from which, we can deduce that:
\begin{itemize}
\item[(i).] If we take
$$
\sigma_{n,0}=0, n\leq N_m, \ \ \sigma_{N_m+1,0}=\frac{1+\epsilon_0}2,
$$
and notice that
$$
\sigma_{n,k}-\sigma_{n,k-1}=\frac{2(1+\gamma)}{\gamma-2}(1+\vartheta),
$$
we can deduce that
$$
 \displaystyle\max_{ N^0_m+1\leq n\leq N_m,\atop  0\leq k \leq n}\{\sigma_{n,k}\}=\sigma_{N_m,N_m},\quad
\max_{ 0\leq n\leq N^0_m,\atop  0\leq k\leq n}\{\sigma_{n,k}\}=\sigma_{N^0_m,N^0_m};
$$
\item [(ii).] If we take  $\widetilde{\ell}_m\geq\frac\gamma2+\frac{(2-4\gamma)\sigma_{N_m,N_m}}{1+\varrho}$ and $l_1^*\geq \widetilde{\ell}_1-\frac\gamma2-\gamma l_1$,
then ${\theta}=\frac{2-4\gamma}{2\widetilde{\ell}_m-\gamma}\leq\frac{1+\varrho}{\sigma_{N_m,N_m}}$ and we have from Lemma \ref{Lemma1-R} that
\begin{eqnarray*}
  &&\sum_{1\leq j_1\leq m-1}\left\|\nabla_x\left[E^{P,\epsilon}(t),E^{j_1,\epsilon}(t)\right]\right\|_{H^{N_m}_x}^{\frac2{{\theta}}}\sum_{0\leq k\leq N_m,\atop k\leq j\leq N_m}\mathcal{D}^{(j,k)}_{f^{m,\epsilon},j,-\gamma l_m+\widetilde{l}_m,1}(t)\\
  &&+\left\| \nabla_x\left[\epsilon^m E^{m,\epsilon}(t),\epsilon^{m+1} B^{m,\epsilon}(t)\right]\right\|_{H^{N_m^0-1}_x}^{\frac2{{\theta}}}\sum_{0\leq k\leq N_m,\atop k\leq j\leq N_m}\mathcal{D}^{(j,k)}_{f^{m,\epsilon},j,\widetilde{\ell}_m-\gamma l_m+1+\frac\gamma2,1}(t)\nonumber\\
&\lesssim&\left\{\Xi_{\textrm{Total}}(t)\right\}^{\frac{1}{{\theta}}}
(1+t)^{-\frac{1+\varrho}{{\theta}}}\sum_{0\leq j\leq N_m,\atop
0\leq k\leq j}\mathcal{D}_{f^{m,\epsilon},j,l_m^*,1}^{(j,k)}(t)\nonumber\\
&\lesssim&\left\{\Xi_{\textrm{Total}}(t)\right\}^{\frac{1}{{\theta}}}
(1+t)^{-\sigma_{N_m,N_m}}\sum_{0\leq j\leq N_m,\atop
0\leq k\leq j}\mathcal{D}_{f^{m,\epsilon},j,l_m^*,1}^{(j,k)}(t).\nonumber
\end{eqnarray*}
Thus we have from Lemma \ref{lemma-f-m} and \eqref{f-R-end-0} that
\begin{eqnarray}\label{f-R-end-1}
&&\frac{d}{dt}\left\{\mathcal{E}_{f^{m,\epsilon},N_m,l_m,-\gamma}(t)+\mathcal{\overline{E}}_{f^{P,\epsilon},N_P,l_P,-\gamma}(t)+\sum_{1\leq j_1\leq m-1}\mathcal{\overline{E}}_{f^{j_1,\epsilon},N_{j_1},l_{j_1},-\gamma}(t)\right\}
\nonumber\\
&&+\mathcal{D}_{f^{m,\epsilon},N_m,l_m,-\gamma}(t)+\mathcal{\overline{D}}_{f^{P,\epsilon},N_P,l_P,-\gamma}(t)+\sum_{1\leq j_1\leq m-1}\mathcal{\overline{D}}_{f^{j_1,\epsilon},N_{j_1},l_{j_1},-\gamma}(t)\nonumber\\
 &\lesssim&
  \left\|\nabla^{N_m+1}f^{m,\epsilon}(t)\right\|^2_\nu   +\left\{\Xi_{\textrm{Total}}(t)\right\}^{\frac{1}{\theta}}
(1+t)^{-\sigma_{N_m,N_m}}\sum_{0\leq j\leq N_m,\atop
0\leq k\leq j}\mathcal{D}_{f^{m,\epsilon},j,l_m^*,1}^{(j,k)}(t)\\
  &&+\mathcal{E}_{f^{m,\epsilon},N_m+1}(t) \mathcal{E}^1_{f^{m,\epsilon},N^0_m,l_m+\frac32-\frac1\gamma,1}(t) .\nonumber
  \end{eqnarray}

\item [(iii).] Similarly, we have by taking $l_m^\sharp\geq \widetilde{\ell}_m-\gamma l^0_m+1+\frac\gamma2$ that
\begin{eqnarray}\label{f-R-end-2}
&&\frac{d}{dt}\left\{\mathcal{\overline{E}}_{f^{m,\epsilon},N^0_m,l^0_m,-\gamma}(t)+\mathcal{\overline{E}}_{f^{P,\epsilon},N_P,l_P,-\gamma}(t)+\sum_{1\leq j_1\leq m-1}\mathcal{\overline{E}}_{f^{j_1,\epsilon},N_{j_1},l_{j_1},-\gamma}(t)\right\}
\nonumber\\
&&+\mathcal{\overline{D}}_{f^{m,\epsilon},N^0_m,l^0_m,-\gamma}(t)+\mathcal{\overline{D}}_{f^{P,\epsilon},N_P,l_P,-\gamma}(t)+\sum_{1\leq j_1\leq m-1}\mathcal{\overline{D}}_{f^{j_1,\epsilon},N_{j_1},l_{j_1},-\gamma}(t)
\nonumber\\
&\lesssim&\eta\left\|\nabla^{N_m^0+1}f^{m,\epsilon}(t)\right\|^2_\nu
+\left\{\Xi_{\textrm{Total}}(t)\right\}^{\frac{1}{\theta}}
(1+t)^{-\sigma_{N^0_m,N^0_m}}
\sum_{0\leq j\leq N^0_m,\atop
0\leq k\leq j}\mathcal{D}_{f^{m,\epsilon},j,l_m^\sharp,1}^{(j,k)}(t).
\end{eqnarray}

\item [(iv).] Since Lemma \ref{Lemma1-R} tells us that
    \begin{eqnarray}
&&\left\|\nabla_x\left[\epsilon^m E^{m,\epsilon},\epsilon^{m+1} B^{m,\epsilon}\right](t)\right\|^2_{L^\infty_x} \left\|\nabla_v\nabla^{N_m}_x{\bf\{I-P\}}f^{m,\epsilon}(t)\langle v\rangle^{1-\frac\gamma2}\right\|^2\nonumber\\
&\lesssim&\epsilon^{2m}\epsilon^{-5-2\varrho}(1+t)^{-\frac52-\varrho}\Xi_{\textrm{Total}}(t)(1+t)^{\sigma_{N_m+1,1}}(1+t)^{-\sigma_{N_m+1,1}}\mathcal{D}^{(N_m+1,1)}_{f^{m,\epsilon},N_m+1,l_m^*,1}\nonumber\\
&\lesssim&\epsilon^{2m-5-2\varrho}(1+t)^{-\frac52-\varrho+\sigma_{N_m+1,1}}\Xi_{\textrm{Total}}(t)(1+t)^{-\sigma_{N_m+1,1}}\mathcal{D}^{(N_m+1,1)}_{f^{m,\epsilon},N_m+1,l_m^*,1}\nonumber\\
&\lesssim&(1+t)^{-\frac52-\varrho+\frac{1+\epsilon_0}{2}+\frac{2(1+\gamma)}{\gamma-2}(1+\vartheta)}\Xi_{\textrm{Total}}(t)(1+t)^{-\sigma_{N_m+1,1}}\mathcal{D}^{(N_m+1,1)}_{f^{m,\epsilon},N_m+1,l_m^*,1}\nonumber\\
&\lesssim&\Xi_{\textrm{Total}}(t)(1+t)^{-\sigma_{N_m+1,1}} \mathcal{D}^{(N_m+1,1)}_{f^{m,\epsilon},N_m+1,l_m^*,1},
\end{eqnarray}
where we used the fact that
\[\epsilon^{2m-5-2\varrho}+ (1+t)^{-\frac52-\varrho+\frac{1+\epsilon_0}{2}+\frac{2(1+\gamma)}{\gamma-2}(1+\vartheta)}\lesssim 1\]
provided that
\[m\geq \frac52+\varrho,\ 0<\vartheta\leq \frac{\gamma-2}{2+2\gamma}\left(2+\varrho-\frac{\epsilon_0}2\right)-1,\]
then if we take $l_m^*\geq N_m+2-\frac\gamma2$, we can get from Lemma \ref{lemma-f-R-high-spatial} that

\begin{eqnarray}\label{f-R-end-3}
  &&\frac{d}{dt}\left\{\mathcal{E}_{f^{m,\epsilon},N_m+1}(t)+\mathcal{\overline{E}}_{f^{P,\epsilon},N_P,l_P,-\gamma}(t)+\sum_{1\leq j_1\leq m-1}\mathcal{\overline{E}}_{f^{j_1,\epsilon},N_{j_1},l_{j_1},-\gamma}(t)\right\}
\nonumber\\
&&+\mathcal{D}_{f^{m,\epsilon},N_m+1}(t)+\mathcal{\overline{E}}_{f^{P,\epsilon},N_P,l_P,-\gamma}(t)+\sum_{1\leq j_1\leq m-1}\mathcal{\overline{D}}_{f^{j_1,\epsilon},N_{j_1},l_{j_1},-\gamma}(t)\nonumber\\
 &\lesssim&
 \Xi_{\textrm{Total}}(t)(1+t)^{-\sigma_{N_m+1,1}}
  \sum_{0\leq k\leq 1}\mathcal{D}^{(N_m+1,k)}_{f^{m,\epsilon},N_m+1,l_m^*,1}.
    \end{eqnarray}

    \item[(v).] By Lemma \ref{Lemma1-R}, one has
    \begin{eqnarray}
 &&\sum_{N_m^0+1\leq n\leq N_m+1,\atop 0\leq|\beta|\leq n }(1+t)^{-\sigma_{n,|\beta|}+2+2\vartheta}\nonumber\\
&&\times\sum_{1\leq j_1\leq m-1}
  \left\|\nabla_x\left[\epsilon^m
 E^{m,\epsilon},\epsilon^{m+1} B^{m,\epsilon}\right](t)\right\|^2_{L^\infty_x}\sum_{0\leq k\leq \min\{|\beta|+1,n\},\atop k\leq j\leq n}\mathcal{D}^{(j,k)}_{f^{m,\epsilon},j,l^*_m,1}(t)\nonumber\\
 &\lesssim&\epsilon^{2m}\sum_{N_m^0+1\leq n\leq N_m+1,\atop 0\leq|\beta|\leq n }(1+t)^{-\sigma_{n,|\beta|}+2+2\vartheta}\epsilon^{-5-2\varrho}(1+t)^{-\frac52-\varrho}\nonumber\\
  &&\times\Xi_{\textrm{Total}}(t)
 \sum_{0\leq k\leq \min\{|\beta|+1,n\},\atop k\leq j\leq n}\mathcal{D}^{(j,k)}_{f^{m,\epsilon},j,l^*_m,1}(t)\nonumber\\
  &\lesssim&\epsilon^{2m-5-2\varrho}\sum_{N_m^0+1\leq n\leq N_m+1,\atop 0\leq|\beta|\leq n }(1+t)^{-\sigma_{n,|\beta|+1}}(1+t)^{\frac{2(1+\gamma)}{\gamma-2}(1+\vartheta)+2+2\vartheta-\frac52-\varrho}\nonumber\\
  &&\times\Xi_{\textrm{Total}}(t)
 \sum_{0\leq k\leq \min\{|\beta|+1,n\},\atop k\leq j\leq n}\mathcal{D}^{(j,k)}_{f^{m,\epsilon},j,l^*_m,1}(t)\nonumber\\
 &\lesssim&\epsilon^{2m-5-2\varrho}\sum_{N_m^0+1\leq n\leq N_m+1,\atop 0\leq|\beta|\leq n }(1+t)^{-\sigma_{n,|\beta|+1}}\nonumber\\
  &&\times\Xi_{\textrm{Total}}(t)\sum_{0\leq k\leq \min\{|\beta|+1,n\},\atop k\leq j\leq n}\mathcal{D}^{(j,k)}_{f^{m,\epsilon},j,l^*_m,1}(t)\nonumber\\
 &\lesssim&\Xi_{\textrm{Total}}(t)
   \sum_{0\leq n\leq N_m+1,\atop 0\leq k\leq n}(1+t)^{-\sigma_{n,k}}\mathcal{D}^{(n,k)}_{f^{m,\epsilon},n,l^*_m,1}(t)\nonumber,
\end{eqnarray}
where we have used the fact that
\[\epsilon^{2m-5-2\varrho}+ (1+t)^{\frac{2(1+\gamma)}{\gamma-2}(1+\vartheta)+2+2\vartheta-\frac52-\varrho}\lesssim 1\]
provided that
\[m\geq \frac52+\varrho,\ 0<\vartheta\leq \frac{\gamma-2}{4\gamma-2}\left(\frac52+\varrho\right)-1.\]
Thus one deduces that
    \begin{eqnarray}\label{f-R-end-4}
  &&\frac{d}{dt}\left\{\sum_{N_m^0+1\leq n\leq N_m+1,\atop 0\leq k\leq n }(1+t)^{-\sigma_{n,k}}\mathcal{E}^{(n,k)}_{f^{m,\epsilon},n,l_m^*,1}(t)+\mathcal{\overline{E}}_{f^{P,\epsilon},N_P,l_P,-\gamma}(t)+\sum_{1\leq j_1\leq m-1}\mathcal{\overline{E}}_{f^{j_1,\epsilon},N_{j_1},l_{j_1},-\gamma}(t)\right\}
\nonumber\\
&&  +\sum_{N_m^0+1\leq n\leq N_m+1,\atop 0\leq k\leq n }
(1+t)^{-\sigma_{n,k}}\mathcal{D}^{(n,k)}_{f^{m,\epsilon},n,l_m^*,1}(t)+\mathcal{\overline{E}}_{f^{P,\epsilon},N_P,l_P,-\gamma}(t)+\sum_{1\leq j_1\leq m-1}\mathcal{\overline{D}}_{f^{j_1,\epsilon},N_{j_1},l_{j_1},-\gamma}(t)\nonumber\\
 &\lesssim&
\mathcal{D}_{f^{m,\epsilon},N_m+1}(t)+\eta(1+t)^{-1-\epsilon_0}\left\|\epsilon \nabla^{N_m+1}E^{m,\epsilon}(t)\right\|^2 +\mathcal{E}_{f^{m,\epsilon},N_m+1}(t)\mathcal{E}^1_{f^{m,\epsilon},N^0_m,l^*_m+1-\frac\gamma2,1}(t)
\nonumber\\
&&+\Xi_{\textrm{Total}}(t)
   \sum_{0\leq n\leq N^0_m,\atop 0\leq k\leq n}(1+t)^{-\sigma_{n,k}}\mathcal{D}^{(n,k)}_{f^{m,\epsilon},n,l^*_m,1}(t),
\end{eqnarray}
where we have used the fact that
\begin{eqnarray}
 &&\sum_{N_m^0+1\leq n\leq N_m+1,\atop 0\leq|\beta|\leq n }(1+t)^{-\sigma_{n,|\beta|}}\sum_{k_1+k_2\leq |\beta|}\left\{\sum_{0\leq j_1\leq n}\mathcal{E}^{(j_1,k_1)}_{f^{m,\epsilon},j_1,l_m^*,1}(t)
  \times\sum_{0\leq j_2\leq n}\mathcal{D}^{(j_2,k_1)}_{f^{m,\epsilon},j_2,l_m^*,1}(t)\right\}\nonumber\\
  &\lesssim&\sum_{N_m^0+1\leq n\leq N_m+1,\atop 0\leq|\beta|\leq n }\sum_{k_1+k_2\leq |\beta|}(1+t)^{-\sigma_{n,k_1}}(1+t)^{-\sigma_{n,k_2}}\nonumber\\
  &&\times\left\{\sum_{0\leq j_1\leq n}\mathcal{E}^{(j_1,k_1)}_{f^{m,\epsilon},j_1,l_m^*,1}(t)
  \times\sum_{0\leq j_2\leq n}\mathcal{D}^{(j_2,k_1)}_{f^{m,\epsilon},j_2,l_m^*,1}(t)\right\}\nonumber\\
    &\lesssim&\sum_{N_m^0+1\leq n\leq N_m+1,\atop 0\leq|\beta|\leq n ,k_1+k_2\leq |\beta|}\left\{\sum_{0\leq j_1\leq n}(1+t)^{-\sigma_{j_1,k_1}}\mathcal{E}^{(j_1,k_1)}_{f^{m,\epsilon},j_1,l_m^*,1}(t)
  \times\sum_{0\leq j_2\leq n}(1+t)^{-\sigma_{j_2,k_2}}\mathcal{D}^{(j_2,k_1)}_{f^{m,\epsilon},j_2,l_m^*,1}(t)\right\}\nonumber\\
  &\lesssim&\Xi_{\textrm{Total}}(t)
   \sum_{0\leq n\leq N_m+1,\atop 0\leq k\leq n}(1+t)^{-\sigma_{n,k}}\mathcal{D}^{(n,k)}_{f^{m,\epsilon},n,l^*_m,1}(t)\nonumber.
\end{eqnarray}
\item [(vi)] In spirit of \eqref{f-R-end-4}, one also deduce
    \begin{eqnarray}\label{f-R-end-5}
  &&\frac{d}{dt}\left\{\sum_{0\leq n\leq N^0_m,\atop 0\leq k\leq n }(1+t)^{-\sigma_{n,k}}\mathcal{E}^{(n,k)}_{f^{m,\epsilon},n,l_m^\sharp,1}(t)+\mathcal{\overline{E}}_{f^{P,\epsilon},N_P,l_P,-\gamma}(t)+\sum_{1\leq j_1\leq m-1}\mathcal{\overline{E}}_{f^{j_1,\epsilon},N_{j_1},l_{j_1},-\gamma}(t)\right\}
\nonumber\\
&&  +\sum_{0\leq n\leq N^0_m,\atop 0\leq k\leq n }\left\{(1+t)^{-\sigma_{n,k}}\mathcal{E}^{(n,|\beta|)}_{f^{m,\epsilon},n,l_m^*,1}(t)
+(1+t)^{-\sigma_{n,k}}\mathcal{D}^{(n,k)}_{f^{m,\epsilon},n,l_m^\sharp,1}(t)\right\}\\
&\lesssim&
\mathcal{D}_{f^{m,\epsilon},N^0_m+1}(t) .\nonumber
\end{eqnarray}
\item[(vii).] Multiplying $(1+t)^{-\epsilon_0}$ to \eqref{f-R-end-3}, one gets
   \begin{eqnarray}\label{f-R-end-6}
    &&\frac{d}{dt}\left\{(1+t)^{-\epsilon_0}\left\{
    \mathcal{E}_{f^{m,\epsilon},N_m+1}(t)
    +\mathcal{\overline{E}}_{f^{P,\epsilon},N_P,l_P,-\gamma}(t)+\sum_{1\leq j_1\leq m-1}\mathcal{\overline{E}}_{f^{j_1,\epsilon},N_{j_1},l_{j_1},-\gamma}(t)\right\}\right\}
\nonumber\\
&&+(1+t)^{-1-\epsilon_0}\mathcal{E}_{f^{m,\epsilon},N_m+1}(t)
    +(1+t)^{-\epsilon_0}\mathcal{D}_{f^{m,\epsilon},N_m+1}(t)\\
    &&+\mathcal{\overline{E}}_{f^{P,\epsilon},N_P,l_P,-\gamma}(t)+\sum_{1\leq j_1\leq m-1}\mathcal{\overline{D}}_{f^{j_1,\epsilon},N_{j_1},l_{j_1},-\gamma}(t)\nonumber\\
   &\lesssim&\Xi_{\textrm{Total}}(t)(1+t)^{-\sigma_{N_m+1,1}}
  \sum_{0\leq k\leq 1}\mathcal{D}^{(N_m+1,k)}_{f^{m,\epsilon},N_m+1,N_m+2-\frac\gamma2,1}\nonumber.
    \end{eqnarray}

\end{itemize}

With the above preparations in hand, by taking a proper linear combination of \eqref{f-R-end-1}, \eqref{f-R-end-2}, \eqref{f-R-end-3}, \eqref{f-R-end-4}, \eqref{f-R-end-5}, and \eqref{f-R-end-6} and by
applying the smallness of $\Xi_{\textrm{Total}}(t)$ and $\eta$, we can deduce by taking the time integration with respect to $t$ from $0$ to $t$ that the resulting differential inequality that
\eqref{end} holds. Thus the proof of Lemma \ref{lemma9} is complete.
\end{proof}

Now we turn to prove Theorem \ref{Th1.3}. To this end, recall the definition of the $\Xi_{\textrm{Total}}(t)-$norm. Lemma \ref{lemma9} tells that for the local solution $[f^{m,\epsilon}(t,x,v), E^{m,\epsilon}(t,x), B^{m,\epsilon}(t,x)]$ to the Cauchy problem (\ref{F-R-s-VMB-Remain-1})-(\ref{IC-s-Compatibility}) defined on the time interval $[0,T]$ for some $0<T\leq +\infty$, if
\begin{equation*}
\Xi_{\textrm{Total}}(t)\leq M
\end{equation*}
holds for all $t\in[0,T]$ and some $\epsilon-$independent positive constant $M$, then there exists a sufficiently small $\epsilon-$independent positive constant $\delta_0>0$ such that if
\begin{equation*}
M\leq \delta^2_0,
\end{equation*}
there exists a $\epsilon-$ independent positive constant $\overline{C}>0$ such that
\begin{equation*}
\Xi_{\textrm{Total}}(t)\leq \overline{C}^2 \mathbb{Y}_{\textrm{Total}, 0}^2
\end{equation*}
holds for all $0\leq t\leq T$.

Thus if the initial perturbation $ \mathbb{Y}_{\textrm{Total}, 0}$ is assumed to be sufficiently small such that
\begin{equation*}
 \mathbb{Y}_{\textrm{Total}, 0}\leq \frac{\delta_0}{\overline{C}},
\end{equation*}
then the global existence follows by combining the local solvability result with the continuation argument in the usual way.
To obtain \eqref{Th2-2},
recall that the expansion form \eqref{F-S-Expansion}, we can deduce that from \eqref{end}
\begin{equation}\label{end-1}
  \sum_{|\alpha|+|\beta|\leq N_m}\left\|\partial^\alpha_\beta \left\{\frac{F^\epsilon-F^P}{\mu^{\frac12}}\right\}\right\|^2+\|E^\epsilon-E^{P,\epsilon}\|^2_{H^{N_m+1}_x}+\|B^\epsilon-B^P\|^2_{H^{N_m+1}_x}\lesssim \epsilon^{2m} \mathbb{Y}_{\textrm{Total}, 0}^2,
\end{equation}
that is what we want. Thus the proof of Theorem \ref{Th1.3} is complete.\qed
\section{Appendix}
\subsection{Appendix A}

In this section, we collect several basic estimates on the linearized Boltzmann collision operator $L$ and the nonlinear term $\Gamma$ for hard sphere interaction. Since these estimates are well-established in \cite{Guo-CPAM-2002, Guo-Invent Math-2003, Strain-CMP-2006} for hard sphere model and in \cite{Duan_Yang_Zhao-M3AS-2013, Strain-Guo-ARMA-2008} for cutoff potentials, we only state the result and omit the proofs for brevity.

$L$ is locally coercive in the sense that

\begin{lemma}\label{Lemma L}(cf. \cite[Lemma 3.1]{Duan-Lei-Yang-Zhao-CMP-2017})
Let $-3<\gamma\leq 1$, one has
\begin{equation}\label{L_0}
  \langle L f, f\rangle\geq |{\bf \{I-P\}}f|^2_\nu.
\end{equation}
Moreover, let $|\beta|>0$, for $\eta>0$ small enough and any $\ell\in\mathbb{R}, \kappa\geq0, 0<q\ll 1, \vartheta\in\mathbb{R}$, there exists $C_{\eta}>0$ such that
\begin{equation}\label{L_v}
  \left\langle w_{\ell,\kappa}^2\partial_{\beta}{ L}f,\partial_{\beta}f\right\rangle\geq \left|w_{\ell,\kappa}\partial_{\beta}f\right|^2_\nu
  -\eta\sum_{|\beta'|<|\beta|}\left|w_{\ell,\kappa}\partial_{\beta'}\{{\bf I-P}\}f\right|_\nu^2-C_{\eta}\left|\chi_{\{| v|\leq2C_{\eta}\}}f\right|^2
\end{equation}
holds.
\end{lemma}

The second lemma is concerned with the corresponding weighted estimates on the nonlinear term $\Gamma$.
For this purpose, similar to that of \cite{Strain-Guo-ARMA-2008}, we can get that
\begin{eqnarray}\label{gamma-0}
\partial^\alpha_\beta\Gamma_\pm(g_1,g_2)&\equiv&\sum C_\beta^{\beta_0\beta_1\beta_2}C_\alpha^{\alpha_1\alpha_2} {\Gamma}^0_\pm\left(\partial^{\alpha_1}_{\beta_1}g_1,\partial^{\alpha_2}_{\beta_2}g_2\right)\\ \nonumber
&\equiv& \sum C_\beta^{\beta_0\beta_1\beta_2}C_\alpha^{\alpha_1\alpha_2}\int_{\mathbb{R}^3\times\mathbb{S}^2}|v-u|^\gamma{\bf b}(\cos\theta)\partial_{\beta_0}[\mu(u)^\frac 12]\left\{\partial^{\alpha_1}_{\beta_1}g_{1\pm}(v')\partial^{\alpha_2}_{\beta_2}g_{2\pm}(u')\right.\\ \nonumber
&&\left.
+\partial^{\alpha_1}_{\beta_1}g_{1\pm}(v')\partial^{\alpha_2}_{\beta_2}g_{2\mp}(u')
-\partial^{\alpha_1}_{\beta_1}g_{1\pm}(v)\partial^{\alpha_2}_{\beta_2}g_{2\pm}(u)
-\partial^{\alpha_1}_{\beta_1}g_{1\pm}(v)\partial^{\alpha_2}_{\beta_2}g_{2\mp}(u)\right\}d\omega du,
\end{eqnarray}
where $g_i(t,x,v)=[g_{i+}(t,x,v), g_{i-}(t,x,v)]$ $(i=1,2)$ and the summations are taken for all $\beta_0+\beta_1+\beta_2=\beta, \alpha_1+\alpha_2=\alpha$.

\begin{lemma}\label{lemma-nonlinear}(cf. \cite[Lemma 3.2]{Duan-Lei-Yang-Zhao-CMP-2017})
Assume $\kappa\geq 0, \ell\geq 0$.
Let $-3<\gamma\leq 1$, $N\geq4$, $g_i=g_i(t,x,v)=[g_{i+}(t,x,v),g_{i-}(t,x,v)]\ (i=1,2,3)$, $\beta_0+\beta_1+\beta_2=\beta$ and $\alpha_1+\alpha_2=\alpha$, we have
the following results:
\begin{itemize}
\item[(i).] When $|\alpha_1|+|\beta_1|\leq N$, we have
\begin{equation}\label{nonlinear-1}
\begin{array}{rl}
\left\langle w_{\ell,\kappa}^2{\Gamma}^0_\pm\left(\partial^{\alpha_1}_{\beta_1}g_1,\partial^{\alpha_2}_{\beta_2}g_2\right), \partial^{\alpha}_{\beta}g_3\right\rangle
\lesssim\sum\limits_{m\leq2}\left\{
\left|\nabla^m_{v}\left\{\mu^\delta\partial^{\alpha_1}_{\beta_1}g_1\right\}\right|+\left|w_{\ell,\kappa}\partial^{\alpha_1}_{\beta_1}g_1\right|\right\}
\left|w_{\ell,\kappa}\partial^{\alpha_2}_{\beta_2}g_2\right|_{L^2_{\nu}}
\left|w_{\ell,\kappa}\partial^{\alpha}_{\beta}g_3\right|_{L^2_{\nu}}
\end{array}
\end{equation}
or
\begin{equation}\label{nonlinear-2}
\begin{array}{rl}
\left\langle w_{\ell,\kappa}^2{\Gamma}^0_\pm\left(\partial^{\alpha_1}_{\beta_1}g_1,\partial^{\alpha_2}_{\beta_2}g_2\right), \partial^{\alpha}_{\beta}g_3\right\rangle
\lesssim\sum\limits_{m\leq2}\left\{
\left|\nabla^m_{v}\left\{\mu^\delta\partial^{\alpha_2}_{\beta_2}g_2\right\}\right|+\left|w_{\ell,\kappa}\partial^{\alpha_2}_{\beta_2}g_2\right|\right\}
\left|w_{\ell,\kappa}\partial^{\alpha_1}_{\beta_1}g_1\right|_{L^2_{\nu}}
\left|w_{\ell,\kappa}\partial^{\alpha}_{\beta}g_3\right|_{L^2_{\nu}}.
\end{array}
\end{equation}
\item[(ii).]
Set $\varsigma(v)=\langle v\rangle^{-\gamma}\equiv \nu(v)^{-1},~l\geq0$, it holds that
\begin{equation}\label{n-3}
\begin{split}
\left|\varsigma^{l}{\Gamma}(g_1,g_2)\right|^2_{L^2_v}
&\lesssim\displaystyle\sum_{|\beta|\leq2}\left|\varsigma^{l-|\beta|}\partial_{\beta}g_1\right|^2_{L^2_{\nu}}
\left|\varsigma^{l}g_2\right|^2_{L^2_{\nu}},\\
\left|\varsigma^{l}{\Gamma}(g_1,g_2)\right|^2_{L^2_v}
&\lesssim\sum_{|\beta|\leq2}\left|\varsigma^{l}g_1\right|^2_{L^2_{\nu}}
\left|\varsigma^{l-|\beta|}\partial_{\beta}g_2\right|^2_{L^2_{\nu}}.
\end{split}
\end{equation}
\end{itemize}
\end{lemma}
Based on Lemma \ref{lemma-nonlinear}, we can get some weighted estimates on the interactions between $f^{P,\epsilon}(t,x,v), f^{1,\epsilon}(t,x,v),$ $\cdots, f^{m,\epsilon}(t,x,v)$ with respect to the weight $w_{\ell-|\beta|,-\gamma}(t,v)$, whose proof will be omitted for brevity.
\begin{lemma}\label{Lemma-5.3} Take $n\geq 8$ and $w_{l_m-|\beta|,-\gamma}
=\langle v\rangle^{-\gamma(l_m-|\beta|)}e^{\frac{q\langle v\rangle^2}{(1+t)^\vartheta}}$, one has the following estimates:
\begin{itemize}
\item [i)] For $|\alpha|\leq n$ with $j_1+j_2\geq m, 0<j_1,j_2<m$, one has
\begin{eqnarray}\label{0-order-gamma-R}
   &&\left(\partial^\alpha\left\{{\Gamma}\left(f^{P,\epsilon},f^{m,\epsilon}\right)
  +{\Gamma}\left(f^{m,\epsilon},f^{P,\epsilon}\right)+\sum_{1\leq j_1\leq m-1}\epsilon^{j_1}
  \left\{{\Gamma}\left(f^{m,\epsilon},f^{j_1,\epsilon}\right)+{\Gamma}\left(f^{j_1,\epsilon},f^{m,\epsilon}\right)\right\}\right.\right.\nonumber\\
  &&\left.\left.
\quad\quad\quad\quad\quad\quad\quad\quad\quad\quad\quad\quad\quad\quad\quad\quad\quad+\sum_{j_1+j_2\geq m,\atop 0<j_1,j_2<m} \epsilon^{j_1+j_2-m}\Gamma\left(f^{j_1,\epsilon},f^{j_2,\epsilon}\right)\right\}, \partial^\alpha f^{m,\epsilon}\right)\\
  &\lesssim&\mathcal{E}_{f^{P,\epsilon},n-1,n-1,-\gamma}(t)\mathcal{D}_{f^{m,\epsilon},n}(t)+\mathcal{E}_{f^{m,\epsilon},n-1,n-1,-\gamma}(t)\mathcal{D}_{f^{P,\epsilon},n}(t) \nonumber\\
  &&+\sum_{1\leq j_1\leq m-1}
  \left\{\mathcal{E}_{f^{j_1,\epsilon},n-1,n-1,-\gamma}(t)\mathcal{D}_{f^{m,\epsilon},n}(t)
  +\mathcal{E}_{f^{m,\epsilon},n-1,n-1,-\gamma}(t)\mathcal{D}_{f^{j_1,\epsilon},n}(t) \right\}\nonumber\\
  &&+\sum_{j_1+j_2\geq m,\atop 0<j_1,j_2<m}\mathcal{E}_{f^{j_1,\epsilon},n-1,n-1,-\gamma}(t)\mathcal{D}_{f^{j_2,\epsilon},n}(t)
  +\eta\left\|\partial^\alpha{\bf \{I-P\}}f^{m,\epsilon}(t)\langle v\rangle^{\frac\gamma2}\right\|^2,\nonumber
     \end{eqnarray}
     and
     \begin{eqnarray}\label{1-typical-low-non-gamma-R}
   &&\left|\left(\partial^\alpha\left\{\epsilon^{m}\Gamma(f^{m,\epsilon},f^{m,\epsilon})\right\},\partial^\alpha f^{m,\epsilon}\right)\right|\nonumber\\
  &\lesssim&
  \mathcal{E}_{f^{m,\epsilon},n-1,n-1,-\gamma}(t)\mathcal{D}_{f^{m,\epsilon},n}(t)
  +\eta\mathcal{D}_{f^{m,\epsilon},n}(t);
\end{eqnarray}

 \item [ii)] For $1\leq|\alpha|\leq n$, one has
\begin{eqnarray}\label{i-weight-gamma-low-R}
  &&\left(\partial^\alpha\left\{{\Gamma}\left(f^{P,\epsilon},f^{m,\epsilon}\right)
  +{\Gamma}\left(f^{m,\epsilon},f^{P,\epsilon}\right)+\sum_{1\leq j_1\leq m-1}\epsilon^{j_1}
  \left\{{\Gamma}\left(f^{m,\epsilon},f^{j_1,\epsilon}\right)+{\Gamma}\left(f^{j_1,\epsilon},f^{m,\epsilon}\right)\right\}\right.\right.\nonumber\\
  &&\left.\left.
  \quad\quad\quad\quad\quad\quad\quad\quad\quad\quad\quad+\sum_{j_1+j_2\geq m,\atop 0<j_1,j_2<m} \epsilon^{j_1+j_2-m}\Gamma\left(f^{j_1,\epsilon},f^{j_2,\epsilon}\right)\right\}, w^2_{l_m,-\gamma}\partial^\alpha f^{m,\epsilon}\right)\\
  &\lesssim&\mathcal{E}_{f^{P,\epsilon},n,l_m,-\gamma}(t)\mathcal{D}_{f^{m,\epsilon},n,l_m,-\gamma}(t)+\mathcal{E}_{f^{m,\epsilon},n,l_m,-\gamma}(t)
  \mathcal{D}_{f^{P,\epsilon},n,l_m,-\gamma}(t)\nonumber\\
  &&+\sum_{1\leq j_1\leq m-1}\left\{\mathcal{E}_{f^{j_1,\epsilon},n,l_m,-\gamma}(t)\mathcal{D}_{f^{m,\epsilon},n,l_m,-\gamma}(t)
  +\mathcal{E}_{f^{m,\epsilon},n,l_m,-\gamma}(t)
  \mathcal{D}_{f^{j_1,\epsilon},n,l_m,-\gamma}(t)\right\}\nonumber\\
  &&+\sum_{j_1+j_2\geq m,\atop 0<j_1,j_2<m}\mathcal{E}_{f^{j_1,\epsilon},n,l_m,-\gamma}(t)\mathcal{D}_{f^{j_2,\epsilon},n,l_m,-\gamma}(t)+\eta\left\|w_{l_m,-\gamma}\partial^\alpha f^{m,\epsilon}(t)\langle v\rangle^{\frac\gamma2}\right\|^2.\nonumber
     \end{eqnarray}
     and
     \begin{eqnarray}\label{1-w-typical-low-non-gamma-R}
   &&\left|\left(\partial^\alpha\left\{\epsilon^{m}\Gamma(f^{m,\epsilon},f^{m,\epsilon})\right\},w^2_{l_m,-\gamma}\partial^\alpha f^{m,\epsilon}\right)\right|\nonumber\\
  &\lesssim&
  \mathcal{E}_{f^{m,\epsilon},n,l_m,-\gamma}(t)\mathcal{D}_{f^{m,\epsilon},n,l_m,-\gamma}(t)
  +\eta\mathcal{D}_{f^{m,\epsilon},n,l_m,-\gamma}(t);
\end{eqnarray}
 \item [iii)] For $|\alpha|+|\beta|\leq n$, one has
\begin{eqnarray}\label{micro-weight-gamma-low-R}
   &&\left(\partial^\alpha_\beta\left\{{\Gamma}\left(f^{P,\epsilon},f^{m,\epsilon}\right)
  +{\Gamma}\left(f^{m,\epsilon},f^{P,\epsilon}\right)+\sum_{1\leq j_1\leq m-1}\epsilon^{j_1}
  \left\{{\Gamma}\left(f^{m,\epsilon},f^{j_1,\epsilon}\right)+{\Gamma}\left(f^{j_1,\epsilon},f^{m,\epsilon}\right)\right\}\right.\right.\nonumber\\
  &&\left.\left.\quad\quad\quad\quad\quad\quad\quad\quad\quad\quad+\sum_{j_1+j_2\geq m,\atop 0<j_1,j_2<m} \epsilon^{j_1+j_2-m}\Gamma\left(f^{j_1,\epsilon},f^{j_2,\epsilon}\right)\right\}, w^2_{l_m-|\beta|,-\gamma}\partial^\alpha{\bf \{I-P\}} f^{m,\epsilon}\right)\nonumber\\
  &\lesssim&\mathcal{E}_{f^{P,\epsilon},n,l_m,-\gamma}(t)\mathcal{D}_{f^{m,\epsilon},n,l_m,-\gamma}(t)+\mathcal{E}_{f^{m,\epsilon},n,l_m,-\gamma}(t)
  \mathcal{D}_{f^{P,\epsilon},n,l_m,-\gamma}(t)\\
  &&+\sum_{1\leq j_1\leq m-1}\left\{\mathcal{E}_{f^{j_1,\epsilon},n,l_m,-\gamma}(t)\mathcal{D}_{f^{m,\epsilon},n,l_m,-\gamma}(t)
  +\mathcal{E}_{f^{m,\epsilon},n,l_m,-\gamma}(t)
  \mathcal{D}_{f^{j_1,\epsilon},n,l_m,-\gamma}(t)\right\}\nonumber\\
  &&+\sum_{j_1+j_2\geq m,\atop 0<j_1,j_2<m}\mathcal{E}_{f^{j_1,\epsilon},n,l_m,-\gamma}(t)\mathcal{D}_{f^{j_2,\epsilon},n,l_i,-\gamma}(t)
  +\eta\left\|w_{l_m-|\beta|,-\gamma}\partial^\alpha_\beta{\bf \{I-P\}} f^{m,\epsilon}(t)\langle v\rangle^{\frac\gamma2}\right\|^2\nonumber
     \end{eqnarray}
     and
     \begin{eqnarray}\label{2-w-typical-low-non-gamma-R}
 &&\left|\left(\partial^\alpha_\beta\left\{\epsilon^{m}\Gamma(f^{m,\epsilon},f^{m,\epsilon})\right\},w^2_{l_m-|\beta|,-\gamma}\partial^\alpha_\beta{\bf\{I-P\}} f^{m,\epsilon}\right)\right|\nonumber\\
  &\lesssim&
  \mathcal{E}_{f^{m,\epsilon},n,l_m,-\gamma}(t)\mathcal{D}_{f^{m,\epsilon},n,l_m,-\gamma}(t)
  +\eta\mathcal{D}_{f^{m,\epsilon},n,l_m,-\gamma}(t).
\end{eqnarray}
\end{itemize}
\end{lemma}
For the corresponding estimates with respect to the weight $w_{\ell_m^*-|\beta|,1}(t,v)$, we have
\begin{lemma}\label{Lemma-5.4} Take $n\geq 8$ and $w_{\ell_m^*-|\beta|,1}(t,v)
=\langle v\rangle^{\ell_m^*-|\beta|}e^{\frac{q\langle v\rangle^2}{(1+t)^\vartheta}}$, one has the following estimates:
\begin{itemize}
 \item [i)] For $1\leq|\alpha|\leq n$, one has
\begin{eqnarray}\label{i-weight-gamma-non-R}
  &&\left(\partial^\alpha\left\{{\Gamma}\left(f^{P,\epsilon},f^{m,\epsilon}\right)
  +{\Gamma}\left(f^{m,\epsilon},f^{P,\epsilon}\right)+\sum_{1\leq j_1\leq m-1}\epsilon^{j_1}
  \left\{{\Gamma}\left(f^{m,\epsilon},f^{j_1,\epsilon}\right)+{\Gamma}\left(f^{j_1,\epsilon},f^{m,\epsilon}\right)\right\}\right.\right.\nonumber\\
  &&\left.\left.
  +\sum_{j_1+j_2\geq m,\atop 0<j_1,j_2<m} \epsilon^{j_1+j_2-m}\Gamma\left(f^{j_1,\epsilon},f^{j_2,\epsilon}\right)+\epsilon^m\Gamma(f^{m,\epsilon},f^{m,\epsilon})\right\}, w^2_{\ell_m^*,1}\partial^\alpha f^{m,\epsilon}\right)\\
   &\lesssim&\left\{\mathcal{E}_{f^{P,\epsilon},n,-\frac{\ell_m^*}\gamma,-\gamma}(t)
   +\mathcal{E}_{f^{m,\epsilon},n-1,n-1,-\gamma}(t)+\sum_{1\leq j_1\leq m-1}\mathcal{E}_{f^{j_1,\epsilon},n,-\frac{\ell_m^*}\gamma,-\gamma}(t)\right\}\nonumber\\
  &&\times
  \left\{\sum_{0\leq j\leq |\alpha|}\mathcal{D}^{(j,0)}_{f^{m,\epsilon},j,\ell_m^*,1}(t)+\mathcal{D}_{f^{m,\epsilon},n}(t)\right\}\nonumber\\
  &&+\chi_{|\alpha|=1}\sum_{1\leq j\leq 2}\mathcal{E}^{(j,0)}_{f^{m,\epsilon},j,\ell_m^*,1}(t)\sum_{1\leq j\leq 2}\mathcal{D}^{(j,0)}_{f^{m,\epsilon},j,\ell^*_m,1}(t)\nonumber\\
  &&+\chi_{|\alpha|\geq 2}\sum_{0\leq j_1\leq |\alpha|}\mathcal{E}^{(j_1,0)}_{f^{m,\epsilon},j_1,\ell_m^*,1}(t)
  \times\sum_{0\leq j_2\leq |\alpha|}\mathcal{D}^{(j_2,0)}_{f^{m,\epsilon},j_2,\ell_m^*,1}(t)\nonumber\\
  &&+\left\{\sum_{0\leq j\leq |\alpha|}\mathcal{E}^{(j,0)}_{f^{m,\epsilon},j,\ell_m^*,1}(t)+\mathcal{E}_{f^{m,\epsilon},n}(t)\right\}\nonumber\\
  &&\left\{\mathcal{D}_{f^{P,\epsilon},n,-\frac{\ell_m^*}\gamma,-\gamma}(t)
  +\sum_{1\leq j_1\leq m-1}
  \mathcal{D}_{f^{j_1,\epsilon},n,-\frac{\ell_m^*}\gamma,-\gamma}(t)\right\}\nonumber\\
  &&+\sum_{j_1+j_2\geq m,\atop 0<j_1,j_2<m}\mathcal{E}_{f^{j_1,\epsilon},n,-\frac{\ell_m^*}\gamma,-\gamma}(t)
  \mathcal{D}_{f^{j_2,\epsilon},n,-\frac{\ell_m^*}\gamma,-\gamma}(t)
  +\eta\left\|w_{\ell_m^*,1}\partial^\alpha f^{m,\epsilon}\langle v\rangle^{\frac\gamma2}\right\|^2;\nonumber
     \end{eqnarray}
 \item [iii)] For $1\leq |\alpha|+|\beta|\leq n$ with $|\beta|\geq 1$ or $|\alpha|=|\beta|=0$, one has
\begin{eqnarray}\label{micro-weight-gamma-R}
   &&\left(\partial^\alpha_\beta\left\{{\Gamma}\left(f^{P,\epsilon},f^{m,\epsilon}\right)
  +{\Gamma}\left(f^{m,\epsilon},f^{P,\epsilon}\right)+\sum_{1\leq j_1\leq m-1}\epsilon^{j_1}
  \left\{{\Gamma}\left(f^{m,\epsilon},f^{j_1,\epsilon}\right)+{\Gamma}\left(f^{j_1,\epsilon},f^{m,\epsilon}\right)\right\}\right.\right.\nonumber\\
  &&\left.\left.+\sum_{j_1+j_2\geq m,\atop 0<j_1,j_2<m} \epsilon^{j_1+j_2-m}\Gamma\left(f^{j_1,\epsilon},f^{j_2,\epsilon}\right)+\epsilon^m\Gamma(f^{m,\epsilon},f^{m,\epsilon})\right\}, w^2_{\ell_m^*-|\beta|,1}\partial^\alpha_\beta{\bf \{I-P\}} f^{m,\epsilon}\right)\nonumber\\
  &\lesssim&\left\{\mathcal{E}_{f^{P,\epsilon},n,-\frac{\ell_m^*}\gamma,-\gamma}(t)
   +\mathcal{E}_{f^{m,\epsilon},n-1,n-1,-\gamma}(t)+\sum_{1\leq j_1\leq m-1}\mathcal{E}_{f^{j_1,\epsilon},n,-\frac{\ell_m^*}\gamma,-\gamma}(t)\right\}
  \nonumber\\
  &&\times\left\{\sum_{0\leq k\leq |\beta|,\atop 0\leq j\leq n}\mathcal{D}^{(j,k)}_{f^{m,\epsilon},j,\ell_m^*,1}(t)+\mathcal{D}_{f^{m,\epsilon},n}(t)\right\}\nonumber\\
  &&+\left\{\sum_{0\leq j\leq n}\mathcal{E}^{(j,0)}_{f^{m,\epsilon},j,\ell_m^*,1}(t)+\mathcal{E}_{f^{m,\epsilon},n}(t)\right\}\left\{\mathcal{D}_{f^{P,\epsilon},n,-\frac{\ell_m^*}\gamma,-\gamma}(t)
  +\sum_{1\leq j_1\leq m-1}
  \mathcal{D}_{f^{j_1,\epsilon},n,-\frac{\ell_m^*}\gamma,-\gamma}(t)\right\}\nonumber\\
  &&+\chi_{|\alpha|+|\beta|\leq1}\sum_{0\leq k_1+k_2\leq |\beta|}\left\{\sum_{k_1\leq j\leq 2}\mathcal{E}^{(j,k_1)}_{f^{m,\epsilon},j,\ell_m^*,1}(t)\sum_{k_2\leq j\leq 2}\mathcal{D}^{(j,k_2)}_{f^{m,\epsilon},j,\ell^*_m,1}(t)\right\}\nonumber\\
  &&+\chi_{|\alpha|+|\beta|\geq 2}\sum_{0\leq k_1+k_2\leq |\beta|}\left\{\sum_{k_1\leq j\leq |\alpha|+|\beta|}\mathcal{E}^{(j,k_1)}_{f^{m,\epsilon},j,\ell_m^*,1}(t)\sum_{k_2\leq j\leq |\alpha|+|\beta|}\mathcal{D}^{(j,k_2)}_{f^{m,\epsilon},j,\ell^*_m,1}(t)\right\}\nonumber\\
  &&+\sum_{j_1+j_2\geq m,\atop 0<j_1,j_2<m}\mathcal{E}_{f^{j_1,\epsilon},n,-\frac{\ell_m^*}\gamma,-\gamma}(t)\mathcal{D}_{f^{j_2,\epsilon},n,-\frac{\ell_m^*}\gamma,-\gamma}(t)
+\eta\left\|w_{\ell_m^*-|\beta|,1}\partial^\alpha_\beta{\bf \{I-P\}} f^{m,\epsilon}\langle v\rangle^{\frac\gamma2}\right\|^2.\nonumber
     \end{eqnarray}
\end{itemize}
\end{lemma}

\subsection{Appendix B}

To deduce the desired energy estimates, we have to obtain the corresponding energy estimates with respect to $f^{P,\epsilon}$ and $f^{i,\epsilon}$ for $1\leq i\leq m$.
The first result is concerned with the estimates on $f^{P,\epsilon}$.

To this end, define
\begin{eqnarray}\label{def-Xt-f-P}
 &&\Xi_{\textrm{Total},P}(t)\equiv\sup_{
 0\leq\tau\leq t}\left\{\sum_{0\leq n\leq N_P,\atop 0\leq k\leq n}(1+\tau)^{-\sigma_{n,k}}{\mathcal{E}}^{(n,k)}_{f^{P,\epsilon},N_P,l^*_P,1}(\tau)+\overline{\mathcal{E}}_{f^{P,\epsilon},N_P,l_P,-\gamma}(\tau)\right\}
 \lesssim  M_{P},
\end{eqnarray}
where $M_P$ is a sufficiently small positive constant independent of $\epsilon$.

Under the {\it a priori} estimates \eqref{def-Xt-f-P}, we have
\begin{proposition}\label{Th1.4}
Take
\begin{itemize}
  \item $-3<\gamma<-1$, $\varrho\in [\frac12,\frac32)$, $N_P\geq 5$;
  \item  $\vartheta$ satisfies \[
    0<\vartheta\leq \min\left\{\frac{\gamma-2}{4\gamma-2}\left(\frac52+\varrho\right)-1,\frac\varrho2-\frac14\right\};
    \]
  \item ${\sigma}_{n,0}=0$ for $n\leq N_P$ and
$$
{\sigma}_{n,k}-{\sigma}_{n,k-1}=\frac{2(1+\gamma)}{\gamma-2}(1+\vartheta),  1\leq k\leq n,1\leq n\leq N_P;
$$
 \item there exists $\widehat{l}_P>\frac{N_P+\varrho}2$, $\ell_P\geq N_P,\ l_P\geq \ell_P+\widehat{l}_P$, $\widetilde{\ell}_P\geq\frac\gamma2-\frac{2\gamma{\sigma}_{N_P,N_P}}{1+\varrho}$ and $l_P^*\geq \widetilde{\ell}_{P}-\gamma l_{P}$;
 \item $F^{P,0}(x,v)=\mu+\sqrt{\mu}f_{P,0}(x,v)\geq0$.
\end{itemize}

If we assume further that
\begin{equation}\label{def-VPB-S-Y_P0}
\begin{split}
\mathbb{Y}_{P,0}=&\sum_{|\alpha|+|\beta|\leq N_P}\left\|\langle v\rangle^{l_P^*-|\beta|}e^{q\langle v\rangle^2}\partial^\alpha_\beta f_0^{P,\epsilon}\right\|+\left\|\left[f_0^{P,\epsilon},E_0^{P,\epsilon}\right]\right\|_{H^{-\varrho}_x}
\end{split}
\end{equation}
is bounded from above by some sufficiently small positive constant independent of $\epsilon$, the Cauchy problem (\ref{f-P-sign}), (\ref{VPB-sign-IC}) admits a unique global solution $\left[f^{P,\epsilon}(t,x,v),E^{P,\epsilon}(t,x)\right]$ satisfying $F^P(t,x,v)=\mu+\sqrt{\mu}f^{P,\epsilon}(t,x,v)\geq0$. Furthermore, one also has
\begin{equation}\label{VPB-S-f-p}
  \frac{d}{dt}\mathcal{E}_{f^{P,\epsilon},N_P,l_P,-\gamma}(t)+\mathcal{D}_{f^{P,\epsilon},N_P,l_P,-\gamma}(t)\lesssim 0,
\end{equation}

\begin{equation}\label{VPB-S-k-estimate}
  \frac{d}{dt}\mathcal{E}^{k}_{f^{P,\epsilon},N_P}(t)+\mathcal{D}^{k}_{f^{P,\epsilon},N_P}(t)\lesssim 0,\ k=0,1,\cdots,N_P-1,
\end{equation}

\begin{equation}\label{VPB-S-k-estimate-w}
  \frac{d}{dt}\mathcal{E}^{k}_{f^{P,\epsilon},N_P,\ell_P,-\gamma}(t)+\mathcal{D}^{k}_{f^{P,\epsilon},N_P,\ell_P,-\gamma}(t)\lesssim 0,\ k=0,1,\cdots,N_P-2,
\end{equation}
and
\begin{equation}\label{VPB-S-th-decay}
    (1+t)^{\varrho+k}\mathcal{E}^k_{f^{P,\epsilon},N_P}(t)\lesssim \mathbb{Y}_{P,0}^2,\ k=0,1,\cdots,N_P-1,
  \end{equation}
  \begin{equation}\label{VPB-S-th-decay-w}
    (1+t)^{\varrho+k}\mathcal{E}^k_{f^{P,\epsilon},N_P,\ell_P,-\gamma}(t)\lesssim \mathbb{Y}_{P,0}^2,\ k=0,1,\cdots,N_P-2,
  \end{equation}
    \begin{equation}\label{VPB-S-th-decay-w-1}
   (1+t)^{-\sigma_{n,k}}{\mathcal{E}}^{(n,k)}_{f^{P,\epsilon},N_P,l^*_P,1}(t)\lesssim \mathbb{Y}_{P,0}^2,\ k=0,1,\cdots,n.
  \end{equation}
where $0\leq n\leq N_P$.
\end{proposition}
\begin{proof}
In fact, $\left[F^{P,\epsilon}(t,x,v), E^{P,\epsilon}(t,x)\right]$ satisfy the  Vlasov-Poisson-Boltzmann system with a given magnetic filed, this proposition can be obtained by a similar argument as \cite{Duan-Lei-Yang-Zhao-CMP-2017, Lei-Zhao-JFA-2014}. We omit its proof for simplicity.
\end{proof}
Based on the above proposition, for $i=1,2,\cdots, m-1$, we can obtain the corresponding results on the global solvability of the Cauchy problem \eqref{f-i-vector}-\eqref{f-i-e-b-compatibility conditions} together with some estimates on its global solution $\left[f^{i,\epsilon}(t,x,v), E^{i,\epsilon}(t,x)\right]$ by the method of mathematical induction.

In fact, if we define
\begin{eqnarray}\label{def-Xt-f-i}
 &&\Xi_{\textrm{Total},i}(t)\nonumber\\
 &\equiv&\sup_{
 0\leq\tau\leq t}\left\{\sum_{0\leq n\leq N_i,\atop 0\leq k\leq n}(1+\tau)^{-\sigma_{n,k}}\mathcal{E}^{(n,k)}_{f^{i,\epsilon},n,l_i^*,1}(\tau)
 +\overline{\mathcal{E}}_{f^{i,\epsilon},N_i,l_i,-\gamma}(\tau)\right.\nonumber\\
&& \left.\quad\quad+\chi_{i\geq 2}\sum_{0<j_1<i}\sum_{0\leq n\leq N_{j_1},\atop 0\leq k\leq n}(1+\tau)^{-\sigma_{n,k}}{\mathcal{E}}^{(n,k)}_{f^{j_1,\epsilon},N_{j_1},l^*_{j_1},1}(\tau)+\chi_{i\geq 2}\sum_{0<j_1<i}\overline{\mathcal{E}}_{f^{j_1,\epsilon},N_{j_1},l_{j_1},-\gamma}(\tau)\right.\nonumber\\
 &&\left.\quad\quad
 +\sum_{0\leq n\leq N_P,\atop 0\leq k\leq n}(1+\tau)^{-\sigma_{n,k}}{\mathcal{E}}^{(n,k)}_{f^{P,\epsilon},N_P,l^*_P,1}(\tau)+\overline{\mathcal{E}}_{f^{P,\epsilon},N_P,l_P,-\gamma}(\tau)\right\}
 \lesssim  M_{i},
\end{eqnarray}
where $M_i$ is a sufficiently small positive constant independent of $\epsilon$, we can get by repeating the arguments used in deducing Lemma \ref{Lemma1-R}, Lemma \ref{lemma2-R}, Lemma \ref{lemma9}, and by the method of mathematical induction that
\begin{proposition}\label{lemma-fi-end}
Assume
\begin{itemize}
\item $\gamma\in(-3,-1)$, $\varrho\in\left[1,\frac32\right)$;
\item  $\vartheta$ satisfies \[
    0<\vartheta\leq \min\left\{\frac{\gamma-2}{4\gamma-2}\left(\frac52+\varrho\right)-1,\frac\varrho2-\frac14\right\};
    \]
\item $N_P,N_i$ and ${\sigma}_{n,k}$ can be taken as
\begin{itemize}
\item $N_i\geq 5$, $N_{j}\geq N_{j+1}+1$ for $1\leq j\leq i-1$, $N_P\geq N_1+1$;
\item ${\sigma}_{n,0}=0$ for $n\leq N_i$ and
$$
{\sigma}_{n,k}-{\sigma}_{n,k-1}=\frac{2(1+\gamma)}{\gamma-2}(1+\vartheta),  1\leq k\leq n,1\leq n\leq N_i;
$$\end{itemize}
\item
\begin{itemize}
 \item there exists $\widehat{l}_i>\frac{N_P+\varrho}2$, $ \widehat{l}_{j-1}>\widehat{l}_{j}$ for $1<j\leq i$, $\widehat{l}_P>\widehat{l}_1$
        where $\widehat{l}_P$ satisfying the assumptions in Proposition \ref{Th1.4};
\item $\widetilde{\ell}_i\geq\frac\gamma2-\frac{2\gamma{\sigma}_{N_i,N_i}}{1+\varrho}$, $\ell_i\geq N_i$, $l_i\geq \ell_i+\widehat{l}_i$ and $l_i^*\geq \widetilde{\ell}_{i}-\gamma l_{i}$;
  \item $\widetilde{\ell}_{j}\geq\frac\gamma2-\frac{2\gamma{\sigma}_{N_{j},N_{j}}}{1+\varrho}$, $\ell_j\geq-\frac{l_{j+1}^*}\gamma+\frac12-\frac 2{\gamma}$, $l_{j}\geq\ell_j+\widehat{l}_j$  and $l_j^*\geq \widetilde{\ell}_{j}-\gamma l_{j}$ for $1\leq j\leq i-1$;
 \item $\widetilde{\ell}_{P}\geq\frac\gamma2-\frac{2\gamma{\sigma}_{N_{P},N_{P}}}{1+\varrho}$,
     $\ell_P\geq -\frac{l_{1}^*}\gamma+\frac12-\frac 2{\gamma}$, $l_{P}\geq\ell_P+\widehat{l}_P$  and $l_P^*\geq \widetilde{\ell}_{P}-\gamma l_{P}$.
        \end{itemize}
\end{itemize}
If we assume further that
\begin{eqnarray}\label{condition-i}
 \mathbb{Y}_{i,0}&\equiv& \sum_{|\alpha|+|\beta|\leq N_i}\left\|\langle v\rangle^{l_i^*-|\beta|}e^{q\langle v\rangle^2}\partial^\alpha_\beta f_0^{i,\epsilon}\right\|+\left\|E_0^{i,\epsilon}\right\|_{H^{N_i}_x} +\left\|\Lambda^{-\varrho}\left[f_0^{i,\epsilon},E_0^{i,\epsilon}\right]\right\|\\
 &&+\sum_{0<j_1<i}\left\{\sum_{|\alpha|+|\beta|\leq N_{j_1}}\left\|\langle v\rangle^{l_{j_1}^*-|\beta|}e^{q\langle v\rangle^2}\partial^\alpha_\beta f_0^{j_1,\epsilon}\right\|+\left\|E_0{j_1,\epsilon}\right\|_{H^{N_{j_1}}_x} +\left\|\Lambda^{-\varrho}\left[f_0^{j_1,\epsilon},E_0^{j_1,\epsilon}\right]\right\|\right\}\nonumber\\
 &&+\sum_{|\alpha|+|\beta|\leq N_P}\left\|\langle v\rangle^{l_P^*-|\beta|}e^{q\langle v\rangle^2}\partial^\alpha_\beta f_0^{P,\epsilon}\right\| +\left\|E_0^{P,\epsilon}\right\|_{H^{N_P}_x} +\left\|\Lambda^{-\varrho}\left[f_0^{P,\epsilon},E_0^{P,\epsilon}\right]\right\|\nonumber
\end{eqnarray}
is bounded from above by some sufficiently small positive constant independent of $\epsilon$, then the Cauchy problem \eqref{f-i-vector}-\eqref{f-i-e-b-compatibility conditions} admits a unique global solution $\left[f^{i,\epsilon}(t,x,v), E^{i,\epsilon}(t,x)\right]$ which satisfies
   \begin{eqnarray}
\Xi_{\textrm{Total},i}(t)
 \lesssim  \mathbb{Y}_{i,0}^2.
\end{eqnarray}

Moreover, one also has
\begin{equation}\label{f-i-end-00}
\frac{d}{dt}\overline{\mathcal{E}}_{f^{i,\epsilon},N_{i},l_{i},-\gamma}(t)+\overline{\mathcal{D}}_{f^{i,\epsilon},N_{i},l_{i},-\gamma}(t)\lesssim 0,\quad\quad 0<i<m
\end{equation}
 and for $k=0,1,2,\cdots, N_i-1$, one has
\begin{eqnarray}\label{Lemma1-1}
&&\frac{d}{dt}\left\{\mathcal{E}^{k}_{f^{i,\epsilon},N_i}(t)+\mathcal{E}^{k}_{f^{P,\epsilon},N_P}(t)+\chi_{i\geq2}\sum_{1\leq j\leq i-1}\mathcal{E}^{k}_{f^{j,\epsilon},N_{j}}(t)\right\}\nonumber\\
&&+\mathcal{D}^{k}_{f^{i,\epsilon},N_i}(t)+\mathcal{D}^{k}_{f^{P,\epsilon},N_P}(t)+\chi_{i\geq2}\sum_{1\leq j\leq i-1}\mathcal{D}^{k}_{f^{j,\epsilon},N_j}(t)\leq 0,\quad 0\leq t\leq T.
\end{eqnarray}

As a consequence of \eqref{Lemma1-1}, we can get that
\begin{eqnarray}\label{lemma-decay-f-i}
(1+t)^{k+\varrho}\left\{\mathcal{E}^{k}_{f^{i,\epsilon},N_i}(t)+\mathcal{E}^{k}_{f^{P,\epsilon},N_P}(t)+\chi_{i\geq2}\sum_{1\leq j\leq i-1}\mathcal{E}^{k}_{f^{j,\epsilon},N_{j}}(t)\right\}
\lesssim \mathbb{Y}^2_{i,0},
\end{eqnarray}
 and for $k=0,1,2,\cdots, N_i-2$, then one has
\begin{eqnarray}\label{Lemma1-1-w}
&&\frac{d}{dt}\left\{\mathcal{E}^{k}_{f^{i,\epsilon},N_i,\ell_i,-\gamma}(t)+\mathcal{E}^{k}_{f^{P,\epsilon},N_P,\ell_i,-\gamma}(t)
+\chi_{i\geq2}\sum_{1\leq j\leq i-1}\mathcal{E}^{k}_{f^{j,\epsilon},N_{j},\ell_i,-\gamma}(t)\right\}\nonumber\\
&&+\mathcal{D}^{k}_{f^{i,\epsilon},N_i,\ell_i,-\gamma}(t)+\mathcal{D}^{k}_{f^{P,\epsilon},N_P,\ell_i,-\gamma}(t)
+\chi_{i\geq2}\sum_{1\leq j\leq i-1}\mathcal{D}^{k}_{f^{j,\epsilon},N_j,\ell_i,-\gamma}(t)\leq 0,\quad 0\leq t\leq T,
\end{eqnarray}
as a consequence of \eqref{Lemma1-1-w}, we can get that
\begin{equation}\label{lemma3-1-R}
 \begin{split}
(1+t)^{k+\varrho}\left\{\mathcal{E}^k_{f^{i,\epsilon},N_i,{\ell_{i}},-\gamma}(t)+\sum_{0<j<i}\mathcal{E}^k_{f^{j,\epsilon},N_j,{\ell_{i}},-\gamma}(t)
+\mathcal{E}^k_{f^{P,\epsilon},N_P,\ell_{i},-\gamma}(t)\right\}
\lesssim& \mathbb{Y}^2_{i,0}.
\end{split}
\end{equation}
Furthermore, for $i=1,2,\cdots, m-1$, one can get that
\begin{eqnarray}\label{time-increasing-i}
 &&\sum_{0\leq n\leq N_i,\atop 0\leq k\leq n}(1+t)^{-\sigma_{n,k}}\mathcal{E}^{(n,k)}_{f^{i,\epsilon},n,l_i^*,1}(t)+\chi_{i\geq 2}\sum_{0<j_1<i}\sum_{0\leq n\leq N_{j_1},\atop 0\leq k\leq n}(1+t)^{-\sigma_{n,k}}{\mathcal{E}}^{(n,k)}_{f^{j_1,\epsilon},N_{j_1},l^*_{j_1},1}(t)\nonumber\\
&&\quad\quad
 +\sum_{0\leq n\leq N_P,\atop 0\leq k\leq n}(1+t)^{-\sigma_{n,k}}{\mathcal{E}}^{(n,k)}_{f^{P,\epsilon},N_P,l^*_P,1}(t)
 \lesssim  \mathbb{Y}^2_{i,0}.
\end{eqnarray}
\end{proposition}

\begin{remark} For the Cauchy problem (\ref{f-P-sign}), (\ref{VPB-sign-IC}) and the Cauchy problem \eqref{f-i-vector}-\eqref{f-i-e-b-compatibility conditions} for the whole range of cutoff intermolecular interactions, since both equations under our considerations are Vlasov-Poisson-Boltzmann type equations with given constant magnetic field, the arguments used to deduce Lemma \ref{Lemma1-R}, Lemma \ref{lemma2-R}, Lemma \ref{lemma9} can be adopted to yield the desired results listed in Propositions \ref{Th1.4} and \ref{lemma-fi-end} by employing two sets of weighted energy estimates with the weights $w_{\ell-|\beta|,-\gamma}(t,v)$ and $w_{\ell-|\beta|,1}(t,v)$ respectively. Although the estimates obtained in Propositions \ref{Th1.4} and \ref{lemma-fi-end} are sufficient to prove Theorem \ref{Th1.3}, these results can indeed be improved.

In fact, for the moderately soft potentials (i.e. for the case of $-2\leq\gamma<0$), by employing the argument developed in \cite{Duan_Yang_Zhao-M3AS-2013}, we can use a single set of weighted energy method with weights $w_{\ell-|\beta|,-\gamma}(t,v)$ to deduce similar global solvability results for both the Cauchy problem (\ref{f-P-sign}), (\ref{VPB-sign-IC}) and the Cauchy problem \eqref{f-i-vector}-\eqref{f-i-e-b-compatibility conditions} and to deduce nice temporal decay estimates on $\left[f^{P,\epsilon}(t,x,v), E^{P,\epsilon}(t,x)\right]$ and $\left[f^{i,\epsilon}(t,x,v), E^{i,\epsilon}(t,x)\right]$ for $i=1,2,\cdots, m-1$ which are sufficient to prove Theorem \ref{Th1.3}, provided that the initial data are assumed
to satisfy the neutral conditions. 

For the very soft potentials (i.e. for the case of $-3<\gamma<-2$), by employing the argument developed in \cite{Xiao-Xiong-Zhao-JFA-2016}, we can still use a single set of weighted energy method with weights $w_{\ell-|\beta|,-\gamma}(t,v)$ to get similar global solvability results for both the Cauchy problem (\ref{f-P-sign}), (\ref{VPB-sign-IC}) and the Cauchy problem \eqref{f-i-vector}-\eqref{f-i-e-b-compatibility conditions} and to deduce the desired temporal decay estimates. But, as in \cite{Xiao-Xiong-Zhao-JFA-2016}, we had to deduce the almost optimal decay estimates on $\left[f^{P,\epsilon}(t,x,v), E^{P,\epsilon}(t,x)\right]$ and $\left[f^{i,\epsilon}(t,x,v), E^{i,\epsilon}(t,x)\right]$ for $i=1,2,\cdots, m-1$ themselves, certain orders of derivatives with respect the space and velocity variables, and even the macroscopic component and microscopic component of them. Since such an improvement will lead to quite complex analysis, we thus state the somewhat weaker results in Propositions \ref{Th1.4} and \ref{lemma-fi-end} for simplicity of presentation.
\end{remark}




\bigskip
\noindent {\bf Acknowledgements:}
Ning Jiang was supported by two grants from National Natural Science Foundation of China under contracts 11731008 and 11971360, respectively. Yuanjie Lei was supported by three grants from National Natural Science Foundation of China under contracts 11601169, 11871335, and 11971187, respectively. Huijiang Zhao was supported by two grants from the National Natural Science Foundation of China under contracts 11731008 and 11671309, respectively.

\end{document}